\documentclass[11pt]{article}

\usepackage[utf8, applemac]{inputenc}
\usepackage{fontenc}
\usepackage[italian,english]{babel}
\usepackage{amsmath}
\usepackage{amsthm}
\usepackage{mathrsfs}
\usepackage{mathtools}
\usepackage{amsfonts}
\usepackage{graphicx}
\usepackage{setspace}
\usepackage{physics}
\usepackage{comment}
\usepackage{amsfonts}
\usepackage{layout}
\usepackage{booktabs,array}
\usepackage{epsfig}	
\usepackage{dsfont}	
\usepackage{slashed}
\usepackage{amssymb}
\usepackage{tensor}
\usepackage{fancyhdr} 
\usepackage{tikz-cd}
\usepackage{booktabs}
\usepackage{colonequals}
\usepackage{tabularx}
\usepackage{tikz}
\usepackage{csquotes}
\usepackage{subcaption}
\usepackage[export]{adjustbox}

\usepackage{jones}
\usepackage{xcolor}

\usetikzlibrary{calc,fit,backgrounds,positioning}

\newcommand{\eqinf}{\overset{\infty}{=}}

\newcommand{\PM}{\mathcal{M}_{\textnormal{perf}}}
\newcommand{\EE}{\mathbb{E}}
\newcommand{\into}{\hookrightarrow}

\newcommand{\loc}{\mathrm{loc}}
\newcommand{\nonloc}{\mathrm{nonloc}}

\newcommand{\samp}{\mathrm{samp}}
\newcommand{\Unif}{\mathrm{Unif}}

\newcommand{\Loc}{\mathrm{Loc}}

\DeclareMathOperator{\CC}{CC}
\DeclareMathOperator{\GCC}{GCC}

\DeclareMathOperator{\type}{type}
\DeclareMathOperator{\cactus}{cactus}

\newcommand{\cyc}{\mathrm{cyc}}

\newcommand{\rt}{\mathrm{root}}
\newcommand{\NC}{\mathrm{NC}}

\newcommand{\core}{\textnormal{core}}
\newcommand{\roott}{\textnormal{root}}
\newcommand{\conn}{\mathrm{conn}}

\DeclareMathOperator{\cycles}{cycles}

\newcommand{\rrom}{\textsf{\upshape r-ROM}\xspace}
\newcommand{\rom}{\textsf{\upshape ROM}\xspace}
\newcommand{\goe}{\textsf{\upshape GOE}\xspace}
\newcommand{\GOE}\goe
\newcommand{\bGOE}{\textsf{\upshape BlockGOE}\xspace}

\DeclareMathOperator{\block}{block}

\title{Universality of first-order methods on random and \\ deterministic matrices}
\author{Nicola Gorini\thanks{Bocconi University. \texttt{nicola.gorini@phd.unibocconi.it}}
\and Chris Jones\thanks{UC Davis. \texttt{chijones@ucdavis.edu}}
\and Dmitriy Kunisky\thanks{Johns Hopkins University. \texttt{kunisky@jhu.edu}}
\and Lucas Pesenti\thanks{ETH Z{\"u}rich. \texttt{lpesenti@ethz.ch}}}
\date{\today}

\begin{document}

\maketitle
\thispagestyle{empty}

\begin{abstract}
    General first-order methods (GFOM) are a flexible class of iterative algorithms which update a state vector by matrix-vector multiplications and entrywise nonlinearities.
    A long line of work has sought to understand the large-$n$ dynamics of GFOM, mostly focusing on ``very random'' input matrices and the approximate message passing (AMP) special case of GFOM whose state is asymptotically Gaussian.
    Yet, it has long remained unknown how to construct iterative algorithms that retain this Gaussianity for more structured inputs, or why existing AMP algorithms can be as effective for some deterministic matrices as they are for random matrices.
    
    We analyze diagrammatic expansions of GFOM via the limiting \emph{traffic distribution} of the input matrix, the collection of all limiting values of permutation-invariant polynomials in the matrix entries, to obtain the following results:
    \begin{enumerate}
    \item We calculate the traffic distribution for the first non-trivial deterministic matrices,
    including (minor variants of) the Walsh--Hadamard and discrete sine and cosine transform matrices.
    This determines the limiting dynamics of GFOM on these inputs, resolving parts of longstanding conjectures of Marinari, Parisi, and Ritort (1994).
    \item We design a new AMP iteration which unifies several previous AMP variants and generalizes to new input types,
    whose limiting dynamics are Gaussian conditional on some latent random variables.
    The asymptotic dynamics hold for a large and natural class of traffic distributions (encompassing both random and deterministic input matrices) and the algorithm's analysis gives a simple combinatorial interpretation of the Onsager correction, answering questions posed recently by Wang, Zhong, and Fan (2022).
    \end{enumerate}
\end{abstract}

\clearpage

\setcounter{tocdepth}{2}
\pagestyle{empty}
\tableofcontents

\clearpage

\pagestyle{plain}
\pagenumbering{arabic}

\section{Introduction}
\label{sec:intro}

Complex systems with a large number of simply interacting pieces underlie many natural processes and, more recently,
have been studied in computer science in an effort to make sense of how simple machine learning algorithms can learn complex structures latent in large, semi-random input data.
Iterative optimization algorithms making sequential updates can be viewed as dynamical systems,
with the main task being to understand how the algorithm evolves over time and what properties the eventual output will have.

When the size of these systems grows very large, a key insight from statistical physics is that 
the macroscopic properties of the system can simplify dramatically:

\begin{quote}
    \textit{As the size of a random, smoothly-interacting dynamical system grows, the effect of individual particles averages out, and the dynamical system's trajectory approximately follows an asymptotic distributional equation.}
\end{quote}

We refer to these distributional equations as \emph{(asymptotic) effective
dynamics}.
We seek to prove this kind of theorem for discrete-time nonlinear iterative algorithms such as those used in modern optimization, statistics, and machine learning.
Concretely, we study \emph{general first-order methods} (GFOM) \cite{celentano2020estimation, montanari2022statistically} which take as input a symmetric matrix $\bA \in \R^{n \times n}$, maintain a vector state $\bmx \in \R^n$, and at each step can perform one of two possible operations:
\begin{enumerate}[1.]
    \item either multiply the state by $\bA$:
        \[ \bmx_{t+1} = \bA \bmx_t\,, \]
    \item or apply a function $f_t : \R^{t+1} \to \R$ componentwise to the previous states:
        \[ \bmx_{t+1} = f_t(\bmx_t, \dots, \bmx_0), \,\,\,\, \text{i.e.,} \,\,\,\, \bmx_{t + 1}[i] = f_t(\bmx_{t}[i], \dots, \bmx_{0}[i]) \text{ for each } i \in [n]\,. \]
\end{enumerate}
The initial state will be either the deterministic all-ones vector $\bmx_0 = \mathbf{1}$, or a random Gaussian vector $\bmx_0 \sim \calN(\mathbf{0}, \bI)$ independent of $\bA$.
Without loss of generality, we may assume that these operations alternate, giving an iteration of the form
\[ \bx_{t + 1} = \bA f_t(\bx_t, \dots, \bx_0)\,. \]
We fix some number of iterations $t$ and view $\bx_t = \bx_t(\bA)$ as the output of the algorithm.

GFOM is a flexible computational model which is expressive enough to capture many types of gradient descent~\cite{celentano2020estimation, gerbelot2022rigorous} and message passing algorithms~\cite{feng2022unifying}.
It may be viewed as a nonlinear version of the power method for estimating top eigenvectors. The alternation of linear and nonlinear steps also closely matches the structure of a feedforward neural network~\cite{cirone2024graph}.
One may view the structural restriction on GFOM as forcing $\bm x_t$ viewed as a function of $\bm A$ to be \emph{permutation-equivariant}: if we apply the same permutation to the rows and columns of $\bm A$, then $\bm x_t$ undergoes the same permutation, a natural condition of an algorithm's not depending on the particular indexing of its inputs.

GFOM and their special case of \emph{approximate message passing} (AMP) are very popular algorithms for many statistical inference tasks and are known to perform optimally in various such settings~\cite{donoho2009message,rangan2011generalized,montanari2012graphical,rangan2016inference,bayati2011dynamics,feng2022unifying}.
In these cases, an algorithm takes as input not an arbitrary matrix $\bA$, but one that contains a corrupted observation of a signal (in a common example, the input $\bA$ is a low-rank $\by\by^{\top}$ plus independent random noise).

GFOM have also been used as optimization algorithms in average-case settings without any such planted structures.
For instance, they are the best known algorithms for optimizing quadratic forms with random coefficients over the non-negative orthant \cite{MR-2015-NonNegative} (the \emph{non-negative PCA} objective function), other convex cones \cite{DMR-2014-ConePCA}, and the hypercube \cite{montanari2021optimization} (the \emph{Sherrington--Kirkpatrick Hamiltonian}), all of which are NP-hard problems in the worst case.
This situation is the main target of our analysis. We receive an input matrix $\bA$ without any particular ``signal'' and wish to output $\bx$ approximately solving an optimization problem parametrized by $\bA$, such as 
\begin{equation}
\label{eq:quad-opt}
\begin{array}{ll} \text{maximize} & \langle \bx, \bA\bx\rangle \\ \text{subject to} & \bx \in S \end{array}
\end{equation}
studied in the above references for various choices of the constraint set $S \subseteq \R^n$.

To view GFOM as an instance of the physical setting sketched above, we consider a growing sequence of matrices $\bA = \bA^{(n)} \in \R^{n \times n}$, and think of the ``particles'' as being the coordinates $\bmx_{t}[i]$ of $\bm x_t \in \R^n$.
To keep notation reasonable, while all of these objects depend on $n$, we omit the $(n)$ superscript whenever possible.
We analyze the \emph{empirical distribution} of our particles, accessed by sampling a random coordinate of a vector:
\[ \samp(\bm x) \colonequals \bmx[i] \in \R \text{ for } i \sim \Unif([n])\,. \]
In order to study a particle's entire trajectory more generally, we may ``stack'' several vectors
and define $\samp((\bx_0, \dots, \bx_t)) := \samp(\bx_0, \dots, \bx_t) \colonequals (\bmx_0[i],\dots, \bmx_t[i]) \in \R^{t + 1}$ for $i \sim \Unif([n])$.

The analysis of GFOM hinges on the observation that these random variables often converge in distribution to certain limiting distributions.
That is, for suitably nice test functions $\varphi: \R^{t + 1} \to \R$,
\[ \lim_{n \to \infty} \E \varphi\big(\samp(\bx_0^{(n)}, \dots, \bm x_t^{(n)})\big) = \lim_{n \to \infty} \E \frac{1}{n} \sum_{i = 1}^n \varphi(\bmx_{0}^{(n)}[i], \dots, \bmx_{t}^{(n)}[i]) = \int \varphi \,\d\nu_{\leq t}^{\infty}\,, \]
for some probability measures $\nu_{\leq t}^{\infty}$.
For example, we can analyze the objective function of a problem like~\cref{eq:quad-opt} in this way: given a GFOM to run for $t$ iterations producing $\bx_t = \bx_t(\bA)$, we extend it to $\bx_{t + 1} = \bA \bx_t$ so that
\[ \E \frac{1}{n}\langle \bx_t, \bA \bx_t\rangle = \EE \frac{1}{n}\langle \bx_t, \bx_{t + 1}\rangle = \EE \frac{1}{n}\sum_{i = 1}^n \bmx_{t}[i] \bmx_{t + 1}[i]\,, \]
a quantity accessible in the above formalism by a suitable choice of $\varphi$.
We can also study the algorithm's convergence by expanding $\frac{1}{n}\|\bx_t - \bx_{t - 1}\|^2_2$ in the same way.

The goal of an asymptotic effective dynamics is then to identify the asymptotic measures $\nu_{\leq t}^{\infty}$\,.
Such a description is a natural first step to designing {\em optimal} GFOM for optimization problems: given an explicit description of the limiting performance of any GFOM, we then optimize the performance over all GFOM~\cite{celentano2020estimation, AMS20:pSpinGlasses, montanari2022equivalence,pesentiThesis}.

The goal of this paper is to study the following three questions regarding effective dynamics:
\begin{enumerate}[1.]
    \item \textbf{Existence:} What are minimal assumptions on the input matrices and the algorithm that ensure the existence of asymptotic effective dynamics?
    \item \textbf{Universality:} What properties of the sequence of input matrices $\bA^{(n)}$ determine the asymptotic effective dynamics?
    In particular, how can we show that two sequences of $\bA^{(n)}$ share the same dynamics?
    \item \textbf{Explicit Calculation:} What are the effective dynamics? In particular, for a given algorithm, how can one describe $\nu_{\leq t}^{\infty}$ for each fixed $t \in \mathbb{N}$?
\end{enumerate}

\subsection{Approximate message passing and simple effective dynamics}

The majority of results to date on effective dynamics for GFOM, including ours, are most useful for \emph{Approximate Message Passing}~(AMP) algorithms. Originating from physicists' work on mean-field spin glass models \cite{mezard1987spinglasstheoryandbeyond,donoho2009message},
AMP algorithms are a special case of GFOM with very simple effective dynamics: each distribution $\nu_t^{\infty}$ (the marginal distribution of $\nu_{\leq t}^{\infty}$ above on the last coordinate) is a Gaussian distribution,
\[ \nu_t^{\infty} = \mathcal{N}(\mu_t, \sigma_t^2)\,, \]
and the effective dynamics gives $(\mu_{t+1}, \sigma_{t+1}^2)$ in terms of $(\mu_{t}, \sigma_{t}^2), \dots, (\mu_{0}, \sigma_0^2)$ via a formula known as the \emph{state evolution} equation.
This gives a simple yet complete description of the leading-order behavior of an algorithm as $n \to \infty$.
In part due to the power afforded by such a description, AMP (and the closely related \emph{belief propagation}, of which AMP is a limit in a suitable sense)
has taken on an indispensable role in statistical physics \cite{mezard1987spinglasstheoryandbeyond, MezardMontanari, charbonneau2023spin} and, more recently, in computational statistics \cite{zdeborova2016statistical, feng2022unifying}.

In fact, while the original appearances of AMP in statistical physics were intrinsically motivated, for statistics applications the simplicity of state evolution is so useful that a line of work has emerged trying to \emph{design} GFOM that have Gaussian $\nu_t^{\infty}$ and effective dynamics given by state evolution~\cite{javanmard2013state,barbierSpatial,vila2015adaptive,fan2022approximate,zhong2024approximate, lovig2025universality}.
The term ``AMP'' is now often used to describe any choice of GFOM for a given family of inputs $\bA^{(n)}$ that has these properties.
While it is not clear that this should be the case \emph{a priori}, a common fortuitous coincidence is that, for various problems, the best GFOM algorithms (in the sense of achieving optimal rates in estimation or inference tasks) happen to be in the special class of AMP.
That is, in many cases, the GFOM with the simplest asymptotic effective dynamics are also the most useful in applications.

Given the successes of AMP, it is a longstanding goal in the literature to identify AMP-like algorithms for as many different choices of inputs and input distributions as possible.
Yet, even to go slightly beyond the simplest choices of matrices $\bA^{(n)}$ has proved challenging and subtle (e.g., random matrices with i.i.d.\ entries~\cite{javanmard2013state, bayati2015universality},
orthogonally invariant distributions~\cite{fan2022approximate}, or
semi-random ensembles~\cite{dudeja2023universality,wang2022universality}).
Constructing AMP algorithms in such settings involves carefully inserting so-called \emph{Onsager correction terms} into the nonlinearities $f_t$ in ways that remain somewhat mysterious yet are crucial to obtain Gaussian limiting behavior.

Here, we will present an approach to the analysis of GFOM that re-derives different existing variants of AMP in a unified way, derives AMP algorithms for new inputs (both random and deterministic), and offers new conceptual insights into the design of these algorithms and into the proof of their asymptotic effective dynamics, in particular giving a clear combinatorial explanation for the Onsager corrections mentioned above.

\subsection{Our contributions: Combinatorial method for GFOM}
\label{sec:results}

We study GFOM
by expressing them as vectors of polynomials in the entries of the input matrix.
For this reason we focus on polynomial $f_t$; it is likely possible to treat more general nonlinearities by approximating them by polynomials (see~\Cref{sec:related} for some discussion).
\begin{definition}
    \label{def:pgfom}
    We call a GFOM as described above a \emph{polynomial GFOM} (pGFOM) if all nonlinearities $f_t: \R^{t + 1} \to \R$ are polynomials.
\end{definition}

Our approach is divided into two parts.
The first is a ``static'' analysis of certain symmetric polynomials in the entries of the input $\bA$.
The second translates this to ``dynamic'' information about vector-valued functions, allowing us to calculate effective dynamics for $O(1)$ iterations of GFOM in a general way.

\subsubsection{Statics of graph polynomials: Traffic distributions and universality}

The basic objects of study for our static analysis are the following \emph{graph polynomials}.
\begin{definition}[Diagram classes]
    We write $\calA = \calA_0$ for the set of finite, undirected, connected (multi)graphs.
    We also write $\calE = \calE_0 \subseteq \calA_0$ for the set of 2-edge-connected (multi)graphs (ones that cannot be disconnected by removing any single edge) and $\calC = \calC_0 \subseteq \calE_0 \subseteq \calA_0$ for the set of \emph{cactus} graphs, ones where every edge belongs to exactly one simple cycle.\footnote{This notion is sometimes more specifically called a \emph{bridgeless cactus}; in this paper we take this to be part of the definition of a cactus.} See~\Cref{fig:cactus}.
\end{definition}
\noindent
The optional subscript ``0'' of the diagram classes refers to the outputs of the polynomials being 0-dimensional, i.e., scalars, which will be useful to distinguish them from vector- and matrix-valued polynomials to be defined later (with subscript ``1'' and ``2'', respectively).

\begin{figure}[ht]
    \centering
    \includegraphics[width=0.4\linewidth]{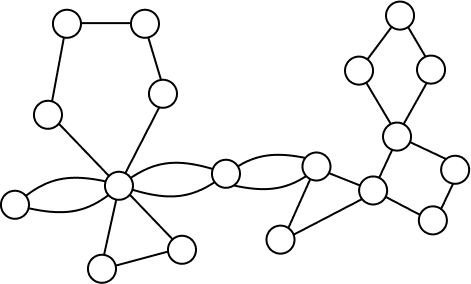}
    \caption{A cactus graph in $\cC$. Intuitively, a cactus is a ``tree of cycles''.}
    \label{fig:cactus}
\end{figure}

\begin{definition}[Scalar graph polynomials]\label{def:w}
    Given $\al \in \calA$ and $\bA \in \R^{n \times n}_{\sym}$\,, define polynomials $w_\al(\bA), z_\al(\bA) \in \R[\bA]$ by:
    \begin{align*}
        w_\al(\bA) &= \sum_{i: V(\al) \to [n]} \prod_{\{u,v\} \in E(\al)} \bA[i(u),i(v)]\,,\\
        z_\al(\bA) &= \sum_{i: V(\al) \embeds [n]} \prod_{\{u,v\} \in E(\al)} \bA[i(u),i(v)]\,.
    \end{align*}
\end{definition}

\noindent That is, $w_\al(\bA)$ and $z_\al(\bA)$ are each multivariate polynomials in the $\frac{n(n + 1)}{2}$ entries on and above the diagonal of the matrix $\bA$
obtained by summing over all labelings of the vertices of $\al$ by $[n] = \{1,2,\dots, n\}$ and with each edge corresponding to an entry of $\bA$.
The only difference between $w_\al(\bA)$ and $z_\al(\bA)$ is that the vertex labeling for $z_\al(\bA)$ is restricted to be injective by the notation $i : V(\al) \embeds [n]$ whereas labels in $w_\al(\bA)$ are allowed to repeat.

Each monomial in the entries of $\bA$ can be represented as a multigraph on $\{1,2,\dots, n\}$.
By summing all monomials with the same ``shape'', the $w_\al(\bA)$ and $z_\al(\bA)$ give two different spanning sets for a subspace of the $S_n$-invariant polynomials in the entries of $\bA$, where $S_n$ acts on $\bA$ by permuting the rows and columns simultaneously.
There are only a few possible distinct shapes for monomials with low degree, so analysis on the $w$ or $z$ polynomials is a highly compressed way to analyze $S_n$-invariant low-degree polynomial functions of $\bA$.

The limiting values of the graph polynomials are a basic set of parameters for the sequence of matrices $\bA^{(n)}$, introduced in random matrix theory by Male \cite{male2020traffic}, who termed them the \emph{traffic distribution}.

\begin{definition}[Traffic distribution]
\label{def:traffic-distribution}
    For a sequence of random\footnote{Deterministic matrices are also allowed just by taking a constant distribution.} matrices $\bA = \bA^{(n)} \in \R^{n \times n}_{\sym}$ we say that
    $\calD: \calA \to \R$ is the \emph{(limiting) traffic distribution} of $\bA$ if 
    \begin{equation}
    \lim_{n \to \infty} \frac{1}{n} \E_\bA w_{\alpha}(\bA) = \calD(\alpha) \text{ for all } \alpha \in \calA. \label{eq:traffic-conv}
    \end{equation}
    We say the (limiting) traffic distribution exists if the limit exists for all $\al \in \calA$.\footnote{
    Note that the diagram $\al$ cannot depend on $n$. It has constant size as $n \to \infty$.
    }
\end{definition}

When the limiting traffic distribution exists, it is easy to show that it determines the
asymptotic behavior of all constant-time GFOM algorithms with input $\bA$:
\begin{claim}
\label{prop:gfom-polynomial}
    Assume that $\bA = \bA^{(n)} \in \R^{n \times n}_{\sym}$ have traffic distribution $\cD$, and that a pGFOM defines $\bx_t = \bx_t(\bA)$ with $\bx_0 = \bm 1$.
    Then, for any fixed $t$ and polynomial $\varphi \in \R[x]$,
    \[ \lim_{n \to \infty} \E_\bA  \frac{1}{n}\sum_{i=1}^n \varphi(\bmx_{t}[i]) = C, \]
    where $C$ is a constant depending only on $\cD$, $(f_s)_{1 \leq s \leq t}$, and $\varphi$.
\end{claim}

Because of this observation, the traffic distribution is a natural way both to show existence of effective dynamics for constant-time GFOM (when the traffic distribution exists then so do effective dynamics) and to characterize the \emph{universality class} of GFOM
(when two sequences of matrices have the same traffic distribution then they have the same effective dynamics).

We now reach our first main contribution: by calculating their limiting traffic distributions, we obtain the first analysis of GFOM on non-trivial completely deterministic inputs.
Namely, we prove that any delocalized orthogonal matrix, after a slight modification, has the same traffic distribution as a corresponding random matrix model, the \emph{regular random orthogonal model} (\rrom; see~\Cref{def:rom}).

\begin{theorem}[See~\Cref{thm:universality-orthogonal}]
    \label{thm:intro-hadamard}
    Let $\bm \Pi = \bm \Pi^{(n)} = \bI - \frac 1n \bm 1 \bm 1^\top$ and $\bH = \bH^{(n)}\in\R_{\sym}^{n\times n}$ be a sequence of orthogonal matrices such that
    \begin{align}
        \max_{1\le i, j\le n} |\bmH[i,j]|\le n^{-\frac 12 + o(1)}\,.\label{eq:delocalization}
    \end{align}
    Then, the traffic distribution of $\mathbf \Pi \bmH \mathbf \Pi$
    exists and equals that of the \rrom.
\end{theorem}
\noindent
The motivating examples for~\cref{thm:intro-hadamard} are ``Fourier transform matrices'' such as the Walsh--Hadamard matrix (\Cref{def:hadamard}), discrete sine transform matrix, or discrete cosine transform matrix (\Cref{def:dst}).
We call conjugating by the projection matrix $\matPi$ \emph{puncturing} the matrix.
\Cref{thm:intro-hadamard} implies
that, after puncturing, the effective dynamics of GFOM
on these matrices are the same as those for the \rrom, which itself is a punctured version of the \emph{random orthogonal model} (\rom) of~\cite{marinari1994replicaI}.
Explicit state evolution equations for these dynamics are given in~\Cref{prop:hadamard-explicit-state}.

Puncturing is necessary in~\Cref{thm:intro-hadamard} and is natural for Fourier transform matrices. For the Walsh--Hadamard matrix, puncturing removes the first row and column, all of whose entries are identically $1/\sqrt{n}$.
This row/column makes $\bH\bm{1}$ have a single large entry; because of that imbalance, without puncturing the traffic distribution of $\bH$ does not exist\footnote{For example, when $\bH$ is the Walsh--Hadamard matrix, the degree-$D$ star diagram $\sigma_D$ satisfies $\frac 1n |w_{\sigma_D}(\bH)| = \Theta(n^{D/2-1})$, which diverges for $D>2$.} and some GFOMs do not have well-defined asymptotic dynamics.
This phenomenon has also been observed experimentally: \cite{schniter2020simple} writes that ``structured matrices (e.g., DCT, Hadamard, Fourier) should work as
well as i.i.d.\ random ones. But, in practice, AMP often diverges with such structured matrices.''
We propose, and our results corroborate, that it is precisely alignment with the all-ones vector that causes this behavior.

Showing that Fourier transform matrices are pseudorandom orthogonal matrices has been a longstanding folklore open problem
in the statistical physics and AMP literature.
It seems to originate in the work of \cite{marinari1994replicaI,marinari1994replicaII,parisi1995mean} in statistical physics, who proposed these matrices as couplings for spin glass models.
Recently (nearly 30 years later),~\cite{dudeja2023universality}
summarized the situation as follows:
\begin{displayquote}
    More generally, numerical studies reported in the literature \emph{[...]}
    suggest that AMP algorithms exhibit universality properties as long as the
    eigenvectors are generic. Formalizing this conjecture remains squarely beyond
    existing techniques, and presents a fascinating challenge.
\end{displayquote}
Similar comments have been made in \cite{subsamplingJavanmard,rangan2019convergence,barbierSpatial}, and relevant numerical experiments can be found in \cite{CO-2019-TAPEquationAMPInvariant,abbara2020universality,dudeja2023universality}. 
Fourier transform matrices are also
favored in compressed sensing applications since they admit fast
multiplications via the Fast Fourier Transform~\cite[Example 2.26]{wang2022universality}.

Although \cref{thm:intro-hadamard} concerns orthogonal 
matrices, we also prove generally that after puncturing, any sequence of delocalized matrices has the same traffic distribution as the orthogonally invariant ensemble with the same eigenvalue distribution, assuming stronger delocalization properties than~\cref{eq:delocalization}. See~\Cref{thm:universality-new} for the formal statement.

\subsubsection{Cactus properties: conditions for simple traffic distributions}

The traffic distribution is a complicated object in general, just because its indexing set $\mathcal{A}$ is very large.
Fortunately, traffic distributions of many common matrices are much simpler.
Specifically, they often satisfy a \emph{cactus property}:
almost all of the graph polynomials $z_\al(\bA)$ are asymptotically negligible as $n \to \infty$, with the only exceptions being the cactus graphs $\al \in \cC \subsetneq \mathcal{A}$ (in the $z$ basis, but \emph{not} in the $w$ basis).

\begin{definition}[Cactus properties and cactus type]\label{def:cactus-property}
    For a sequence of symmetric matrices $\bA = \bA^{(n)} \in \R^{n \times n}$, we say that:
    \begin{enumerate}[(i)]
        \item $\bA$ has the \emph{strong cactus property} if $\lim_{n \to \infty} \frac{1}{n}\E_\bA z_{\alpha}(\bA) = 0$ for all $\alpha \in \calA \setminus \calC$.
        \item $\bA$ has the \emph{weak cactus property} if $\lim_{n \to \infty} \frac{1}{n} \E_\bA z_{\alpha}(\bA) = 0$ for all $\alpha \in \calE \setminus \calC$.
        \item $\bA$ has the \emph{factorizing (strong or weak) cactus property} if it has the (strong or weak) cactus property, and for each $\sig \in \cC$ we have $\lim_{n \to \infty}\frac 1n \E_\bA z_\sig(\bA) = \prod_{\rho \in \cyc(\sig)} \kappa_{|\rho|}$ for some real numbers $\kappa_q$, where $\cyc(\sig)$ is the set of cycles of a cactus and $|\rho|$ is the length of a cycle.\footnote{In the traffic probability literature, the factorizing strong cactus property has been referred to as a traffic distribution being \emph{of cactus type} \cite{cebron2024traffic}. The parameters $\kappa_q$ are the \emph{free cumulants} appearing in free probability theory.}
    \end{enumerate}
\end{definition}

The idea that the non-negligible diagrams for many random matrix models are cactuses appeared in the physics literature as early as the 1990s \cite{parisi1995mean, MFCKMZ-2019-PlefkaExpansionOrthogonalIsing}
and we will show in~\Cref{sec:feynman-physics} how it can be derived from the \emph{Feynman diagram expansion} widely used in physics.
More recent mathematical work~\cite{male2020traffic,cebron2024traffic} reviewed in~\cref{sec:trafficRandom} has rigorously established the strong cactus property for Wigner matrices and unitarily invariant matrices whose eigenvalue distributions converge weakly.
In fact,  
the factorizing strong cactus property is essentially {equivalent} to $\bA$ having the same limiting traffic distribution as some orthogonally invariant random matrix model.

The strong cactus property implies that the traffic distribution is specified only by the limiting values associated to $\sig \in \mathcal{C}$, a much smaller set of graphs than $\mathcal{A}$.
Another way to say this is that, under the strong cactus property, the traffic distribution contains no extra information beyond the considerably simpler \emph{diagonal distribution}, introduced by \cite{wang2022universality}.

\begin{definition}[Diagonal distribution]\label{def:diagonal-distribution}
    For a sequence of symmetric matrices $\bA = \bA^{(n)} \in \R^{n \times n}$, we say that
    $\calD: \calC \to \R$ is the \emph{limiting diagonal distribution} of $\bA$ if
    \[ \lim_{n \to \infty} \frac{1}{n}\E_\bA  w_{\sig}(\bA) = \calD(\sig) \text{ for all } \sig \in \calC. \]
    We say the diagonal distribution exists if the limit exists for all $\sig \in \calC$.
\end{definition}

Let us make several important observations about the definitions of the
traffic distribution, the diagonal distribution, and the cactus
properties.

First, note that~\cref{def:cactus-property} is stated in the $z$-polynomial basis, whereas~\cref{def:traffic-distribution,def:diagonal-distribution} are stated in the $w$-polynomial
basis.
Throughout the paper, it will be helpful to move back and forth
between these bases, since some properties are most natural (or even are
only true) in one basis or the other.
This can be done via M{\"o}bius inversion, as described
in~\Cref{sec:mobius}.

Second, neither the diagonal distribution nor the traffic
distribution is an actual probability distribution.
Instead, they should be interpreted as specifying limiting moments of
certain empirical distributions, namely, the empirical distributions
of the entries of vector graph polynomials.\footnote{
The reason for the name of the diagonal distribution $\calD$ is that it can also be interpreted as specifying the moments of the empirical distribution over the diagonal of certain matrices, namely those that can be formed from $\bA$ by matrix multiplication and the operation of zeroing out the off-diagonal entries of a matrix \cite{wang2022universality}.}

Third, one can view the diagonal and traffic distributions as
generalizations of the limiting spectral distribution of a sequence of matrices.
The spectral moments are $\frac1n \Tr(\bA^q)=\frac1n w_\alpha(\bA)$,
where $\alpha$ is the $q$-cycle diagram, so they are included in both
the diagonal and traffic distributions:
\[ \text{`` spectral distribution} \,\, \subsetneq \,\, \text{diagonal distribution} \,\, \subsetneq \,\, \text{traffic distribution ''} \]

Just as the empirical spectral distribution characterizes the
limiting behavior of all polynomials in $\bA$ that are invariant under
the action of the orthogonal group $O(n)$
(acting by $\bQ \cdot \bA = \bQ\bA\bQ^{\top}$), the traffic distribution characterizes the limiting behavior of the larger space of
polynomials invariant under the smaller symmetric group
$S_n$, i.e.,\ where $\bQ$ is restricted to be a permutation matrix.

Finally, the strong cactus properties describe when these inclusions can
be reversed:
if the strong cactus property holds, then the traffic distribution
contains no more information than the diagonal distribution. If the
factorizing strong cactus property holds, then the diagonal
distribution, in turn, contains no more information than the spectral
distribution.

Due to the effect of the puncturing operation, the strong cactus property actually is \emph{not} satisfied by the pseudorandom
matrices or \rrom matrices appearing in our~\Cref{thm:intro-hadamard}.
But, these matrices satisfy the weak cactus property,
and establishing this is a key step in the analysis of these matrices (in fact, the weak cactus property holds for the Fourier transform matrices without puncturing, as we show in Part 1 of \cref{thm:universality-new}).

\subsubsection{Dynamics of graph polynomials: asymptotic GFOM state and treelike AMP}

Recall that our final goal is to describe the state $\bm x_t = \bm x_t(\bm A)$ of a GFOM.
Since $\bm x_t \in \R^n$ we use vector diagrams for this task.
Compared to the scalar diagrams in $\cA_0$, the only extra information in these diagrams is that one of the vertices is specially marked as the ``root'', whose label specifies the coordinate of the vector output.

\begin{definition}[Vector diagram classes]
    We write $\cA_1$ and $\cC_1$ for the set of graphs in $\cA$ and $\cC$ respectively, further decorated with a distinguished root vertex.
    For $\alpha \in \cA_1$, we write $\rt(\alpha) \in V(\alpha)$ for its root vertex.
\end{definition}

\begin{definition}[Vector graph polynomials]
        Given $\al \in \calA_1$ and $\bA \in \R^{n \times n}_{\sym}$\,, define vectors of polynomials $w_\al(\bA), z_\al(\bA) \in (\R[\bm A])^n$ by,
    \begin{align*}
        \bm w_\al(\bA)[i] &\defeq \sum_{\substack{j: V(\al) \to [n] \\ j(\rt(\al)) = i}} \prod_{\{u,v\} \in E(\al)} \bA[j(u),j(v)]\,,\\
        \bm z_\al(\bA)[i] &\defeq \sum_{\substack{j: V(\al) \embeds [n] \\ j(\rt(\al)) = i}} \prod_{\{u,v\} \in E(\al)} \bA[j(u),j(v)]\,,
    \end{align*}
    for all $i\in [n]$.
\end{definition}

To analyze the vector graph polynomials, we compute the moments of the empirical distribution of their entries.
We will see that these are matched (asymptotically) by a family of scalar random variables $Z_{\alpha}^{\infty}$, so the empirical distribution of the entries of $\bm z_{\alpha}(\bA)$ converges in a suitable sense to $Z_\al^\infty$ as $n\to\infty$.
Further, when $\bA$ has the strong cactus property, an analogous property is inherited by these limiting distributions, only a small number of $\alpha \in \cA_1$ having a non-negligible limit.
\begin{definition}[Treelike diagrams]\label{def:treelike}
    We say that $\alpha \in \cA_1$ is \emph{treelike} if it is a tree with \emph{hanging cactuses} attached to the leaves of the tree (see~\Cref{fig:treelike}).
    We denote the set of treelike diagrams by $\cT_1$, and denote by $\cG_1 \subseteq \cT_1$ the set of treelike diagrams in which, after removing
    hanging cactuses, the root has degree exactly 1. 
\end{definition}

\begin{figure}[ht]
    \centering

    \begin{subfigure}[t]{0.31\linewidth}
        \centering
        \begin{minipage}[c][1.8in][c]{\linewidth}
            \centering
            \includegraphics[height=1.8in]{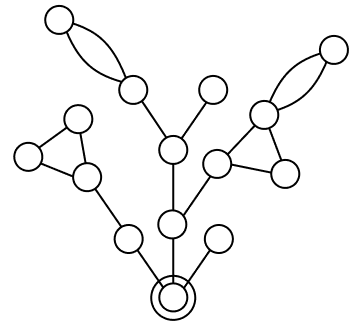}
        \end{minipage}
        \caption{A treelike diagram in $\calT_1$}
    \end{subfigure}
    \hfill
    \begin{subfigure}[t]{0.31\linewidth}
        \centering
        \begin{minipage}[c][1.8in][c]{\linewidth}
            \centering
            \includegraphics[height=1.5in]{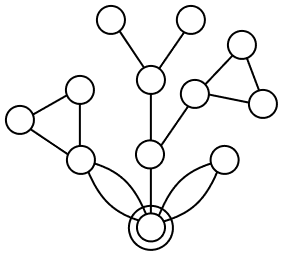}
        \end{minipage}
        \caption{A Gaussian diagram in $\calG_1$}
    \end{subfigure}
    \hfill
    \begin{subfigure}[t]{0.31\linewidth}
        \centering
        \begin{minipage}[c][1.8in][c]{\linewidth}
            \centering
            \includegraphics[height=1.5in]{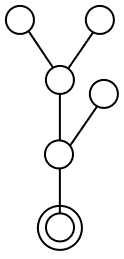}
        \end{minipage}
        \caption{The diagram in $\calG_1$ after removing its hanging cactus.}
    \end{subfigure}

    \caption{Examples of treelike and Gaussian diagrams. The root vertex is
    circled.}
    \label{fig:treelike}
\end{figure}

\begin{theorem}[Vector polynomial limits; see~\cref{thm:convergence-distribution}]
\label{prop:convergence-distribution}
    Assume that $\bA = \bA^{(n)}$ has the strong cactus property with limiting diagonal distribution $\cD$. 
    Assume also that the sequence of random variables 
    $(\|\bm A^{(n)}\|)_{n\ge 1}$ is \emph{tight},\footnote{If the matrices $\bm A^{(n)}$ are deterministic, this should be understood as $(\|\bm A^{(n)}\|)_{n\ge 1}$ being bounded.} i.e., that
    \begin{equation}
        \text{for all } \varepsilon > 0 \text{ there exists } K > 0 \text{ such that } \sup_{n\ge 1}\Pr(\|\bA^{(n)}\|>K)\le \varepsilon\,.\label{eq:tightnessNorm}
    \end{equation}
    Write $\bz_{\cA_1}(\bA) \in (\R^{\cA_1})^n$ for the stacking of values of all $\bz_{\alpha}(\bA)$ for $\alpha \in \cA_1$.
    Then,
    \[ \samp(\bz_{\cA_1}(\bA)) \xrightarrow[n \to \infty] {\textnormal{(d)}} (Z_{\alpha}^{\infty})_{\alpha \in \cA_1}\,, \]
    for a family of (partially dependent) random variables $(Z_{\alpha}^{\infty})_{\alpha \in \cA_1}$ such that $Z_{\alpha}^{\infty} = 0$ for all $\alpha$ not treelike, and which can be sampled as follows for $\al \in \calT_1$:
    \begin{enumerate}
        \item Draw $(Z_{\sig}^{\infty})_{\sig \in \cC_1}$ from a distribution determined by $\cD$.
        \item Draw $(Z_{\gamma}^{\infty})_{\gamma \in \cG_1} \sim \calN(\bm 0, \bm\Sigma^\infty)$ from a centered Gaussian distribution with countably infinite covariance matrix $\bm\Sigma^\infty$ depending on $(Z_{\sigma}^{\infty})_{\sigma\in\calC_1}$.
        \item Set $(Z_{\alpha}^{\infty})_{\alpha \in \cT_1  \setminus (\cG_1 \cup \cC_1)}$ to be certain deterministic polynomial functions of $(Z_{\alpha}^{\infty})_{\alpha\in \calG_1\cup \calC_1}$.
    \end{enumerate}
\end{theorem}
\noindent
We note that $\samp(\bz_{\cA_1}(\bA))$ is a random variable taking values in $\R^{\cA_1}$, a countable product space.
Thus, its convergence in distribution is the same as convergence in distribution of any finite-dimensional projection; see~\Cref{sec:convergence-process}.

The application to pGFOM is as follows.
Analogously to~\Cref{prop:gfom-polynomial},
it is easy to see that the iterates $\bmx_t(\bmA)$
of a pGFOM admit a diagrammatic expansion
of the form
\begin{equation}
\bm x_t(\bm A) = \sum_{\alpha \in \cA_1} c_{t, \alpha} \bm z_{\alpha}(\bm A)\,, \label{eq:z-expansion-of-x}
\end{equation}
for finitely supported coefficients
$(c_{t,\alpha})_{\alpha\in\calA_1}$.
Given the limits of the individual diagrams above, for a given GFOM, number of iterations $t$, and coefficients as in~\cref{eq:z-expansion-of-x}, we write
\begin{equation*}
    (X_0^{\infty}, \dots, X_t^{\infty}) \defeq \left(\sum_{\alpha \in \cA_1} c_{0, \alpha} Z_{\alpha}^{\infty}, \dots, \sum_{\alpha \in \cA_1} c_{t, \alpha} Z_{\alpha}^{\infty}\right)\,,\label{eq:asymptoticStateIntro}
\end{equation*}
a random variable in $\R^{t + 1}$ that describes the joint empirical distribution of the first $t$ steps of the GFOM.
We call this the \emph{asymptotic state} of a GFOM (\Cref{def:asymptotic-state}). By~\Cref{prop:convergence-distribution}, the asymptotic state describes limiting empirical averages over the GFOM states, in the sense that
\[ \lim_{n \to \infty} \E_{\bA} \frac{1}{n}\sum_{i = 1}^n  \varphi(\bx_0[i], \dots, \bx_t[i]) = \EE\, \varphi(X_{0}^{\infty}, \dots, X_t^{\infty}) \]
for any $\varphi: \R^{t + 1} \to \R$ either a polynomial or a bounded continuous function (\Cref{claim:cvImpliesState}).

In particular, if the only nonzero $c_{t, \alpha}$ in~\cref{eq:z-expansion-of-x} are non-treelike $\al$ or treelike $\alpha \in \cG_1$, then the GFOM has an asymptotic state that is Gaussian conditional on $(Z_{\sigma}^{\infty})_{\sigma\in\calC_1}$.
This observation leads to our second main contribution: a new family of \emph{treelike AMP} algorithms simultaneously generalizing Orthogonal Approximate Message Passing (OAMP) algorithms~\cite{rangan2019vector, fan2022approximate} for orthogonally invariant matrices, and Generalized Approximate Message Passing (GAMP) algorithms~\cite{rangan2011generalized, javanmard2013state} for matrices with independent entries that are not necessarily identically distributed.\footnote{The second comparison is with the caveat that GAMP uses a certain class of ``non-separable'' nonlinearities (applying a different function $f_t$ to each coordinate of $\bmx_t$) which are not directly covered by our result \cite{rangan2011generalized, javanmard2013state}.}

\begin{theorem}[Treelike AMP; see~\Cref{thm:full-onsager}]
\label{thm:intro-onsager}
    Assume that $\bA = \bA^{(n)}$ satisfies the assumptions of \Cref{prop:convergence-distribution}.
    Given polynomial functions $f_t : \R \to \R$, define the pGFOM:
    \begin{equation*}
    \begin{alignedat}{2}
        &\bmx_0 \defeq \bm{1}\,, \qquad
        &&\bmx_t \defeq \bA \bmf_{t-1}
          - \sum_{s=0}^{t-1}\bmb_{s,t} \cdot \bmf_s \,,
          \qquad
          \text{(The product $\bmb_{s,t} \cdot \bmf_s$ is entrywise.)}
        \\
        &\bmf_t \defeq f_t(\bmx_t)\,, \qquad
        &&\bmf'_t \defeq f'_t(\bmx_t)\,.
    \end{alignedat}
    \end{equation*}
    \begin{equation*}
    \begin{aligned}
        \bmb_{s,t}[i]
        &\defeq
        \sum_{\substack{
            i_s,\dots,i_{t-1}=1\\
            \textnormal{distinct}\\
            i_s=i
        }}^n
        \bA[i_s,i_{t-1}] \bmf'_{t-1}[i_{t-1}]
        \bA[i_{t-1},i_{t-2}] \bmf'_{t-2}[i_{t-2}]
        \cdots
        \bmf'_{s+1}[i_{s+1}]
        \bA[i_{s+1},i_s]\,.
    \end{aligned}
    \end{equation*}
    Then, for any fixed $t$ as $n \to \infty$, the asymptotic state $(X^\infty_1,\ldots,X^\infty_t)$, conditional on $(Z_{\sigma}^{\infty})_{\sigma\in\calC_1}$, is a centered Gaussian vector.
    A formula for its covariance is given in~\cref{lem:cov-cactus}.
\end{theorem}

The subtracted terms $\sum_{s=0}^{t-1} \bm b_{s,t} \cdot \bm f_s$ generalize the ``Onsager correction terms'' appearing in different variants of AMP.
\cref{thm:intro-onsager} and its proof address two questions posed in~\cite{wang2022universality},
namely (1)~to obtain a combinatorial interpretation of the Onsager correction for OAMP algorithms,
and (2)~to identify a more general class of AMP algorithms whose state evolution is characterized by the diagonal distribution of the input matrix.
\cref{thm:intro-onsager} shows that (2) is possible for arbitrary matrices satisfying the strong cactus property, and explicitly describes such an algorithm and its conditionally Gaussian asymptotic states. We
show in~\Cref{sec:se-examples} how the treelike
AMP iteration simultaneously generalizes several
variants of AMP introduced in prior work.

We emphasize that, in contrast to all existing state evolution results we are aware of, we derive an Onsager correction and state evolution formula \emph{without assuming an explicit random model for $\bA$}.
The iteration in~\Cref{thm:intro-onsager} is the same regardless of the limiting diagonal distribution of $\bA$, provided that these matrices (random or deterministic) satisfy the strong cactus property and have \emph{some} limiting diagonal distribution (which will affect the covariance formula in~\cref{lem:cov-cactus}).
Note that the matrices in our universality result (\cref{thm:intro-hadamard}) and their random counterparts (the \rrom), satisfy the weak cactus property instead of the strong cactus one.
Nevertheless, the Onsager correction and the state evolution can still be determined by a reduction to the strong-cactus-property setting, as we explain in~\Cref{sec:dynamics-reg}.

\subsection{Related work}
\label{sec:related}

\paragraph{Moment method for AMP.}
Our overall approach to graph polynomials generalizes prior work for the case of Wigner matrices~\cite{jones2025fourier}.
Similar techniques have also appeared in prior works using the moment method to study AMP algorithms~\cite{bayati2015universality, wang2022universality, montanari2022equivalence, dudeja2023universality, ivkov2023semidefinite, dudeja2024spectral}.
The $w$ and $z$ polynomials are rather fundamental objects which,
along with their vector, matrix, and tensor generalizations,
have variously been called
``graph monomials'' or ``traffics'' in free probability,
``graph matrices'' in computer science,
``graph homomorphism polynomials'' in combinatorics,
and are also related to ``tensor networks'' and ``Feynman diagrams'' in physics.

\paragraph{Polynomial vs.\ non-polynomial GFOM.}
    In random and semi-random models, general
    first-order methods with a constant number of
    iterations using (1) only polynomial
    nonlinearities or (2) arbitrary Lipschitz
    nonlinearities are generally 
    expected to have the same
    computational power. Using polynomial approximation arguments, this has been
    made precise in several previous works~\cite{montanari2022equivalence,ivkov2023semidefinite,wang2022universality}. For example,
    \cite[Lemma 2.12]{wang2022universality} gives
    an abstract reduction
    showing
    that if state evolution for AMP on rotationally-invariant matrices holds for polynomial
    nonlinearities, then it also
    holds for arbitrary Lipschitz nonlinearities.
    While we study more general matrix models, we expect the assumption of polynomial nonlinearities is not essential.

\paragraph{AMP vs.\ GFOM.}
A simple reduction shows that every algorithm in the GFOM class can be expressed as a certain post-processing of an AMP algorithm (allowing ``memory terms'')~\cite{celentano2020estimation}.
Therefore, these two classes of algorithms are equivalent from the standpoint of computational power.
In our analysis, this is mirrored by the fact that, in~\Cref{prop:convergence-distribution}, all possible non-Gaussian limits after conditioning on the draw of $(Z^\infty_{\sigma})_{\sigma\in \calC_1}$ are  deterministic functions of the possible Gaussian limits.

\paragraph{GFOM on independent entry matrices.}
The analysis of GFOM and AMP on Wigner matrices or inhomogeneous versions thereof was the first case widely considered in the literature, and goes back to the origins of the mathematical analysis of AMP in the statistical physics literature on spin glasses~\cite{bolthausen2014iterative,donoho2009message,bayati2011dynamics,montanari2012graphical,barbierSpatial,rush2018finite,LW-2022-NonAsymptoticAMPSpiked}. See~\cite{feng2022unifying} for a survey of many of these works.
Further, see~\cite{bayati2015universality,chen2021universality} for universality results over such models allowing for different entry distributions (but still requiring entrywise independence),~\cite{donoho2013information,javanmard2013state} for results on block-structured variance profiles along the lines of our block GOE model, and~\cite{gueddari2025approximate,bao2025leave} for recent progress on more general variance profiles.

\paragraph{GFOM on orthogonally invariant matrices.} The correct form of AMP (to ensure Gaussian limiting distributions)
in orthogonally invariant models was first predicted non-rigorously for physics applications by~\cite{opper2016theory} using dynamical mean-field theory (DMFT), and then proved by~\cite{fan2022approximate}. Precursors for special ``divergence-free'' forms of AMP were also obtained by~\cite{CO-2019-TAPEquationAMPInvariant,ma2017orthogonal,rangan2019vector,takeuchi2019rigorous} under the names of \emph{Vector AMP} and \emph{Orthogonal AMP}.
Related calculations for a more general statistical physics framework subsuming these AMP variants are carried out in~\cite{MFCKMZ-2019-PlefkaExpansionOrthogonalIsing}; in particular, this work includes special cases of and discusses the more general form of the calculations we detail in~\Cref{app:weingarten}.
See the discussion in~\cite{fan2022approximate} for a more thorough overview of these distinctions.

\paragraph{Universality principles for GFOM.}
Beyond the above results, the main ones we are aware of that reduce the amount of randomness required for AMP are the recent works ~\cite{wang2022universality,dudeja2023universality}, which, modulo technical differences, both prove universality results over random matrices whose distribution is invariant under signed permutations.
In other words, they treat broad classes of matrices provided that these are conjugated by random signed permutations, a considerable reduction in randomness from, e.g., conjugating by random Haar-distributed orthogonal or unitary matrices as in OAMP.
Numerous experimental works have found universality phenomena for ``sufficiently pseudorandom'' deterministic matrices, but we are not aware of any rigorous results for completely deterministic matrices prior to our work.
See discussion in \cite{CO-2019-TAPEquationAMPInvariant,schniter2020simple,abbara2020universality,dudeja2023universality}.

\subsection{Organization of the paper}

We give preliminaries on the matrices considered
in this work and modes of convergence for
our limiting theorem in~\Cref{sec:prelims}.
We introduce our definitions of diagrams and
consequences of M{\"o}bius inversions for the
traffic distribution in~\Cref{sec:diagrams}. In~\Cref{sec:trafficRandom}, to build intuition on
traffic distributions,
we describe them for several random matrix 
ensembles.
\Cref{sec:universality} is dedicated to the proof
of our first main result, the polynomial
universality of delocalized deterministic matrices (\Cref{thm:intro-hadamard}).
\Cref{sec:dynamics} details and proves
the effective dynamics of GFOM under the 
strong cactus property (\Cref{prop:convergence-distribution,thm:intro-onsager}).

We illustrate
two viable approaches to computing the 
traffic distribution
of orthogonally invariant matrix models:~\Cref{sec:feynman-physics} is based on Feynman
diagrams and~\Cref{app:weingarten} relies on
Weingarten calculus. \Cref{sec:convergence-process} provides background on convergence
of stochastic processes, and~\Cref{sec:omitted} contains omitted proofs.

\subsection{Acknowledgments}

Thanks to Zhou Fan, Cynthia Rush, and Subhabrata Sen for helpful discussions over the course of this project.
CJ was supported in part by the European Research Council (ERC) under the European Union's Horizon 2020 research and innovation programme (grant agreement No.
101019547).
LP's work was supported by the Swiss National Science Foundation (SNSF), grant no. 10004947.

\newpage
\section{Preliminaries}
\label{sec:prelims}

\subsection{Matrix notation}

Given matrices
$\bA, \bB \in \R^{n\times n}$, we will use:
\begin{itemize}
    \item $\bA \in \R^{n\times n}_\sym$ to
    specify that $\bA$ is symmetric.
    \item $\bA \in O(n) \subseteq \R^{n \times n}$ to specify that $\bA$ is orthogonal.
    \item $\bA[i,j]$ to denote its $(i,j)$-th
    entry for $i,j\in [n] \defeq \{1,\ldots,n\}$.
    \item $\|\bA\| \defeq \max_{\|\bm x\|_2=1} \|\bA \bm x\|_2$ to denote its spectral or operator
    norm.
    \item $\|\bA\|^2_\frob \defeq \sum_{i,j=1}^n \bA[i,j]^2$ to denote its
    Frobenius norm.
    \item $\Tr(\bm A) \defeq \sum_{i=1}^n \bm A[i,i]$ to denote its trace.
    \item $\lambda_1(\bA)\ge \ldots \ge \lambda_n(\bA)$ to
    denote its eigenvalues when $\bm A$ is symmetric.
    \item $\bA \odot \bB$ to denote the entrywise or Hadamard product with entries $(\bm A[i,j]\bm B[i,j])_{i,j\in [n]}$.
\end{itemize}

\begin{definition}[Puncturing]\label{def:puncturing}
    Let $\bH\in \R_\sym^{n\times n}$ and $\bm \Pi \defeq \bm I - \frac 1n \bm 1 \bm 1^\top$ be
    the projection orthogonal to the all-ones
    direction. The {puncturing} of $\bH$ is
    the matrix $\bA = \bm \Pi \bmH \bm \Pi$.
\end{definition}

\begin{definition}[\goe]
\label{def:goe}
    The (normalized) {Gaussian Orthogonal Ensemble} \goe
    is the distribution of random matrices $\bA\in\R_\sym^{n\times n}$
    with $\bA[i,j]=\bA[j,i]\sim \calN(0,1/n)$ independently for all $1\le i<j\le n$, and
    $\bA[i, i]\sim \calN(0,2/n)$ independently
    for all $i\in [n]$.
\end{definition}

\begin{definition}[Hadamard matrices]
    \label{def:hadamard}
    When $n$ is a power of $2$, the (normalized) Walsh--Hadamard matrix $\bH_{\textnormal{had}}^{(n)}\in\R_\sym^{n\times n}$ is defined 
    recursively by
    \[
        \bH_{\textnormal{had}}^{(1)}=\begin{bmatrix}1\end{bmatrix}\,,\qquad 
        \bH_{\textnormal{had}}^{(2n)} \defeq \frac{1}{\sqrt{2}}
        \begin{bmatrix}
        \bH_{\textnormal{had}}^{(n)} & \bH_{\textnormal{had}}^{(n)}\\
        \bH_{\textnormal{had}}^{(n)} & -\bH_{\textnormal{had}}^{(n)}
        \end{bmatrix}.
    \]
\end{definition}
\noindent
$\bH_{\textnormal{had}}^{(n)}$ is a 
symmetric orthogonal matrix with entries in $\pm 1/\sqrt n$.

\begin{definition}[DST and DCT matrices]
    \label{def:dst}
    The discrete sine transform matrices
    $\bH_{\sin}^{(n)}\in\R^{n\times n}_\sym$ are
    \[
        \bH_{\sin}^{(n)}[i,j] \defeq \sqrt{\frac 2 {n+1}} \sin\left(\frac {\pi ij}{n+1}\right)\quad \forall i,j\in [n]\,.
    \]
    The discrete cosine transform matrices
    $\bH_{\cos}^{(n)}\in\R^{n\times n}_\sym$ are 
    \[
        \bH_{\cos}^{(n)}[i,j]\defeq \sqrt{\frac 2n}\cos\left(\frac{\pi(i-\tfrac12)(j-\tfrac12)}{n}\right)\quad \forall i,j\in [n]\,.
    \]
\end{definition}
\noindent
$\bH_{\cos}^{(n)}$ and $\bH_{\sin}^{(n)}$ are
symmetric orthogonal matrices with entries
at most $O(1/\sqrt n)$ in magnitude.

\begin{definition}[\rom and \rrom]
\label{def:rom}
    The {Random Orthogonal Model} \rom is the
    distribution of random matrices $\bH = \bm Q \bm D \bm Q^\top$, where $\bm Q\in O(n)$ is
    Haar-distributed, and $\bm D$ is a diagonal
    matrix with i.i.d. $\textnormal{Unif}(\{-1,1\})$ entries, independent from $\bm Q$.
    The {Regular Random Orthogonal Model} \rrom is the distribution of the puncturing of
    $\bH$, when $\bH$ is sampled from the \rom.
\end{definition}
\noindent
Random matrices from the \rom are symmetric orthogonal matrices, satisfying $\bH^2 = \bm I$.
They are a special case of the
orthogonally invariant models we discuss in~\Cref{sec:rot-inv}.

\subsection{Modes of convergence}
\label{sec:prelim-modes}

We will use a few standard modes of convergence from scalar-valued probability theory.
\begin{definition}[Modes of convergence: scalars]
    For a sequence of random variables $x^{(n)}\in \R$, we say that:
    \begin{itemize}
        \item $x^{(n)}$ \emph{converge in expectation} if, for some $c \in \R$, $\lim_{n \to \infty} \EE x^{(n)} = c$.
        \item $x^{(n)}$ \emph{converge in probability} if, for some $c \in \R$, for all $\varepsilon > 0$, $\lim_{n \to \infty} \PP[|x^{(n)} - c| > \varepsilon] = 0$.
        \item $x^{(n)}$ \emph{converge in $L^2$} if they converge in expectation and $\lim_{n \to \infty} \EE(x^{(n)} - c)^2 = 0$, or equivalently if they converge in expectation and $\lim_{n \to \infty} \Var x^{(n)} = 0$.
    \end{itemize}
    We write a symbol $\cM \in \{\EE, \PP, L^2\}$ to indicate these modes of convergence, and in this notation say that the $x^{(n)}$ \emph{converge in $\cM$}.
\end{definition}

Moreover, we say a sequence of random vectors
$\bx^{(n)}\in\R^d$ in fixed dimension
$d\ge 1$ {\em converges in distribution} to a random vector
$\bx\in\R^d$ if for every bounded continuous
function $\varphi \colon \R^d\to \R$,
\[
    \E \varphi(\bx^{(n)})\underset{n\to\infty}\longrightarrow \E \varphi(\bx)\,,
\]
in which case we write $\bx^{(n)}\tod \bx$.
See~\Cref{sec:convergence-process} for a
generalization to random variables
indexed by a countably infinite index
set.

\begin{definition}[Modes of convergence: tracial moments]
    For a mode of convergence $\cM$, we say that a sequence of random matrices $\bA \in \R^{n \times n}$ \emph{converges in tracial moments in $\cM$} if, for every $k \geq 1$, $\frac{1}{n}\Tr \bA^k$ converges in $\cM$.
    We say that it converges in tracial moments in $\cM$ to a probability measure $\mu$ over $\R$ if \[\frac{1}{n}\Tr \bA^k \to \int x^k\, \d\mu(x)\]in the mode of convergence $\cM$.
\end{definition}

\subsection{Matchings and Wick calculus}

Given a set $S$, let $\calM(S)$ denote the set of matchings on $S$.
Let $\PM(S)$ denote the subset of perfect matchings.
The elements of $M \in \calM(S)$ are written as pairs $\{i,j\} \subseteq S$.
For several sets $S_1, \dots, S_k$, denote by $\calM(S_1, \dots, S_k)$ the set of matchings on the disjoint union $S_1 \sqcup \cdots \sqcup S_k$ that do not match any two elements of the same $S_i$.
For two sets $S_1, S_2$ of the same size, denote by $\PM(S_1, S_2)$ the bipartite perfect matchings of $S_1 \sqcup S_2$ that only match elements of $S_1$ to ones of $S_2$.
We will abbreviate $\calM(\{1,2,\dots, q\})$ as $\calM(q)$.

\begin{lemma}[Wick lemma]\label{lem:wick}
    Let $X_1, \dots, X_q$ be jointly Gaussian random variables with mean zero. Then:
    \begin{equation*}
        \E[X_1\cdots X_{q}] = \sum_{M \in \PM(q)} \prod_{ij \in M} \E[X_iX_j]\,.
    \end{equation*}
\end{lemma}

The Wick products are the multivariate generalization of the Hermite polynomials to correlated Gaussians~\cite[Chapter 3]{Janson:GaussianHilbertSpaces}.

\begin{definition}[Wick product]
\label{def:wick}
    Let $I$ be an index set, $\bX = (X_i)_{i \in I}$ be formal variables, and $\matSig \in \R_{\sym}^{I \times I}$. The {\em Wick products} are defined by, for each finitely supported $\alpha \in \N^I$,
    \[
       \He_{\al}(\bX\,;\, \matSig) \defeq \sum_{M \in \calM(\al)} (-1)^{|M|} \prod_{uv \in M}\matSig[u,v] \prod_{u\notin M} X_u\,,
    \]
    where $\calM(\al)$ denotes the set of matchings on a collection consisting of $\al_i$ copies of each $i \in I$.
\end{definition}

\noindent When $|I| = 1$, $X\sim \calN(0,1)$, and $\Sig = 1$, then $\He_{(p)}(X\,;\, \Sig)$ equals the $p$th
Hermite polynomial.

When the $X_i$ are mean-zero Gaussian random variables and $\matSig$
is their covariance matrix, the Wick products satisfy the
(partial) orthogonality property that for each
finitely supported $\alpha, \beta\in\N^I$ with $\sum_i \alpha_i \neq \sum_i \beta_i$,
\[
    \E\left[\He_{\al}(\bX\,;\, \matSig) \He_{\beta}(\bX\,;\, \matSig)\right] = 0\,.
\]
In general, we have
\[ \E\left[\He_{\al}(\bX\,;\, \matSig) \He_{\beta}(\bX\,;\, \matSig)\right] = \sum_{M \in \PM(\alpha, \beta)} \prod_{uv \in M} \bm\Sigma[u, v]\,. \]
Since by the Wick lemma $\E\left[\prod_{i\in \alpha}X_i \cdot \prod_{j\in \beta} X_j\right]$ equals the same sum over all matchings of $\alpha \sqcup \beta$, the Wick products achieve a general ``partial orthogonalization'' that removes all terms from this covariance where any pairs within $\alpha$ or within $\beta$ are matched.

For each choice of $\matSig \in \R^{I \times I}_\sym$,
the Wick products are a basis for polynomials in the $X_i$.
Multiplication of polynomials gives an algebra structure to this space
which we call the {\em Wick algebra} of $\bX$.
Below is a combinatorial formula for multiplication in the Wick algebra.

\begin{proposition}[{\cite[Theorem 3.15]{Janson:GaussianHilbertSpaces}}]
\label{prop:wick}
    Let $I$ be an index set, $\bX = (X_i)_{i \in I}$ be formal variables, and $\matSig \in \R_\sym^{I \times I}$.
    Let $\al^1,\dots, \al^k \in \N^I$. Then:
    \begin{align*}
        \prod_{j=1}^k \He_{\al^j}(\bX; \matSig) = \sum_{M \in \calM( \al^1, \dots,  \al^k)} \prod_{uv \in M} \matSig[u, v] \He_{U(M)}(\bX\,;\, \matSig)\,,
    \end{align*}
     where $\al^j$ is a multiset of size $|\alpha^j|$ with $\al^j_{i}$ copies of each $i \in I$.
     Here $U(M) \in \N^I$ for $M$ a matching of $\al^1 \sqcup \cdots \sqcup \al^k$ counts the number of unmatched elements of each type.
\end{proposition}

\noindent In the special case where each
group $\alpha^j$ consists of a
single element, we obtain:

\begin{corollary}\label{lem:wick-recursion}
    For every $i_1, \ldots, i_k\in I$,
    \[
        \prod_{j=1}^k X_{i_j} = \sum_{M\in \calM(k)} \prod_{uv\in M} \bm \Sigma[i_u,i_v] \He_{U(M)}(\bm X\,;\,\bm \Sigma)\,.
    \]
\end{corollary}

\section{Diagrams and the \texorpdfstring{$w$- and $z$-Bases}{w- and z-Bases} of Polynomials}
\label{sec:diagrams}

All graphs considered in this paper are {\em multigraphs} (loops and multiedges are allowed) and will be
denoted by Greek letters ($\alpha, \beta, \gamma,\ldots$). 
We use the terms {\em graphs} and {\em diagrams}
interchangeably in this paper.
Given a diagram $\alpha$, we use $V(\alpha)$
to denote its vertex set and $E(\alpha)$ to denote
its edge set. We 
denote by $\alpha[S]$ the subgraph of $\alpha$
induced by $S\subseteq V(\alpha)$.
We count self-loops as contributing 2 to the degree of a vertex.

\subsection{Classes of diagrams}

Each diagram can have either $0$, $1$, or an ordered pair of $2$ special vertices called its \emph{root(s)}. With the exception of the class of graphs defined in~\Cref{def:open-cactus}, the roots of a graph can be arbitrary vertices (in particular, they might be equal if there are two of them).

\begin{notation}
    Let $\calA=\calA_0$ (resp.\ $\calA_1$ or $\calA_2$) be the set of all connected graphs with no 
    root (resp.\ $1$ root or $2$ roots). We also
    refer to such graphs as scalar (resp. vector or matrix) diagrams.
\end{notation}

Given $\alpha\in\calA$, an edge
$e\in E(\alpha)$ is a {\em bridge} of $\alpha$ if deleting $e$
would disconnect the graph. $\alpha\in\calA$ is {\em 2-edge-connected} if it contains no bridge. In general, $\alpha\in\calA$ can be decomposed into 
a tree of 2-edge-connected components connected by
bridges.

\begin{notation}
    Let $\calE=\calE_0\subseteq\calA$ (resp.\ $\calE_1\subseteq \calA_1$ or $\calE_2\subseteq \calA_2$) be the set of all 2-edge-connected
    scalar  (resp.\ vector or matrix) diagrams.
\end{notation}

Given $\alpha\in\calA$, a vertex $u\in V(\alpha)$ 
is an {\em articulation point} of $\alpha$ if removing $u$ and its incident edges
disconnects the graph. 
$\alpha$ is 
{\em 2-vertex-connected} if it has no articulation point.
Any $\alpha\in\calA$ decomposes into its 2-vertex-connected components
(blocks), which refine the 2-edge-connected components. The block-cut
graph (whose vertices are the articulation points and the blocks, with edges for
incidence) is a tree.

A connected graph is a {\em cactus} if every edge
lies on exactly one simple cycle.
Thus, cactuses are in a sense the minimal 2-edge-connected graphs.

\begin{notation}
    Let $\calC=\calC_0\subseteq \calA$ (resp. $\calC_1\subseteq \calA_1$) be
    the set of all scalar (resp. vector) cactus diagrams.
\end{notation}

\noindent For a cactus $\sigma$, we will denote
by $\cyc(\sigma)$ the set of (unrooted)
cycles of $\sigma$. 

Finally, as in~\Cref{def:treelike}, we will denote the treelike diagrams by $\calT_1$ and the treelike diagrams such that the root has degree 1 after deleting all hanging cactuses by $\calG_1$.

\subsection{Graph polynomials}
\label{sec:polynomial-definitions}

Each diagram represents different scalar-, vector-, or matrix-valued polynomials 
in a matrix input, depending on whether it is viewed
in the $w$-basis or the $z$-basis. In the following
definitions,
we fix $\bA\in \R_\sym^{n\times n}$, $\alpha$ to be a scalar, vector, or matrix diagram, and
$i,j\in [n]$.

\begin{definition}
    \label{def:w-basis}
    Define $w_\alpha(\bA)
    \in \R$, $\bw_\alpha(\bA)\in \R^n$, and $\bW_\alpha(\bA)\in \R^{n\times n}$ by
    \begin{alignat*}{2}
        w_\alpha(\bA) &= \sum_{\substack{\varphi\colon V(\alpha)\to [n]}} \prod_{\{u,v\}\in E(\alpha)} \bA[\varphi(u),\varphi(v)] \quad&& \text{if $\alpha$ is a scalar diagram,}\\[2mm]
        \bw_\alpha(\bA)[i] &= \sum_{\substack{\varphi\colon V(\alpha)\to [n]\\\varphi(r) = i}} \prod_{\{u,v\}\in E(\alpha)} \bA[\varphi(u),\varphi(v)] \quad&&\text{if $\alpha$ is a vector diagram with root $r$,}\\
        \bW_\alpha(\bA)[i,j] &= \sum_{\substack{\varphi\colon V(\alpha)\to [n]\\\varphi(r_1)=i,\varphi(r_2)=j}} \prod_{\{u,v\}\in E(\alpha)} \bA[\varphi(u),\varphi(v)] \quad&&\text{if $\alpha$ is a matrix diagram with roots $(r_1,r_2)$.}
    \end{alignat*}   
\end{definition}

\begin{definition}
    \label{def:z-basis}
    Define $z_\alpha(\bA)
    \in \R$, $\bz_\alpha(\bA)\in \R^n$, and $\bZ_\alpha(\bA)\in \R^{n\times n}$ by
    \begin{alignat*}{2}
        z_\alpha(\bA) &= \sum_{\substack{\varphi\colon V(\alpha)\hookrightarrow [n]}} \prod_{\{u,v\}\in E(\alpha)} \bA[\varphi(u),\varphi(v)] &&\text{if $\alpha$ is a scalar diagram,}\\[2mm]
        \bz_\alpha(\bA)[i] &= \sum_{\substack{\varphi\colon V(\alpha)\hookrightarrow [n]\\\varphi(r) = i}} \prod_{\{u,v\}\in E(\alpha)} \bA[\varphi(u),\varphi(v)] && \text{if $\alpha$ is a vector diagram with root $r$,}\\
        \bZ_\alpha(\bA)[i,j] &= \sum_{\substack{\varphi\colon V(\alpha)\hookrightarrow [n]\\\varphi(r_1)=i,\varphi(r_2)=j}} \prod_{\{u,v\}\in E(\alpha)} \bA[\varphi(u),\varphi(v)] \quad&& \text{if $\alpha$ is a matrix diagram with roots $(r_1,r_2)$.}
    \end{alignat*}    
\end{definition}

\noindent The only difference between the $w$- and $z$-bases
is the summation domain: \Cref{def:z-basis} sums over injective embeddings $\varphi$, whereas \Cref{def:w-basis} sums over all
embeddings.

Finally, we define two extensions of~\Cref{def:w-basis}
that we will need in the proofs.
The following allows us to use a different matrix on
each edge of the graph:

\begin{definition}
    \label{def:non-uniform-label}
    Let $\alpha$ be a matrix diagram with roots $(r_1,r_2)$ and $\bm \calA = (\bmA_e)_{e\in E(\alpha)}$ be
    such that $\bmA_e\in \R_\sym^{n\times n}$ for
    all $e\in E(\alpha)$. Define $\bW_\alpha(\bm \calA)\in \R^{n\times n}$ by
    \[
        \bW_\alpha(\bm \calA)[i,j] = \sum_{\substack{\varphi\colon V(\alpha)\to [n]\\\varphi(r_1)=i,\varphi(r_2)=j}} \prod_{e=\{u,v\}\in E(\alpha)} \bA_e[\varphi(u),\varphi(v)]\,.
    \]
\end{definition}

The following is an intermediate quantity
between~\Cref{def:w-basis} and~\Cref{def:z-basis} which only restricts the sum
over injective labelings on two vertices:

\begin{definition}
    Let $\bA\in \R_\sym^{n\times n}$, $\alpha$ be a scalar/vector/matrix diagram,
    $i,j\in [n]$, and $s,t\in V(\alpha)$. Define $w_\alpha^{s\neq t}\in \R$, $\bw_\alpha^{s\neq t}\in \R^n$, and $\bW_\alpha^{s\neq t}\in \R^{n\times n}$ by
    \begin{alignat*}{2}
        w_\alpha^{s\neq t}(\bA) &= \sum_{\substack{\varphi\colon V(\alpha)\to [n]\\\varphi(s)\neq \varphi(t)}} \prod_{\{u,v\}\in E(\alpha)} \bA[\varphi(u),\varphi(v)] && \text{if $\alpha$ is a scalar diagram,}\\
        \bw^{s\neq t}_\alpha(\bA)[i] &= \sum_{\substack{\varphi\colon V(\alpha)\to [n]\\\varphi(s)\neq \varphi(t)\\\varphi(r) = i}} \prod_{\{u,v\}\in E(\alpha)} \bA[\varphi(u),\varphi(v)] && \text{if $\alpha$ is a vector diagram with root $r$,}\\
        \bW^{s\neq t}_\alpha(\bA)[i,j] &= \sum_{\substack{\varphi\colon V(\alpha)\to [n]\\\varphi(s)\neq \varphi(t)\\\varphi(r_1)=i\\\varphi(r_2)=j}} \prod_{\{u,v\}\in E(\alpha)} \bA[\varphi(u),\varphi(v)] \quad && \text{if $\alpha$ is a matrix diagram with roots $(r_1, r_2)$.}
    \end{alignat*}
\end{definition}

\subsection{Partitions, change of basis, and M{\"o}bius inversion}
\label{sec:mobius}

While $(z_\alpha(\bA))_{\alpha\in \cA}$ and
$(w_\alpha(\bA))_{\alpha\in \cA}$ span the
same space of $S_n$-invariant polynomials in the entries of $\bm A$, some
properties are better expressed in one basis
than the other. 
Here we take a closer look at these bases and derive
change-of-basis formulas.

Given a set $S$, let $\calP(S)$ denote the set of all \emph{partitions} of $S$, sets of non-empty disjoint subsets of $S$ whose union is all of $S$.
We call the parts of a partition \emph{blocks}.
Each block is a set, and $P$ is the set of blocks, so we denote the blocks by $b \in P$.

For a (scalar, vector, or matrix) diagram $\al$ and a partition $P \in \calP(V(\al))$,
we define a new diagram $\al_P$
by identifying the vertices within each block of $P$ into a single vertex.
The vertices of $\al_P$ may thus be identified with the blocks of $P$.
$\al_P$ retains all edges of $\al$, which may become multiedges or self-loops.
The status of being one of the (0, 1, or 2) roots of $\al$ is inherited by the block containing that root. 

To change from the $w$- to the $z$-basis, we then simply sum over all partitions:

\begin{claim}
    \label{claim:wtoz}
    For all (scalar, vector, or matrix) diagrams $\alpha$, 
    \[
        w_{\alpha}(\bA) = \sum_{P \in \calP(V(\alpha))} z_{\alpha_P}(\bA)\,.
    \]
\end{claim}

Define the relation $\al \psdleq \beta$ on scalar diagrams if there exists a partition $P \in \calP(V(\beta))$
such that $\al = \beta_P$.
It is easy to check that this relation gives a partial ordering, inherited from the standard partial ordering on partitions.
We write $\alpha\prec \beta$ as a shorthand for
$\alpha\preceq \beta$ and $\alpha\neq \beta$.

\begin{lemma}
    \label{claim:mobius}
    There exist $(c_{\alpha, \beta})_{\alpha, \beta \in \calA}$ and $(c'_{\alpha, \beta})_{\alpha,\beta\in\calA}$ not depending on $n$ such that $c_{\alpha, \beta} \in \N$, $c'_{\al,\beta} \in \Z$ and for any $\alpha,\beta\in\calA$, 
    \[
        w_\beta(\bA) = \sum_{\al \preceq \beta} c_{\alpha, \beta} z_\alpha(\bA)\,,\qquad 
        z_{\beta}(\bA) = \sum_{\alpha \preceq \beta} c'_{\alpha,\beta} w_{\alpha}(\bA)\,.
    \]
\end{lemma}
\begin{proof}
    The coefficients in the left equation count symmetries in \cref{claim:wtoz}, i.e.,
    $c_{\al,\beta}$ equals the number of ways to choose a partition $P \in \calP(V(\beta))$
    such that $\beta_P$ is isomorphic to $\alpha$.
    Reciprocally, since $\preceq$ is a partial ordering, this
    transformation can be inverted using M{\"obius} inversion~\cite{Rota-1964-Foundations} on this poset.
    Although an explicit formula for $c'_{\al, \beta}$
    is available in terms of the
    combinatorial structure of the graphs, we will not need it in this paper.
\end{proof}

\subsection{The example of cycles: Moments versus free cumulants}
\label{sec:moments-free-cumulants}

The difference
between the $w$- and $z$-bases is illustrated nicely by the
special case of the diagrams $\sigma_q$ which are cycles of length $q\ge 1$.
In this case, $\frac 1n w_{\sigma_q}(\bA)$ and $\frac 1n z_{\sigma_q}(\bA)$ are versions of the limiting spectral moments and free cumulants, respectively, for finite-dimensional matrices.

Let $\calP(q)$ denote the set of partitions of $\{1,2,\dots, q\}$
and let $\textnormal{NC}(q)$ denote the subset of non-crossing partitions (partitions such that there does not exist $i < j < k < \el$ with $i,k$ in the same block and $j, \el$ in the same block, different from the one $i,k$ are in).
It is convenient to view these as partitions of the vertices of the $q$-cycle
so that the term \emph{non-crossing} may be interpreted visually: in a non-crossing partition, the blocks do not intersect one another when drawn as ``blobs'' inside the cycle.

In the $w$-basis, we have
\begin{align}
    \frac 1n w_{\sigma_q}(\bA) = \frac 1n \sum_{i_1,\ldots,i_q=1}^n \bA[i_1, i_2] \bA[i_2, i_3] \ldots \bA[i_q,i_1] = \frac 1n \Tr(\bA^q) = \frac 1n \sum_{i=1}^n \lambda_i(\bA)^q\,.\label{eq:empiricalDist}
\end{align}
Suppose that the expression in~\cref{eq:empiricalDist} converges as
$n\to\infty$ to the $q$th moment
$m_q\in \R$ of a limiting spectral distribution, $m_q = \int \lambda^q\,\d \mu(\lambda)$.

The free cumulants
are defined from the moments by a formula similar to
the classical cumulants vis-{\`a}-vis the moments of a random variable:

\begin{definition}[Free cumulant]\label{def:free-cumulant}    
    The free cumulants $(\kappa_q)_{q\ge 1}$ corresponding to $(m_q)_{q\ge 1}$ are defined implicitly by:
\begin{equation}\label{eq:free-cumulant}
    m_q = \sum_{\sig \in \textnormal{NC}(q)} \prod_{b \in \sig} \kappa_{|b|}\,.
\end{equation}
\end{definition}
\noindent The $\kappa_q$ can be computed explicitly in terms of the $m_q$ by applying M{\"o}bius inversion to~\cref{eq:free-cumulant}; see~\cref{eq:free-cumulant-explicit}.

Analogous to \cref{eq:empiricalDist} which is in the $w$-basis,
it appears to be folklore\footnote{This is for example explicitly stated in \cite[Theorem 1 and Appendix D.1]{MFCKMZ-2019-PlefkaExpansionOrthogonalIsing}.} that if $\bA$ is drawn
from an orthogonally invariant matrix ensemble
with free
cumulants $(\kappa_q)_{q\ge 1}$, then
\begin{align}
    \frac 1n \E z_{\sigma_q}(\bA) \underset{n\to\infty}{\longrightarrow} \kappa_q\,.\label{eq:freeCumulantRelation}
\end{align}
The quantity $\frac 1n z_{\sigma_q}(\bA)$ has also been called the $q$th \emph{injective trace} of $\bA$.
Below in \cref{lem:diagonal-w-z}, we prove \cref{eq:freeCumulantRelation} using a change of basis from $w$ to $z$.

For example, below are the parameters $m_q$ and $\kappa_q$
for the \goe and the \rom, whose limiting empirical spectral distribution
are the Wigner semicircle distribution and the Rademacher distribution, respectively.

\begin{claim}\label{claim:rom-cumulants}
    Let $\textnormal{Cat}(k) \colonequals \frac{1}{k + 1}\binom{2k}{k}$ be the $k$th Catalan number.
    For the \goe, the limiting spectral moments and free cumulants are:
    \begin{align*}
        m_q = \left\{\begin{array}{ll}
            \textnormal{Cat}(q/2) & \textnormal{if $q$\text{ is even}}\\
            0 & \textnormal{if $q$\text{ is odd}}
        \end{array}
        \right\}, \qquad 
        \kappa_q = \left\{\begin{array}{ll}
            1 & \textnormal{if $q = 2$}\\
            0 & \textnormal{otherwise}
        \end{array}\right\}.
    \end{align*}        
    For the \rom, the limiting spectral moments and free cumulants are:
    \begin{align}
    m_q = \left\{\begin{array}{ll}
        1 & \textnormal{ if $q$ is even}\\
        0 & \textnormal{ if $q$ is odd}
    \end{array}\right\}, \qquad
    \kappa_q = \left\{\begin{array}{ll}
        (-1)^{q/2-1} \textnormal{Cat}(q/2 - 1) & \textnormal{ if $q$ is even}\\
        0 & \textnormal{ if $q$ is odd}
    \end{array}\right\}.\label{eq:freeCumulantROM}
\end{align}
\end{claim}

\subsection{Solving equations in the traffic distribution}

The traffic distribution is defined as the limiting values of all $w$-basis polynomials, but we show now how it can be derived from various combinations of limits of $w$- and $z$-basis polynomials.
In our other arguments, we will also find it convenient to describe
the traffic distribution of sequences of matrices (random or deterministic)
using the two bases simultaneously.
While~\Cref{claim:mobius} shows
that we could in principle express all these results
in a single basis, this would involve
precisely tracking 
very complicated combinatorial coefficients
(in fact, this was a major technical obstacle in
previous diagrammatic analyses of AMP).

As we have discussed, when a matrix satisfies the strong cactus property,
its traffic distribution is determined by its values on the cactus diagrams (equivalently, by the diagonal distribution),
and when it satisfies the factorizing strong cactus property,
its traffic distribution is determined by the spectral distribution.
We show that one can use either the $w$-basis or $z$-basis for these determinations.

\begin{lemma}
    \label{lem:diagonal-w-z}
    Suppose that $\bA = \bA^{(n)}$ satisfies the
    weak cactus property, i.e., for all $\alpha\in \calE\setminus \calC$,
    \[
        \frac 1n \E_{\bmA} z_\alpha(\bmA)\underset{n\to\infty}{\longrightarrow} 0\,.
    \]
    Then the following are equivalent:
    \begin{enumerate}[(i)]
        \item For all $\sigma \in \calC$ there exists $m_\sigma \in \R$ such that $\frac 1n \E_\bA w_\sigma(\bA) \underset{n\to\infty}{\longrightarrow} m_\sigma$.
        \item For all $\sigma \in \calC$ there exists $k_\sigma \in \R$
        such that $\frac 1n \E_\bA z_\sigma(\bA) \underset{n\to\infty}{\longrightarrow} k_\sigma$.
    \end{enumerate}
    Furthermore, when they exist, $(m_\sigma)_{\sigma \in \calC}$ and $(k_\sigma)_{\sigma \in \cC}$ determine each other.
    The following are also equivalent:
    \begin{enumerate}[(i)]
        \item There exist real numbers $(m_q)_{q \in \N}$ such that for all $\sig \in \calC$, $\frac 1n \E_\bA w_\sig(\bA) \underset{n\to\infty}{\longrightarrow} \prod_{\rho \in \cyc(\sig)} m_{|\rho|}$.
        \item There exist real numbers $(\kappa_q)_{q \in \N}$ such that for all $\sig \in \calC$, $\frac 1n \E_\bA z_\sig(\bA) \underset{n \to \infty}{\longrightarrow} \prod_{\rho \in \cyc(\sig)} \kappa_{|\rho|}$.
    \end{enumerate}
    Furthermore, when they exist, $(m_q)_{q \in \N}$ and $(\kappa_q)_{q \in \N}$ are related by \cref{eq:free-cumulant}.
\end{lemma}

We use the
following observation which will be used repeatedly in \cref{sec:universality}:

\begin{lemma}\label{lem:contract-2ec}
    If $\al \in \calE$ and $\beta \psdleq \al$, then $\beta\in\calE$.
\end{lemma}

\begin{proof}[Proof of~\Cref{lem:contract-2ec}]
    By Menger's theorem, a graph is 2-edge-connected 
    if and only if there exist two edge-disjoint 
    paths between every pair of distinct vertices.
    These paths are maintained when $\al$ is contracted into $\beta$.
\end{proof}

\begin{proof}[Proof of~\Cref{lem:diagonal-w-z}]
    {\em (ii)} $\implies$ {\em (i)}. Using \Cref{claim:wtoz},
    \[
        \frac 1n \E_{\bA} w_\sigma(\bA) = \frac 1n \sum_{\beta \psdleq \sigma} c_{\beta, \sigma} \E_{\bA} z_\beta(\bA)\,.
    \]
    Every diagram $\beta \psdleq \sigma$ remains 2-edge-connected by \cref{lem:contract-2ec}.
    There are only finitely many terms in the
    sum, so we can directly
    take the $n\to\infty$ limit
    and use the assumptions
    to obtain that $\frac 1n \E_\bA w_\sigma(\bA)$ converges to $\sum_{\beta \psdleq \sig} c_{\beta, \sig} k_\beta$.

    Note that by the weak cactus property,
    the only asymptotically nonzero $\beta \preceq \sig$ are when $\beta$ is a cactus.
    Assuming furthermore that $k_\beta = \prod_{\rho \in \cyc(\beta)} \kappa_{|\rho|}$ factors over the cycles of each cactus $\beta$ we will derive the second part of the lemma.

    Using the more specific result of \Cref{claim:wtoz}, we have
    \begin{align*}
        \lim_{n \to \infty} \frac{1}{n} \E_{\bA} w_{\sigma}(\bA)
        &= \sum_{P \in \calP(V(\sigma))} \lim_{n \to \infty} \frac{1}{n} \E_{\bA} z_{\sigma_P}(\bA)
        \intertext{Since $\bA$ has the weak cactus property and $\sigma$ is a cactus, only the terms where $\sigma_P$ is a cactus contribute. These are precisely the terms where $P$ restricted to each cycle of $\sigma$ is non-crossing. Given $P_{\rho} \in \NC(V(\rho))$ for each $\rho \in \cyc(\sigma)$, let us write $P(P_{\rho}: \rho \in \cyc(\sigma))$ for the partition obtained by composing these partitions of each cycle, and let us write, following our previous notation, $\cyc(\rho_{P_{\rho}})$ for the set of cycles created when the single cycle $\rho$ is contracted according to $P_{\rho}$. Then, we have}
        &= \sum_{\substack{P_{\rho} \in \NC(V(\rho)) \\ \text{for each } \rho \in \cyc(\sigma)}} \lim_{n \to \infty} \frac{1}{n} \E_{\bA} z_{\sigma_{P(P_{\rho}: \rho \in \cyc(\sigma))}}(\bA) \\
        &= \sum_{\substack{P_{\rho} \in \NC(V(\rho)) \\ \text{for each } \rho \in \cyc(\sigma)}} \prod_{\rho \in \cyc(\sigma)} \prod_{\pi \in \cyc(\rho_{P_{\rho}})} \kappa_{|\pi|} \\
        &= \prod_{\rho \in \cyc(\sigma)} \left(\sum_{P \in \NC(V(\rho))} \prod_{\pi \in \cyc(\rho_P)} \kappa_{|\pi|}\right) \\
        &= \prod_{\rho \in \cyc(\sigma)} m_{|\rho|}.
    \end{align*}
    Thus we have the claimed factorization.
    Further, the coefficients $m_q$ and $\kappa_q$ indeed have the relation between moments and free cumulants from~\cref{eq:free-cumulant}:
    \[
        m_q = \sum_{\sig \in \NC(q)} \prod_{b \in \sig} \kappa_{|b|}\,.
    \]
    
    \noindent {\em (i) $\implies$ (ii)}. This direction uses a recursive change of basis technique that will be very useful in~\Cref{sec:universality}.
    Using~\Cref{claim:mobius} in both directions, we get
    \begin{align*}
        \frac 1n \E_{\bmA} z_\sigma(\bmA) &= \frac 1n \sum_{\substack{\beta\preceq \sigma\\\beta\in\calC}} c'_{\beta,\sigma} \E_{\bmA} w_\beta(\bmA) + \frac 1n \sum_{\substack{\beta\prec \sigma\\\beta\in\calE\setminus\calC}} c'_{\beta,\sigma} \E_{\bmA} w_\beta(\bmA)\\
        &= \frac 1n \sum_{\substack{\beta\preceq \sigma\\\beta\in\calC}} c'_{\beta,\sigma} \E_{\bmA} w_\beta(\bmA) + \frac 1n \sum_{\substack{\beta\prec \sigma\\\beta\in\calE\setminus\calC}} c'_{\beta,\sigma} 
        \sum_{\alpha\preceq \beta} c_{\alpha,\beta} \E_{\bmA} z_\alpha(\bmA)\\
        &= \frac 1n \sum_{\substack{\beta\preceq \sigma\\\beta\in\calC}} c'_{\beta,\sigma} \E_{\bmA} w_\beta(\bmA) +  \frac 1n  \sum_{\alpha\prec \sigma}\left(\sum_{\substack{\beta\in\calE\setminus\calC\\ \alpha \preceq \beta \prec \sigma}} c'_{\beta,\sigma} c_{\alpha,\beta}\right) \E_{\bmA} z_\alpha(\bmA)
    \end{align*}
    Note that every diagram in this expansion
    remains 2-edge-connected by~\Cref{lem:contract-2ec}.
    
    Every contraction identifying
    a non-empty subset of vertices decreases
    the number of vertices in the graph, and
     the $w$ and $z$ bases coincide for 1-vertex
    graphs. Therefore, we can
    apply the same steps inductively on terms
    for which $\alpha\in\calC$ to finally obtain
    \[
        \frac 1n \E_{\bmA} z_\sigma(\bmA) = \frac 1n \sum_{\substack{\beta\preceq \sigma\\\beta\in \calC}} c''_{\beta,\sigma} \E_{\bmA} w_\beta(\bmA) + \frac 1n \sum_{\substack{\beta\preceq \sigma\\\beta\in \calE\setminus \calC}} c''_{\beta,\sigma} \E_{\bmA} z_\beta(\bmA)\,.
    \]
    for some coefficients $\{c''_{\alpha,\beta}\}$ independent of $n$.
    Take the $n\to\infty$ limit
    to obtain
    \[
        \lim_{n\to\infty} \frac 1n \E_{\bmA} z_\sigma(\bmA) = \sum_{\substack{\beta\preceq \sigma\\\beta\in \calC}} c''_{\beta,\sigma} m_\beta\,,
    \]
    which finishes the proof of the first equivalence.
    Assuming furthermore that $m_\beta$ factors over the cycles of each cactus $\beta$,
    then $\frac 1n \E_\bA z_\sig(\bA)$
    also asymptotically factors over its cycles: $\frac 1n \E_\bA z_\sig(\bA) \longrightarrow \prod_{\rho \in \cyc(\sig)} \kappa_{|\rho|}$ for some numbers $\kappa_q$.
    This is because the cactuses $\beta \psdleq \sig$ still only arise by contracting a separate non-crossing partition for each cycle of $\sig$, and so we can perform the above recursive analysis separately inside each cycle.
\end{proof}

The following lemma shows that the properties
of graph polynomials
we will establish for delocalized deterministic matrices
in~\Cref{sec:universality} characterize their traffic distribution.
We emphasize our use of a combination of assumptions on limits of the $w$- and $z$-bases that makes this formulation convenient.

\begin{lemma}
    \label{lem:eqDiffMobius}
    Suppose that $\bA = \bA^{(n)}$ satisfies:
    \begin{enumerate}
        \item The weak cactus property, i.e., that for all $\alpha\in \calE\setminus \calC$, $\frac 1n \E_{\bmA} z_\alpha(\bmA)\underset{n\to\infty}{\longrightarrow} 0$.
        \item For all $\alpha\in \calA\setminus \calE$, $\frac 1n \E_{\bmA} w_\alpha(\bmA)\underset{n\to\infty}{\longrightarrow} 0$.
        \item For all $\sigma\in \calC$,
        there exists $m_\sigma\in \R $ such that
        $\frac 1n \E_{\bmA} w_\sigma(\bmA)\underset{n\to\infty}{\longrightarrow} m_\sigma$.
    \end{enumerate}
    Then the traffic distribution of $\bmA$
    exists
    and only depends on $\{m_\sigma : \sigma\in\calC\}$.
\end{lemma}

\begin{proof}
    We want to show that for every $\alpha\in \calA$, $ \lim_{n\to\infty} \frac 1n \E_{\bmA} w_\alpha(\bmA)$ exists and
    only depends on $\{m_\sigma : \sigma\in\calC\}$. By assumption, it suffices to prove it for $\alpha\in \calE\setminus \calC$. By~\Cref{claim:mobius},
    \[
        \frac 1n \E_{\bmA} w_\alpha(\bmA) = \frac 1n \sum_{\beta\preceq \alpha} c_{\beta,\alpha} \E_{\bmA} z_\beta(\bmA)\,.
    \]
    By~\Cref{lem:contract-2ec}, every $\beta$
    in the support of the sum is 2-edge-connected. If $\beta\in\calC$, then the value of $\lim_{n\to\infty} \frac 1n  \E_{\bmA} z_\beta(\bmA)$ exists and only
    depends on $\{m_\sigma : \sigma\in\calC\}$
    by~\Cref{lem:diagonal-w-z}. Otherwise,
    $\beta\in\calE\setminus \calC$, and $\lim_{n\to\infty} \frac 1n  \E_{\bmA} z_\beta(\bmA) = 0$ by assumption. This
    implies that $\lim_{n\to\infty}\frac 1n \E_{\bmA} w_\alpha(\bmA)$ exists and 
    only depends on $\{m_\sigma : \sigma\in\calC\}$, which concludes the proof.
\end{proof}
\noindent
Note that, more generally, by~\Cref{lem:diagonal-w-z}, the same statement will hold with Condition~3 of~\Cref{lem:eqDiffMobius} taken in terms of either the $w$- or $z$-basis.

\subsection{Products and concentration of traffic observables}

Recall that the traffic distribution specifies the limits of $\frac 1n \E_\bA w_\al(\bA)$
for all $\al \in \calA$.
In all of the random matrix models we consider,
these expectations are highly concentrated.
We say that {\em the traffic distribution concentrates for $\bA$} if the
following property holds, studied in~\cite{male2020traffic}.

\begin{definition}\label{def:traffic-concentration}
    Let $\bA = \bA^{(n)} \in \R^{n \times n}_\sym$ and assume that the traffic distribution of $\bA$ exists. We say that the traffic distribution concentrates for $\bA$
    if for all $k \geq 2$ and $\al_1, \dots, \al_k \in \calA$,
    \[
        \lim_{n \to \infty} \E_{\bA} \left[ \prod_{j=1}^k \frac 1n w_{\alpha_j}(\bA)\right] = \prod_{j=1}^k \lim_{n \to \infty} \frac 1n \E_{\bA} w_{\alpha_j}(\bA)\,.
    \]
\end{definition}
\noindent
The case $k = 2$ and $\alpha_1 = \alpha_2 = \alpha$ of the definition specializes to the statement:
\begin{lemma}\label{lem:traffic-l2}
    Let $\bA = \bA^{(n)} \in \R^{n \times n}_{\sym}$ have traffic distribution $\calD$. 
    If the traffic distribution concentrates for $\bA$, then $\frac 1n \E_{\bA} w_\al(\bA)$ converges to $\calD(\al)$ in $L^2$.
\end{lemma}
The full condition may be viewed as a strengthening of this straightforward notion of concentration.
We note that the product of several $w$-basis polynomials is equivalent to taking the disjoint union of their diagrams:
\[ w_{\alpha_1}(\bA) \cdots w_{\alpha_k}(\bA) = w_{\alpha_1 \sqcup \cdots \sqcup \alpha_k}(\bA).\]
Therefore, \cref{def:traffic-concentration} says that the values of disconnected diagrams asymptotically factor over the components.
This justifies defining
$\cA$ and the traffic distribution to include only connected diagrams.
The following shows that concentration may equally well be considered in the $z$-basis.

\begin{lemma}[{\cite[Lemma 2.9]{male2020traffic}}] \label{lem:concentration-z}
    Let $\bA = \bA^{(n)} \in \R^{n \times n}_\sym$ and assume that the traffic distribution of $\bA$ exists. The traffic distribution concentrates for $\bA$ if and only if, for all $k \geq 2$ and $\al_1, \dots, \al_k \in \calA$,
    \[
        \lim_{n \to \infty} \E_{\bA} \left[ \prod_{j=1}^k \frac 1n z_{\alpha_j}(\bA)\right] = \prod_{j=1}^k \lim_{n \to \infty} \frac 1n \E_{\bA} z_{\alpha_j}(\bA)\,.
    \]
\end{lemma}

For vector diagrams, the componentwise or Hadamard product is
\[
    \bmw_{\al_1}(\bA) \cdots \bmw_{\al_k}(\bA) = \bw_{\al_1 \oplus \cdots \oplus \al_k}(\bA)\,,
\]
where $\al_1 \oplus \cdots \oplus \al_k$ is the diagram formed by taking the disjoint union of $\al_1$ through $\al_k$ and then identifying the roots together into a single root.
We sometimes refer to this operation as \emph{grafting} $\al_1, \dots, \al_k$ at the root.

\section{Traffic Distributions of Random Matrices}
\label{sec:trafficRandom}

As both a technical preliminary for our results and useful background, this section describes the traffic distributions of several common random matrix ensembles. A common theme is that all of these classical models satisfy the strong cactus property. Most of these results have appeared previously in the literature, though we provide some extensions and new interpretations.

\subsection{Wigner random matrices}
\label{sec:traffic-wigner}

A \emph{Wigner matrix} is
a random symmetric matrix with i.i.d.\ entries on and above the diagonal.
Changes to the diagonal entries such as setting them to zero (which is the convention used in some works), or taking the diagonal variances to be twice the off-diagonal ones (as in the GOE model), do not affect the results.

The limiting traffic distribution of a sequence of Wigner matrices was derived by Male~\cite{male2020traffic}, by generalizing the combinatorial
proof of the semicircle limit theorem for the limiting spectral distribution~\cite{AGZ-2010-RandomMatrices}.
The same result was re-discovered in~\cite{jones2025fourier} in the context of analyzing
pGFOM on such matrices.
\begin{theorem}[Traffic distribution of Wigner matrices]
\label{thm:goe}
Let $\nu$ be a probability measure on $\R$ with all moments finite, mean 0, and variance 1.
For all $n\ge 1$, let $\widetilde{\bA}^{(n)} \in \R^{n \times n}_{\sym}$ have entries on and above the diagonal drawn i.i.d.\ from $\nu$.
Define $\bA^{(n)} \defeq \frac{1}{\sqrt{n}} \widetilde{\bA}^{(n)}$.
Then, for all $\al \in \calA$,
\[
    \lim_{n \to \infty}\frac{1}{n} \E z_\al(\bA^{(n)}) = \begin{cases}
        1 & \text{if $\al$ is a cactus of 2-cycles}, \\
        0 & \text{otherwise}.
    \end{cases}
\]
\end{theorem}
\noindent
The same result holds for normalized GOE matrices.
Note that a cactus of 2-cycles may equivalently be viewed as a ``doubled tree'', a tree where every edge is repeated exactly twice, which is the formulation used in the previous works~\cite{male2020traffic,jones2025fourier}.

Thus, sequences of Wigner matrices have the factorizing strong cactus property, with the especially simple sequence of free cumulants $\kappa_2 = 1$ and $\kappa_q = 0$ for all $q \neq 2$. These are also the free cumulants of the semicircle law, which is the limiting eigenvalue distribution of $\bA^{(n)}$.

\subsection{Orthogonally invariant random matrices}
\label{sec:rot-inv}

Let the orthogonal group $O(n)$ act on $\R^{n \times n}_{\sym}$ by conjugation, with $\bQ \in O(n)$ acting as $\bQ \cdot \bA \defeq \bQ^{\top}\bA\bQ$.
Let $\mu$ denote a probability measure on $\R^{n \times n}_{\sym}$ that is invariant under this action of $O(n)$.
In this case, we call $\bA \sim \mu$ an \emph{orthogonally invariant random matrix}.

If $\mu$ has a density on $\R^{n \times n}_{\sym}$, an equivalent condition is that the density at $\bA \in \R^{n \times n}_{\sym}$ depends only on the unordered multiset of eigenvalues of $\bA$.
An important class of
examples in physics is
given by {\em matrix models
with potential} $V:\R\to\R$,
whose density
is proportional to
$\exp(-\Tr V(\bm A))$.
For example, the
GOE model corresponds to $V(t)=t^2/2$.
We will come back to these examples 
in~\Cref{sec:feynman-physics}.

For the
complex-valued variant
where $O(n)$ is replaced by the unitary group $U(n)$, the limiting traffic distribution of such \emph{unitarily invariant} random matrices is described in~\cite[Theorem~1.1]{cebron2024traffic}.
The same description holds in the orthogonal case.
The proof is a straightforward generalization of the unitarily invariant case, but for the sake of completeness we present it in detail in \Cref{app:weingarten}.

\begin{theorem}[Traffic distribution of orthogonally invariant random matrices]
    \label{thm:moments}
    Let $\bA^{(n)} \in \R^{n \times n}_\sym$ be a sequence of orthogonally invariant random matrices that converges in tracial moments in $L^2$ to a probability measure $\mu$.
    Then, for all $\al \in \calA$,
    \begin{equation}
        \lim_{n \to \infty} \frac{1}{n} \E z_\al(\bA^{(n)}) = \begin{cases}
            \displaystyle\prod_{\sig \in \cyc(\al)} \kappa_{|\sig|} & \text{if }\al\in\cC, \\
            0 & \text{otherwise}.
        \end{cases}
        \label{eq:traffic-dist-invar}
    \end{equation}
    where $\kappa_q$ is the $q$th \emph{free cumulant} of $\mu$ (\Cref{def:free-cumulant}),
    and $|\sig|$ denotes the length of the cycle.
\end{theorem}

\cref{eq:traffic-dist-invar} shows
that the factorizing strong cactus property holds for orthogonally invariant random matrices,
and in particular their limiting traffic distribution is supported only on cactus diagrams in the $z$-basis.

Actually, in this case the strong cactus
property is non-trivial only
for the Eulerian diagrams,
since the non-Eulerian ones have identically zero expectation for each fixed dimension $n$:
\begin{claim}\label{lem:eulerian}
    Let $\bA^{(n)} \in \R^{n \times n}_\sym$ be an orthogonally invariant random matrix.
    Then for all $\al \in \calA$ which are not Eulerian,
    $\E z_\al(\bA^{(n)}) = 0$.
\end{claim}
\noindent 
We show this at the beginning of our proof in~\Cref{app:weingarten}.

Both the proof of~\cite[Theorem 1.1]{cebron2024traffic} and our proof of~\Cref{thm:moments} are based on the \emph{Weingarten calculus}, a combinatorial description of the entrywise moments of Haar-distributed matrices from a matrix group.
In~\cref{sec:feynman-physics}, we present an alternative (albeit non-rigorous) derivation of~\cref{thm:moments} using the {\em Feynman diagram} method from physics.
Arguably, the combinatorics of the Feynman diagram method is simpler than that of the Weingarten calculus proof.

\subsection{Block-structured random matrices}
\label{sec:wigner-block}

Wigner random matrices and orthogonally invariant random matrices both
extend the GOE in different directions, while still satisfying the
factorizing strong cactus property.
We now consider a third
generalization, block matrices, which typically do
{\em not} satisfy
the factorizing property.

Fix $q\in \N$. For $r,c \in [q]$, let
$\bm A_{r,c} = \bA_{r,c}^{(n)} \in \R_{\sym}^{n/q \times n/q}$ be a sequence of
random matrices with $\bm A_{r,c} = \bm A_{c,r}$.
The corresponding {\em block matrix model} is the symmetric $n$-by-$n$ matrix
whose rows and columns are partitioned into
blocks of sizes $n/q$ which has blocks $(\bA_{r,c})_{r,c \in [q]}$.
We let $\block(i) \in [q]$ denote the block label of
$i \in [n]$.

The simplest example of a block matrix model is the {\em block GOE} model,
which has previously been studied in the context of the Generalized AMP algorithm~\cite{javanmard2013state}.\footnote{In this paper, we study a slightly more symmetric variant, in which the blocks themselves are symmetric. This modification is made purely for technical reasons, since we work in our other definitions only with symmetric matrices.}

\begin{definition}[Block GOE model]\label{def:blockgoe}
    Let $q\in \N$ and let $\matSig \in \R^{q \times q}$ be a symmetric with nonnegative entries.
    For $1 \leq r \leq c \leq q$, let $\bm A_{r,c}\in \R_\sym^{n/q \times n/q}$ be a symmetric random matrix whose entries on and above the diagonal are independent Gaussians with mean $0$ and variance $\bm\Sigma[r,c]/n$, and let $\bm A_{r,c}=\bm A_{c,r}$ for $q \geq r > c \geq 1$.
    The \emph{block GOE} model $\bA \sim \bGOE(n,\bm\Sigma)$ is the block matrix with blocks $(\bm A_{r,c})_{r,c\in[q]}$.
\end{definition}

Following the arguments of~\cite{male2020traffic,jones2025fourier}, one can prove that the block GOE model with fixed parameter $\bm\Sigma$ satisfies the strong cactus property.
Indeed, as in~\Cref{thm:goe}, it is still only the doubled trees or cactuses of 2-cycles that have non-zero value in the traffic distribution.
However, these values depend non-trivially on $\bm\Sigma$, and in general the block GOE model does {not} satisfy the factorizing strong cactus property.\footnote{If the row sums of $\bm\Sigma$ are constant, yielding what is sometimes called a \emph{generalized Wigner matrix}, then up to rescaling the traffic distribution is again that of the GOE and the factorizing property does hold.}

\paragraph{Traffic independence.} 
We study block models through the notion of {\em traffic independence}.
Traffic independence was
introduced by Male~\cite{male2020traffic}
as a generalization of
free independence of matrices. Free
independence is a property
of the mixed traces of
several random matrices
(in our notation, these traces are
represented by cycle diagrams),
whereas traffic independence
is a property of {\em all}
diagrams. Using this concept, below we 
prove a general result that 
block-structured matrices have the
strong cactus property 
provided that {\em (i)}
each of the blocks separately has
the strong cactus property,
and {\em (ii)} those blocks
are asymptotically traffic
independent.

For a sequence of symmetric matrices $(\bA_1, \dots, \bA_k) \in (\R^{n \times n}_{\sym})^k$, we generalize the graph polynomials to 
$w_\al(\bA_1, \dots, \bA_k)$ and $z_\al(\bA_1, \dots, \bA_k)$, where $\al$ is a multigraph whose edges are additionally colored by $\bA_1, \dots, \bA_k$.
The graph polynomial defined by $\al$ uses the entries of $\bA_i$ on each edge whose color is $\bA_i$, as in~\cref{def:non-uniform-label}.

Define a \emph{colored component} to be a maximal connected subgraph of $\al$ whose edges all have the same label $\bA_i$.
Let $\CC(\al)$ denote the set of colored components.
Define the \emph{graph of colored components} $\GCC(\al)$ to be the bipartite graph $\chi$ with:
\begin{align*}
    V(\chi) &= \CC(\al) \cup \{u \in V(\al): u \text{ belongs to at least two colored components}\}, \\
    E(\chi) &= \{(\calC, u): u \text{ belongs to the colored component }\calC\}.
\end{align*}

\begin{definition}[Traffic independence]
    \label{def:traffic independence}
    Let $(\bA_1, \dots, \bA_k) = (\bA_1^{(n)},\dots, \bA_k^{(n)}) \in (\R^{n \times n}_{\sym})^k$ be sequences of symmetric random matrices, with respective limiting traffic distributions
    $\cD_1, \ldots, \cD_k$.
    We say that $\bA_1,\dots, \bA_k$ are {asymptotically traffic independent} if,
    for all connected undirected multigraphs $\al$ with edges labeled by $\bA_1, \dots, \bA_k$,
    \[
        \lim_{n \to \infty} \frac 1n \E_{\bA_1, \dots, \bA_k}z_\al(\bA_1, \dots, \bA_k) = \begin{cases}
            \displaystyle\prod_{\calC \in \CC(\al)} \cD_{i(\calC)}(\calC) & \textnormal{if GCC($\al$) is a tree}\\
            0 & \textnormal{otherwise}
        \end{cases}
    \]
    Here, $i(\calC)$ denotes the matrix label associated with the colored component $\calC$.
\end{definition}

Next, we prove that traffic independence of the blocks preserves the strong cactus property:

\begin{proposition}
    \label{prop:block-matrix-strong-cactus}
    Let $q \in \N$. For $r,c\in [q]$, let $\bA_{r,c} = \bA_{r,c}^{(n)} \in \R^{n/q \times n/q}_\sym$ be a sequence of symmetric random matrices such that $\bm A_{r,c} = \bm A_{c,r}$.
    Assume that each $\bm A_{r,c}$ has a limiting traffic distribution that satisfies the strong cactus property and $(\bm A_{r,c})_{1 \leq r\leq c \leq q}$
    are asymptotically traffic independent. Then, the block matrix $\bmA\in \R^{n\times n}_{\sym}$ with blocks $(\bA_{r,c})_{r,c\in [q]}$ also has a limiting traffic distribution that satisfies the strong cactus property.
\end{proposition}

\begin{proof}
    Let $\al \in \calA$.
    In the graph polynomial $z_\al(\bA)$ we partition the sum based on the block of each vertex:
    \begin{align*}
        \frac 1n z_\al(\bA) &= \frac 1n \sum_{\chi : V(\al) \to [q]} \sum_{\substack{i : V(\al) \to [\frac nq]}} \prod_{uv \in E(\al)} \bA_{\chi(u),\chi(v)}[i(u), i(v)]\,.
    \end{align*}
    We can interpret the inner summation as a generalized graph polynomial whose edges are labeled by the matrices $\bA_{r,c}$.
    Call this diagram $\al_\chi$ and write:
    \[
        \frac 1n z_\al(\bA) = \sum_{\chi : V(\al) \to [q]} \frac 1n z_{\al_\chi}((\bA_{r,c})_{r,c\in [q]})\,.
    \]
    Taking the expectation and the limit $n\to \infty$, by traffic independence, all limits exist (so the block matrix has a limiting traffic distribution), and the nonzero terms on the right-hand side are those for which $\GCC(\al_\chi)$ is a tree.
    By the strong cactus property for each $\bA_{rc}$, each colored component must be a cactus.
    Therefore, any nonzero $\al$ is formed by gluing several cactuses along a tree, which forms a bigger cactus.
\end{proof}

Finally, 
traffic independence is shown in~\cite{male2020traffic} to hold quite generally for independent random matrices $\bA_i$, each of which has a permutation-invariant distribution.

\begin{theorem}[{\cite[Theorem 1.8]{male2020traffic}}]
\label{thm:male-traffic-independence}
    Let $\bA_1, \dots, \bA_k \in \R_\sym^{n \times n}$ be independent random matrices such that for each $i\in [k]$,
    \begin{enumerate}[(i)]
        \item The law of $\bA_i \in \R^{n \times n}_\sym$ is $S_n$-invariant (i.e., invariant under the simultaneous action of $S_{n}$ on the rows and columns of $\bA_i$).
        \item The limiting traffic distribution of $\bA_i$ exists.
        \item The traffic distribution concentrates for $\bA_i$ (\cref{def:traffic-concentration}).
    \end{enumerate}
    Then $\bA_1, \dots, \bA_k$ are asymptotically traffic independent.
\end{theorem}

Together with~\Cref{prop:block-matrix-strong-cactus}, \Cref{thm:male-traffic-independence}
implies that block-structured
matrices with independent blocks, each
satisfying the strong cactus property and Conditions {\em (i), (ii), (iii)} also satisfy
the strong cactus property (such as the block
GOE matrix).
We note that Condition~{\em (i)} can be ensured by applying an independent random permutation to the rows and columns of each $\bA_i$.
Condition~{\em (iii)} is proven for orthogonally invariant random matrices in~\Cref{lemma:full-factorization}.

\section{Universality for Deterministic Matrices}
\label{sec:universality}

Recall the definition of puncturing (\Cref{def:puncturing}) and
of the \rrom~(\Cref{def:rom}).
Our main theorem in this section is:

\begin{theorem}
    \label{thm:universality-orthogonal}
    Let $\bH = \bH^{(n)}\in\R_{\sym}^{n\times n}$ be a sequence of symmetric orthogonal matrices such that
    \begin{align}
        \max_{1\le i\le j\le n} |\bmH[i,j]|\le n^{-\frac 12 + o(1)}\,.\label{eq:delocOrtho}
    \end{align}
    Then, the limiting traffic distribution of the puncturing of $\bm H$ 
    exists and equals that of the \rrom.
\end{theorem}

\noindent \Cref{thm:universality-orthogonal}
directly applies to $\bH$ being the sequence
of Walsh-Hadamard matrices, discrete sine
transform matrices, or discrete cosine 
transform matrices.
\Cref{thm:universality-orthogonal}
follows from the more general~\Cref{thm:universality-new} below, which applies
to symmetric matrices that are not
necessarily orthogonal, but have a
limiting diagonal distribution and 
satisfy
a generalized delocalization assumption.

\begin{assumption}
    \label{assumption:punctured}
    Let $\bH=\bH^{(n)} \in \R^{n \times n}_{\sym}$ and $\eps = \eps^{(n)}>0$.
    We introduce the assumptions:
    \begin{alignat}{2}
        \|\bH\|&\le 1, &&\label{eq:assNorm}\\[3.5mm]
        \max_{1\le i<j\le n} |\bW_\alpha(\bH)[i,j]|&\le \eps&\qquad&\text{for each open cactus $\alpha$ (\Cref{def:open-cactus})}, \label{eq:assOpen}\\
        \frac 1 {\sqrt n}\|\mathbf \Pi \bmw_\sigma(\bH) \|_2&\le \eps&\qquad&\text{for all $\sigma\in\calC_1$}, \label{eq:assOne}
    \end{alignat}
    where $\mathbf \Pi=\mathbf \Pi^{(n)} = \mathbf I - \frac 1n \bm 1 \bm 1^\top$ denotes the projection orthogonal to the all-ones direction. 
\end{assumption}

\noindent For example, one of the constraints of~\cref{eq:assOpen} is that $|\bH^k[i,j]| \leq \eps$ uniformly for all $k,n \in \N$ and distinct $i,j \in [n]$ (a bound which is uniform in $n,i,j$ but may depend on $k$ would also be sufficient, but we omit this for simplicity).

\begin{theorem}[Universality]
    \label{thm:universality-new}
    Let $\bH = \bH^{(n)}\in\R^{n\times n}_\sym$, $\bm A$ be the puncturing
    of $\bH$, and $\eps^{(n)} > 0$.
    \begin{enumerate}
        \item If $\bH$ satisfies~\cref{eq:assNorm,eq:assOpen}, then for all $\alpha\in \calE\setminus \calC$, 
        \[
            \frac 1n |z_\alpha(\bH)|\le O_\alpha\left( \varepsilon^{(n)} + \frac 1 {\sqrt n}\right)\qquad\text{and}\qquad\frac 1n |z_\alpha(\bA)|\le O_\alpha\left( \varepsilon^{(n)} + \frac 1 {\sqrt n}\right)\,.
        \]
        In particular, if $\eps^{(n)}=o(1)$, then
        both $\bH$ and $\bA$ satisfy the weak cactus property.
        \item If $\bH$ satisfies~\cref{eq:assNorm,eq:assOpen,eq:assOne}, then for all $\alpha\in \calA\setminus \calE$, 
        \[
            \frac 1n |w_\alpha(\bA)|\le \frac 1 {\sqrt n}\cdot \left(1 + \varepsilon^{(n)} \sqrt n\right)^{O_\alpha(1)}\,.
        \]
        In particular, if $\eps^{(n)} = n^{-\frac 12 + o(1)}$, then the right-hand side is $n^{-\frac 12 + o_\alpha(1)}$.
    \end{enumerate}
    Hence, if $\bH$ satisfies~\cref{eq:assNorm,eq:assOpen,eq:assOne} with $\eps^{(n)} = n^{-\frac 12 + o(1)}$, and the diagonal distribution of $\bH$ exists, then the traffic distribution of $\bA$ exists
    and is determined by the diagonal distribution of $\bH$.
\end{theorem}

\noindent We emphasized in the statement 
that all constants in the $O$
notations depend on $\alpha$.
We will drop this dependency in the rest
of the section.

\paragraph{Comparison with prior work.}
    In~\cite[Theorem 2.8]{wang2022universality}, the authors assume {\em (i)} delocalization of open cactuses
    (\cref{eq:assOpen}) and {\em (ii)} the existence of a limiting diagonal
    distribution. They show that, after conjugation by a randomly signed
    permutation matrix, the resulting ``semi-random'' matrix lies in the same
    universality class (in the sense of AMP dynamics) as an orthogonally
    invariant matrix with the same diagonal distribution.
    \Cref{thm:universality-new} shows that the same conclusion holds for
    deterministic matrices, if we replace random conjugation with puncturing.

    The universality result of \cite{wang2022universality} can also be extended in a black-box way to
    deterministic matrices, but only for GFOM with {\em odd} nonlinearities~\cite{dudeja2023universality,
    zhong2024approximate}. 
    This assumption lets one only consider the limiting traffic distribution evaluated on {\em Eulerian} diagrams. Under the same
    assumption, our proof would also 
    significantly simplify. Indeed, in~\Cref{thm:universality-orthogonal}, the number of monomials appearing in $w_\al(\bH)$ is $O(n^{|V(\al)|})$, and each term has magnitude $\max_{i,j\in [n]}|\bH[i,j]|^{|E(\al)|}\le n^{-|E(\al)|/2+o(1)}$, giving the upper bound $|w_\al(\bH)| \leq n^{o(1)}$ if $\al$ has minimum degree 4. It only remains to incorporate paths of degree-2 vertices, which simply compute $\bH^k\in \{\bm I, \bm H\}$ for some $k\ge 1$.
    
\subsection{Calculation of cactus diagrams and diagonal distribution}
\label{sec:diagonal-hadamard}

To apply~\Cref{thm:universality-new}, one needs to
compute the diagonal distribution of $\bH$ and small strengthenings of it
in order to verify~\Cref{assumption:punctured}.
Notice that the only diagrams involved in the assumptions are cactuses, so this is a much simpler task than calculating the entire traffic distribution.
In this subsection, we do this calculation directly to prove~\cref{thm:universality-orthogonal} assuming~\cref{thm:universality-new}.

Let $\bH$ be a delocalized orthogonal matrix
satisfying the assumption of~\cref{thm:universality-orthogonal}. Note that it
satisfies $\bH^2 = \bI$. Hence,~\cref{eq:assNorm} is automatic.
Next, we define the notion of {\em open cactus} appearing in~\cref{eq:assOpen}. An {open cactus} is
a matrix diagram with two roots such that merging the roots yields a cactus.
\begin{definition}
    \label{def:open-cactus}
    An { open cactus} is a graph obtained from
    a simple path by
    attaching vertex-disjoint cactuses to 
    each vertex of the path.
    Formally, $\alpha=(V(\alpha),E(\alpha))$
    is an open cactus if there exist $k\ge 2$,
    vertex-disjoint cactuses $\beta_1, \ldots, \beta_{k}$, and
    distinct vertices $u_1\in V(\beta_1), \ldots, u_{k}\in V(\beta_{k})$ with
    \[
        V(\alpha) = \bigcup_{i=1}^{k} V(\beta_i)\,,\quad E(\alpha) = \{\{u_i,u_{i+1}\}:i\in \{1,\ldots, k-1\}\}\cup \bigcup_{i=1}^{k} E(\beta_i)\,.
    \]
    We call $(u_1,u_{k})$ the {endpoints} of $\alpha$, and $(u_1, \ldots, u_k)$ the {base
    path} of $\alpha$.
    Unless specified otherwise, we will view
    an open cactus $\alpha\in\calA_2$ as a
    matrix diagram rooted at its two ordered endpoints.
\end{definition}

In general, if $\al$ is a matrix diagram and $\al^{\prime}$ is the scalar diagram formed by merging the roots of $\al$, then $\Tr(\bW_{\al}(\bA)) = w_{\al^{\prime}}(\bA)$.
For an open cactus $\al$, this $\al^{\prime}$ is a cactus, and so $w_{\al^{\prime}}(\bA)$ is one of the quantities whose limit is included in the diagonal distribution of $\bA$; further, all values of the diagonal distribution can be obtained in this way from the \emph{diagonal} entries of open cactus matrices.
From this perspective,
\cref{eq:assOpen} is a natural counterpart to the diagonal distribution since it concerns all of the \emph{off-diagonal} entries of the open cactus matrices.

We compute the open cactus matrices for $\bH$ in the following lemma.

\begin{lemma}\label{lem:diagonal-hadamard}
    Let $\sigma$ be an open cactus and let $\bH$ satisfy~\cref{eq:delocOrtho}.
    If all cycles in all of the hanging cactuses have even length, then $\bW_\sig(\bH) = \bI$ if the base path has even length and $\bW_\sig(\bH) = \bH$ if the base path has odd length. Otherwise, $\norm{\bW_\sig(\bH)} \leq n^{-\frac 12 + o(1)}$.
\end{lemma}

\begin{proof}
    First, the
    leaf 2-vertex-connected components of 
    $\sigma$
    consisting of cycles of even length
    can be iteratively removed without
    changing the value of $\bW_\sigma(\bH)$.
    This is because a hanging cycle of
    even length $k$ contributes $\text{diag}(\bH^{k}) = \text{diag}(\bI) = \bm 1$ in the 
    definition of $\bW_\sigma$.
    Therefore, if all cycles in all hanging cactuses have even length, then $\bmW_\sig(\bH) = \bH^\ell \in \{\bI, \bH\}$ where $\ell$ is the length of the base path.

    In the remaining case where $\sig$ has an odd cycle, we use induction.
    Let $\beta_1, \dots, \beta_k$ be the hanging cactuses of $\sig$.
    We convert each $\beta_i$
    into an open cactus diagram $\beta'_i$ by splitting the vertex at which $\beta_i$ meets $\sig$.
    With this notation, we have the matrix factorization:
    \begin{equation*}\label{eq:hadamard-cactus}
        \bW_{\sig}(\bH) = \diag(\bmW_{\beta'_1}(\bH)) \bH \diag(\bW_{\beta'_2}(\bH)) \bH  \ldots \bH\text{diag}(\bW_{\beta'_{k}}(\bH))\,.
    \end{equation*}

    The odd cycle in $\sig$ has either become an odd-length base path in some $\beta'_i$ or it continues
    to be an odd cycle in some $\beta'_i$.
    In the second case, by sub-multiplicativity of the spectral norm,
    \[
        \norm{\bW_\sig(\bH)} \leq \norm{\diag(\bW_{\beta'_i}(\bH))} \leq \norm{\bW_{\beta'_i}(\bH)} \le n^{-\frac 12 + o(1)}
    \]
    with the last inequality by induction.
    In the first case, we have $\bW_{\beta'_i}(\bH) = \bH$.
    Then 
    \[ \norm{\diag(\bW_{\beta'_i}(\bH))} = \norm{\diag(\bH)} \leq n^{-\frac 12 + o(1)}\] 
    by the delocalization assumption, and this case is also complete.
\end{proof}

We use the lemma to complete the proof of~\cref{thm:universality-orthogonal}.
\begin{proof}[Proof of~\Cref{thm:universality-orthogonal} from~\Cref{thm:universality-new}]
    \cref{eq:assNorm} holds automatically for $\bH$ a symmetric orthogonal matrix.
    Verifying~\cref{eq:assOpen},~\Cref{lem:diagonal-hadamard} implies that the off-diagonal entries of all open cactus matrices satisfy
    \[
        \max_{1\le i<j\le n} \left|\bW_\sig(\bH)[i,j]\right|\le \|\bW_\sig(\bH)\|\le n^{-\frac 12+o(1)}
    \]
    when $\sig$ has an odd cycle, and the remaining cases $\bW_\sig(\bH) = \bH$ or $\bW_\sig = \bI$ are easily checked.

    Next, each vector cactus diagram $\sig \in \cC_1$ satisfies $\bmw_\sig(\bH) = \diag(\bW_{\sig'}(\bH))$ where $\sig'$ is an open cactus obtained by splitting the root of $\sig$.
    By~\cref{lem:diagonal-hadamard} the diagonal of an open cactus matrix is either $\bm 1$ (in which case~\cref{eq:assOne} is satisfied with $\eps = 0$) or it satisfies
    \[
        \frac{1}{\sqrt{n}} \norm{\diag(\bW_{\sig'}(\bH))}_2 \le \norm{\diag(\bW_{\sig'}(\bH))}_\infty \leq n^{-\frac 12 + o(1)}\,,
    \]
    in which case~\cref{eq:assOne} is satisfied with $\eps = n^{-\frac 12 +o(1)}$.
    
    The diagonal distribution is computed by averaging the diagonal entries of open cactus matrices:
    \[
        \lim_{n\to\infty} \frac 1n w_\sigma(\bH) = \lim_{n\to\infty} \frac 1n \sum_{i=1}^n \bmW_{\sig'}[i,i] = \begin{cases}
            1 & \text{ if all cycles in $\sigma$ have even length}\\
            0 & \text{ otherwise}
        \end{cases}
    \]
    where on the left-hand side, we convert $\sigma\in\calC_0$ to an open cactus
    diagram $\sig'$ by rooting it arbitrarily and splitting the root.
    The right-hand side is by~\cref{lem:diagonal-hadamard}.
    That is, the diagonal distribution of $\bH$ is just the indicator function that all cycles of the cactus are even.

    Thus, we showed that~\cref{eq:assNorm,eq:assOpen,eq:assOne} hold and the diagonal distribution converges
    to the same fixed limit for any orthogonal matrix with delocalized entries. By~\Cref{thm:universality-new}, the
    traffic distribution of such matrices exists
    and is always the same.
    
    Finally, we show that the \rrom is also in this class, by showing that, after conditioning on a suitable high-probability event, the above argument applies to an \rrom matrix as well. 
    Let
    $\bH_{\rom} = \bQ \bD \bQ^\top$, where $\bQ$ is Haar-distributed and $\bD$ is diagonal
    with i.i.d. $\pm 1$ entries, independent of $\bQ$.
    \begin{claim}
        There exists $c>0$ such that for any $t>0$, 
        \begin{align}
            \max_{i,j\in [n]} |\bH_{\textnormal{\rom}}[i,j]|\le t^2 n^{-\frac 12}\label{eq:event-small-norm}
        \end{align}
        holds with probability at least $1-n^2 e^{-ct^2}$\,.
    \end{claim}
    \begin{proof}
    Since
    every entry of $\bQ$ is $O(n^{-1/2})$-subgaussian, by a union bound\[\max_{i,j\in [n]} |\bQ[i,j]|\le t n^{-\frac 12}\] holds with probability at least $1-n^2 e^{-\Omega(t^2)}$. Next, we have $\bH_{\textnormal{\rom}}[i,j] = \sum_{k=1}^n \bD[k,k] \bQ[i,k] \bQ[j,k]$, which, conditioned on $\bQ$, is a sum of independent
    random variables.
    By Hoeffding's bound, any fixed
    entry of $\bH_{\textnormal{\rom}}$ is $O(\sigma)$-subgaussian with parameter
    \[
        \sigma^2 \defeq \sum_{k=1}^n \bQ[i,k]^2 \bQ[j,k]^2\le \max_{i,j\in [n]} \bQ[i,j]^2\,,
    \]
    since every row of $\bQ$ has $\ell_2$-norm $1$. The conclusion follows from a union bound over all entries. 
    \end{proof}

    Fix $\alpha\in\calA$. Let $E_n$ denote the event~\cref{eq:event-small-norm},
    with $t = n^{o(1)}$. By the law of total expectation, we decompose
    \[
        \frac 1n \E w_\alpha(\bm \Pi \bH_{\textnormal{\rom}} \bm \Pi)
        =
        \frac 1n \E\!\left[w_\alpha(\bm \Pi \bH_{\textnormal{\rom}} \bm \Pi)\mid E_n\right]
        \Pr(E_n)
        +
        \frac 1n \E\!\left[w_\alpha(\bm \Pi \bH_{\textnormal{\rom}} \bm \Pi)\mid E_n^c\right]
        \Pr(E_n^c)\,.
    \]
    The left-hand side converges to the traffic distribution of the \rrom
    evaluated at $\alpha$. Moreover, since
    $\|\bm \Pi \bH_{\textnormal{\rom}} \bm \Pi\|\le 1$, we may crudely bound the
    second term by
    \[
        \frac 1n \E \left[w_\alpha(\bm \Pi \bH_{\textnormal{\rom}} \bm \Pi) \mid E_n^c\right]\cdot \Pr(E_n^c)\le n^{|V(\alpha)|-1} \Pr(E_n^c)\underset{n\to\infty}{\longrightarrow } 0\,.
    \]
    Since $\Pr(E_n)\underset{n\to\infty}{\longrightarrow} 1$, we deduce that
    \[
        \lim_{n\to\infty} \frac 1n \E \left[w_\alpha(\bm \Pi \bH_{\textnormal{\rom}} \bm \Pi) \mid E_n\right] = \lim_{n\to\infty} \frac 1n \E w_\alpha(\bm \Pi \bH_{\textnormal{\rom}} \bm \Pi)\,.
    \]
    Finally, on the event $E_n$, the matrix $\bH_{\textnormal{\rom}}$ satisfies the
    assumptions of~\Cref{thm:universality-orthogonal}. Consequently, the traffic
    distribution of punctured delocalized orthogonal matrices coincides with that
    of the \rrom, as desired.
\end{proof}

As a consequence of the above argument, the traffic distribution of the \rrom is specified implicitly as the solution to the following equations:
\begin{enumerate}
    \item For every $\alpha\in\calA\setminus \calE$, $\frac 1n \E w_\alpha(\bA) \underset{n\to\infty}{\longrightarrow}0$.
    \item For every $\alpha\in \calE\setminus \calC$, $\frac 1n \E z_\alpha(\bA) \underset{n\to\infty}{\longrightarrow}0$.
    \item For every $\sigma\in\calC$, $\frac 1n \E w_\sigma(\bA)\underset{n\to\infty}{\longrightarrow} 1$ if all cycles of $\sig$ are even and 0 otherwise.
\end{enumerate}
These equations determine a unique traffic distribution by~\cref{lem:eqDiffMobius}.
It is possible to give an explicit but much more complicated description using the Weingarten calculus, which we do in~\Cref{sec:puncturedWeingarten}.
However, the above characterization is arguably the conceptually clearer one, and we emphasize that it involves both the $w$- and $z$-bases.

We note also as a point of reference that the last part, the limiting values of cactuses in the $w$-basis, are the same as those for the (unpunctured) \rom, as follows from combining~\Cref{claim:rom-cumulants} with~\Cref{lem:diagonal-w-z}, and corresponds simply to the moments of the Rademacher distribution being 1 for moments of even order and 0 for ones of odd order.

\subsection{The fundamental theorem of graph polynomials}

The main proof of~\Cref{thm:universality-new} throughout the rest of the section relies on the
``fundamental theorem of graph polynomials'' of Bai and Silverstein~\cite{bai2010spectral}.
This result can be used to easily bound
2-edge-connected graph polynomials 
expressed in the $w$-basis, which is one reason that it is convenient to restrict to such diagrams in our definition of the weak cactus property.
The proof of the fundamental theorem uses a spectral bound on tensor powers of $\bA$; see \cite{mingo2012sharp} for another related result.

\begin{theorem}[{\cite[Theorems A.31 and A.32]{bai2010spectral}}]\label{thm:fundamental-theorem}
    For every $n\ge 1$, 
    $\alpha\in \calE\cup \calE_1\cup \calE_2$
    and collection of $n\times n$ symmetric matrices 
    $\bm \calA = (\bA_e)_{e\in E(\alpha)}$,
    \begin{alignat*}{2}
        \frac 1 n |w_\alpha(\bm \calA)|&\le \prod_{e\in E(\alpha)} \|\bA_e\| \quad&&\text{if $\alpha\in\calE$}, \\
        \|\bm w_\alpha(\bm \calA)\|_\infty&\le \prod_{e\in E(\alpha)} \|\bA_e\| \quad&&\text{if $\alpha\in\calE_1$}, \\
        \|\bW_\alpha(\bm \calA)\|&\le \prod_{e\in E(\alpha)} \|\bA_e\| \quad&&\text{if $\alpha\in\calE_2$}.
    \end{alignat*}
\end{theorem}

\noindent The result of~\cite{bai2010spectral} only covers scalar and matrix diagrams, but we provide a quick reduction
of the vector case to the scalar case.

\begin{proof}[Proof of vector case of~\cref{thm:fundamental-theorem}.]
    For all $q \ge 1$, we can diagrammatically express $\norm{\bw_\al(\bm \calA)}_{2q}^{2q}$ as the diagram formed by merging $2q$ copies of $\al$ at the root, and then forgetting the identity of the root to obtain a scalar diagram.
    Let $\al_{2q} = \alpha^{\oplus 2q}$ denote this diagram.
    The graph $\al_{2q}$ remains 2-edge-connected, therefore by the scalar case of the result we have:
    \[
        \norm{\bw_\al(\bm \calA)}_{2q}^{2q} = w_{\al_{2q}}(\bm \calA) \leq n \cdot \left(\prod_{e \in E(\al)} \norm{\bA_e}\right)^{2q}\,.
    \]
    Taking $q \to \infty$ with $n$ fixed, we obtain $\norm{\bw_\alpha (\bm \calA)}_\infty \leq \prod_{e \in E(\al)} \norm{\bA_e}$\,.
\end{proof}

We will apply the fundamental theorem by decomposing a general graph into its 2-edge-connected components,
which are joined together by a tree of bridge edges.
Decomposing diagrams into their 2-edge-connected components is also a fundamental idea in physics, where a 2-edge-connected Feynman diagram is called a ``1-particle-irreducible diagram''.

\subsection{Main structural lemma: Open cactus decomposition}

To prove the weak cactus property of~\Cref{thm:universality-new}, we begin by observing that
any 2-edge-connected non-cactus graph contains 
three edge-disjoint
paths between some pair of vertices.
How can we quantify that such a graph is a cactus plus excess edges?
We answer this question
by introducing the {\em open cactus decomposition}.
Our main structural result is that one can identify an ``extra'' open cactus subgraph inside any 2-edge-connected graph which is not a cactus, in the sense that the subgraph can be removed without spoiling 2-edge-connectedness.

\begin{proposition}
    \label{prop:open-cactus-decomposition}
    For any $\alpha\in\calE_1\setminus \calC_1$, there exist distinct $s,t\in V(\alpha)$ and
    an induced subgraph $\beta$ of $\al$ such that
    \begin{enumerate}
        \item $\beta$ is an open cactus with endpoints $\{s,t\}$.
        \item $\alpha\left[V(\alpha)\setminus (V(\beta)\setminus \{s,t\})\right]$ 
        is 2-edge-connected.
        \item $\rt(\al)\notin V(\beta)\setminus \{s,t\}$.
    \end{enumerate}
\end{proposition}

\begin{figure}[ht]
    \centering
    \includegraphics[width=0.35\linewidth]{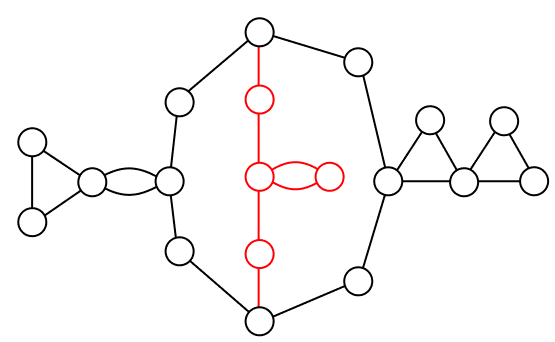}
    \caption{Example for~\cref{prop:open-cactus-decomposition} of a 2-edge-connected graph which is not a cactus. If the open cactus in red is removed, the graph remains 2-edge-connected.}
    \label{fig:2ec}
\end{figure}

\noindent To prove~\Cref{prop:open-cactus-decomposition},
we will consider the last ear
in an {\em ear decomposition} of $\alpha$. We
prove a small variant of the classical
ear decomposition (see~\cite{robbins1939theorem} or~\cite[\S 5.3]{bondyMurty}) which lets us exclude a specified vertex from the internal vertices of the last ear.

\begin{lemma}
    \label{lem:ear-decomposition-lemma}
    Let $\alpha\in\calE_1$ be 2-edge-connected with at least 2
    vertices.
    There exists a path $\pi=(u_1, \ldots, u_k)$ in $\alpha$ with $k\ge 2$
    such that:
    \begin{enumerate}
        \item Each internal vertex $u_2, \ldots, u_{k-1}$ has degree 2 in $\alpha$.
        \item Each internal vertex $u_2, \ldots, u_{k-1}$ satisfies $u_i\neq \rt(\al)$.
        \item $u_1, \ldots, u_{k}$ are pairwise distinct, except possibly $u_1=u_k$.
        \item Removing internal vertices
        and edges of $\pi$ from $\alpha$
        leaves $\alpha$ 2-edge-connected.
    \end{enumerate}
\end{lemma}

\begin{proof}[Proof of~\Cref{lem:ear-decomposition-lemma}]
    Consider the following sequence $(\alpha_t)_{t\ge 0}$ of 2-edge connected subgraphs of $\alpha$:
    \begin{enumerate}
        \item Start from $\alpha_0$ being 
        any cycle of $\alpha$ containing $\rt(\al)$.
        \item Let $t\ge 0$. If $\alpha_t$ spans all vertices of $\alpha$, then stop.
        \item Otherwise, there exists
        $\{u_1,u_2\}\in E(\alpha)$ such that $u_1\in V(\alpha_t)$ and $u_2\notin V(\alpha_t)$. Since $\alpha$ is
        2-edge-connected, there exists a
        simple path
        $(u_2, \ldots, u_k)$ in $\alpha\setminus\{\{u_1,u_2\}\}$
        such that
        $u_i\notin V(\alpha_t)$ for all $2\le i\le k-1$, and $u_k\in V(\alpha_t)$.
        Set 
        \[
            \alpha_{t+1}=(V(\alpha_t)\cup \{u_2, \ldots, u_{k-1}\}, E(\alpha_t) \cup \{\{u_i,u_{i+1}\} : 1\le i<k\})\,.
        \]
    \end{enumerate}
    For any $t\ge 0$, $\alpha_t$ is 2-edge-connected.
    Therefore, if at the end of the algorithm $V(\alpha_t)=V(\alpha)$ but $E(\alpha_t)\neq E(\alpha)$, then any
    edge in $E(\alpha)\setminus E(\alpha_t)$
    is a length-1 path that satisfies the conclusion of the lemma.
    Otherwise, this means that $\alpha$ is
    obtained from
    $\alpha_{t-1}$ (which is 2-edge-connected) by adding a path of
    internal degree-2 vertices in $\alpha$
    which must all 
    be distinct from $\rt(\al)\in V(\alpha_0)\subseteq V(\alpha_{t-1})$. This concludes the proof.
\end{proof}

\begin{proof}[Proof of~{\Cref{prop:open-cactus-decomposition}}]
    Starting with the graph $\alpha$, consider the following procedure:
    \begin{enumerate}
        \item Delete all self-loops in $\alpha$.
        \item If no leaf 2-vertex-connected
        component (i.e., a 2-vertex-connected
        component meeting the rest of the graph
        at a single articulation point)
        consists of a single cycle,
        then stop.
        \item Otherwise, choose an arbitrary
        such component. Let $v$ be the articulation point connecting this component to the rest of graph. Delete all edges 
        of this component
        from the graph.
        \item Delete newly isolated vertices; exactly
        one vertex of the component remains, namely $v$. Since
        $\alpha\notin \calC_1$, the procedure
        does not delete the entire graph.
    
        \item If the root was removed in Step 4, set $v$ as the new root of the diagram.
        \item Return to Step 1.
    \end{enumerate}
    Call $\beta\in\calA_1$ the
    resulting rooted graph. Note that $\beta$
    is still 2-edge-connected, so
    by~\Cref{lem:ear-decomposition-lemma}, we can find a path
    $\pi = (u_1, \ldots, u_k)$ in $\beta$
    with internal degree-2 vertices. $\pi$ 
    cannot be a cycle because of our
    initial step of removing cyclic 2-vertex-connected components.
    Therefore, $\pi$ is a simple
    path and the root of $\beta$ is not an
    internal vertex of $\pi$.

    \begin{observation}\label{obs2-open}
        For $2\le i<k$, let $\sigma_i$ be the connected
        component of $u_i$ in $\alpha\setminus E(\pi)$. Then $\alpha' \defeq \pi\cup \sigma_2\cup \ldots \cup \sigma_{k-1}$ is
        an open cactus in $\alpha$ with endpoints $u_1,u_k$. Moreover, $\rt(\alpha)$
        is not an internal vertex of the open cactus.
    \end{observation}

    \begin{proof}
        $\pi$ is a simple
        path in $\beta$, and adding back loops and
        cyclic 2-vertex-connected components we removed from $\alpha$, we obtain an open cactus.
        The recursive pruning procedure we used 
        to transfer
        the root ensures that $\rt(\alpha)$
        is not in any of the cyclic 2-vertex-connected
        components that are added to $\pi$.
    \end{proof}

    \begin{observation}\label{obs1-open}
        $\alpha\left[V(\alpha)\setminus (V(\alpha')\setminus \{u_1,u_k\})\right]$
        is 2-edge-connected.
    \end{observation}

    \begin{proof}
        By~\Cref{lem:ear-decomposition-lemma},
        $\beta\left[V(\beta)\setminus \{u_2, \ldots, u_{k-1}\}\right]$ is 2-edge-connected.
        Adding 2-vertex-connected cyclic components
        to this graph preserves 2-edge-connectivity.
    \end{proof}
    \noindent
    \Cref{obs2-open} and~\Cref{obs1-open} conclude
    the proof of~\Cref{prop:open-cactus-decomposition}.
\end{proof}

\subsection{The effect of puncturing}

The main result of this subsection is:

\begin{proposition}
    \label{prop:to-punctured}
    Let
    $\bH \in \R^{n \times n}_\sym$ such that $\|\bH\|\le 1$ and $\bm u\in\R^n$ be a unit vector. Denote by $\bA = (\bI - \bm u \bm u^\top) \bH (\bI - \bm u\bm u^\top)$.
    Then for any open cactus $\alpha\in\calA_2$,
    \[
        \|\bW_\alpha(\bA) - \bW_\alpha(\bH)\|_{\textnormal F}\le |E(\alpha)|\cdot \|\bm A - \bm H\|_\frob\le 3|E(\alpha)|\,.
    \]
\end{proposition}

We deduce in the following that puncturing does not change the diagonal distribution.
In particular, matrices such as the \rom and the \rrom have the same diagonal distribution.

\begin{corollary}
    \label{cor:cactus-puncturing-same}
    Let $\bH$ and $\bA$ be as in~\Cref{prop:to-punctured}. Then for any $\sigma\in\calC_1$
    \[
        \|\bw_\sigma(\bH) - \bw_\sigma(\bA)\|_2 \le O(1)\,,
    \]
    and for any $\sigma\in \calC$,
    \[
        \frac 1n |w_\sigma(\bH) - w_\sigma(\bA)| \le O\left(\frac 1 {\sqrt n}\right)\,.
    \]
\end{corollary}

\begin{proof}[Proof of~\Cref{cor:cactus-puncturing-same} from~\Cref{prop:to-punctured}]
    Except for the case where $\sigma \in \cC_1$ has one vertex (in which case the statement holds because the diagonal entries are bounded), $\rt(\sigma)$ has degree $\ge 2$. Create
    two copies $r_1,r_2$ of $\rt(\sigma)$ and re-assign the edges
    incident to $\rt(\sig)$ to $r_1$ or $r_2$ in such a way that $r_1$ and $r_2$
    have degree at least $1$. The resulting
    graph is an open cactus $\alpha$ with endpoints $r_1$ and $r_2$ such that merging
    these endpoints yields back $\sigma$.
    Hence,
    \[
        \|\bw_\sigma(\bH) - \bw_\sigma(\bA)\|_2 = \|\textnormal{diag}(\bW_\alpha(\bH)) - \textnormal{diag}(\bW_\alpha(\bA))\|_{\textnormal F}\le O(1)\,.
    \]
    The second statement then follows from Cauchy-Schwarz:
    \[
        |w_\sigma(\bH)-w_\sigma(\bA)| = |\langle \bm 1, \bw_\sigma(\bH)-\bw_\sigma(\bA)\rangle| \le \sqrt n \cdot  \|\bw_\sigma(\bH) - \bw_\sigma(\bA)\|_2\le O(\sqrt n)\,.
    \]
    This concludes the proof.
\end{proof}

However, $\bH$ and its punctured version $\bA$ may  
{\em not} have the same traffic distribution, even on scalar open
cactuses.
Thus, the diagonal distribution (i.e., the values of cactus diagrams) is not sensitive to the behavior of $\bH$ in any single direction $\bu$, while some diagrams in the traffic distribution \emph{are} sensitive to the behavior in the $\bm 1$ direction.

\begin{example}[Puncturing of the Walsh-Hadamard matrix]\label{ex:puncture-hadamard}
    Let $\bH^{(n)}$ be the normalized
    Walsh-Hadamard matrices (\Cref{def:hadamard}). Then for the 2-path diagram $\alpha$ (which is an open cactus),
    \[
        \frac 1n (w_\alpha(\bH) - w_\alpha(\bA)) = \frac 1n \langle \bm 1, (\bH^2 - \bA^2) \bm 1\rangle  \underset{n\to\infty}{\longrightarrow} 1\,.
    \]
    This does not contradict~\Cref{prop:to-punctured}: $\bE = \bW_\alpha(\bH) - \bW_\alpha(\bA)$ indeed satisfies
    \[
        \sum_{i,j=1}^n \bE[i,j]^2 \le O(1)\quad \text{ and }\quad \left|\sum_{i,j=1}^n \bE[i,j] \right| = \Omega(n)\,.
    \]

\end{example}

In general, as the following example demonstrates, the off-diagonal structure of the error matrix
$\bE = \bW_\alpha(\bH) - \bW_\alpha(\bA)$ in~\Cref{prop:to-punctured} may
be intricate. 
In the following example, $\bE$ has
entries of magnitude $\Omega(1)$, even though its
Frobenius norm remains bounded.

\begin{example}[Puncturing of the DST matrix]
        Let $\bH^{(n)}$ be the discrete sine transform matrices (\Cref{def:dst}).
        Then for any fixed odd $i\ge 1$, the 
        normalized sum of the $i$th row of
        $\bH^{(n)}$ is
        \[
            \frac {1} {\sqrt n} \sum_{j=1}^n \bH[i,j] = (\sqrt 2+o(1)) \int_0^1 \sin(\pi i t) \,\d t \underset{n\to\infty}{\longrightarrow} \frac {2 \sqrt 2} {i\pi }\,.
        \]
        Consider the 2-path diagram $\alpha$. While the
        off-diagonal entries of $\bW_\alpha(\bH) = \bH^2$
        vanish (since $\bH$ is a symmetric
        orthogonal matrix), on the other hand, for any
        fixed distinct odd numbers $i,j\ge 1$, 
        \[
            \bW_\alpha(\bA)[i,j] = (\bA^2)[i,j]\underset{n\to\infty}{\longrightarrow} -\frac {8} {ij \pi^2}\,,
        \]
        which is $\Omega(1)$ for constant $i\neq j$.
\end{example}

\noindent The proof of~\Cref{prop:to-punctured}
relies on expanding $\bA$ in terms of
$\bm u\bm u^\top$ and $\bH$. All rank-1 terms
can be neglected thanks to the following
lemma:

\begin{lemma}
    \label{lem:low-rank}
    Let $\alpha$ be an open cactus, $e^*\in E(\alpha)$, and $\bm \calA = (\bA_e)_{e\in E(\alpha)}$
    be a collection of matrices such
    that $\|\bA_e\|\le 1$ for all
    $e\in E(\alpha)\setminus \{e^*\}$. Then,
    \[
        \|\bW_\alpha(\bm \calA)\|_{\textnormal F}\le \|\bA_{e^*}\|_{\textnormal F}\,.
    \]
\end{lemma}

\begin{proof}
    We first
    run a pruning procedure that iteratively removes parts of $\alpha$ not containing
    $e^*$, without decreasing the
    Frobenius norm of $\bW_\alpha(\bm \calA)$
    during the procedure.
    To this end, we use repeatedly the standard inequalities:

    \begin{claim}    
    \label{claim:frobenius-product}
        $\|\bmM_1 \bmM_2\|_\frob\le \|\bmM_1\|_\frob \|\bmM_2\|$.
    \end{claim}
    
    \begin{claim}
    \label{claim:frobenius-hadamard}
    $\|\bmM_1 \odot \bmM_2\|_\frob\le \|\bmM_1\|_\frob\cdot  \max_{1\le i,j\le n} \left|\bmM_2[i,j]\right|\le \|\bmM_1\|_\frob \|\bmM_2\|$, where $\odot$ denotes entrywise or Hadamard product.
    \end{claim}
    
    Initially, let $u_{\textnormal L}$
    be one of the endpoints of $\alpha$.
    \begin{enumerate}
        \item If $e^*$ belongs to a
        cactus hanging from $u_{\textnormal L}$, then stop.
        \item Otherwise, remove
        the cactus hanging from $u_{\textnormal L}$ from the diagram. Using~\Cref{claim:frobenius-hadamard} and~\Cref{thm:fundamental-theorem} 
        (the spectral norm of the cactus matrix
        diagram with a double root at $u_{\textnormal L}$ is at most 1), this does not 
        decrease
        the Frobenius norm. 
        \item At this point, $u_{\textnormal L}$ must have degree equal to 1 in the current graph. If $e$ is the edge 
        adjacent to $u_{\textnormal L}$, then stop. 
        \item Otherwise, remove the edge
        adjacent to $u_{\textnormal L}$.
        By~\Cref{claim:frobenius-product} and the assumption, this does not decrease the
        Frobenius norm. Set $u_{\textnormal L}$ to be the
        vertex that was adjacent to
        $u_{\textnormal L}$, and go
        back to the first step.
    \end{enumerate}
    Then, apply the symmetric procedure
    from the other endpoint $u_{\textnormal R}$ of $\alpha$. At
    this point, there are two cases. If $u_{\textnormal L} \neq u_{\textnormal R}$, then the resulting graph
    must consist of the single edge
    $e^* = \{u_{\textnormal L}, u_{\textnormal R}\}$, so we get
    the desired upper bound on the Frobenius norm. Therefore, we
    assume from now on that $u_{\textnormal L} = u_{\textnormal R}$.
    
    The resulting graph must be a cactus rooted at $u_{\textnormal L} = u_{\textnormal R}$, and $e^*$ is
    one of the edges of this cactus.
    If there are several cycles
    incident to $u_{\textnormal L}$,
    we use again~\Cref{claim:frobenius-hadamard} and~\Cref{thm:fundamental-theorem} to remove all such cycles
    not containing $e^*$ without
    decreasing the Frobenius norm.
    
    Finally, we bound the Frobenius norm of
    the diagonal cactus matrix rooted at
    $u_{\textnormal L}$ by the
    Frobenius norm of an open cactus
    obtained by creating two copies
    of the root and turning the unique
    cycle hanging at $u_{\textnormal L}$ into a simple path between
    these two copies (we used a similar procedure
    in~\Cref{cor:cactus-puncturing-same}). We claim that this open cactus has strictly less edges
    than the one we started with
    before running the pruning procedure. Indeed,
    the base path had at least one edge, which
    was removed during the pruning stage when
    $u_{\textnormal L} = u_{\textnormal R}$ at the end. We
    conclude by induction on the
    number of edges of the open cactus.
\end{proof}

\begin{proof}[Proof of {\Cref{prop:to-punctured}}]
    We replace iteratively $\bm H$ by
    $\bm A$ in the graph polynomial $\bm W_\alpha(\bm H)$: let $e_1, \ldots, e_{|E(\alpha)|}$
    be the edges of $\alpha$, and write
    \[
        \bm W_\alpha(\bm A) - \bm W_\alpha(\bm H) = \sum_{i=1}^{|E(\alpha)|} \bm W_\alpha(\bm \calA_{i})\,,
    \]
    where $\bm \calA_i[e_j] = \bm H$ if $j<i$, $\bm \calA_i[e_j] = \bm A$ if $j>i$, and $\bm \calA_i[e_i] = \bm A - \bm H$.
    For each $i\in [|E(\alpha)|]$, we apply~\Cref{lem:low-rank} with $e^* = e_i$.
    We have $\|\bm A\|\le 1$ and $\|\bm H\|\le 1$
    so the assumptions of the lemma are satisfied, and we deduce
    \[
        \|\bm W_\alpha(\bm \calA_i)\|_\frob\le \|\bm A - \bm H \|_\frob\,,
    \]
    and by the triangle inequality
    \[
        \|\bm W_\alpha(\bm A) - \bm W_\alpha(\bm H)\|_\frob\le |E(\alpha)|\cdot \|\bm A - \bm H\|_\frob\,.
    \]
    Finally, we have
    \[
        \bA - \bH = \langle \bu,\bH \bu\rangle \bu\bu^\top - (\bH \bu\bu^\top + \bu\bu^\top \bH)\,.
    \]
    Since $\|\bH\|\le 1$ and $\bu$ is a unit vector, we have $|\langle \bu, \bH \bu\rangle|\le 1$ and $\|\bH \bu\|_2\le 1$, so $\|\bm A - \bm H\|_{\textnormal F}\le 3$.
\end{proof}

\subsection{Support of the \texorpdfstring{$z$}{z}-basis}

Let $\bH=\bmH^{(n)}$ be a family of matrices
satisfying~\cref{eq:assNorm,eq:assOpen}
and $\bmA=\bmA^{(n)}$ be their puncturing.
The main result of this subsection is
that $\bA$ and $\bH$ satisfy the weak cactus
property, that is, their traffic 
distribution in the $z$-basis is supported on
cactuses and graphs with bridges.

\begin{proposition}
    \label{prop:Z-support}
    For any $\alpha\in \calE\setminus \calC$,
    \[
        \frac 1n |z_\alpha(\bH)|\le O\left(\eps + \frac 1 {\sqrt n}\right)\quad\text{and}\quad \frac 1n |z_\alpha(\bA)|\le O\left(\eps + \frac 1 {\sqrt n}\right)\,.
    \]
\end{proposition}

The fundamental theorem of graph polynomials can be used to show that these quantities are $O(1)$ (after converting between the $z$ and $w$-bases).
The idea of~\Cref{prop:Z-support} is
to isolate an open cactus in $\alpha$ by~\Cref{prop:open-cactus-decomposition}
and apply~\Cref{assumption:punctured}
to gain an additional $\eps$
factor.

We emphasize that analogous bounds in the
$w$-basis are false in general;
summation over some distinct indices
is necessary to prove~\Cref{prop:Z-support}.
We prove that, using the notation in~\cref{sec:polynomial-definitions}:

\begin{lemma}    
    \label{prop:Z-support-vector}
    Let $\alpha\in\calE_1 \setminus \calC_1$ and
    let $s,t$ be the endpoints of an open
    cactus in $\alpha$ satisfying the guarantees
    of~\Cref{prop:open-cactus-decomposition}. Then
    \begin{align}
        \frac 1 {\sqrt n} \|\bw_\alpha^{s\neq t}(\bA)\|_2 \le O\left(\eps + \frac 1 {\sqrt n}\right)\quad\text{and}\quad \frac 1 {\sqrt n} \|\bw_\alpha^{s\neq t}(\bH)\|_2\le O\left(\eps + \frac 1 {\sqrt n}\right)\,.\label{eq:vector-extension}
    \end{align}
\end{lemma}

The constraint
$s\neq t$ in~\cref{eq:vector-extension} ensures that we only use off-diagonal
entries of the open cactuses in the graph 
polynomial.
These are the only entries assumed to be small in~\Cref{assumption:punctured} (and indeed,
the diagonal entries of $\bW_\al(\bH)$ can be large,
for example, in the 2-path diagram).

\begin{proof}[Proof of~\Cref{prop:Z-support} from~\Cref{prop:Z-support-vector}]
    Let $\bmM\in \{\bmA, \bmH\}$ and
    $s,t$ be two distinct vertices of $\alpha$ to be fixed later. Using M\"{o}bius inversion
    (\Cref{claim:mobius})
    recursively, we can expand
    \[
        z_\alpha(\bmM) = c_{\alpha} w_\alpha^{s\neq t}(\bmM) + \sum_{\beta \prec \alpha} c_{\beta} z_{\beta}(\bmM)\,,
    \]
    for some constant coefficients $c_\beta \in \R$. Since all $\beta\prec \alpha$ are
    2-edge-connected by~\Cref{lem:contract-2ec} and have
    strictly less vertices than $\alpha$, by induction on
    the number of vertices of $\alpha$,
    it suffices to prove:
    \begin{align}
        \frac 1n |w^{s\neq t}_\alpha(\bmM)|\le O\left(\eps + \frac 1 {\sqrt n}\right)\,.\label{eq:vector-scalar-last}
    \end{align}
    But~\cref{eq:vector-scalar-last} follows
    from~\Cref{prop:Z-support-vector}: pick $s,t$
    to be the endpoints of an open cactus decomposition provided by~\Cref{prop:open-cactus-decomposition}, so that by Cauchy-Schwarz
    \[
        \frac 1n |w^{s\neq t}_\alpha(\bmM)| = \frac 1 n \left|\langle \bw_\alpha^{s\neq t}(\bmM),\bm1\rangle\right|\le \frac 1 {\sqrt n} \|\bw_\alpha^{s\neq t}(\bmM)\|_2\le O\left(\eps + \frac 1 {\sqrt n}\right)\,,
    \]
    which concludes the proof.
\end{proof}

We now move to the proof of~\Cref{prop:Z-support-vector}. A useful concept will be the following
graphical interpretation of squaring the polynomial expressed by a diagram:

\newcommand\lift{\textnormal{Lift}}

\begin{definition}[Lift]
    \label{def:lift}
    Let $\alpha\in\calA$ and $T\subseteq V(\alpha)$. Let $S_1$ and
    $S_2$ be two new disjoint sets of size $|V(\alpha)|-|T|$ (also disjoint from $V(\alpha)$). For $i\in \{1,2\}$, let $p_i$ be a bijection between $V(\alpha)\setminus T$ and $S_i$, which is extended to $V(\alpha)$
    by $p_i(u)=u$ for
    all $u\in T$.

    The
    {\em lift} of $\alpha$ with respect to $T$
    is the graph $\lift_T(\alpha)$ with
    \[
    V(\lift_T(\alpha)) = T\cup S_1\cup S_2\,,\quad
        E(\lift_T(\alpha)) = \{\{p_i(u), p_i(v)\} : i\in \{1,2\}, \{u, v\}\in E(\alpha)\}\,.
    \]
\end{definition}

\begin{claim}
    \label{claim:lift}
    Let $\alpha\in\calA_2$  with roots
    $(s,t)$, and $T\subseteq V(\alpha)$ be such that $\{s,t\}\subseteq T$. Then for any $\bmM\in \R_\sym^{n\times n}$,
    \[
        \bW_{\lift_T(\alpha)}(\bmM)[i,j] = \sum_{\substack{\varphi\colon T\to [n]\\\varphi(s)=i,\varphi(t)=j}} \left(\sum_{\varphi\colon V(\alpha)\setminus T\to [n]} \prod_{\{u,v\}\in E(\alpha)} \bmM[\varphi(u),\varphi(v)]\right)^2\,.
    \]
\end{claim}

\begin{lemma}
    \label{lem:structural-2ec}
    Let $\alpha\in\calE$, let $\pi$ be a connected
    subgraph of $\alpha$, let $\alpha_1$ be any connected component
    of $\alpha\setminus E(\pi)$, and let $\alpha_2$ the graph spanned by $E(\alpha)\setminus E(\alpha_1)$.
    Then for all $j\in \{1,2\}$, 
    $\lift_{V(\alpha_1)\cap V(\alpha_2)}(\alpha_j)$ is 2-edge-connected.
\end{lemma}

\begin{proof}
    First, $\alpha_1$ is connected by definition. Since $\alpha$ is connected, every connected component
    in $(V(\alpha), E(\alpha)\setminus E(\pi))$ must be connected to
    $\pi$. Together with the fact
    that $\pi$ itself is connected,
    we get that $\alpha_2$ is connected. In particular, the lifts
    of $\alpha_1$ and $\alpha_2$ are connected.

    Fix $j\in \{1,2\}$ and an edge $e'$ in the lift of $\alpha_j$. We need to show that $e'$ belongs to at least one simple cycle in the lift of $\alpha_j$.
    There exist $i\in \{1,2\}$ and
    $e=\{x,y\}\in E(\alpha_j)$ such that $e' = \{p_i(x),p_i(y)\}$ (where $p_1,p_2$ are the lift maps from~\Cref{def:lift}). Since $\alpha$ is 2-edge-connected, $e$ belongs to a simple cycle
    in $\alpha$. Consider the
    longest subpath of this cycle
    containing $e$ and consisting 
    only of vertices in $V(\alpha_j)$.
    If this subpath is the entire cycle,
    then we have found a cycle containing $e$ in
    $\alpha_j$, and so a cycle containing $e'$ in its lift.
    Otherwise, the endpoints of this path
    must be in $V(\alpha_1)\cap V(\alpha_2)$.
    The images of this
    path through the lift maps
    $p_1$ and $p_2$ are
    disjoint, so their union forms a 
    cycle in the lift of $\alpha_j$ containing $e'$.
\end{proof}

\begin{lemma}
    \label{lem:abs-value-pushed}
    Let $\alpha\in \calE_2$ have two distinct roots. Let $\beta$
    be a leaf 2-vertex-connected component of $\alpha$
    (i.e., removing internal vertices of $\beta$ leaves $\alpha$
    connected) that does not contain the roots
    of $\alpha$.
    We view $\beta\in\calE_1$ as a vector
    diagram rooted at the articulation point
    connecting $\beta$ to the rest of $\alpha$.
    For any distinct $s',t'\in V(\beta)$ and $\bmM\in \R^{n\times n}_\sym$
    such that $\|\bmM\|\le 1$,
    \[
        \sum_{i,j=1}^n \left|\bW_\alpha^{s'\neq t'}(\bmM)[i,j]\right|\le \sqrt n \cdot \|\bw^{s'\neq t'}_\beta(\bmM)\|_2\,.
    \]
\end{lemma}

\begin{proof}
    Let $(s,t)$ be the roots of $\alpha$.
    Since $\alpha$ is 2-edge-connected, there
    exist two edge-disjoint simple paths 
    between $s$
    and $t$. Let $\pi$ be one of them. Let
    $\alpha_1$ be the connected component of $s$
    in 
    $(V(\alpha),E(\alpha)\setminus E(\pi))$, and
    $\alpha_2$ be the graph spanned by $E(\alpha)\setminus E(\alpha_1)$ (including only the vertices incident with one of these edges).
    Finally, let
    $S = V(\alpha_1)\cap V(\alpha_2)$.

    \begin{claim}
        $\{s,t\}\subseteq S$.
    \end{claim}

    \begin{proof}
        On the one hand, $E(\pi)\subseteq E(\alpha_2)$ and $\{s,t\}$
        are the endpoints of $\pi$, so $\{s,t\}\subseteq V(\alpha_2)$. On the other
        hand, $s\in V(\alpha_1)$ by definition, and there is an $s$--$t$ path
        in $\alpha\setminus E(\pi)$, so $t\in V(\alpha_1)$.
    \end{proof}

    \begin{claim}
        \label{obs:makes-sense}
        For any $\{u,v\}\in E(\alpha)$ with
        $u\in V(\alpha_1)$ and $v\in V(\alpha_2)$, we have $u\in S$ or $v\in S$.
    \end{claim}

    \begin{proof}
        Suppose that
        $v\notin V(\alpha_1)$. 
        Since $u \in V(\alpha_1)$, $u$ is connected to $s$ by edges of $E(\alpha) \setminus E(\pi)$, and since $v \notin V(\alpha_1)$, $v$ is not connected to $s$ by these edges.
        But, $\{u, v\} \in E(\alpha)$, so it must be that $\{u, v\} \in E(\pi)$.
        And, $E(\pi) \subseteq E(\alpha_2)$, so
        $u\in V(\alpha_2)$.
    \end{proof}
    
    As $\pi$ is a simple path and $\beta$ is connected to the rest of $\alpha$ at an articulation vertex, $\pi$
    does not contain any edge of $\beta$, so it
    must be that either $E(\beta)\subseteq E(\alpha_1)$ or $E(\beta)\subseteq E(\alpha_2)$. Assume without loss of generality that this holds for $\alpha_1$ (the argument will be exactly symmetric for $\alpha_2$, as we will only use the fact that these subgraphs satisfy the conclusion of~\Cref{lem:structural-2ec}).
    In particular, we then have $s^{\prime}, t^{\prime} \in V(\alpha_1)$.

    We first use the triangle
    inequality to push the absolute value inside
    the sum over labelings of vertices in $S$:
    \begingroup
    \allowdisplaybreaks
    \begin{align}
        &\sum_{\varphi(s),\varphi(t)=1}^n \left|\bW_\alpha^{s'\neq t'}(\bmM)[\varphi(s),\varphi(t)]\right|\nonumber\\
        &\le \sum_{\varphi\colon S \to [n]} \left|\sum_{\substack{\varphi\colon V(\alpha)\setminus S\to [n]\\\varphi(s')\neq \varphi(t')}} \prod_{\{u,v\}\in E(\alpha)} \bmM[\varphi(u),\varphi(v)]\right|\label{eq:tmpcs-crazy}\\
        &=\sum_{\varphi\colon S \to [n]} \left| \prod_{j=1}^2 \sum_{\substack{\varphi\colon V(\alpha_j)\setminus S\to [n]\\\varphi(s')\neq \varphi(t')\text{ if $j = 1$}}} \prod_{\{u,v\}\in E(\alpha_j)} \bmM[\varphi(u),\varphi(v)]\right|\nonumber\\
            &\le \left[\prod_{j=1}^2 \sum_{\varphi\colon S \to [n]} \left(  \sum_{\substack{\varphi \colon V(\alpha_j)\setminus S\to [n]\\\varphi(s')\neq \varphi(t')\text{ if $j = 1$}}} \prod_{\{u,v\}\in E(\alpha_j)} \bmM[\varphi(u),\varphi(v)]\right)^2\right]^{\frac 12}\,,\label{eq:cs-crazy}
    \end{align}
    \endgroup
where we applied Cauchy-Schwarz in the second inequality. Note that~\cref{eq:tmpcs-crazy} is
well-defined by~\Cref{obs:makes-sense}.

By~\Cref{lem:structural-2ec} and~\Cref{claim:lift}, the term
for $j = 2$ in~\cref{eq:cs-crazy} is a 2-edge-connected
graph polynomial, so by~\Cref{thm:fundamental-theorem} and the assumption $\|\bmM\|\le 1$, this term is bounded by
\[
    \sum_{\varphi\colon S \to [n]} \left(  \sum_{\substack{\varphi \colon V(\alpha_2)\setminus S\to  [n]}} \prod_{\{u,v\}\in E(\alpha_2)} \bmM[\varphi(u),\varphi(v)]\right)^2\le n\,.
\]
We now switch to the term $j = 1$ in~\cref{eq:cs-crazy}. This graph polynomial
can be interpreted as
\[
    \sum_{\varphi\colon S \to [n]} \left(  \sum_{\substack{\varphi \colon V(\alpha_1)\setminus S \to [n]\\\varphi(s')\neq \varphi(t')}} \prod_{\{u,v\}\in E(\alpha_1)} \bmM[\varphi(u),\varphi(v)]\right)^2 = \langle \bw_\beta^{s'\neq t'}(\bmM), \bW_{\alpha'}(\bmM) \bw_\beta^{s'\neq t'}(\bmM)\rangle\,,
\]
where $\alpha'$ is the lift of 
$\alpha\left[V(\alpha)\setminus (V(\beta)\setminus \{r\})\right]$
with respect to $S$ (here $r$ denotes the root of $\beta$, the articulation vertex connecting $\beta$ to the rest of $\alpha$),
and we add two roots in $\alpha'$ at the two
copies of $r$ created during the lift operation.

Hence,
\[
    \sum_{\varphi\colon S \to [n]} \left(  \sum_{\substack{\varphi \colon V(\alpha_1)\setminus S\to [n]\\\varphi(s')\neq \varphi(t')}} \prod_{\{u,v\}\in E(\alpha_j)} \bmM[\varphi(u),\varphi(v)]\right)^2 \le \|\bW_{\alpha'}(\bmM)\| \cdot \|\bw_\beta^{s'\neq t'}(\bmM)\|_2^2 \,.
\]
Note that $\alpha'$ is 2-edge-connected by~\Cref{lem:structural-2ec}, so that $ \|\bW_{\alpha'}(\bmM)\|\le 1$ by~\Cref{thm:fundamental-theorem}. Putting everything together, we obtain
\[
    \sum_{i,j=1}^n \left|\bW_\alpha^{s'\neq t'}(\bmM)[i,j]\right|\le \sqrt n \cdot \|\bw_\beta^{s'\neq t'}(\bmM)\|_2\,,
\]
as desired.
\end{proof}

\begin{proof}[Proof of~\Cref{prop:Z-support-vector}]
    Let $\bmM\in \{\bmA, \bmH\}$.
    Consider $\beta\in\calA_2$ defined by:
    \begin{enumerate}
        \item Start from the lift of $\alpha$
        with respect to its root. 
        Let $p_1$ and $p_2$ be the lift maps.
        \item Delete the edges and
        internal vertices of the image under $p_1$ of the
        open cactus in $\alpha$.
        \item Root the resulting graph at $p_1(s)$ and $p_1(t)$.
    \end{enumerate}
    Recall that $s$ and $t$ are the endpoints of the ``extra'' open cactus in $\alpha$.
    Thus, $\beta$ is, in short, $\alpha$ grafted to its mirror image at the roots, with just \emph{one} of the copies of that extra open cactus deleted except for its endpoints, and those endpoints made the roots of the matrix diagram $\beta$.
    See~\Cref{fig:lift-proof} for an illustration of this and the rest of the proof.
    
    Let $\sigma$ be the image of the open cactus in
    $\alpha$ under the lift map $p_2$, and let $s'$ and $t'$ be the images
    of the endpoints of this open cactus through the lift map $p_2$.
    Thus $s^{\prime}$ and $t^{\prime}$ are the mirror images of the vertices chosen to be the roots of $\beta$ above.
     We can then rewrite
    \begin{align*}
        &\|\bw^{s\neq t}_\alpha(\bmM)\|_2^2 \\
        &= \sum_{\substack{i,j=1\\i\neq j}}^n  \bmW_\beta^{s'\neq t'}(\bmM)[i,j] \bmW_\sigma(\bmM)[i,j]\numberthis \label{eq:lift-proof1}\\
            &= \sum_{\substack{i,j=1\\i\neq j}}^n  \bmW_\beta^{s'\neq t'}(\bmM)[i,j]   \bmW_\sigma(\bmH)[i,j] +  \sum_{\substack{i,j=1\\i\neq j}}^n  \bmW_\beta^{s'\neq t'}(\bmM)[i,j]  (\bmW_\sigma(\bmM) - \bmW_\sigma(\bmH))[i,j]\\
            &\le \max_{1\le i<j\le n} \left|\bW_\sigma(\bmH)[i,j]\right| \sum_{i,j=1}^n \left|\bmW_\beta^{s'\neq t'}(\bmM)[i,j]\right| + \|\bmW_\sigma(\bmM) - \bmW_\sigma(\bmH)\|_{\frob}  \|\bmW_\beta^{s'\neq t'}(\bmM)\|_{\frob}\,, \numberthis\label{eq:lift-proof2}
    \end{align*}
    using H{\"o}lder on the first term 
    and Cauchy-Schwarz on the second.
    We further bound the first term with~\Cref{assumption:punctured} and~\Cref{lem:abs-value-pushed}:
    \[
        \max_{1\le i<j\le n} \left|\bW_\sigma(\bmH)[i,j]\right|\cdot \sum_{i,j=1}^n \left|\bmW_\beta^{s'\neq t'}(\bmM)[i,j]\right|\le \eps \sqrt n \cdot \|\bw_\alpha^{s\neq t}(\bmH)\|_2\,.
    \]
    For the second term, observe that
    by~\Cref{prop:to-punctured}, we know that
    the change due to puncturing is small in Frobenius norm, i.e.,
    $\|\bmW_\sigma(\bmM) - \bmW_\sigma(\bmH)\|_{\frob}\le O(1)$ for $\bmM\in \{\bmA, \bmH\}$. Moreover, in the other factor,
    $\|\bmW_\beta^{s'\neq t'}(\bmM)\|_{\frob}^2$
    is nothing but the lift of $\beta$ with
    respect to $\{p_1(s), p_1(t)\}$. This lift
    can be interpreted as:
    \begin{align*}
        \|\bmW_\beta^{s'\neq t'}(\bmM)\|_{\frob}^2 = \langle \bmw_\alpha^{s\neq t}(\bmM), \bW_{\beta'} (\bmM) \bmw_\alpha^{s\neq t}(\bmM)\rangle\,, \numberthis\label{eq:lift-proof3}
    \end{align*}
    where $\beta'$ is the lift of 
    $\alpha\left[V(\alpha)\setminus (V(\sigma)\setminus \{s,t\})\right]$ with
    respect to $\{s,t\}$. By the guarantees of~\Cref{prop:open-cactus-decomposition},
    $\alpha\left[V(\alpha)\setminus (V(\sigma)\setminus \{s,t\})\right]$ is already
    2-edge-connected, and therefore 
    so is $\beta'$. As a result, by~\Cref{thm:fundamental-theorem},
    \begin{align*}
        \|\bmW_\sigma(\bmM) - \bmW_\sigma(\bmH)\|_{\frob} \cdot \|\bmW_\beta^{s'\neq t'}(\bmM)\|_{\frob}&\le O(1)\cdot \|\bW_{\beta'} (\bmM)\|^{\frac 12} \|\bw_\alpha^{s\neq t}(\bmM)\|_2\\
        &\le O(1)\cdot \|\bw_\alpha^{s\neq t}(\bmM)\|_2\,.
    \end{align*}
We obtain
\[
    \|\bw^{s\neq t}_\alpha(\bmM)\|_2^2\le O(1 + \eps \sqrt n)\cdot \|\bw^{s\neq t}_\alpha(\bmM)\|_2\,,
\]
and the result follows after rearranging the inequality.
\end{proof}

\begin{figure}[ht!]
    \centering
    \vspace{-0.8cm}
    \begin{subfigure}[t]{0.45\linewidth}
        \centering
        \begin{minipage}[c][3.8cm][c]{\linewidth}
            \centering
            \includegraphics[width=0.4\linewidth]{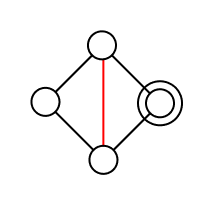}
            \vspace{-1cm}
        \end{minipage}
        \caption{Example of vector diagram $\al$ with extra open cactus in red}
    \end{subfigure}
    \begin{subfigure}[t]{0.5\linewidth}
        \centering
        \begin{minipage}[c][3.8cm][c]{\linewidth}
            \centering
            \includegraphics[width=0.5\linewidth]{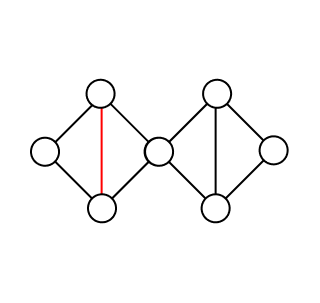}
            \vspace{-1cm}
        \end{minipage}
        \caption{Lift of $\alpha$ at the root (\cref{eq:lift-proof1})}
    \end{subfigure}

    \vspace{-0.5cm}
    \begin{subfigure}[t]{0.45\linewidth}
    \centering
    \begin{minipage}[c][4.2cm][c]{\linewidth}
        \centering
        \includegraphics[width=0.6\linewidth]{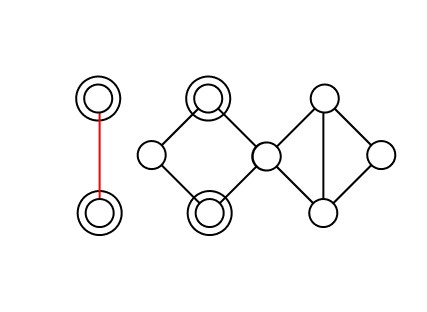}
    \end{minipage}
    \vspace{-1cm}
    \caption{Separate out open cactus (\cref{eq:lift-proof2})}
    \end{subfigure}
    \begin{subfigure}[t]{0.5\linewidth}
    \centering
    \begin{minipage}[c][4.2cm][c]{\linewidth}
        \centering
        \includegraphics[width=0.6\linewidth]{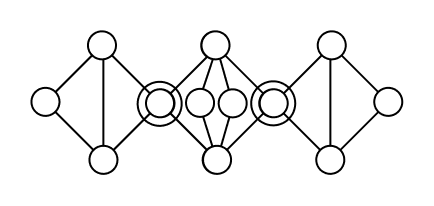}
    \end{minipage}
    \vspace{-1cm}
    \caption{Lift again. The left and right sides are copies of $\al$ while the inner part is 2-edge-connected (\cref{eq:lift-proof3}).}
    \end{subfigure}
    \vspace{1em}
    \caption{Illustration of the main diagrammatic manipulations in the proof of~\cref{prop:Z-support-vector}.}
    \label{fig:lift-proof}
\end{figure}

\subsection{Support of the \texorpdfstring{$w$}{w}-basis}

In this subsection we prove the second part
of~\Cref{thm:universality-new}:

\begin{proposition}
    \label{prop:W-support}
    Suppose that $\bH$ satisfies~\cref{eq:assNorm,eq:assOpen,eq:assOne}. Then for any $\alpha\in \calA\setminus \calE$,
    \[
        \frac 1n |w_\alpha(\bA)|\le \frac 1 {\sqrt n} \cdot (1 + \eps \sqrt n)^{O(1)}\,.
    \]
\end{proposition}

These calculations are simpler than the previous ones, but this is also the point in the proof of~\Cref{thm:universality-new} where puncturing is essential (note that it was not used at all in the previous section, and indeed those results apply equally well to the original $\bH$ or the puncturing $\bA$).

But, without puncturing, the values of graph polynomials that contain bridges can fail to be universal.
For instance, when $\alpha$ is the degree-$d$ star,
Walsh--Hadamard matrices $\bH = \bH^{(n)}$ satisfy
$w_\alpha(\bH) = \Theta(n^{d/2})$, so the limiting traffic distribution does not even exist when $d\ge 3$.
As~\Cref{prop:W-support} shows, puncturing effectively forces all such diagrams to vanish in the traffic distribution.

To prove~\Cref{prop:W-support}, we will isolate a
bridge edge in the graph, and show by induction over the tree of 2-edge-connected components that:

\begin{lemma}
    \label{lem:AW-all}
    For all $\alpha\in \calA_1$,
    \[
       \|\bA \bw_\alpha(\bA)\|_2 \le (1 + \eps \sqrt n)^{O(1)}\,.
    \]
\end{lemma}

\begin{proof}[Proof of~\Cref{prop:W-support} from~\Cref{lem:AW-all}]
    Decompose $\alpha=\alpha_1 \sqcup \alpha_2 \sqcup \{u,v\}$, where $\{u,v\}\in E(\alpha)$ is a bridge edge,
    $\alpha_1\in\calE_1$ is
    rooted at $u$, and
    $\alpha_2\in\calA_1$ is rooted at $v$. Then,
    \[
        |w_\alpha(\bA)| = |\langle \bw_{\alpha_1}(\bA), \bA \bw_{\alpha_2}(\bA)\rangle|\le \|\bw_{\alpha_1}(\bA)\|_2 \|\bA \bw_{\alpha_2}(\bA)\|_2 \le \sqrt n \cdot (1 + \eps \sqrt n)^{O(1)}\,,
    \]
    using~\Cref{thm:fundamental-theorem} on
    the first term
    and~\Cref{lem:AW-all} on the second.
\end{proof}

We prove~\Cref{lem:AW-all} by first treating 
the cactus special case
(\Cref{lem:AW-cactus}), then the 
2-edge-connected special case
(\Cref{lem:AW-2ec}), and then finally the general case by the induction mentioned above.

\begin{lemma}
    \label{lem:AW-cactus}
    For any $\alpha\in\calC_1$,
    \[
        \left\|\bA \bw_\alpha(\bA)\right\|_2\le O(1 + \eps \sqrt n)\,.
    \]
\end{lemma}

\begin{proof}
    We first decompose $\bw_\alpha(\bmA)$ as:
    \[
        \bw_\alpha(\bmA) =  (\bw_\alpha(\bmA) - \bw_\alpha(\bmH)) + \bm \Pi \bw_\alpha(\bmH) + \frac 1n \langle \bm 1, \bw_\alpha(\bmH)\rangle \bm 1\,.
    \]
    Since $\bmA \bm 1 = 0$ and $\|\bmA \|\le 1$
    by assumption, by the triangle inequality we have
    \[
        \|\bmA \bw_\alpha(\bmA)\|_2 \le \|\bw_\alpha(\bmA) - \bw_\alpha(\bmH)\|_2 + \|\bm \Pi \bw_\alpha(\bmH)\|_2\,.
    \]
    By~\Cref{cor:cactus-puncturing-same}, the first
    term is $O(1)$, and by our assumption in~\cref{eq:assOne}, the second term is at most $\eps \sqrt n$.
\end{proof}

\begin{lemma}
    \label{lem:AW-2ec}
    For all $\alpha\in\calE_1$,
    \[
        \|\bA \bw_\alpha(\bA)\|_2\le O(1 + \eps \sqrt n)\,.
    \]
\end{lemma}

\begin{proof}
    We proceed by induction on $|V(\alpha)|$.
    For $\alpha\in\calC_1$ (in particular, if $\alpha$ has only one vertex, which is our base case), the claim
    follows from~\Cref{lem:AW-cactus}.
    For $\alpha\in \calE_1\setminus \calC_1$, 
    we apply~\Cref{prop:open-cactus-decomposition}: there is an open cactus $\sigma$
    induced in $\alpha$ such that removing the
    internal vertices and edges from $\sigma$
    leaves $\alpha$ rooted and 2-edge-connected.
    Let $\{s,t\}$ be the endpoints of $\sigma$,
    and $\beta$ be the graph obtained from $\alpha$ by merging $s$ and $t$. Then, we
    can decompose:
    \[
        \bw_\alpha(\bA) = \bw_\alpha^{s\neq t}(\bA) + \bw_\beta(\bA)\,.
    \]
    On the one hand, $\beta\in\calE_1$ by~\Cref{lem:contract-2ec} 
    and has strictly less vertices than $\alpha$,
    so by induction
    \[
        \|\bA \bw_\beta(\bA)\|_2\le O(1 + \eps \sqrt n)\,.
    \]
    On the other hand, by~\Cref{prop:Z-support-vector},
    \[
        \|\bw_\alpha^{s\neq t}(\bmA)\|_2\le O(1 + \eps \sqrt n)\,.
    \]
    Putting everything together and using $\|\bmA\|\le 1$ and the triangle inequality, we obtain
    \[
        \|\bA \bw_\alpha(\bA)\|_2\le O(1 + \eps \sqrt n)\,,
    \]
    which concludes the induction.
\end{proof}

\begin{lemma}
    \label{lem:new-2ec-construction}
    Let $\al\in\calE$, $v \in V(\al)$, and
    $S$ the set of edges adjacent to $v$ in
    $\alpha$.
    Then there exists $\beta\in\calE$ and $v_1, v_2 \in V(\beta)$ such that
    \begin{align*}
        V(\beta) &= (V(\alpha)\setminus \{v\}) \cup \{v_1, v_2\}\,,\\
        E(\beta) &= (E(\alpha)\setminus S)\cup \phi(S)\cup \{\{v_1, v_2\}\}\,,
    \end{align*}
    where $\phi(e)\in \{\{v_1, u\}, \{v_2, u\}\}$ for all $e=\{v,u\}\in S$.
\end{lemma}

\begin{proof}
    We use the ear decomposition construction
    from the proof of~\Cref{lem:ear-decomposition-lemma}.
    Consider the step of the ear decomposition which adds $v$. During this step, $v$ is a new interior vertex of a path or cycle added to $\al$.
    We define $\beta$ by splitting $v$ into two vertices $v_1, v_2$ with a new edge between them.
    When other ears attach to $v$ in $\al$, we can attach them to either $v_1$ or $v_2$ in $\beta$.
    This process yields an ear decomposition for $\beta$, hence $\beta$ is also 2-edge-connected.
\end{proof}

\begin{proof}[Proof of~\Cref{lem:AW-all}]
    We proceed by induction on the number of
    2-edge-connected components in $\alpha$. If $\alpha$ is
    2-edge-connected, then the result 
    follows by~\Cref{lem:AW-2ec}. We assume
    from now on that $\alpha$ is not 2-edge-connected.

    Let $C$ be the 2-edge-connected component of the root of $\alpha$. Let $\beta_1, \ldots, \beta_k$ 
    ($k\ge 1$) be the 
    connected components disjoint from $\beta$ in 
    the graph obtained
    after removing $E(C)$ and all bridges
    incident to $C$. We root $\beta_i$ at the
    (unique) vertex of $V(\beta_i)$ adjacent to $\beta$. We also consider
    $u_1\in V(\alpha)$, the unique vertex in $V(C)$
    that is adjacent to $V(\beta_1)$.

    Let $\beta\in\calA_2$ be the graph obtained
    from $\alpha$ by adding a second root at $u_1$, and deleting $V(\beta_1)$, $E(\beta_1)$, and the bridge between $u_1$ and $\beta_1$.
    Then, for $i = 2, \ldots, k$, we
    iteratively apply the graph transformation from~\Cref{lem:new-2ec-construction},
    label the new edge $e=\{v_1,v_2\}$ by $\bm A_e = \text{diag}(\bm A \bm w_{\beta_i}(\bmA))$, and transfer the old labels for all
    other edges. In this way, we obtain
    a 2-edge-connected graph $\beta'\in\calE_2$ and a family of matrices $\bm \calA = (\bm A_{e})_{e\in E(\beta')}$ such that \[\bW_{\beta'}(\bm \calA) = \bW_{\beta}(\bA)\,.\] All involved matrices $\bA_e\in\bm \calA$ are either $\bm A$ or of the form 
    $\text{diag}(\bm A \bm w_{\beta_i}(\bmA))$ for some $i\in \{2, \ldots, k\}$, so they
    satisfy
    $\|\bA_e\|\le (1 + \eps \sqrt n)^{O(1)}$ by induction.
    Next, applying~\Cref{thm:fundamental-theorem}, we get
    \[
        \|\bW_\beta(\bA)\| = \|\bW_{\beta'}(\bm \calA)\|\le (1 + \eps \sqrt n)^{O(1)}\,.
    \]
    As a result,
    \[
        \|\bA \bw_\alpha(\bA)\|_2 = \|\bA \bW_\beta(\bA) \bA \bw_{\beta_1}(\bA)\|_2\le \|\bA\| \cdot \|\bW_\beta(\bA)\| \cdot \|\bA \bw_{\beta_1}(\bA)\|_2 \le (1 + \eps \sqrt n)^{O(1)}\,,
    \]
    using again induction on $\|\bA \bw_{\beta_1}(\bA)\|_2$. This concludes the induction.
\end{proof}

\subsection{Putting everything together\texorpdfstring{: Proof of~\Cref{thm:universality-new}}{}}

\begin{proof}[Proof of~\Cref{thm:universality-new}]
    The first part
    follows from~\Cref{prop:Z-support}, and  the second part
    follows from~\Cref{prop:W-support}.
    For the third part, suppose that $\bH$ satisfies~\cref{eq:assNorm,eq:assOpen,eq:assOne} with $\eps^{(n)} = n^{-\frac 12 + o(1)}$. Summarizing, we know that:
    \begin{enumerate}
        \item For all $\sigma\in\calC$, $\frac 1n w_\sigma(\bA)\underset{n\to\infty}{\longrightarrow} m_\sigma\in \R$ by
        assumption.
        \item For all $\alpha\in\calE\setminus \calC$, $\frac 1n z_\alpha(\bA)\underset{n\to\infty}{\longrightarrow} 0$ by the first part.
        \item For all $\alpha\in\calA\setminus \calE$, $\frac 1n w_\alpha(\bA)\underset{n\to\infty}{\longrightarrow} 0$ by the second part.
    \end{enumerate}
    By~\Cref{lem:eqDiffMobius}, the traffic distribution
    of $\bA$ then exists and is uniquely determined
    by $\{m_\sigma : \sigma\in \calC\}$, completing the proof.
\end{proof}

\section{From Diagrams to Asymptotic GFOM Dynamics}
\label{sec:dynamics}

The traffic distribution captures the limiting behavior of all scalar-valued,
permutation-invariant polynomials. In this section, we show how to leverage this
information to derive the limiting empirical laws of vector-valued,
permutation-invariant polynomials. Our main application is a description of the
limiting dynamics of GFOM.

We will mostly work under the assumption that the input matrices $\bA = \bA^{(n)}$ satisfy the strong cactus property, which we recall is the statement that $\frac{1}{n}\E_\bA z_{\alpha}(\bA) \to 0$ as $n \to \infty$ for all non-cactus $\alpha$ (i.e., all $\alpha \in \cA \setminus \cC$, a statement about \emph{scalar} graph polynomials).
In~\Cref{sec:dynamics-reg} we will briefly suspend this assumption to discuss punctured matrices, so as to connect to the setting of~\Cref{sec:universality}.

We tackle two tasks in this section:
\begin{enumerate}
    \item First, we study the joint asymptotic limit of the empirical distributions of the vector diagrams $\bmz_\al(\bA)$ over $\al \in \calA_1$. 
    Assuming the strong cactus property, we show that only the small subset of \emph{treelike} $\al \in \cT_1$ are asymptotically nonzero in the $z$-basis, in a sense to be made precise below.
    We then show that the asymptotic algebra of the treelike diagrams is isomorphic to a \emph{Wick algebra}, an algebra defined by a family of Gaussian random variables.
    This will give a precise version of \Cref{prop:convergence-distribution}.
    \item Second, we work with the asymptotic limit of treelike diagrams to identify a generalized Onsager correction, derive a treelike Approximate Message Passing algorithm, and prove its state evolution over arbitrary input matrices having the strong cactus property and a limiting diagonal distribution.
    This will give a precise version of \Cref{thm:intro-onsager}.
\end{enumerate}

\subsection{Asymptotic limit of the vector diagrams}
\label{sec:vector-limit}

In this section, given a family $(X_i)_{i\in I}$ and
$J\subseteq I$, we will write as a shortcut $X_J = (X_j)_{j\in J}$. 

Recall that $\cC_1$ denotes the set of rooted cactuses and $\cT_1 \subseteq \cA_1$ denotes the set of rooted trees with hanging cactuses.
We call the diagrams in $\cT_1$ \emph{treelike}, and we call \emph{Gaussian trees} the subset of diagrams $\cG_1 \subseteq \cT_1$ such that the root has degree exactly $1$ after removing hanging cactuses.

\begin{definition}[Type]
    For each $\tau\in\cT_1$, let
    $\type(\tau) \in \N^{\cG_1\cup \calC_1}$, where $\type(\tau)_\alpha$ count the
    number of copies of $\alpha\in \calG_1\cup \calC_1$ attached to the root of $\tau$,
    with the additional convention that
    $\type(\tau)_\alpha=0$ for all
    $\alpha\in\calG_1$ that has cactuses hanging at the root.
\end{definition}

The following theorem identifies the limiting distribution of $\bz_{\calA_1}(\bA)$ under the strong cactus property.
We refer the reader to~\Cref{sec:convergence-process} for the definition of convergence in distribution for random elements indexed by countably infinite index sets.

\begin{theorem}
\label{thm:convergence-distribution}
    Assume that $\bA=\bA^{(n)}$ satisfies~\cref{eq:tightnessNorm}, has the strong cactus property, and a limiting diagonal distribution.
    Then,
    \[
        \samp(\bmz_{\cA_1}(\bA)) \tod Z_{\cA_1}^{\infty}\,,
    \]
    where $Z_{\calA_1}^{\infty} \in \R^{\calA_1}$ is a random variable satisfying the following properties:
    \begin{enumerate}
        \item $Z_{\alpha}^\infty = 0$ for all non-treelike $\alpha$.

        \item Conditioned on $Z^\infty_{\calC_1}$, $Z^\infty_{\calG_1}$ is a centered
        Gaussian process with covariance $\bm \Sigma^\infty$ from~\cref{eq:covariance}.
        
        \item Let $\He$ denote the Wick product (\Cref{def:wick}). Then
        for every $\tau\in\calT_1$,
        \[
            Z^\infty_\tau
            = \He_{\type(\tau)}(Z^\infty_{\calG_1}\,;\,\bm \Sigma^\infty)\cdot \prod_{\sigma\in \calC_1} (Z^\infty_{\sig})^{\type(\tau)_\sigma}\,.
        \]
    \end{enumerate}
\end{theorem}

\Cref{thm:convergence-distribution} shows how
the limiting algebra $Z^\infty_{\calA_1}$
of permutation-invariant, vector-valued polynomials in
$\bA$ can be derived from $Z^\infty_{\calC_1}$.
Although we have not specified the description of the law of $Z^\infty_{\calC_1}$, it is fully
determined by the limiting diagonal distribution of $\bA$. For example, when $\bA$ further
satisfies the
{\em factorizing} strong cactus property, $Z^\infty_{\calC_1}$ is deterministic:

\begin{proposition}
    \label{prop:factorizing-deterministic}
    If $\bA$ satisfies the factorizing strong 
    cactus property and~\cref{eq:tightnessNorm},
    then the conclusion of~\Cref{thm:convergence-distribution} holds with the additional property that for every
    $\sigma\in\calC_1$,
    \[
        Z^\infty_{\sigma} = \prod_{\rho\in\cyc(\sigma)} \left(\lim_{n\to\infty} \frac 1n \E z_\rho(\bA)\right)\,.
    \]
\end{proposition}

\begin{proof}
    Let $\sigma\in\calC_1$.
    The first moment of
    $\samp(\bz_\sigma(\bA))$ is
    \begin{align}
        \E \samp(\bz_\sigma(\bA)) = \frac 1n \sum_{i=1}^n \E_\bA \bz_\sigma(\bA)[i] = \frac 1n \E_\bA z_{\sigma_0}(\bA)\,,\label{eq:first-prop53}
    \end{align}
    where $\sigma_0$ is the unrooted version of $\sigma$. As $n\to\infty$,~\cref{eq:first-prop53} converges to the deterministic
    constant
    \[
        \kappa_{\sigma_0} \defeq \lim_{n\to\infty} \frac 1n \E_\bA z_{\sigma_0}(\bA) = \prod_{\rho\in\cyc(\sigma_0)} \left(\lim_{n\to\infty} \frac 1n \E z_\rho(\bA)\right)
    \]
    by the factorizing cactus property.
    
    We now switch to the second
    moment,
    \[
        \E \samp(\bz_\sigma(\bA))^2 =  \frac 1n \E_\bA \sum_{i=1}^n  \bz_\sigma(\bA)[i]^2\,.
    \]
    Expand the scalar polynomial $\sum_{i=1}^n \bz_\sigma(\bA)[i]^2$
    in the $z$-basis. The support of that
    expansion is
    the set of diagrams that can be
    obtained by grafting two copies
    of $\sigma$ at the root and merging
    pairs of vertices across the two
    different copies. By the strong cactus property, it suffices to find
    which cactuses can be obtained in
    this way. By~\Cref{lem:prod-cactus-cactus}, the only
    cactus that can occur in this way
    has no merging, and it contributes
    $\kappa_{\sig_0}^2$ by the factorizing cactus property.
    Thus,
    \[
        \E \samp(z_\sig(\bA))^2 = \kappa_{\sig_0}^2 + o(1) = \left(\E \samp(z_\sig(\bA)\right)^2 + o(1)\,.
    \]

    We showed that $\samp (z_\sigma(\bA))$
    converges to the desired 
    deterministic quantity in expectation,
    and furthermore that
    its variance converges to $0$.
    This implies that it converges to
    the constant in distribution. By
    unicity of the limit in distribution,
    $Z_\sigma^\infty$
    equals that constant almost surely.
\end{proof}

However, if we drop the factorizing cactus property assumption,
the variables $Z^\infty_{\calC_1}$ may no longer be deterministic. 
For example, this can be the case
when $\bA$ is a block-structured matrix as
in~\Cref{sec:wigner-block}:
 
 \begin{example}\label{ex:block-structured}
    Let $\bm A_1^{(n)}$ and $\bm A_2^{(n)}$
    be two $n\times n$ matrices
    satisfying the assumptions of~\Cref{thm:convergence-distribution}. Define
    the $2n\times 2n$ matrix,
    \[
        \bmA^{(2n)} = \begin{bmatrix} \bmA_1^{(n)} & \bm 0\\\bm 0 & \bmA_2^{(n)} \end{bmatrix}
    \]
    From the block-diagonal structure, for any $\alpha\in \calA_1$,
    \[
        \bmz_{\alpha}(\bmA)[i] = \begin{cases} \bmz_{\alpha}(\bmA_1)[i] & \text{if $i\in [n]$}\\
        \bmz_{\alpha}(\bmA_2)[i-n] & \text{if $i\in [2n]\setminus [n]$}\end{cases}
    \]
    Hence, the law of $Z^\infty_{\calC_1}(\bmA)$ is a
    uniform mixture of the law of $Z^\infty_{\calC_1}(\bmA_1)$
    and that of $Z^\infty_{\calC_1}(\bmA_2)$.
\end{example}

\noindent We will prove a generalization of~\Cref{ex:block-structured} later; see~\Cref{lem:block-matrix-cactus-limits}.

In~\Cref{ex:block-structured}, the randomness of $Z^\infty_{\calC_1}$
may be viewed as coming solely from the $\samp(\cdot)$ operator,
but this is not always the case.
For instance, our model also captures orthogonally invariant distributions
that do not satisfy the traffic concentration property:

\begin{example}\label{ex:finetti}
    Let $(\lambda_n)_{n\ge 1}$
    be an exchangeable sequence of random variables in $[-1,1]$ and consider
    \[\bA^{(n)} = (\bQ^{(n)})^\top \diag(\lambda_1, \ldots, \lambda_n) \bQ^{(n)}\,,\]
    for Haar-distributed matrices $\bQ^{(n)} \in O(n)$, independent from $(\lambda_n)$. By de Finetti's
    theorem, there exists a latent random
    probability measure $\mu$ almost
    surely supported on $[-1,1]$ such that
    conditionally on $\mu$, $\lambda_1,\lambda_2,\ldots$ are i.i.d.\ with common law $\mu$.
    By~\Cref{thm:moments}, $\bA^{(n)}$ satisfies the strong
    cactus property conditionally on
    $\mu$, so it also satisfies the strong
    cactus property unconditionally.

    Applying~\Cref{thm:convergence-distribution} and \Cref{prop:factorizing-deterministic},
    we get that conditionally on $\mu$,
    $\samp(z_{\calC_1}(\bA))$ converges in
    distribution to
    \begin{align}
        \left(\prod_{\rho\in\cyc(\sigma)} \kappa_{|\rho|}(\mu)\right)_{\sigma\in \calC_1}\,,\label{eq:random-cumulants}
    \end{align}
    where $(\kappa_q(\mu))_{q\ge 1}$
    are the free cumulants of $\mu$.
    Therefore, {unconditionally}, 
    $\samp(z_{\calC_1}(\bA))$ converges
    in distribution to the {\em random}
    quantity~\cref{eq:random-cumulants}.
\end{example}

\noindent Note that \Cref{ex:block-structured,ex:finetti} do not contradict~\Cref{prop:factorizing-deterministic} because in these
examples,
$\bA^{(n)}$
typically does not 
satisfy the factorizing
cactus property.

\subsubsection{Non-treelike diagrams are asymptotically negligible}
\label{sec:treelike}

The remainder of~\Cref{sec:vector-limit} is dedicated to
the proof of~\Cref{thm:convergence-distribution}. In the whole proof, we drop the dependence
of $z_\alpha$ and $w_\alpha$ on $\bA$ 
to lighten notation.
We start by proving that
non-treelike diagrams are negligible.

\begin{lemma}
    \label{lem:non-treelike}
    Suppose that $\bA$ satisfies the
    strong cactus property. Then for each non-treelike $\alpha$,
    \[
        \samp(\bz_\alpha)\overset{L^2}{\longrightarrow} 0\,.
    \]
\end{lemma}

\begin{proof}
    By definition, we have
    \begin{align*}
        \E \samp(\bz_\alpha)^2 &= \frac 1n \sum_{i=1}^n \E \left[\left(\bz_\alpha^{ 2}\right)[i]\right]
    \intertext{By~\Cref{claim:only-trees}, we can expand $\bz_\alpha^{ 2}$ in the
    $\bz$-basis to obtain, for some constant coefficients
    $c_\beta$}
        &=  \frac 1 n \sum_{i=1}^n  \sum_{\beta\in \calA_1\setminus \calT_1} c_\beta \E \bz_\beta[i]\\
        &= \frac 1n \sum_{\beta\in \calA_0\setminus \calT_0} c_\beta' \E z_\beta
    \end{align*}
    for some other constant coefficients $c_\beta'$. Since no diagram in
    $\calA_0\setminus \calT_0$ is a cactus, by the strong
    cactus property, we get $\E \samp(\bz_\alpha)^2\underset{n\to\infty}\longrightarrow 0$, as desired.
\end{proof}

\subsubsection{Asymptotic limit of the treelike diagrams}
\label{sec:asymptotic-space}

Next, we analyze the treelike diagrams. All results in~\Cref{sec:asymptotic-space} are purely
combinatorial, meaning that they hold for arbitrary
$\bA\in\R^{n\times n}_\sym$.

The covariance of treelike diagrams is
defined in terms of {\em homeomorphic
matchings} between them. We start by defining this new concept.

\begin{definition}[Core]
    Let $\tau\in\calT_1$.
    Define $\textnormal{core}(\tau)$ to be
    the rooted tree obtained from
    $\tau$ by
    \begin{enumerate}
        \item Removing all hanging cactuses.
        \item Removing all non-root
        degree-2 vertices and the two edges they are incident with, and adding back a new edge between their two neighbors.
    \end{enumerate}
\end{definition}
\noindent
Note that the vertex set $V(\core(\tau))$ may be identified with a subset of $V(\tau)$, even though the second rule may lead to edges being present in $\core(\tau)$ that do not exist in $\tau$.

\begin{definition}[Homeomorphic matchings]\label{def:homeomorphic}
    Let $\tau_1,\tau_2\in\calT_1$. We
    say that a partial matching
    $P\subseteq V(\tau_1)\times V(\tau_2)$ of $\tau_1$ and $\tau_2$ is
    {\em homeomorphic} if
    \begin{enumerate}
        \item $(\textnormal{root}(\tau_1),\textnormal{root}(\tau_2))\in P$.
        \item Restricted to $V(\textnormal{core}(\tau_1))\times V(\textnormal{core}(\tau_2))$, $P$ is a rooted graph isomorphism between $\textnormal{core}(\tau_1)$ and $\textnormal{core}(\tau_2)$.

        \item Let $\{u,u'\}\in E(\textnormal{core}(\tau_1))$,
        let $(u=u_1, \ldots, u_k=u')$ be the path between $u$ and $u'$ in $\tau_1$. Let $v=P(u)$,  $v'=P(u')$, and $(v=v_1,\ldots,v_\ell=v')$ be the
        path between $v$ and $v'$ in $\tau_2$.
        Then there is no matching edge between
        $\{u_1, \ldots, u_k,v_1,\ldots,v_\ell\}$
        and its complement.
        Moreover,
        for all $(u_i,v_j)\in P$ and $(u_{i'},v_{j'})\in P$, we have $i\le i' \iff j\le j'$ (the matching restricted to the vertices in the paths is non-crossing).

        \item No inner vertices from the 
        hanging cactuses are matched.
    \end{enumerate}
    We denote by $H(\tau_1,\tau_2)$ the set
    of homeomorphic matchings between
    $\tau_1$ and $\tau_2$. 
\end{definition}

\noindent This definition
is motivated by the following lemma 
stating that, when 
computing the covariance of two treelike
diagrams, the matchings giving rise to
cactuses are precisely the homeomorphic ones.

\begin{lemma}
    \label{lem:homeomorphic}
    Let $\tau_1,\tau_2\in\calT_1$ and $\tau = \tau_1 \sqcup \tau_2$.
    For any matching $P\subseteq V(\tau_1)\times V(\tau_2)$ such that $(\roott(\tau_1), \roott(\tau_2)) \in P$, we have $\tau_P\in\calC_1$ if and only
    if $P\in H(\tau_1,\tau_2)$.
\end{lemma}

\noindent In particular, if $\tau_1, \tau_2 \in \cC_1$, only the matching $P = \{(\roott(\tau_1), \roott(\tau_2))\}$ creates a cactus $\tau_P$.
We are now ready to describe the algebra of
treelike diagrams:

\begin{lemma}\label{lem:product-gaussian}
    For all $\gamma_1,\ldots,\gamma_\ell\in\calG_1\cup \calC_1$,
    \begin{align}
        \prod_{j=1}^\ell \bz_{\gamma_j} - \sum_{M\in\calM(\ell)} \sum_{\substack{P_{uv}\in H(\gamma_u,\gamma_v)\\\forall uv\in M}} \bz_{\bigoplus_{uv\in M}\gamma_{P_{u,v}} \,\oplus\, \bigoplus_{u\notin M} \gamma_u} \in \spn(\bz_{\calA_1\setminus \calT_1})\,,\label{eq:product-gaussian}
    \end{align}
    where $\oplus$ denotes the grafting at the root.
\end{lemma}

\noindent The proofs of~\Cref{lem:homeomorphic,lem:product-gaussian} are deferred to~\Cref{sec:combinatorics}.
Note that the error in~\Cref{lem:product-gaussian}
is measured in terms of
non-treelike diagrams. 

By inverting \cref{eq:product-gaussian},
we can formulate the algebra of treelike diagrams in the language of {\em Wick products} (\cref{def:wick}).

\begin{corollary}\label{lem:treelike-hermite}
    For all $\tau\in\calT_1$,
    \begin{align}
        \bz_{\tau} - \He_{\type(\tau)}(\bz_{\calG_1}\,;\bm \Sigma) \prod_{\sigma\in\calC_1} (Z^\infty_{\sigma})^{\type(\tau)_\sigma} \in \spn(\bz_{\calA_1\setminus \calT_1})\,, \label{eq:reformulated-hermite}
    \end{align}
    where for all $\gamma,\gamma'\in\cG_1$, we defined
    the ``finite-$n$'' covariance matrix
    \begin{align}
        \bm \Sigma[\gamma,\gamma'] \defeq \sum_{P\in H(\gamma, \gamma')} \bz_{\gamma_P}\,.\label{eq:covariance-finite}
    \end{align}
\end{corollary}

\begin{proof}
    We proceed by induction on the number
    of vertices of $\tau$. First,~\cref{eq:reformulated-hermite} trivially holds
    if $\tau$ has one vertex, which proves
    the base case.
    Now, suppose that $\tau = \gamma_1\oplus \ldots \oplus \gamma_\ell$ is the
    grafting at the root of $\gamma_1,\ldots,\gamma_{\ell}\in\calG_1\cup \calC_1$.
    By~\Cref{lem:product-gaussian},
    \begin{align*}
        \bmz_\tau + \sum_{\substack{M\in \calM(\ell)\\M\neq \varnothing}} \sum_{\substack{P_{uv}\in H(\gamma_u,\gamma_v)\\\forall uv\in M}} \bz_{\bigoplus_{uv\in M}\gamma_{P_{u,v}} \,\oplus\, \bigoplus_{u\notin M} \gamma_u} - \prod_{j=1}^\ell \bmz_{\gamma_j}\in \spn(\bmz_{\calA_1\setminus \calT_1})\,.
    \end{align*}
    Applying the induction hypothesis and
    using additivity of types, we have:
    \begin{align*}
        \bz_{\bigoplus_{uv\in M}\gamma_{P_{u,v}} \,\oplus\, \bigoplus_{u\notin M} \gamma_u} - \prod_{uv\in M} \bmz_{\gamma_{P_{uv}}} \prod_{\substack{u\notin M\\\gamma_u\in \calC_1}} \bmz_{\gamma_u} \cdot \He_{\sum_{\substack{u\notin M\\\gamma_u\in\calG_1}} \type(\gamma_u)}(\bmz_{\cG_1}\,;\,\bm \Sigma)\in \spn(\bmz_{\calA_1\setminus \calT_1})\,.
    \end{align*}
    Since cactuses are not matched
    by homeomorphic matchings by
    definition, the
    product over cactuses $\bmz_{\gamma_u}$ is over all
    $u$ such that $\gamma_u\in \calC_1$, which is independent of $M$ and can be 
    factorized out. Therefore,
    in the rest of the proof
    we assume that $\gamma_i\in \calG_1$ for all $i\in [\ell]$.
    Using~\Cref{claim:non-treelike},
    we obtain
    \begin{equation}
         z_\tau + \sum_{\substack{M\in \calM(\ell)\\M\neq \varnothing}} \prod_{uv\in M} \underbrace{\sum_{P\in H(\gamma_u,\gamma_v)} \bmz_{\gamma_P}}_{\bm \Sigma[\gamma_u,\gamma_v]} \cdot  \He_{\sum_{u\notin M} \type(\gamma_u)} (\bmz_{\cG_1}\,;\, \bm \Sigma) - \prod_{j=1}^\ell \bmz_{\gamma_j}\in \spn(\bmz_{\calA_1\setminus \calT_1})\,.\label{eq:summary-wick}
    \end{equation}
    By the recursive formula of the
    Wick products (\Cref{lem:wick-recursion}),
    \begin{equation}
        \sum_{\substack{M\in \calM(\ell)\\M\neq \varnothing}} \prod_{uv\in M} \bm \Sigma[\gamma_u,\gamma_v] \cdot  \He_{\sum_{u\notin M} \type(\gamma_j)} (\bmz_{\cG_1}\,;\, \bm \Sigma) + \He_{\type(\tau)}(\bmz_{\cG_1}\,;\,\bm \Sigma)  = 
         \prod_{j=1}^\ell \bmz_{\gamma_j}\,.\label{eq:wick-application}
    \end{equation}
    Combining~\cref{eq:summary-wick,eq:wick-application} concludes the proof.
\end{proof}

Finally, if we reduce~\Cref{lem:product-gaussian} 
modulo the larger class of non-cactus diagrams
(which are the negligible diagrams {\em in expectation} under
the strong cactus property),
we deduce that the 
joint moments of the diagrams in
$\cG_1$ have an asymptotically Gaussian structure.

\begin{corollary}\label{lem:treelike-gaussian}
    For all $\gamma_1,\ldots,\gamma_\ell\in \calG_1$ and $\sigma_1,\ldots,\sigma_k\in\calC_1$,
     \[
        \prod_{i=1}^k \bmz_{\sigma_i} \left[\prod_{j=1}^\ell \bz_{\gamma_j}  - \sum_{M\in \PM(\ell)} \prod_{xy\in M} \bm \Sigma[\gamma_x,\gamma_y] \right]\in\spn(\bz_{\calA_1\setminus \calC_1})\,,
     \]
     where $\bm \Sigma$ is defined in~\cref{eq:covariance-finite}.
\end{corollary}

\begin{proof}
    Every non-treelike term in~\cref{eq:product-gaussian}
    is a fortiori not a cactus. Also, the only
    cactuses in the subtracted term occur
    when $M$ is a perfect matching. In other words,
    \[
        \prod_{i=1}^k \bmz_{\sigma_i}\left[\prod_{j=1}^\ell \bmz_{\gamma_j} - \sum_{M\in \PM(\ell)} \sum_{\substack{P_{uv}\in H(\gamma_u, \gamma_v)\\\forall uv\in M}} \bmz_{\bigoplus_{uv\in M} \gamma_{P_{uv}}}\right]\in \spn(\bmz_{\calA_1\setminus \calC_1})\,.
    \]
    Therefore, by~\Cref{claim:only-trees}, we deduce
    \[
        \bmz_{\bigoplus_{uv\in M} \gamma_{P_{uv}}} - \prod_{uv\in M} \bmz_{\gamma_{P_{uv}}}\in \spn(\bmz_{\calA_1\setminus \calC_1})\,,
    \]
    and the desired statement follows.
\end{proof}

\subsubsection{Proof of~\Cref{thm:convergence-distribution}}
\label{sec:proof-convergence}

\begin{claim}
    \label{claim:existence-joint-moments}
    Suppose that the traffic distribution
    of $\bA$ exists. Then, for any
    $\alpha_1, \ldots, \alpha_k\in\calA_1$,
    the sequence $\E \samp(\bz_{\alpha_1} \cdots \bz_{\alpha_k})$
    converges as $n\to\infty$.
\end{claim}

\begin{proof}
    This is straightforward, as
    \[
        \E \samp(\bz_{\alpha_1} \cdots \bz_{\alpha_k}) = \frac 1n \EE \sum_{i=1}^n \bz_{\alpha_1}[i] \cdots \bz_{\alpha_k}[i],
    \]
    and the inner polynomial is a scalar polynomial of $\bA$ that can be expanded in the $z$-basis of scalar diagrams as a linear combination of various quotients of the scalar diagram formed by forgetting the identity of the root in $\alpha_1 \oplus \cdots \oplus \alpha_k$.
\end{proof}

\noindent \Cref{claim:existence-joint-moments}
implies in particular 
that the sequence $\samp(\bz_{\calA_1})$ is
tight. In the rest of the proof, we 
show that the limit in distribution
actually exists and characterize it.
The following lemma is a direct consequence
of the fundamental theorem of graph
polynomials.

\begin{lemma}
    \label{lem:2ec-bounded}
    If $\|\bA\|\le O(1)$, then for each $\alpha\in\calE_1$, there
    exists $C_\alpha>0$ such that
    $|\samp(\bz_\alpha)|\le C_\alpha$.
\end{lemma}

\begin{proof}
    By~\Cref{claim:mobius} and~\Cref{lem:contract-2ec}, we can
    expand for some coefficients $c_\beta = c_\beta(\alpha)\in\R$,
    \[
        \bz_\alpha = \sum_{\beta\in\calE_1} c_\beta \bw_\beta\,.
    \]
    By~\Cref{thm:fundamental-theorem},
    it holds for every $\beta\in\calE_1$
    that $\|\bw_\beta\|_\infty\le \|\bA\|^{|E(\beta)|}$, which is at most
    $O_\alpha(1)$ by assumption.
    The lemma follows by the triangle inequality.
\end{proof}

\begin{lemma}
    \label{lem:cactus-converge}
    Suppose that the traffic distribution
    of $\bA$ exists and that~\cref{eq:tightnessNorm} holds. Then,
    $\samp(\bz_{\calC_1})$
    converges in distribution to some
    stochastic process $Z^\infty_{\calC_1}$.
\end{lemma}

\begin{proof}    
    First, assume that $\sup_{n\ge 1} \|\bm A^{(n)}\|\le K$ holds almost surely, for some universal constant $K>0$. All the moments of $\samp(\bz_{\calC_1})$ converge
    by~\Cref{claim:existence-joint-moments}.
    Since cactuses are 2-edge-connected, by~\Cref{lem:2ec-bounded}, all random
    variables $\samp(\bz_\alpha)$
    for $\alpha\in\calC_1$ are uniformly
    bounded in $n$. Hence, the moments satisfy the growth condition~\cref{eq:subexp-growth-condition}, so that
    $\samp(\bz_{\calC_1})$ converges
    in distribution by~\Cref{cor:bounded-limit-existence}. 
    Finally, if we assume~\cref{eq:tightnessNorm} rather
    than uniform boundedness, the result
    can be deduced from the latter case
    using~\Cref{lem:truncation}.
\end{proof}

\begin{proof}[Proof of~\Cref{thm:convergence-distribution}]
In the rest of the proof, we assume that
the assumptions of~\Cref{thm:convergence-distribution} are satisfied. 
We start by analyzing convergence of the subtracted term from~\Cref{lem:treelike-gaussian}. By convergence in distribution of the cactuses (\Cref{lem:cactus-converge}) and the continuous mapping theorem, we have for any 
$\gamma_1,\ldots,\gamma_\ell\in\calG_1$ and $\sigma_1,\ldots,\sigma_k\in \calC_1$,
\begin{align}
    \samp\left(\prod_{i=1}^k \bmz_{\sigma_i} \sum_{M\in \PM(\ell)} \prod_{xy\in M} \bm \Sigma[\gamma_x,\gamma_y]\right) \tod \prod_{i=1}^k Z^\infty_{\sigma_i} \sum_{M\in \PM(k)} \prod_{xy\in M} \bm \Sigma^\infty[\gamma_x,\gamma_y]\,,\label{eq:gaussian-cmt}
\end{align}
where we defined, for any $\gamma_1,\gamma_2\in\calG_1$, the ``limiting'' covariance matrix
\begin{align}
    \bm \Sigma^\infty[\gamma_1,\gamma_2] \defeq  \sum_{P\in H(\gamma_1, \gamma_2)} Z^\infty_{\gamma_P}\,.\label{eq:covariance}
\end{align}
Since all joint moments converge
by~\Cref{claim:existence-joint-moments},
the sequence
of random
variables on the left-hand side of~\cref{eq:gaussian-cmt} is uniformly
integrable. So we also get convergence 
of the mean,
\[
    \E \samp\left(\prod_{i=1}^k \bmz_{\sigma_i} \sum_{M\in \PM(\ell)} \prod_{xy\in M} \bm \Sigma[\gamma_x,\gamma_y]\right) \underset{n\to\infty}\longrightarrow \E\left[\prod_{i=1}^k Z^\infty_{\sigma_i} \sum_{M\in \PM(k)} \prod_{xy\in M} \bm \Sigma^\infty[\gamma_x,\gamma_y]\right]\,.
\]   
Combining with~\Cref{lem:treelike-gaussian} and the strong cactus property,
\begin{align}
    \E \samp\left(\prod_{i=1}^k \bmz_{\sigma_i} \prod_{j=1}^\ell \bz_{\gamma_j}\right) \underset{n\to\infty}\longrightarrow \E\left[ \prod_{i=1}^k Z^\infty_{\sigma_i} \sum_{M\in \PM(\ell)} \prod_{xy\in M} \bm \Sigma^\infty[\gamma_x,\gamma_y]\right]\,.\label{eq:covariance-converge}
\end{align}
The right-hand side of~\cref{{eq:covariance-converge}}
coincides with the moments of $Z^\infty_{\calG_1\cup \calC_1}$. Recall that the law of
$Z^\infty_{\calG_1}$ satisfies that after sampling
$Z^\infty_{\calC_1}$ from its marginal
(which is bounded almost surely by~\Cref{lem:2ec-bounded}), then $Z^\infty_{\calG_1}$
conditioned on $Z^\infty_{\calC_1}$ is a Gaussian
process with covariance kernel given by~\cref{eq:covariance}. This
object satisfies the moment growth
condition~\cref{eq:subexp-growth-condition}. So~\Cref{cor:bounded-limit-existence}
applies and we obtain convergence in
distribution of $\samp(\bz_{\calG_1\cup \calC_1})$
to $Z^\infty_{\calG_1\cup \calC_1}$.

By~\Cref{lem:non-treelike}, the non-treelike diagrams converge in $L^2$ to 0,
so by Slutsky's lemma, we obtain joint
convergence in distribution, except for the remaining treelike, non-Gaussian trees. By~\Cref{lem:treelike-hermite}, these are continuous images
of cactuses and non-treelike diagrams, so
by the continuous mapping theorem, all
diagrams
converge jointly in distribution to $Z^\infty_{\calA_1}$.
\end{proof}

\subsection{The treelike AMP algorithm}

Now we turn to studying the dynamics of GFOM operations.

\begin{definition}[Asymptotic state]\label{def:asymptotic-state}
    Let $(\bm x_i)_{i\in \calI}$ be a
    family of random vectors,
    $\bm x_i\in \R^{n}$.
    We say that a stochastic process $(X_i)_{i\in \calI}$
    is the asymptotic state of $(\bm x_i)_{i\in \calI}$ if, for any $k\ge 1$, $i_1, \ldots, i_k\in \calI$, and any bounded continuous
    or polynomial function $\varphi:\R^k\to \R$,
    \begin{equation}
        \lim_{n\to\infty} \frac 1n \sum_{j=1}^n \E\varphi(\bm x_{i_1}[j], \ldots, \bm x_{i_k}[j]) = \E\varphi(X_{i_1}, \ldots, X_{i_k})\,.\label{eq:defAsymptoticState}
    \end{equation}
\end{definition}

\noindent \Cref{def:asymptotic-state} requires in particular for $(\bm x_i)_{i\in \mathcal I}$ to converge in distribution to $(X_i)_{i\in \calI}$. 
As with convergence in distribution in general, this suffers from the caveat that the law of the limit in distribution of $(X_i)_{i\in \calI}$ is unique, but the probability space on which the limit $(X_i)_{i\in \calI}$ is realized is not.
Thus when we speak of ``the asymptotic state'' we refer to a specific law, not a specific collection of random variables.
Nonetheless, the sampling procedure in \Cref{thm:convergence-distribution} suggests a natural
way to sample an asymptotic state of the
iterates of a pGFOM, since, provided we know how to sample from $Z_{\calC_1}^{\infty}$ (which we must address on a case-by-case basis), the other $Z_{\alpha}^{\infty}$ are conditionally Gaussian or deterministic functions thereof.

Translating the limiting variables $Z_{\alpha}^{\infty}$ from \Cref{thm:convergence-distribution} to a construction of an asymptotic state, we find:
\begin{lemma}\label{claim:cvImpliesState}
    Assume that $\bA = \bA^{(n)}$ satisfies the assumptions of~\Cref{thm:convergence-distribution}.
    Let 
    \begin{equation}
        \bm x = \sum_{\alpha\in \calA_1} c_\alpha \bmz_\alpha(\bmA)\label{eq:assumeState}
    \end{equation}
    for
    some finitely supported coefficients
    $(c_\alpha)_{\alpha\in \calA_1}$. Then,
    \begin{equation}
        X \defeq \sum_{\alpha\in \calA_1} c_\alpha Z^\infty_\alpha\label{eq:asymptState}
    \end{equation}
    is the asymptotic state of $\bm x$.
    Moreover, if $\bm x_t$ is of the form~\cref{eq:assumeState} for any $t\ge 1$ and $X_t$
    is correspondingly defined as in~\cref{eq:asymptState}, then $(X_t)_{t\ge 1}$ is the asymptotic
    state of $(\bm x_t)_{t\ge 1}$. 
\end{lemma}

\noindent We emphasize that the index set
$t\ge 1$ is independent of $n$, and so our
results hold for all fixed iterates $t$
independent of $n$, in the limit $n\to\infty$.

\begin{proof}
    The statement for bounded continuous test
    functions $\varphi$ follows from~\Cref{thm:convergence-distribution} and
    the continuous mapping theorem.
    For polynomial $\varphi$, we proceed by a truncation argument.
    Let $S_n \defeq \samp(\bm x_1, \ldots, \bm x_t)$ and $S \defeq (X_1, \ldots, X_t)$. Fix a cutoff $K>0$ and consider
    any bounded continuous function $\varphi_K$ such
    that $|\varphi_K|\le |\varphi|$, $\varphi_K(s) = \varphi(s)$ for all
    $\|s\|_2\le K$ and $\varphi_K(s)=0$ for all
    $\|s\|_2>2K$ (standard approximations show that such a function exists). First, $\left|\E \varphi_K(S_n)-\E \varphi_K(S)\right|$ converges to $0$
    as $n\to\infty$
    by the bounded continuous case. Next,
    \begin{align*}
        \left|\E \varphi(S_n) - \E \varphi_K(S_n)\right|&\le \E \left[|\varphi(S_n)| \mathbf 1_{\|S_n\|_2>K}\right]\tag*{(Definition of the truncated function)}\\
        &\le (\E \varphi(S_n)^2)^{\frac 12} \Pr(\|S_n\|_2>K)^{\frac 12}\tag*{(Cauchy-Schwarz inequality)}\\
        &\le (\E \varphi(S_n)^2)^{\frac 12} \frac {(\E \|S_n\|_2^2)^{\frac 12}} {K}\tag*{(Markov inequality)}
    \end{align*}
    Note that these quantities are respectively equal to
    \[
        \E \varphi(S_n)^2 = \frac 1n \sum_{i=1}^n \varphi(\bm x_1[i], \ldots, \bm x_t[i])^2\quad\text{and} \quad\E \|S_n\|_2^2 = \sum_{s=1}^t \frac 1n  \sum_{i=1}^n \bm x_s[i]^2\,,
    \]
    which both converge as $n\to\infty$ by existence
    of the traffic distribution, and in particular are bounded uniformly in $n$. Hence, there exists $C>0$ independent of $n$ and $K$ such that
    \[\limsup_{n\to\infty} \left|\E \varphi(S_n) - \E \varphi_K(S_n)\right|\le \frac C K\,,\] and the
    same holds for $\E \varphi(S)-\E \varphi_K(S)$ 
    by the same argument, using the fact that all moments exist in the space generated by $Z_{\calA_1}^\infty$. We obtain $\limsup_{n\to\infty} \left|\E \varphi(S_n) - \E \varphi_K(S_n)\right|\le 2C/K$, and the claim
    follows from taking the limit $K\to\infty$.
\end{proof}

By~\Cref{prop:gfom-polynomial}, the iterates
of a pGFOM are of the form~\cref{eq:assumeState}, so they have an asymptotic state. By definition, these asymptotic states determine the
limiting distribution of any (bounded continuous or polynomial) observable.
Motivated by this,
we introduce a family of approximate message passing algorithms whose
asymptotic states are conditionally Gaussian.

\begin{theorem}[Treelike AMP]
\label{thm:full-onsager}
Assume that $\bA=\bA^{(n)}$ satisfies the assumptions of \Cref{thm:convergence-distribution}. 
Let $f_t : \R \to \R$ be polynomial functions.\footnote{For ease of exposition $f_t$ is assumed to be ``memoryless'', meaning that it only takes the most recent $\bx_{t}$ as input.} Define:
\begin{equation}
    \begin{aligned}
    \bmx_0 &= \bm{1}\,,\qquad 
    \bmx_t = \bA \bmf_{t-1}
      - \sum_{s=0}^{t-1} \bmb_{s,t} \cdot \bmf_s\,,\\
    \bmb_{s,t}[i]
    &\defeq \sum_{\substack{i_s,\ldots,i_{t-1}\in[n] \textnormal{ distinct}\\i_s=i}} \left(\prod_{r=s+1}^{t-1} \bm A[i_{r-1}, i_{r}] \bm f'_{r}[i_{r}] \right)\bA[i_{t-1},i_s]\,,\\
    \bmf_t &\defeq f_t(\bmx_t)\,, \qquad
    \bmf'_t \defeq f'_t(\bmx_t)\,,\qquad \bm f_0 = \bm 1\,.\label{eq:tree-amp}
    \end{aligned}
\end{equation}
Then $\bx_t \in \linspan(\bmz_{\cG_1 \cup (\cA_1 \setminus \cT_1)}(\bA))$. Therefore, the asymptotic state of $(\bm x_t)_{t\ge 1}$ defined
in~\cref{eq:asymptState}
is a centered Gaussian process conditionally on $Z^\infty_{\calC_1}$.
\end{theorem}

\noindent To prove~\Cref{thm:full-onsager}, motivated by the results in \cref{sec:vector-limit}, we introduce
the following handy notations:
\newcommand\gauss{\textnormal{gaussian}}

\begin{definition}[Equality modulo non-treelike diagrams]\label{def:eqinf}
    For $\bm x,\bm y\in \spn(\bmz_{\calA_1})$, we write $\bm x \eqinf \bm y$ if $\bm x - \bm y\in \spn(\bmz_{\calA_1\setminus \calT_1})$.
    We denote by $\cactus(\bmx)$ the
    projection of $\bmx$ onto the span
    of the cactus diagrams $\cC_1$,
    and by $\gauss(\bmx)$ the
    projection of $\bmx$ onto the span
    of the Gaussian diagrams $\cG_1$.
\end{definition}

The iterates of the treelike AMP
algorithm~\cref{eq:tree-amp} are engineered to 
asymptotically generate a self-avoiding walk.
That is, whenever the algorithm performs a matrix multiplication operation,
the Onsager correction terms in \cref{eq:tree-amp} (the subtracted terms involving $\bb_{s,t}$) are chosen to subtract
off the terms in the resulting diagram expansion which (1)~are treelike and (2)~revisit an existing vertex in any diagram.

\begin{example}[Self-avoiding walk]
For intuition, consider
the case of~\Cref{thm:full-onsager}
where $f_t(x) = x$.
Let $\pi_t$ be the $t$-path
diagram and $\rho_t$ the $t$-cycle
diagram. We can expand exactly:
\begin{align*}
    \bA \bmz_{\pi_t} &= \bmz_{\pi_{t+1}} + \sum_{s = 0}^t \bmz_{\rho_{s+1} \oplus \pi_{t-s}}\,.
\end{align*}
For each term on the right-hand side, we have the approximate factorization (by~\Cref{lem:product-gaussian}) $\bmz_{\rho_{s+1} \oplus \pi_{t-s}} \eqinf \bmz_{\rho_{s+1}} \cdot \bmz_{\pi_{t-s}}$,
which holds up to non-treelike terms.
Then, we define a self-avoiding version of power iteration by:
\[
    \bmx_0 = \bm{1}, \qquad \bmx_{t+1} = \bA \bmx_t - \sum_{s = 0}^t \bmz_{\rho_{s+1}} \cdot \bmx_{t-s}\,.
\]
By construction, we have $\bmx_t \eqinf \bmz_{\pi_t}$ and therefore, assuming the conditions of \cref{thm:convergence-distribution}, the asymptotic state $X_t$ of $\bm x_t$ is Gaussian.
\end{example}

To analyze a general iteration in the proof of~\Cref{thm:full-onsager}, we separate the diagram expansion of $\bmf_t$ into its linear and nonlinear parts:
\begin{equation*}\label{eq:linear-nonlinear}
    \bmf_t = \underbrace{\sum_{\gam \in \cG_1} c_\gam \bmz_\gam(\bA)}_{=: \bmf^1_t} + \underbrace{\sum_{\tau \in \cT_1 \setminus \cG_1} c_\tau \bmz_\tau(\bA)}_{=: \bmf^{\neq 1}_t} + \sum_{\al \in \cA_1 \setminus \cT_1} c_\al \bmz_\al(\bA)\,.
\end{equation*}

\noindent We call $\bmf^1_t = \gauss(\bmf_t)$ the ``linear part'' since it should be thought of as the degree-1 part of the Hermite expansion
of $\bmf_t$ with respect to the Gaussian vectors $\bmz_{\cG_1}(\bA)$,
while $\bmf^{\neq 1}_t$ equals all of the other components of the Hermite expansion.
More precisely, when $\bmf_t$ is of the form $f_t(\bmx_{t})$ for some Gaussian vector $\bmx_{t}$, which is the situation for AMP, then $\bmf^{1}_t$ has the following simple form.

\begin{lemma}\label{lem:f-linear}
    Let $\bmx \in \linspan(\bmz_{\cG_1})$ and let $f : \R \to \R$ be a polynomial.
    Then,
    \[
        \gauss(f(\bm x)) \eqinf \cactus(f'(\bm x)) \cdot \bm x\,.
    \]
\end{lemma}
\begin{proof}
    Suppose that $f(x) = x^\ell$ for some integer $\ell\ge 0$ (the general case follows by linearity). By~\Cref{lem:product-gaussian},
    the product of
    diagrams in $\calG_1$
    yields a diagram in $\calG_1$ only
    when every diagram except one is matched. Formally, write $\bmx = \sum_{\gamma\in\calG_1} c_\gamma \bmz_{\gamma}$, then by~\Cref{lem:product-gaussian},
    \begin{align}
        \bmf^1 \eqinf \ell \sum_{\gamma_1,\ldots,\gamma_\ell\in \calG_1} c_{\gamma_1}\ldots c_{\gamma_{\ell}} \sum_{M\in \PM(\ell-1)} \sum_{\substack{P_{uv}\in H(\gamma_u,\gamma_v)\\\forall uv\in M}} \bz_{\bigoplus_{uv\in M}\gamma_{P_{u,v}} \oplus \gamma_\ell} \,.\label{eq:l-power}
    \end{align}
    Viewing $\bigoplus_{uv\in M}\gamma_{P_{u,v}}$ as a fixed cactus, by~\Cref{lem:product-gaussian},
    every term on the right-hand side
    satisfies
    \begin{align}
        \bz_{\bigoplus_{uv\in M}\gamma_{P_{u,v}} \oplus \gamma_\ell} \eqinf \bz_{\bigoplus_{uv\in M}\gamma_{P_{u,v}}} \cdot \bmz_{\gamma_\ell}\,.\label{eq:factorize-inter}
    \end{align}
    Applying~\Cref{lem:product-gaussian}
    one last time, it remains to observe 
    that the cactus part of $f'(\bm x) = \ell \bm x^{\ell-1}$ is
    \begin{align}
        \cactus(f'(\bmx)) \eqinf \ell \sum_{\gamma_1,\ldots,\gamma_{\ell-1}\in \calG_1} c_{\gamma_1} \ldots c_{\gamma_{\ell-1}} \sum_{M\in \PM(\ell-1)} \sum_{\substack{P_{uv}\in H(\gamma_u,\gamma_v)\\\forall uv\in M}} \bz_{\bigoplus_{uv\in M}\gamma_{P_{u,v}}}\,.\label{eq:l-power-minus}
    \end{align}
    Combining~\cref{eq:l-power,eq:factorize-inter,eq:l-power-minus} yields the desired claim.
\end{proof}

The next key~\cref{lem:memory-formula} derives an explicit asymptotic formula for the AMP iterates: $\bmx_t$ is generated by taking a self-avoiding walk from each nonlinear term $\bmf_s^{\neq 1}$.

\begin{definition}\label{def:self-avoiding}
Let $\bmc_t = \cactus(f'_t(\bmx_t))$.
Define the \emph{self-avoiding walk matrix} $\bmB_{s,t}$ generated by the iteration between time $s$ and $t$ to be:
\begin{align*}
    \bmB_{s,t}[i,j] &\defeq \sum_{\substack{i_{s},\dots, i_{t} \in [n] \textnormal{ distinct}\\i_{s} = j,\; i_{t} = i}} 
    \bmA[i_s, i_{s+1}] \cdots \bmA[i_{t-1},i_t] \cdot \bm c_{s+1}[i_{s+1}] \cdots \bm c_{t-1}[i_{t-1}]\,.
\end{align*}
\end{definition}

\noindent Recalling \Cref{def:open-cactus}, $\bmB_{s,t}$
is a linear combination of
open cactus matrices in the $z$-basis (up to non-treelike terms which arise from intersections involving the $\bm c_{s+i}$).
For example, $\bmB_{t-1,t}$ equals $\bA$ with the diagonal elements set to zero.
We note the analogy between $\bmb_{s,t}$ and $\bmB_{s,t}$, which contain similar self-avoiding walks that return to the start and do not return to the start, respectively.

\begin{lemma}\label{lem:memory-formula}
    Define $\bmx_t, \bmf_t$ by \cref{eq:tree-amp} and let $\bmc_t = \cactus(f'_t(\bmx_t))$. Then for $t \geq 1$:
    \begin{align*}
        \bmx_t \eqinf \sum_{s = 0}^{t-1} \bmB_{s,t} \bmf_s^{\neq 1} \quad \text{and} \quad \bmf_t \eqinf \bmc_t \cdot \sum_{s=0}^{t-1}  \bmB_{s,t} \bmf_s^{\ne 1} + \bmf_{t}^{\neq 1}\,.
    \end{align*}
\end{lemma}
\begin{proof}
    First, note that for a fixed $t$, the second equation follows from the first:
    \begin{align*}
        \bmf_t &\eqinf \bmf_t^1 + \bmf_t^{\neq 1}\\
        &\eqinf \bc_t  \cdot \bmx_t + \bmf_t^{\neq 1}  \tag{\cref{lem:f-linear}} \\
        &\eqinf \bc_{t} \cdot \sum_{s = 0}^{t-1}  \bmB_{s,t} \bmf_s^{\neq 1} + \bmf_{t}^{\neq 1} \tag{first equation and~\Cref{claim:only-trees}}
    \end{align*}
    To establish the equations, we use induction on $t$. In the base case $t = 1$ we have $\bmx_1 = \bA \bmf_0 - \bmb_{0,1}\cdot \bmf_0 = \bmB_{0,1} \bmf_0$ as needed.
    Now, assume that the formulas
    hold for $0, \ldots, t$.
    Denote by $\bmC_t$ the diagonal matrix with entries $\bmc_t$.   
    The equation for $\bmf_t$ implies:
    \begin{align*}
        \bA \bmf_{t} &\eqinf \sum_{s = 0}^{t-1} \bA \bmC_{t} \bmB_{s, t} \bmf_s^{\neq 1}  + \bA \bmf_{t}^{\neq 1} \numberthis \label{eq:Af_t}
    \end{align*}
    If we expand the matrix product $\bA \bC_{t} \bB_{s,t}$ we can partition the sum based on whether the matrix $\bA$ revisits a vertex already on the walk:
    \begingroup
    \begin{align}
        (\bA\bC_{t} \bB_{s, t})[i,j] &= \sum_{k=1}^n \bmA[i,k] \bm c_t[k] \sum_{\substack{i_{s},\dots, i_{t} \in [n] \textnormal{ distinct}\\i_{s} = j,\; i_{t} = k}} 
    \bmA[i_s, i_{s+1}] \cdots \bmA[i_{t-1},i_t] \cdot \bm c_{s+1}[i_{s+1}] \cdots \bm c_{t-1}[i_{t-1}]\nonumber\\
        &= \sum_{\substack{i_{s},\dots, i_{t+1} \in [n] \textnormal{ distinct}\\i_{s} = j,\; i_{t+1} = i}} 
    \bmA[i_s, i_{s+1}] \cdots \bmA[i_{t},i_{t+1}] \cdot \bm c_{s+1}[i_{s+1}] \cdots \bm c_{t}[i_{t}]\label{eq:first-term-gaussian}\\
    &+ \sum_{r=s}^t \sum_{\substack{i_{s},\dots, i_{t} \in [n] \textnormal{ distinct}\\i_{s} = j,\; i_{r} = i}} 
    \bmA[i_s, i_{s+1}] \cdots \bmA[i_{t-1},i_t] \cdot \bmA[i_r,i_t] \cdot \bm c_{s+1}[i_{s+1}] \cdots \bm c_{t}[i_{t}]\,.\label{eq:second-term-gaussian}
    \end{align}

    The first term~\cref{eq:first-term-gaussian} is self-avoiding and equals $\bmB_{s,t+1}[i,j]$.
    In the second term~\cref{eq:second-term-gaussian}, the term $r=s$ is diagrammatically a cycle
    and is equal to $\bmb_{s,t+1}[i]$ when $i=j$, and $0$ otherwise:

    \begin{claim}\label{lem:bst-c}
    We have:
    \[
        \bmb_{s,t} \eqinf \left(\sum_{\substack{i_s,\ldots,i_{t-1}\in[n] \textnormal{ distinct}\\i_s=i}}
        \bA[i_s,i_{s+1}] \cdots \bA[i_{t-2},i_{t-1}] \bA[i_{t-1},i_s] \cdot \bmc_{s+1}[i_{s+1}] \cdots \bmc_{t-1}[i_{t-1}]\right)_{i\in [n]}\,.
    \]
    \end{claim}
    \begin{proof}
        The only difference between this formula and the definition of $\bb_{s,t}$ is that the vectors $\bmf'_t$ at the internal vertices of the cycle have been replaced by $\bmc_t$.
        This holds up to non-treelike terms since placing a non-cactus diagram at any internal vertex of the cycle will create only non-treelike diagrams.
    \end{proof}    
    \noindent The remaining terms in~\cref{eq:second-term-gaussian} are a cycle and a path joined together at vertex $r$:
    
    \begin{claim}\label{claim:factorize-inside-path-cycle}
        Let $r\in \{s+1, \ldots, t\}$. For $i\in [n]$, let
        \[
            \bm u[i] =         \sum_{\substack{i_{s},\dots, i_{t} \in [n] \textnormal{ distinct}\\i_{r} = i}} 
    \bmA[i_s, i_{s+1}] \cdots \bmA[i_{t-1},i_t] \cdot \bmA[i_r,i_t] \cdot \bm c_{s+1}[i_{s+1}] \cdots \bm c_{t}[i_{t}] \bmf_s^{\neq 1}[i_s]\,.
        \]
        Then $\bm u \eqinf \bm b_{r,t+1} \cdot \bm C_r \bm B_{s,r} \bmf_s^{\neq 1}$.
    \end{claim}

    \begin{proof}
        By expanding definitions, we can conveniently interpret
        \[
            \bm b_{r,t+1} \cdot \bm C_r \bm B_{s,r} \bmf_s^{\neq 1}[i] = \sum_{\substack{i_{s},\dots, i_{t} \in [n]\\
            i_s,\ldots,i_{r} \textnormal{ distinct}\\
            i_r,\ldots,i_{t}\textnormal{ distinct}\\i_{r} = i}} 
    \bmA[i_s, i_{s+1}] \cdots \bmA[i_{t-1},i_t] \cdot \bmA[i_r,i_t] \cdot \bm c_{s+1}[i_{s+1}] \cdots \bm c_{t}[i_{t}] \bmf_s^{\neq 1}[i_s]\,.
        \]
        Since the diagram induced on
        $\{i_r, \ldots, i_t\}$ is
        a cycle, any intersection between the
        vertices $\{i_s, \ldots, i_r\}$
        and $\{i_r, \ldots, i_t\}$
        would create a non-treelike 
        diagram.
    \end{proof}
    
    Plugging~\Cref{claim:factorize-inside-path-cycle} in to~\cref{eq:Af_t}, we have:
    \begin{align*}
        \bA \bmf_t &\eqinf \sum_{s = 0}^{t-1} \bB_{s,t+1} \bmf_s^{\neq 1} + \sum_{s=0}^{t-1}\bmb_{s,t+1} \cdot \bmf_s^{\neq 1} + \sum_{s=0}^{t-1}\sum_{r=s+1}^t \bmb_{r,t+1}\cdot (\bmC_r \bmB_{s,r} \bmf^{\neq 1}_s)+ \underbrace{\bA \bmf_t^{\neq 1}}_{= \bB_{t,t+1} \bmf_t^{\neq 1} + \bmb_{t,t+1} \cdot \bmf_t^{\neq 1}}\\
        &= \sum_{s = 0}^{t} \bB_{s,t+1} \bmf^{\neq 1}_s + \sum_{r = 0}^t \bmb_{r,t+1} \cdot \left(\bmC_r \sum_{s = 0}^{r-1} \bmB_{s,r}\bmf_s^{\neq 1} + \bmf_r^{\neq 1} \right)\\
        &\eqinf \sum_{s = 0}^{t} \bB_{s,t+1} \bmf^{\neq 1}_s + \sum_{r = 0}^t \bmb_{r,t+1} \cdot \bmf_r
    \end{align*}
    \endgroup
    The last equality is the inductive formula for $\bmf_r$.
    The Onsager correction subtracts off the second sum, leaving only the desired first sum for $\bmx_{t+1}$.
\end{proof}

\begin{proof}[Proof of \cref{thm:full-onsager}.]
    We prove the following purely combinatorial claim about~\cref{eq:tree-amp}:
    $\bm x_t$ is in
    the span of non-treelike
    diagrams and {\em Gaussian} treelike
    diagrams.  
    By~\Cref{claim:cvImpliesState}, this will
    imply
    that conditioned on $Z^\infty_{\calC_1}$, the asymptotic state of $(\bm x_t)_{t\ge 1}$ is a centered Gaussian process, as desired.

    To show the claim, 
    we start from the conclusion of~\cref{lem:memory-formula}:
    \[
        \bmx_t \eqinf \sum_{s=0}^{t-1}\bB_{s,t}\bmf_s^{\neq 1}\,.
    \]
    The diagrams in $\bB_{s,t}\bmf_{s}^{\neq 1}$ are obtained by: (1)~choose a diagram from $\bmf_{s}^{\neq 1}$, (2)~choose a cactus diagram from $\bc_r$ at each internal vertex of $\bB_{s,t}$
    (i.e.\ each internal vertex along a path of length $t-s$),
    (3)~multiply these diagrams together.
    Since none of the diagrams in $\bmf_s^{\neq 1}$ have degree 1 at the root by definition,
    the only treelike terms in the product are formed by grafting the diagrams together without intersections.
    In particular, the root is the endpoint of the path in $\bB_{s,t}$ and has degree 1.
    This concludes the proof.
\end{proof}

\subsubsection{Covariance structure of treelike AMP}

While~\Cref{thm:full-onsager} shows that the treelike AMP iterates are
asymptotically Gaussian, it does not identify their covariance. We calculate the covariance ``combinatorially'' by calculating the cactus diagrams appearing in
the expansion of $\langle \bmx_s, \bmx_t\rangle$.

\begin{proposition}\label{lem:cov-cactus}
    Let $\bm x_t$ follow the iteration~\cref{eq:tree-amp}. Then for any $s,t\ge 1$,
    \begin{equation}\label{eq:cov-finite}
        \bm x_s \cdot \bm x_t - \sum_{s'=0}^{s-1} \sum_{t'=0}^{t-1} \bm B_{s'st't} (\bm f_{s'} \cdot \bm f_{t'}) \in \spn(\bm z_{\calA_1\setminus \calC_1})\,,
    \end{equation}
    where for $0\le s'\le s, 0\le t'\le t$, we define the matrix $\bmB_{s'st't} \in \R^{n\times n}$ by
    \begin{equation*}
        \bm B_{s'st't}[i,j] \defeq \sum_{\substack{i_{s'},\ldots,i_{s}, j_{t'}, \ldots, j_{t}\in [n]\\
        \textnormal{distinct except}\\
        i_{s'}=j_{t'}=j,\, i_s=j_t=i}} \prod_{r=s'}^{s-1} \bm A[i_r,i_{r+1}] \prod_{r=t'}^{t-1} \bm A[j_r,j_{r+1}] \prod_{r=s'+1}^{s-1} \bm c_r[i_r] \prod_{r=t'+1}^{t-1} \bm c_r[j_r]\,.
    \end{equation*}
\end{proposition}

When we apply this lemma, we will average \cref{eq:cov-finite}
over the coordinates $i \in [n]$ and over $\bA$.
Since the error terms are all in the span of non-cactus diagrams,
all error terms converge to 0 by the strong cactus property.
On the other hand,
the average of the term $\bmx_s \cdot \bmx_t$ converges to the covariance $\E[X_sX_t]$
which we want to calculate.
The subtracted terms involving the $\bB_{s't'st}$ matrices converge to limits depending on the asymptotic values of the cactuses $Z_{\cC_1}^\infty$.
For some settings (such as \cref{prop:factorizing-deterministic}), the values $Z_{\cC_1}^\infty$ are deterministic.
For other settings, the values $Z_{\cC_1}^\infty$ are random, and we will condition on them in order to obtain the conditional covariance.

\begin{proof}
    Since $\bm x_s$ and $\bm x_t$ have
    degree exactly one at the root, in
    order to form a cactus in $\bm x_s\cdot \bm x_t$, the paths
    from the root of $\bm x_s$ and $\bm x_t$ in the expansion from~\Cref{lem:memory-formula} must meet at some point. This
    intersection
    cannot happen at a vertex from $\bm f_{s'}^{\neq 1}$ or $\bm f_{t'}^{\neq 1}$ (that would create edges in two cycles).
    Let
    $s'\in \{0,\ldots,s-1\}$ and $t'\in \{0,\ldots,t-1\}$ denote the integers
    such that the first intersection
    corresponds to the indices
    $i_{s'}$ (for $\bm x_s$) and $i_{t'}$ (for $\bm x_t$) in~\Cref{def:self-avoiding}. Then,
    we can decompose
    \[
        \bm x_s\cdot \bm x_t - \sum_{s'=0}^{s-1} \sum_{t'=0}^{t-1} \bm B_{s'st't} ((\bm c_{s'} \cdot \bm x_{s'} + \bm f_{s'}^{\neq 1}) \cdot (\bm c_{t'} \cdot \bm x_{t'} + \bm f_{t'}^{\neq 1}))\in \spn(\bm z_{\calA_1\setminus \calC_1})\,,
    \]
    and the conclusion follows from
    the equality (\Cref{lem:f-linear}) $\bm f_s\eqinf \bm c_s \cdot \bm x_s + \bm f_s^{\neq 1}$.
\end{proof}

\noindent The 
cactus expansion of $\bm B_{s'st't} (\bm f_{s'} \cdot \bm f_{t'})$ can be obtained explicitly by
combining a cycle of length $s-s'+t-t'$
along the edges of $\bm B_{s'st't}$,
a cactus from $\bm c_r$ hanging at every vertex $r$ in the cycle, and a homeomorphic
matching of the tree components of $\bm f_{s'}$ and $\bm f_{t'}$ (\Cref{def:homeomorphic}). 

\subsection{Examples of state evolution}\label{sec:se-examples}

In this section, we specialize~\cref{thm:full-onsager}
to obtain a more explicit description
of the state evolution of the
treelike AMP algorithm for several concrete matrix models. 

\begin{notation}
    \label{not:empirical}
    For a vector $\bm x\in \R^n$, we will use the following notation for empirical averages:
    \[
        \langle \bm x \rangle \defeq \frac 1n \sum_{i=1}^n x_i\,.
    \]
\end{notation}

Technically, most algorithms in this section are not pGFOM since they
calculate empirical averages.
Assuming that the traffic distribution concentrates and the vector $\bm x$ lies in the diagram basis,
then the empirical average $\langle{\bm x}\rangle$ concentrates,
and we can replace $\langle{\bm x}\rangle$ by its limit $\E X$ without changing the asymptotic state of the algorithm.
This is formally proven in \cref{lem:asymptotic-state-empirical-average}.

\subsubsection{Orthogonally invariant random matrices}
\label{sec:treelike-oamp}

In the special case that $\bA$ is drawn from an orthogonally invariant random matrix ensemble,
the treelike AMP algorithm recovers the orthogonal AMP algorithm of Fan~\cite{fan2022approximate},
giving a new proof of this result.

\begin{theorem}[State evolution for orthogonally invariant matrices]\label{thm:orthogonal-amp}
    Let $\bA = \bA^{(n)} \in \R_\sym^{n \times n}$ be an orthogonally invariant random matrix converging in tracial moments in $L^2$ to a probability measure with free cumulants $(\kappa_q)_{q\ge 1}$. Assume $\bm A$ satisfies~\cref{eq:tightnessNorm}.
    Let $f_t : \R \to \R$ be polynomial functions and define the iteration
    \begin{alignat*}{2}
        \bmx_0 &= \bm{1}, \qquad &\bmx_t &= \bA f_{t-1}(\bm x_{t-1}) - \sum_{s = 0}^{t-1} \kappa_{t-s} \left(\prod_{r=s+1}^{t-1} \langle f'_r(\bm x_r)\rangle \right)f_s(\bm x_s)\quad \forall t\ge 1\,. \numberthis\label{eq:rotation-amp}
    \end{alignat*}
    Then,
    the asymptotic state of $(\bm x_t)_{t\ge 1}$ is a centered Gaussian process $(X_t)_{t\ge 1}$ with covariance
    \begin{align*}
        \E\left[X_s X_t\right] &= \sum_{s' = 0}^{s-1} \sum_{t' = 0}^{t-1} \kappa_{s-s'+t-t'} \left(\prod_{r = s'+1}^{s-1} \E f'_r(X_r) \right) \left(\prod_{r=t'+1}^{t-1} \E f'_r(X_r)\right) \E\left[f_{s'}(X_{s'}) f_{t'}(X_{t'})\right]\quad \forall s,t\ge 1\,,
    \end{align*}
    with $X_0\defeq 1$.
\end{theorem}

\begin{proof}
    By~\cref{thm:moments}, $\bA$ satisfies the factorizing strong cactus property and its diagonal distribution exists,
    so the assumptions of~\Cref{thm:convergence-distribution} and \Cref{thm:full-onsager} are satisfied.
    Therefore, the treelike AMP algorithm in \cref{eq:tree-amp} has Gaussian asymptotic state.
    
    We now specialize the Onsager correction term in~\cref{eq:tree-amp} to
    this model. The term $\bm b_{s,t}$ is represented by a cycle of length
    $t-s$, with $f'_r(\bm x_r)$ attached to the $r$th vertex of the cycle for each
    $s<r<t$. By~\Cref{claim:only-trees}, we
    only need to look at treelike contributions in $\bm b_{s,t}$.
    Because
    of the base cycle, these are only cactuses, obtained by attaching
    cactuses from $(f'_r(\bm x_r))_{s<r<t}$
    along the base cycle. 
    By~\Cref{prop:factorizing-deterministic},
    $\bm b_{s,t}$ has constant asymptotic
    state equal to
    $\kappa_{t-s}\prod_{r=s+1}^{t-1}\E f'_r(X_r)$.
    The cactuses in $\bm b_{s,t}$ persist until the end of the algorithm,
    so that they will eventually contribute this value towards the asymptotic state.
    Hence it does not affect the asymptotic state to replace $\bm b_{s,t}$ immediately by its limiting constant value.
    
    Moreover, by~\Cref{lem:asymptotic-state-empirical-average} and \cref{lemma:full-factorization}, we may replace $\E f'_r(X_r)$ by the empirical average
    $\langle f'_r (\bm x_r)\rangle$ to obtain~\cref{eq:rotation-amp} without affecting the asymptotic state.
    Now the asymptotic state $X_t$ of~\cref{eq:rotation-amp} matches that of~\cref{eq:tree-amp}, and we may apply~\cref{thm:full-onsager} to deduce that $X_t$ is Gaussian.

    To calculate the covariance $\E\left[X_sX_t\right]$, we average \cref{lem:cov-cactus} over the coordinates $i \in [n]$ and take the limit $n \to \infty$.
    On the right side of \cref{eq:cov-finite},
    the cycle of $\bB_{s'st't}$ contributes $\kappa_{s-s'+t-t'}$ and the hanging diagrams $f'_r(\bm x_r)$ inside $\bB_{s'st't}$ contribute
    $\E f'_r(X_r)$ by the factorizing cactus property.
    The cactuses in $f_{s'}(\bm x_{s'}) \cdot f_{t'}(\bm x_{t'})$ contribute $\E\left[f_{s'}(X_{s'})f_{t'}(X_{t'})\right]$, which establishes the desired recurrence.
\end{proof}

\noindent Note that this proof only
uses the strong factorizing cactus
property and the concentration of the
traffic distribution, which explains
why~\Cref{thm:orthogonal-amp} also holds
for non-orthogonally invariant matrix
models such as Wigner matrices (\Cref{sec:traffic-wigner}).

\subsubsection{Punctured random and deterministic matrices}
\label{sec:dynamics-reg}

The punctured matrices studied in~\Cref{sec:universality} do not satisfy the strong cactus property, so we cannot directly apply~\Cref{thm:full-onsager} to derive an AMP iteration for them. However, a reduction allows us to derive the state evolution of punctured orthogonally invariant random matrices from that of their unpunctured counterparts. These matrices are central because, by~\Cref{thm:universality-new}, they provide an intermediate step in deriving the state evolution of sequences of punctured {\em deterministic} matrices satisfying~\Cref{assumption:punctured}.

Note that a GFOM run on a punctured matrix must be initialized with a random vector $\bmx_0 \sim \calN(\bm{0}, \bI)$, rather than $\bm{x}_0 = \bm{1}$, to avoid triviality.

\begin{theorem}[State evolution for punctured matrices]\label{prop:hadamard-explicit-state}
    Let $\bH = \bH^{(n)} \in \R_\sym^{n \times n}$ be a sequence of orthogonally invariant random matrices satisfying~\cref{eq:tightnessNorm} and converging in tracial moments in $L^2$ to a probability measure with free cumulants $(\kappa_q)_{q\ge 1}$. 
    Let $\bA$ denote the puncturing of $\bH$ (\Cref{def:puncturing}).
    Let $f_t : \R \to \R$ be polynomial functions with $f_0(x)=x$, and consider the pGFOM:
    \begin{alignat*}{2}
        \bm x_0 &\sim \calN(\bm 0, \bm I)\,,\quad &\bm x_t &= \bm A { f_{t-1}(\bm x_{t-1})} - \sum_{s=0}^{t-1} \kappa_{t-s} \left(\prod_{r=s+1}^{t-1} \langle f_r'(\bm x_r)\rangle\right) (f_s(\bm x_s) - \langle f_s(\bm x_s)\rangle \bm{1})\quad \forall t\ge 1\,.\label{eq:ampPunctured}
    \end{alignat*}
    Then for any $t\ge 1$ and any polynomial $\varphi:\R^t \to \R$, we have
    \[
        \lim_{n\to\infty}  \E_{\bm H,\bm x_0} \langle \varphi(\bm x_1, \ldots, \bm x_t)\rangle = \E \varphi(X_1, \ldots, X_t)\,,
    \]
    where $(X_t)_{t\ge 1}$ is a centered Gaussian process with covariance given by
    \begin{align*}
        \E\left[X_s X_t\right] &= \sum_{s' = 0}^{s-1}\sum_{t'=0}^{t-1} \kappa_{s-s'+t-t'} \left(\prod_{r= s'+1}^{s-1} \E f'_r(X_r) \right) \left(\prod_{r=t'+1}^{t-1} \E f'_r(X_r)\right) \E \left[\overline{F_{s'}} \, \overline{F_{t'}}\right]\quad \forall s,t\ge 1\,,\\
        \overline{F_0} &\defeq 1\,,\quad \overline{F_t} \defeq f_t(X_t) - \E f_t(X_t)\quad \forall t\ge 1\,.
    \end{align*}
\end{theorem}

\noindent By~\Cref{thm:universality-orthogonal}, the conclusion of~\Cref{prop:hadamard-explicit-state} also holds for any sequence of deterministic matrices satisfying the delocalization assumption~\Cref{assumption:punctured} and having a limiting diagonal distribution that factorizes over cycles (that is, matches the diagonal distribution of some orthogonally invariant random matrix ensemble).
In particular, the conclusion holds for the Walsh-Hadamard matrices and the Discrete Cosine and Sine Transform matrices, for which the $\kappa_q$ are the free cumulants of the \rom  (\cref{eq:freeCumulantROM}).

The proof of~\Cref{prop:hadamard-explicit-state} proceeds by reducing to the following iteration on the original, non-punctured matrix, initialized at the all-ones vector:
\begin{alignat}{2}
    \bm u_0 &= \bm 1\,,\qquad
    \bm u_t
    = \bm H \overline{\bm f_{t-1}}
    - \sum_{s=0}^{t-1} \bm b_{s,t} \cdot \overline{\bm f_s} \quad \forall t\ge 1\,,\nonumber\\
    \text{where } {\bm b}_{s,t}[i]
    &\defeq \sum_{\substack{i_s,\ldots,i_{t-1}\in[n] \textnormal{ distinct}\\i_s=i}}
    \left(\prod_{r=s+1}^{t-1} \bm H[i_{r-1}, i_{r}] \bm f'_{r}[i_{r}] \right)
    \bH[i_{t-1},i_s]\quad \forall t> s\ge 0\,,\label{eq:backtrackPunctured}\\
    \overline{\bm f_0} &\defeq \bm u_0\,,\quad \overline{\bm f_t} \defeq \bm \Pi f_t(\bm u_t)\quad \forall t\ge 1\,.\nonumber
\end{alignat}

\begin{lemma}\label{lem:switchInit}
    For any $t\ge 1$ and any polynomial $\varphi:\R^t \to \R$,
    \begin{align*}
        \lim_{n\to\infty} \E_{\bm H,\bm x_0} \langle \varphi(\bm x_1, \ldots, \bm x_t)\rangle &= \lim_{n\to\infty} \E_{\bm H} \langle \varphi(\bm u_1, \ldots, \bm u_t)\rangle\,.
    \end{align*}
\end{lemma}

\noindent The proof of~\Cref{lem:switchInit}
is deferred to~\Cref{appendix:initialization}.

\begin{proof}[Proof of~\Cref{prop:hadamard-explicit-state}]
    We apply~\Cref{thm:orthogonal-amp} to $\bm u_t$ after replacing iteratively each occurrence of $\bm \Pi f_t(\bm u_t)$ by $f_t(\bm u_t) - \E\left[f_t(U_t)\right] \cdot \bm 1$ (where $U_t$ is the asymptotic state of $\bm u_t$ as predicted by~\cref{thm:orthogonal-amp}). By~\Cref{lem:asymptotic-state-empirical-average}, this transformation does not change the asymptotic state of $\bm u_t$.
    The state evolution formula for polynomial test functions then transfers to $\bm x_t$ by~\Cref{lem:switchInit}.
\end{proof}

\subsubsection{Block-structured random matrices}

\newcommand{\Community}{\textnormal{Community}}

Our final example is the class of block-structured matrices whose blocks satisfy the factorizing strong cactus property, which we introduced in~\Cref{sec:wigner-block}.
As anticipated in~\Cref{ex:block-structured}, these matrices do not
themselves
satisfy the factorizing strong
cactus property.
Therefore, we start by describing the random
limit $Z^\infty_{\calC_1}$.

\begin{lemma}\label{lem:block-matrix-cactus-limits}
    Let $q \in \N$. For $r,c\in [q]$, let $\bA_{r,c} = \bA_{r,c}^{(n)} \in \R^{\frac nq \times \frac nq}_\sym$ be a sequence of symmetric random matrices such that $\bm A_{r,c} = \bm A_{c,r}$.
    Let $\bA \in \R^{n \times n}_\sym$ be the block matrix with blocks $(\bA_{r,c})_{r,c \in [q]}$.
    Assume that each $\bm A_{r,c}$ satisfies the factorizing strong cactus property,
    and $(\bm A_{r,c})_{1 \leq r\leq c \leq q}$
    are asymptotically traffic independent. Let $(\kappa^{\{r,c\}}_\ell)_{\ell\ge 1}$ be the limiting
    free cumulants of $\bm A_{r,c}$.
    Then,
    \[
        (\block(i), \bmz_{\cC_1}(\bm A)[i]) \tod  (R, Z_{\cC_1}^\infty(R))\,, \qquad i \sim \Unif([n])\,, \quad R \sim \Unif([q])\,,
    \]
    where  the (deterministic) sequence $Z_{\cC_1}^\infty(r)$ for $r\in [q]$ is defined recursively by:
    \begin{enumerate}[(i)]
        \item For the singleton cactus, $Z^\infty_{\mathrm{singleton}}(r) \defeq 1$.
        
        \item Suppose $\sigma \in \mathcal C_1$ is rooted at a vertex $u_1$ of degree $2$. Let $(u_1, \ldots, u_{\ell})$ be
        the cycle incident to the root.
        Let
        $\sigma_2, \dots, \sigma_{\ell} \in \mathcal C_1$ be the rooted cactuses
        attached to the vertices of the cycle. Then
        \[
            Z^\infty_\sigma(r) \defeq
            \begin{cases}
                \displaystyle
                \sum_{c\in [q]} \left[\kappa_\ell^{\{r,c\}}
                \prod_{\substack{k=2\\k\textnormal{ odd}}}^{\ell} Z^\infty_{\sigma_k}(r) \prod_{\substack{k=2\\k\textnormal{ even}}}^{\ell} Z^\infty_{\sigma_k}(c)\right]
                & \text{if $\ell$ is even} \\
                \displaystyle
                \kappa_\ell^{\{r,r\}}
                \prod_{k=2}^{\ell} Z^\infty_{\sigma_k}(r)
                & \text{if $\ell$ is odd}
            \end{cases}
        \]
    
        \item If $\sigma \in \mathcal C_1$ decomposes as
        $\sigma = \bigoplus_{k=1}^\ell \sigma_k$, then $Z^\infty_\sigma(r) \defeq \prod_{k=1}^\ell Z^\infty_{\sigma_k}(r)$.
    \end{enumerate}
    In particular, the law of the limit $Z_{\calC_1}^{\infty}$ is $\Unif(\{Z_{\calC_1}^{\infty}(r): r \in [q]\})$.
\end{lemma}

\noindent The proof is deferred to~\Cref{app:block-matrix}. By unicity of the limit in distribution,~\Cref{lem:block-matrix-cactus-limits}, together with~\Cref{thm:convergence-distribution}, determines
the law of $Z^\infty_{\calA_1}$.
The joint convergence with $\block(i)$ clarifies the source of the randomness
of $Z^\infty_{\calC_1}$: it arises because of the random choice of the block an entry belongs to in the $\samp(\cdot)$ operation.

Using~\Cref{lem:block-matrix-cactus-limits},
we can specialize the treelike
AMP iteration and its state evolution
to concrete block-structured models.
We start with block GOE matrices (\cref{def:blockgoe}).
A family of AMP iterations for such matrices was derived in~\cite{rangan2011generalized, javanmard2013state}. As we will discuss below, these iterations
have the same asymptotic state as
treelike AMP.

\begin{theorem}[State evolution for the block GOE model]\label{thm:block-goe-state}
    Let $\bA \sim \bGOE(n, \bm\Sigma)$, where  $\bm \Sig \in \R_{\ge 0}^{q \times q}$ is a symmetric matrix.
    Given polynomial functions $f_t : \R \to \R$, let
    \begin{alignat}{2}
        \bmx_0 &= \bm{1}\,,\quad \bm x_1 = \bm A f_0(\bm x_0)\,, \quad &\bmx_t &= \bA f_{t-1}(\bm x_{t-1}) - (\bm A^{\odot 2}  f_{t-1}'(\bm x_{t-1})) \cdot f_{t-2}(\bm x_{t-2})\quad \forall t\ge 2\,.\label{eq:amp-block-goe}
    \end{alignat}
    Then the asymptotic state of $(\bm x_t)_{t\ge 1}$
    is a mixture $\frac 1q\sum_{r \in [q]} \mu_r$ where $\mu_r$ denotes the law of a centered Gaussian process $(X_t)_{t \geq 1}$ with covariance kernel $\bm \Gamma_r$ defined recursively by
    \[
        \bm \Gamma_{r}[s,t] = \sum_{c\in  [q]}\bigg[\bSig[r,c] \E_{(X_{T})_{T \ge 1} \sim \mu_c} \left[f_{s-1}(X_{s-1})f_{t-1}(X_{t-1})\right]\bigg]\quad \forall r \in [q]\,,\quad  \forall s,t\ge 1\,,
    \]
    with $X_0\defeq 1$.
\end{theorem}

\begin{proof}
    By the discussion after~\Cref{thm:male-traffic-independence}, $\bA$ has a traffic distribution and satisfies the strong cactus
    property\footnote{One can also verify that it satisfies~\cref{eq:tightnessNorm}. But note that this was only needed in the proof of~\Cref{thm:convergence-distribution} to ensure the existence of the limit $Z^\infty_{\calC_1}$, which we established directly in~\Cref{lem:block-matrix-cactus-limits}.}, so it satisfies
    the assumption of~\Cref{thm:full-onsager}.
    We consider the treelike AMP iteration $\bm x_t$
    in~\cref{eq:tree-amp} applied to
    $\bm A$.
    We show that this iteration
    has the same asymptotic state
    as~\cref{eq:amp-block-goe} by
    simplifying the Onsager correction term.
    
    The free cumulants of the GOE are $0$ except for $\kappa_2$, so by~\cref{lem:block-matrix-cactus-limits}, the only asymptotically non-negligible cactuses are those such that every cycle is a 2-cycle.
    For any
    $s < t-2$, $\bm b_{s,t}$ contains an injective cycle
    of length larger than $2$ that cannot be
    destroyed by later operations. For $s=t-2$,
    we have:
    \begin{align*}
        \bmb_{t-2,t}[i] &= \sum_{\substack{j = 1\\j \neq i}}^n \bA[i,j]^2 \bmf'_{t-1}[j] = (\bm A^{\odot 2} \bm f_{t-1}')[i] - \bm A[i,i]^2 \bm f_{t-1}'[i]\,.
    \end{align*}
    Both $\bm A[i,i]^2 \bm f_{t-1}'[i]$ and
    $\bm b_{t-1,t}$ contain a self-loop that also
    cannot be destroyed by later operations.
    In conclusion, the treelike AMP algorithm from~\cref{eq:tree-amp} and the iteration in~\cref{eq:amp-block-goe} are equal up to negligible diagrams.
    By~\Cref{thm:full-onsager},
    the asymptotic state $(X_t)_{t\ge 1}$ of $(\bmx_t)_{t\ge 1}$ in~\cref{eq:amp-block-goe} exists and
    is Gaussian conditionally on
    $Z^\infty_{\calC_1}$, and so, in the construction from~\Cref{lem:block-matrix-cactus-limits}, it is Gaussian conditionally on the random variable $R$. 

    Next, we specialize the covariance formula given by~\Cref{lem:cov-cactus}. Since
    only cactuses of 2-cycles are nonzero in the traffic distribution of $\bm A$ (this may be induced from~\Cref{lem:block-matrix-cactus-limits}), only the term
    for $s'=s-1$ and $t'=t-1$ is non-negligible
    in the expansion of $\frac 1 n \E \bm x_s \cdot \bm x_t$ given by~\Cref{lem:cov-cactus}. The expansion into cactuses of that
    term is obtained by grafting together a 2-cycle at the root, and cactuses of 2-cycles
    from $\bm f_{s-1}$ and $\bm f_{t-1}$ at the
    child of the root. 
    Applying the recursive formula for $Z^\infty_{\calC_1}(r)$ in~\Cref{lem:block-matrix-cactus-limits}, we obtain:
    \begin{align*}
        \E \left[X_s X_t\mid R=r\right]
        &=  \sum_{c\in [q]} \bm \Sigma[r,c] \E\left[f_{s-1}(X_{s-1}) f_{t-1}(X_{t-1})\mid R=c\right]\qquad \forall r\in [q]\,.
    \end{align*}
    Thus, we have shown that, conditionally on $R\sim \Unif([q])$, $(X_t)_{t\ge 1}$
    is a Gaussian process with the required covariance. The result follows by taking $\mu_r$ to be the law of $(X_t)_{t\ge 1}$ conditionally
    on $R=r$.
\end{proof}

To illustrate the modularity of
our approach, we also study a different
block-structured matrix
model whose blocks
are not all GOE.

\begin{theorem}[State evolution for the community model]\label{thm:communitySimple}
    Let $\bm M\in \R_{\sym}^{\frac n q\times \frac n q}$ be
    an orthogonally invariant
    random matrix converging in
    tracial moments to a probability measure with free cumulants $(\kappa_q)_{q\ge 1}$
    such that $\kappa_2 = \frac 1 q$.
    Let $\bm A$ be the random symmetric $n\times n$ matrix with blocks $(\bA_{r,c})_{r,c \in [q]}$ given by $\bA_{1,1} = \bM$ and 
    for all $1\le r\le c\le q$ with $(r,c)\neq (1,1)$, 
    $\bA_{r,c}$ are i.i.d. $\frac n q\times \frac n q$ GOE matrices with entries of variance $\frac 1n$ (and we set $\bm A_{r,c} =\bm A_{c,r}$).
    
    Let $\bm x_t$ be the treelike AMP iteration~\cref{eq:tree-amp} run
    on $\bm A$ with arbitrary polynomial
    nonlinearities.
    Then the asymptotic state $(X_t)_{t\ge 1}$ of $(\bm x_t)_{t\ge 1}$ is the mixture $(1-\frac 1q)\mu_0 + \frac 1 q \mu_1$, where $\mu_i$ is the law
    of a centered Gaussian process
    $(X_t)_{t\ge 1}$ with covariance kernel $\bm \Gamma_i$
    defined recursively by, for
    all $s,t\ge 1$:
    \begin{align*}
        & \bm \Gamma_0[s,t] = \E\left[F_{s-1}F_{t-1}\right]\,\\
        & \bm \Gam_1[s,t] = \E\left[F_{s-1}F_{t-1}\right]  + \underset{(s',t') \neq (s-1,t-1)}{\sum_{s'=0}^{s-1} \sum_{\substack{t'=0}}^{t-1}} \kappa_{s-s'+t-t'} 
        \left(\prod_{r=s'+1}^{s-1} \E_{\mu_1}F'_r\right) \left(\prod_{r=t'+1}^{t-1}\E_{\mu_1} F'_r \right) \E_{\mu_1}\left[F_{s'}F_{t'}\right]\,,\\
        & F_t := f_t(X_t), \qquad F'_t := f'_t(X_t), \qquad X_0 = 1\,,
    \end{align*}
    where $\E_{\mu_1}$ denotes
    expectation with respect to $(X_t)_{t\ge 1}\sim \mu_1$.
\end{theorem}

\begin{proof}
    The assumptions of~\Cref{lem:block-matrix-cactus-limits,thm:communitySimple} are satisfied. All
    blocks except the one in position $(1,1)$ have the same
    free cumulants (the GOE free cumulants, normalized so that $\kappa_2 = \frac 1 q$). Therefore, in the construction of~\Cref{lem:block-matrix-cactus-limits}, we have $Z^\infty_{\calC_1}(r) = Z^\infty_{\calC_1}(s)$ for all $r,s>1$. Let $\mu_0$ (resp. $\mu_1$) be
    the law of the asymptotic state
    $(X_t)_{t\ge 1}$ of the treelike AMP iteration $(\bm x_t)_{t\ge 1}$ conditioned on $R>1$ (resp. $R=1$). By~\Cref{thm:full-onsager}, both $\mu_0$ and $\mu_1$ are the laws of centered Gaussian processes.
    It remains to specialize the formula of \Cref{lem:cov-cactus} for their covariance 
    to the present setting. 
    
    Conditionally
    on $R>1$ (that is, outside the community), only 2-cycles
    at the root contribute to $Z^\infty_{\calC_1}$. Thus, by
    combining the strong cactus property,~\Cref{lem:cov-cactus},
    and~\Cref{lem:block-matrix-cactus-limits}, we obtain
    \begin{align*}
        \E \left[X_s X_t\mid R >1\right] &= \frac 1q \left(\E \left[f_{s-1}(X_s) f_{t-1}(X_t) \mid R = 1\right] + (q-1) \E \left[f_{s-1}(X_s) f_{t-1}(X_t) \mid R > 1\right]\right)\\
        &= \E \left[f_{s-1}(X_s) f_{t-1}(X_t)\right]\,.
    \end{align*}
    
    Conditionally on $R=1$ (that is, inside
    the community), we also obtain 
    a contribution of $\E \left[f_{s-1}(X_s) f_{t-1}(X_t)\right]$ from the term
    $s'=s-1$ and $t'=t-1$ in~\Cref{lem:cov-cactus} (again using the normalization $\kappa_2 = \frac 1q$ inside the community).
    For all of the remaining terms $s',t'$,
    when $s-s' + t-t'$ is an even
    integer larger than $2$, we obtain
    a contribution only from $c=1$
    in~\Cref{lem:block-matrix-cactus-limits}, namely
    \[
        \kappa_{s-s'+t-t'} \E\left[ F_{s'} F_{t'}\mid R=1\right]\prod_{r=s'+1}^{s-1} \E\left[ F'_{r}\mid R=1\right]\prod_{r=t'+1}^{t-1} \E\left[ F'_{r}\mid R=1\right]\,.
    \]
    When $s-s'+t-t'$ is odd,~\Cref{lem:block-matrix-cactus-limits} yields exactly the same expression as the even case.
    Altogether, we obtain the recursion
    \begin{align*}
        &\qquad\E\left[X_s X_t\mid R=1\right] = \E \left[F_{s-1} F_{t-1}\right] + \\
        &\underset{(s',t') \neq (s-1,t-1)}{\sum_{s'=0}^{s-1} \sum_{\substack{t'=0}}^{t-1}} \kappa_{s-s'+t-t'} \E\left[ F_{s'} F_{t'}\mid R=1\right] \prod_{r=s'+1}^{s-1} \E\left[ F'_{r}\mid R=1\right]\prod_{r=t'+1}^{t-1} \E\left[ F'_{r}\mid R=1\right] \,.
    \end{align*}
    These are the desired covariance
    formulas for $\mu_0$ and $\mu_1$, and the mixing weights
    of the events $(R=1)$ and $(R>1)$ are indeed $\frac 1q$ and $1-\frac 1 q$, respectively.
\end{proof}

\subsubsection{Further extensions}

There are several possible technical extensions of the methods we have developed here,
whose full development is left for future work.

First,~\Cref{lem:block-matrix-cactus-limits} applies to general orthogonally invariant distributions within the blocks, not just the GOE. In principle, one can then derive a corresponding state evolution formula mechanically for non-identically distributed orthogonally invariant blocks with arbitrary free cumulants, although the resulting expression is quite complicated.

Second, for technical reasons, we assumed that the blocks are square and symmetric, so that we could work with undirected graphs. The results of~\cite{male2020traffic,cebron2024traffic} extend to general matrices, and our techniques should also extend to the setting of varying block sizes and asymmetric matrices, leading to non-uniform mixtures in the recursion for the covariance kernel.

One caveat of the treelike AMP algorithm
is that the Onsager correction term
in~\cref{eq:tree-amp} is not obviously efficient
to compute in practice.\footnote{The $\bm b_{s,t}$ can be approximated with high probability to negligible error for all
$0\le s<t\le T$ in time $2^{O(T)}\poly(n)$ using the color coding
technique~\cite{colorCoding, heavyTailedWigner}, but the exponential
dependence on $T$ makes this
algorithm impractical to implement
for large $T$.}
On the other hand,
the vectors $\bb_{s,t}$ have asymptotically constant entries in many settings,
so that the Onsager correction can be replaced by a simpler asymptotically equivalent term, like in \Cref{thm:orthogonal-amp,prop:hadamard-explicit-state}.
This should also hold for block-structured models, as in the generalized AMP algorithm of Javanmard and Montanari~\cite{javanmard2013state}.
For example, \cref{eq:amp-block-goe} is expected to be asymptotically equivalent to:
\begin{alignat*}{2}
    \bmx_0 &= \bm{1}\,, \qquad &\bmx_t &= \bA f_{t-1}(\bm x_{t-1}) - \bmb_{t-2,t} \cdot f_{t-2}(\bm x_{t-2})\,,\\
     && \bmb_{t-2,t}[i] &= \sum_{c = 1}^q \matSig[\block(i), c] \langle f'(x_{t-1}) \cdot \bm{1}_{\block = c} \rangle\,,
\end{alignat*}
where $\bm{1}_{\block = c} \in \{0,1\}^{n}$ indicates the entries in block $c \in [q]$.
The treelike AMP algorithm for~\cref{thm:communitySimple} is expected to be asymptotically equivalent to:
\begin{alignat*}{2}
    \bmx_0 &= \bm{1}\,, \qquad &\bmx_t &= \bA \bmf_{t-1} - \langle \bmf'_{t-1} \rangle \bmf_{t-2}
    - \sum_{\substack{s = 0\\ s \neq t-2}}^{t-1} \kappa_{t-s} \left(\prod_{r=s+1}^{t-1} \langle \bmf_r' \cdot \bm{1}_{\block = 1} \rangle \right) \bmf_s \cdot \bm{1}_{\block = 1}\,, \\
    \bmf_{t} &\defeq f_t(\bmx_t)\,, \qquad &\bmf'_t &\defeq f'_t(\bx_t)\,,
\end{alignat*}
where $\bm{1}_{\block = 1} \in \{0,1\}^n$ indicates the entries in the first block.
Because these expressions involve the blockwise indicators
$\bm{1}_{\block = c}$, they could be represented and analyzed using an extended
diagram basis in which certain indices are constrained to lie in a prescribed
block. We leave the full development of this extension to future work.

A final open question is to characterize traffic distributions satisfying the
(not necessarily factorizing) strong cactus property. Sequences of block matrices with orthogonally invariant
blocks provide one general construction of matrices with the strong cactus
property. If a sequence of matrices has the strong cactus property, must its
traffic distribution arise as the limit (in an appropriate sense) of traffic
distributions of block matrices with orthogonally invariant blocks (allowing the number of blocks to tend to infinity)?

\bibliographystyle{alpha}
{\footnotesize\bibliography{references}}

@article{etingof2024mathematical,
  title={Mathematical ideas and notions of quantum field theory},
  author={Etingof, Pavel},
  journal={arXiv preprint arXiv:2409.03117},
  year={2024}
}

@book{billingsleyProbabilityBook,
  title={Probability and measure},
  author={Billingsley, Patrick},
  year={1995},
  publisher={John Wiley \& Sons},
  edition={Third}
}

@article{cotler2017black,
  title={Black holes and random matrices},
  author={Cotler, Jordan S. and {Gur-Ari}, Guy and Hanada, Masanori and Polchinski, Joseph and Saad, Phil and Shenker, Stephen H. and Stanford, Douglas and Streicher, Alexandre and Tezuka, Masaki},
  journal={Journal of High Energy Physics},
  volume={2017},
  number={5},
  pages={1--54},
  year={2017},
  publisher={Springer}
}

@article{saad2019jt,
  title={{JT gravity as a matrix integral}},
  author={Saad, Phil and Shenker, Stephen H. and Stanford, Douglas},
  journal={arXiv preprint arXiv:1903.11115},
  year={2019}
}

@article{cebron2024traffic,
  title={{Traffic distributions and independence II: Universal constructions for traffic spaces}},
  author={C{\'e}bron, Guillaume and Dahlqvist, Antoine and Male, Camille},
  journal={Documenta Mathematica},
  volume={29},
  number={1},
  pages={39--114},
  year={2024}
}

@article{lovig2025universality,
  title={On Universality of Non-Separable Approximate Message Passing Algorithms},
  author={Lovig, Max and Wang, Tianhao and Fan, Zhou},
  journal={arXiv preprint arXiv:2506.23010},
  year={2025}
}

@article{robbins1939theorem,
  title={A theorem on graphs, with an application to a problem of traffic control},
  author={Robbins, Herbert Ellis},
  journal={The American Mathematical Monthly},
  volume={46},
  number={5},
  pages={281--283},
  year={1939},
  publisher={JSTOR}
}

@book{bai2010spectral,
  title={Spectral analysis of large dimensional random matrices},
  author={Bai, Zhidong and Silverstein, Jack W.},
  volume={20},
  year={2010},
  publisher={Springer}
}

@article{mingo2012sharp,
  title={Sharp bounds for sums associated to graphs of matrices},
  author={Mingo, James A. and Speicher, Roland},
  journal={Journal of Functional Analysis},
  volume={262},
  number={5},
  pages={2272--2288},
  year={2012},
  publisher={Elsevier}
}

@article{cirone2024graph,
  title={Genus Expansion for Non-linear Random Matrix Ensembles with Applications to Neural Networks},
  author={Cirone, Nicola Muca and Hamdan, Jad and Salvi, Cristopher},
  journal={arXiv preprint arXiv:2407.08459},
  year={2024}
}

@book{male2020traffic,
  title={Traffic distributions and independence: permutation invariant random matrices and the three notions of independence},
  author={Male, Camille},
  volume={267:1300},
  year={2020},
  publisher={American mathematical society}
}

@article{dudeja2024spectral,
  title={Spectral universality in regularized linear regression with nearly deterministic sensing matrices},
  author={Dudeja, Rishabh and Sen, Subhabrata and Lu, Yue M.},
  journal={IEEE Transactions on Information Theory},
  year={2024},
  publisher={IEEE}
}

@article{marinari1994replicaII,
  title={{Replica field theory for deterministic models. II. A non-random spin glass with glassy behaviour}},
  author={Marinari, Enzo and Parisi, Giorgio and Ritort, Felix},
  journal={Journal of Physics A: Mathematical and General},
  volume={27},
  number={23},
  pages={7647},
  year={1994},
  publisher={IOP Publishing}
}

@article{marinari1994replicaI,
  title={Replica field theory for deterministic models: I. Binary sequences with low autocorrelation},
  author={Marinari, Enzo and Parisi, Giorgio and Ritort, Felix},
  journal={Journal of Physics A: Mathematical and General},
  volume={27},
  number={23},
  pages={7615},
  year={1994},
  publisher={IOP Publishing}
}

@article{diFrancesco1995gravity,
  title={{2D gravity and random matrices}},
  author={Di Francesco, Philippe and Ginsparg, Paul and Zinn-Justin, Jean},
  journal={Physics Reports},
  volume={254},
  number={1-2},
  pages={1--133},
  year={1995},
  publisher={Elsevier}
}

@book{AGZ-2010-RandomMatrices,
  title={An introduction to random matrices},
  author={Anderson, {Greg W.} and Guionnet, Alice and Zeitouni, Ofer},
  year={2010},
  publisher={Cambridge University Press}
}

@article{gueddari2025approximate,
  title={{Approximate Message Passing for General Non-Symmetric Random Matrices}},
  author={Gueddari, Mohammed-Younes and Hachem, Walid and Najim, Jamal},
  journal={Journal of Theoretical Probability},
  year={2026},
  volume = 39
}

@article{parisi1995mean,
  title={Mean-field equations for spin models with orthogonal interaction matrices},
  author={Parisi, Giorgio and Potters, Marc},
  journal={Journal of Physics A: Mathematical and General},
  volume={28},
  number={18},
  pages={5267},
  year={1995},
  publisher={IOP Publishing}
}

@inproceedings{jones2025fourier,
  title={Fourier Analysis of Iterative Algorithms},
  author={Jones, Chris and Pesenti, Lucas},
  booktitle={52nd International Colloquium on Automata, Languages, and Programming (ICALP 2025)},
  pages={102--1},
  year={2025},
  organization={Schloss Dagstuhl--Leibniz-Zentrum f{\"u}r Informatik}
}

@article{petersenEquivalence,
    title = {{On the relation between the multidimensional moment problem and the one-dimensional moment problem}},
    author = {Petersen, {L.C.}},
    year = 1982,
    volume = 51,
    journal = {Mathematica Scandinavica},
    pages = {361--366}
}

@article{mcLaughlin,
    title = {{Asymptotics of the partition function for random matrices via Riemann-Hilbert techniques and applications to graphical enumeration}},
    author = {Ercolani, {Nicholas M.} and McLaughlin, {Kenneth}},
    year = 2003,
    journal = {International Mathematics Research Notices},
    volume = 2003,
    pages = {755--820},
    issue = 14
}

@article{garouf,
    title = {{Analyticity of the planar limit of a matrix model}},
    author = {Garoufalidis, Stavros and Popescu, Ionel},
    year = 2013,
    journal = {Annales Henri Poincar\'{e}},
    volume = 14,
    pages = {499--565}
}

@article{tHooft1974planar,
title = {A planar diagram theory for strong interactions},
journal = {Nuclear Physics B},
volume = {72},
number = {3},
pages = {461-473},
year = {1974},
issn = {0550-3213},
doi = {https://doi.org/10.1016/0550-3213(74)90154-0},
url = {https://www.sciencedirect.com/science/article/pii/0550321374901540},
author = {Gerard 't Hooft}
}

@article{brezin1978planar,
  title={Planar diagrams},
  author={Br{\'e}zin, Edouard and Itzykson, Claude and Parisi, Giorgio and Zuber, Jean-Bernard},
  journal={Communications in Mathematical Physics},
  volume={59},
  pages={35--51},
  year={1978},
  publisher={Springer}
}

@article{wang2022universality,
  title={Universality of approximate message passing algorithms and tensor networks},
  author={Wang, Tianhao and Zhong, Xinyi and Fan, Zhou},
  journal={Annals of Applied Probability},
  volume = 34,
  issue = 4,
  pages = {3943--3994},
  year={2022}
}

@article{dudeja2023universality,
  title={Universality of approximate message passing with semirandom matrices},
  author={Dudeja, Rishabh and Lu, Yue M. and Sen, Subhabrata},
  journal={Annals of Probability},
  volume={51},
  number={5},
  pages={1616--1683},
  year={2023},
  publisher={Institute of Mathematical Statistics}
}

@book{charbonneau2023spin,
  title={{Spin Glass Theory and Far Beyond: Replica Symmetry Breaking after 40 Years}},
  author={Charbonneau, Patrick and Marinari, Enzo and Parisi, Giorgio and Ricci-Tersenghi, Federico and Sicuro, Gabriele and Zamponi, Francesco and M{\'e}zard, Marc},
  year={2023},
  publisher={World Scientific}
}

@article{takeuchi2019rigorous,
  title={Rigorous dynamics of expectation-propagation-based signal recovery from unitarily invariant measurements},
  author={Takeuchi, Keigo},
  journal={IEEE Transactions on Information Theory},
  volume={66},
  number={1},
  pages={368--386},
  year={2019},
  publisher={IEEE}
}

@article{ma2017orthogonal,
  title={{Orthogonal AMP}},
  author={Ma, Junjie and Ping, Li},
  journal={IEEE Access},
  volume={5},
  pages={2020--2033},
  year={2017},
  publisher={IEEE}
}

@article{gerbelot2022rigorous,
  title={{Rigorous Dynamical Mean-Field Theory for Stochastic Gradient Descent Methods}},
  author={Gerbelot, Cedric and Troiani, Emanuele and Mignacco, Francesca and Krzakala, Florent and Zdeborov{\'a}, Lenka},
  journal={SIAM Journal on Mathematics of Data Science},
  year={2024},
  volume = 6,
  issue = 2,
  pages = {400--427}
}

@article{rangan2016inference,
  title={{Inference for generalized linear models via alternating directions and Bethe free energy minimization}},
  author={Rangan, Sundeep and Fletcher, Alyson K. and Schniter, Philip and Kamilov, Ulugbek S.},
  journal={IEEE Transactions on Information Theory},
  volume={63},
  number={1},
  pages={676--697},
  year={2016},
  publisher={IEEE}
}

@inproceedings{rangan2011generalized,
  title={Generalized approximate message passing for estimation with random linear mixing},
  author={Rangan, Sundeep},
  booktitle={IEEE International Symposium on Information Theory (ISIT 2011)},
  pages={2168--2172},
  year={2011},
  organization={IEEE}
}

@article{donoho2009message,
  title={Message-passing algorithms for compressed sensing},
  author={Donoho, David L. and Maleki, Arian and Montanari, Andrea},
  journal={Proceedings of the National Academy of Sciences},
  volume={106},
  number={45},
  pages={18914--18919},
  year={2009},
  publisher={National Acad Sciences}
}

@article{LW-2022-NonAsymptoticAMPSpiked,
  title={A non-asymptotic framework for approximate message passing in spiked models},
  author={Li, Gen and Wei, Yuting},
  journal={arXiv preprint arXiv:2208.03313},
  year={2022}
}

@article{javanmard2013state,
  title={State evolution for general approximate message passing algorithms, with applications to spatial coupling},
  author={Javanmard, Adel and Montanari, Andrea},
  journal={Information and Inference: A Journal of the IMA},
  volume={2},
  number={2},
  pages={115--144},
  year={2013},
  publisher={OUP}
}

@article{barbierSpatial,
    title = {{Approximate message-passing with spatially coupled structured operators, with applications to compressed sensing and sparse superposition codes}},
    author = {Barbier, Jean and Schulke, Christophe and Krzakala, Florent},
    year = 2015,
    journal = {Journal of Statistical Mechanics: Theory and Experiment}
}

@inproceedings{vila2015adaptive,
  title={Adaptive damping and mean removal for the generalized approximate message passing algorithm},
  author={Vila, Jeremy and Schniter, Philip and Rangan, Sundeep and Krzakala, Florent and Zdeborov{\'a}, Lenka},
  booktitle={IEEE International Conference on Acoustics, Speech and Signal Processing (ICASSP 2015)},
  pages={2021--2025},
  year={2015},
  organization={IEEE}
}

@article{fan2022approximate,
  title={Approximate message passing algorithms for rotationally invariant matrices},
  author={Fan, Zhou},
  journal={Annals of Statistics},
  volume={50},
  number={1},
  pages={197--224},
  year={2022},
  publisher={Institute of Mathematical Statistics}
}

@article{zhong2024approximate,
  title={Approximate message passing for orthogonally invariant ensembles: Multivariate non-linearities and spectral initialization},
  author={Zhong, Xinyi and Wang, Tianhao and Fan, Zhou},
  journal={Information and Inference: A Journal of the IMA},
  volume={13},
  number={3},
  pages={iaae024},
  year={2024},
  publisher={Oxford University Press}
}

@article{MR-2015-NonNegative,
  title={Non-negative principal component analysis: message passing algorithms and sharp asymptotics},
  author={Montanari, Andrea and Richard, Emile},
  journal={IEEE Transactions on Information Theory},
  volume={62},
  number={3},
  pages={1458--1484},
  year={2015},
  publisher={IEEE}
}

@inproceedings{DMR-2014-ConePCA,
  title={Cone-constrained principal component analysis},
  author={Deshpande, Yash and Montanari, Andrea and Richard, Emile},
  booktitle={Advances in Neural Information Processing Systems (NIPS 2014)},
  pages={2717--2725},
  year={2014},
}

@incollection{montanari2012graphical,
  author    = {Andrea Montanari},
  title     = {{Graphical models concepts in compressed sensing}},
  booktitle = {{Compressed Sensing: Theory and Applications}},
  editor    = {Yonina C. Eldar and Gitta Kutyniok},
  pages     = {394--438},
  year      = {2012},
  publisher = {Cambridge University Press},
}

@article{bao2025leave,
  title={A leave-one-out approach to approximate message passing},
  author={Bao, Zhigang and Han, Qiyang and Xu, Xiaocong},
  journal={Annals of Applied Probability},
  volume={35},
  number={4},
  pages={2716--2766},
  year={2025},
  publisher={Institute of Mathematical Statistics}
}

@article{rangan2019convergence,
  title={On the convergence of approximate message passing with arbitrary matrices},
  author={Rangan, Sundeep and Schniter, Philip and Fletcher, Alyson K. and Sarkar, Subrata},
  journal={IEEE Transactions on Information Theory},
  volume={65},
  number={9},
  pages={5339--5351},
  year={2019},
  publisher={IEEE}
}

@inproceedings{subsamplingJavanmard,
  author       = {Adel Javanmard and
                  Andrea Montanari},
  title        = {Subsampling at information theoretically optimal rates},
  booktitle    = {Proceedings of the International Symposium on Information
                  Theory ({ISIT} 2012)},
  pages        = {2431--2435},
  publisher    = {{IEEE}},
    year = 2012
}

@article{rush2018finite,
  title={Finite sample analysis of approximate message passing algorithms},
  author={Rush, Cynthia and Venkataramanan, Ramji},
  journal={IEEE Transactions on Information Theory},
  volume={64},
  number={11},
  pages={7264--7286},
  year={2018},
  publisher={IEEE}
}

@article{rangan2019vector,
  title={Vector approximate message passing},
  author={Rangan, Sundeep and Schniter, Philip and Fletcher, Alyson K.},
  journal={IEEE Transactions on Information Theory},
  volume={65},
  number={10},
  pages={6664--6684},
  year={2019},
  publisher={IEEE}
}

@article{donoho2013information,
  title={Information-theoretically optimal compressed sensing via spatial coupling and approximate message passing},
  author={Donoho, David L. and Javanmard, Adel and Montanari, Andrea},
  journal={IEEE Transactions on Information Theory},
  volume={59},
  number={11},
  pages={7434--7464},
  year={2013},
  publisher={IEEE}
}

@article{opper2016theory,
  title={{A theory of solving TAP equations for Ising models with general invariant random matrices}},
  author={Opper, Manfred and {\c{C}}akmak, Burak and Winther, Ole},
  journal={Journal of Physics A: Mathematical and Theoretical},
  volume={49},
  number={11},
  pages={114002},
  year={2016},
  publisher={IOP Publishing}
}

@article{zdeborova2016statistical,
  title={Statistical physics of inference: Thresholds and algorithms},
  author={Zdeborov{\'a}, Lenka and Krzakala, Florent},
  journal={Advances in Physics},
  volume={65},
  number={5},
  pages={453--552},
  year={2016},
  publisher={Taylor \& Francis}
}

@book{mezard1987spinglasstheoryandbeyond,
  title={{Spin glass theory and beyond: An Introduction to the Replica Method and Its Applications}},
  author={M{\'e}zard, Marc and Parisi, Giorgio and Virasoro, Miguel Angel},
  volume={9},
  year={1987},
  publisher={World Scientific}
}

@article{bolthausen2014iterative,
  title={{An Iterative Construction of Solutions of the TAP Equations for the Sherrington-Kirkpatrick Model}},
  author={Bolthausen, Erwin},
  journal={Communications in Mathematical Physics},
  volume={325},
  number={1},
  pages={333--366},
  year={2014},
  publisher={Springer}
}

@article{montanari2022statistically,
    author = {Montanari, Andrea and Wu, Yuchen},
    title = {{Statistically optimal first-order algorithms: a proof via orthogonalization}},
    journal = {Information and Inference: A Journal of the IMA},
    volume = {13},
    year = 2024
}

@inproceedings{montanari2021optimization,
  author       = {Andrea Montanari},
  title        = {{Optimization of the Sherrington-Kirkpatrick Hamiltonian}},
  booktitle    = {60th {IEEE} Annual Symposium on Foundations of Computer Science ({FOCS}
                  2019)},
  pages        = {1417--1433},
  publisher    = {{IEEE}},
  year         = {2019},
}

@article{chen2021universality,
  title={Universality of approximate message passing algorithms},
  author={Chen, Wei-Kuo and Lam, Wai-Kit},
  journal={Electronic Journal of Probability},
  volume={26},
  pages={1--44},
  year={2021}
}

@article{bayati2011dynamics,
  title={The dynamics of message passing on dense graphs, with applications to compressed sensing},
  author={Bayati, Mohsen and Montanari, Andrea},
  journal={IEEE Transactions on Information Theory},
  volume={57},
  number={2},
  pages={764--785},
  year={2011},
  publisher={IEEE}
}

@article{feng2022unifying,
  title={{A Unifying Tutorial on Approximate Message Passing}},
  author={Feng, Oliver Y. and Venkataramanan, Ramji and Rush, Cynthia and Samworth, Richard J.},
  journal={Foundations and Trends in Machine Learning},
  volume={15},
  number={4},
  pages={335--536},
  year={2022},
  publisher={Now Publishers, Inc.}
}

@article{montanari2022equivalence,
  title={Equivalence of approximate message passing and low-degree polynomials in rank-one matrix estimation},
  author={Montanari, Andrea and Wein, Alexander S.},
  journal={Probability Theory and Related Fields},
  year={2025},
  volume= 191,
  pages = {181--233}
}

@inproceedings{ivkov2023semidefinite,
  author       = {Misha Ivkov and
                  Tselil Schramm},
  title        = {Semidefinite Programs Simulate Approximate Message Passing Robustly},
  booktitle    = {Proceedings of the 56th Annual {ACM} Symposium on Theory of Computing ({STOC} 2024)},
  pages        = {348--357},
  publisher    = {{ACM}},
  year         = {2024}
}

@article{bayati2015universality,
  title={Universality in polytope phase transitions and message passing algorithms},
  author={Bayati, Mohsen and Lelarge, Marc and Montanari, Andrea},
 journal = {Annals of Applied Probability},
 number = {2},
 pages = {753--822},
 publisher = {Institute of Mathematical Statistics},
 volume = {25},
 year = {2015}
}

@inproceedings{celentano2020estimation,
  title={The estimation error of general first order methods},
  author={Celentano, Michael and Montanari, Andrea and Wu, Yuchen},
  booktitle={Conference on Learning Theory (COLT 2020)},
  pages={1078--1141},
  year={2020},
  organization={PMLR}
}

@book{MezardMontanari,
  title={Information, Physics, and Computation},
  author={M{\'e}zard, Marc and Montanari, Andrea},
  year={2009},
  publisher={Oxford University Press}
}

@article{AMS20:pSpinGlasses,
  title={Optimization of mean-field spin glasses},
  author={Alaoui, Ahmed El and Montanari, Andrea and Sellke, Mark},
  journal={Annals of Probability},
  volume={49},
  number={6},
  pages={2922--2960},
  year={2021},
  doi={10.1214/21-AOP1519}
}

@article {Holyer81:EdgePartitioning,
    AUTHOR = {Holyer, Ian},
     TITLE = {The {NP}-completeness of some edge-partition problems},
  JOURNAL = {SIAM Journal on Computing},
    VOLUME = {10},
      YEAR = {1981},
    NUMBER = {4},
     PAGES = {713--717},
      ISSN = {0097-5397},
   MRCLASS = {68C25 (68E10)},
  MRNUMBER = {635429},
MRREVIEWER = {V. B. Alekseev},
       DOI = {10.1137/0210054},
       URL = {https://doi.org/10.1137/0210054},
}

@book {Janson:GaussianHilbertSpaces,
    AUTHOR = {Janson, Svante},
     TITLE = {Gaussian {H}ilbert spaces},
    SERIES = {Cambridge Tracts in Mathematics},
    VOLUME = {129},
 PUBLISHER = {Cambridge University Press, Cambridge},
      YEAR = {1997},
     PAGES = {x+340},
      ISBN = {0-521-56128-0},
   MRCLASS = {60G35 (60H15 62M20 81T08)},
  MRNUMBER = {1474726},
MRREVIEWER = {Amarjit Budhiraja},
       DOI = {10.1017/CBO9780511526169},
       URL = {https://doi.org/10.1017/CBO9780511526169},
}

@article{MFCKMZ-2019-PlefkaExpansionOrthogonalIsing,
  title={High-temperature expansions and message passing algorithms},
  author={Maillard, Antoine and Foini, Laura and Castellanos, Alejandro Lage and Krzakala, Florent and M{\'e}zard, Marc and Zdeborov{\'a}, Lenka},
  journal={Journal of Statistical Mechanics: Theory and Experiment},
  volume={2019},
  number={11},
  pages={113301},
  year={2019},
  publisher={IOP Publishing}
}

@article{Banica-2010-OrthogonalWeingartenFormula,
  title={{The orthogonal Weingarten formula in compact form}},
  author={Banica, Teodor},
  journal={Letters in Mathematical Physics},
  volume={91},
  number={2},
  pages={105--118},
  year={2010},
  publisher={Springer}
}

@article{CS-2006-HaarMeasureMoments,
  title={Integration with respect to the {Haar} measure on unitary, orthogonal and symplectic group},
  author={Collins, Beno{\^\i}t and {\'S}niady, Piotr},
  journal={Communications in Mathematical Physics},
  volume={264},
  number={3},
  pages={773--795},
  year={2006},
  publisher={Springer}
}

@inproceedings{KMW-2024-TensorCumulantsInvariantInference,
  author={Kunisky, Dmitriy and Moore, Cristopher and Wein, Alexander S.},
  booktitle={65th Annual Symposium on Foundations of Computer Science (FOCS 2024)}, 
  title={Tensor Cumulants for Statistical Inference on Invariant Distributions}, 
  year={2024},
  volume={},
  number={},
  pages={1007-1026},
}

@book{bondyMurty,
  title={{Graph Theory}},
  author={Bondy, Adrian and Murty, {U.S.R.}},
  year={2008},
  publisher={Springer}
}

@book{NS-2006-LecturesCombinatoricsFreeProbability,
  title={Lectures on the combinatorics of free probability},
  author={Nica, Alexandru and Speicher, Roland},
  volume={13},
  year={2006},
  publisher={Cambridge University Press}
}

@article{Rota-1964-Foundations,
  title={On the foundations of combinatorial theory: {I}. {Theory} of {M{\"o}bius} functions},
  author={Rota, Gian-Carlo},
  journal={Probability Theory and Related Fields},
  volume={2},
  number={4},
  pages={340--368},
  year={1964},
  publisher={Springer}
}

@article{abbara2020universality,
  title={On the universality of noiseless linear estimation with respect to the measurement matrix},
  author={Abbara, Alia and Baker, Antoine and Krzakala, Florent and Zdeborov{\'a}, Lenka},
  journal={Journal of Physics A: Mathematical and Theoretical},
  volume={53},
  number={16},
  pages={164001},
  year={2020},
  publisher={IOP Publishing}
}

@article{CO-2019-TAPEquationAMPInvariant,
  title={{Memory-free dynamics for the Thouless-Anderson-Palmer equations of Ising models with arbitrary rotation-invariant ensembles of random coupling matrices}},
  author={{\c{C}}akmak, Burak and Opper, Manfred},
  journal={Physical Review E},
  volume={99},
  number={6},
  pages={062140},
  year={2019},
  publisher={APS}
}

@article{schniter2020simple,
  title={{A simple derivation of AMP and its state evolution via first-order cancellation}},
  author={Schniter, Philip},
  journal={IEEE Transactions on Signal Processing},
  volume={68},
  pages={4283--4292},
  year={2020},
  publisher={IEEE}
}

@phdthesis{pesentiThesis,
  title={Algorithms beyond the union bound: polynomial optimization and discrepancy theory},
  author={Pesenti, Lucas},
  year={2026},
  school={Bocconi University}
}

@article{colorCoding,
    title = {Color-coding},
    author = {Alon, Noga and Yuster, Raphael and Zwick, Uri},
    journal = {Journal of the ACM},
    volume = 42,
    issue = 4,
    pages = {844--856},
    year = 1995
}

@inproceedings{heavyTailedWigner,
    title = {{Estimating Rank-One Spikes from Heavy-Tailed Noise via Self-Avoiding Walks}},
    author = {Ding, Jingqiu and Hopkins Samuel and Steurer, David},
    booktitle = {Advances in Neural Information Processing Systems (NeurIPS 2020)},
    volume = {33},
    pages = {5576--5586},
    year = 2020
}

\appendix

\newpage
\section{Traffic Distributions via Feynman Diagrams}
\label{sec:feynman-physics}

One of our motivations is to connect graph polynomials with the celebrated Feynman diagram technique from physics.
In quantum field theory, Feynman diagram expansion is used to reduce matrix integrals into graphical calculations.
We show in this section that this method can (heuristically) derive the traffic distribution of orthogonally invariant distributions (\cref{thm:moments}).

The matrix model that we consider in this section is specified by a {\em potential
function} $V:\R\to\R$, and has partition function
\begin{equation}\label{eq:action}
    Z \defeq \int_{\mathcal M}\d\bA\ e^{-\frac n2\Tr V(\bA)}\,,
\end{equation}
where $\mathcal M \defeq \R^{n\times n}_{\sym}$ is the space of symmetric $n\times n$
matrices. Equivalently, this is the partition function of the random matrix
$\bA\in\mathcal M$ sampled from the probability measure
$\mu_V(\bA)\propto \exp(-\frac n2\Tr V(\bA))$, which is a special case of an
orthogonally invariant distribution (\cref{sec:rot-inv}).

In physics, matrix integrals such as~\cref{eq:action} are viewed as a
$0$-dimensional theory: the variable is a matrix, and the partition
function is a finite-dimensional integral rather than a functional integral over
fields on space-time. The large-$n$ expansion of such integrals is organized into diagrammatic
contributions indexed by {\em Feynman diagrams}.
\begin{enumerate}
    \item In the limit $n\to\infty$, only planar diagrams contribute at leading
    order, an observation going back to foundational work of 't~Hooft~\cite{tHooft1974planar, brezin1978planar}. Related planarity phenomena also appear in
    mathematics, for example in the connections between large random matrices and
    non-crossing pairings.

    \item In special scaling limits of the potential
    with $n$, the Feynman diagram expansion can be interpreted in terms of physical theories such as 2D
    gravity and certain string-theoretic models~\cite{diFrancesco1995gravity, cotler2017black, saad2019jt}.
\end{enumerate}

The combinatorial approach in this paper fits naturally into this perspective. First,
our results are formulated in the large-$n$ limit, and the dominant combinatorial
objects in that limit are planar, as in the 't~Hooft limit. Second, we show that our
$w$- and $z$-polynomials are planar dual to the Feynman diagrams
traditionally used in physics. Third, while the Feynman diagram method is based on perturbative expansion around the GOE potential $V(x) = x^2/2$, our rigorous results~\cref{thm:moments,thm:convergence-distribution} still remain valid
beyond the radius of convergence for perturbative methods.

We present 
in this section the traditional approach
for computing~\cref{eq:action} based on
Feynman diagrams. The argument is ``combinatorially rigorous''
(true at the level of generating functions),
but not sufficient to rigorously derive the probabilistic conclusions.

\subsection{Calculation of the free energy}

For now, we restrict to the case where the potential in~\cref{eq:action} is $V(\bA) = \frac 12 \bA^2 + \frac{g}{4}\bA^4$, where the \emph{coupling constant} $g$ measures the strength of the quartic interaction in the model. Such potentials appear in string theory, statistical physics (the $\lambda\phi^4$ theory), and the theory of integrable systems. The quartic term $\frac g4 \Tr(\bA^4)$ can be viewed as a correction term to the \goe model, for which $Z_\goe = \int_\calM \d\bA \exp(-n\Tr(\bA^2)/4)$.

The idea of the Feynman diagram technique is to perturbatively expand this correction
term, reducing to a problem on Gaussian variables.
We illustrate this by computing
the free energy
of the quartic model, namely the quantity $\ln Z$ (this example can be found in physics textbooks).
For an observable quantity $\cO$,
we write 
$\langle \cO \rangle := \E_{\bA \sim \mu_V}[\cO]$, and 
$\langle \cO \rangle_{\goe} := \E_{\bA \sim \goe}[\cO]$.
We have
\begin{align*}
    Z &= \int_{\calM}\d\bA \exp(-\frac n4 \Tr(\bA^2) - \frac{gn}{8}  \Tr(\bA^4))\\
    &= Z_\goe\cdot \left\langle \exp(-\frac{gn}{8} \Tr(\bA^4))\right\rangle_\goe\,.
\end{align*}
A simple calculation shows that $Z_\goe = 2^{\frac{n}{2}}\left(\frac{2\pi}{n}\right)^{\frac{n(n+1)}{4}}$. We Taylor expand the remaining part and integrate term-by-term:
\begin{align}
    \left\langle \exp\left(-\frac{gn}8 \Tr(\bA^4)\right)\right\rangle_\goe &= \sum_{s = 0}^\infty \frac{1}{s!}\left(-\frac{gn}8 \right)^s\left\langle\Tr(\bA^4)^s\right\rangle_\goe\,. \label{eq:taylor-exp}
\end{align}
The quantities $\langle \Tr(\bA^4)^s\rangle_\goe$ on the right-hand side are expectations over Gaussian random variables, and can be computed by Wick's lemma (\cref{lem:wick}) to be a sum over all \emph{Wick contractions} between the variables (in graph-theoretic terms, a sum over all perfect matchings).
The \emph{propagator} for a single contraction with a \goe matrix is
the covariance of the Gaussians,
\begin{equation}\label{eq:propagator}
    \langle\bA[i,j]\bA[k,\ell]\rangle_\goe = \frac 1n \delta_{ik}\delta_{j\ell} + \frac 1n \delta_{i\ell}\delta_{jk}\,,\quad \textnormal{where $\delta_{ij} \defeq \begin{cases}1 & \textnormal{if $i=j$}\\0 & \textnormal{otherwise}\end{cases}$}\,.
\end{equation}

A \emph{Feynman diagram} represents a combinatorial type of Wick contractions.
In the graphical notations of this paper, we would visualize each $\Tr(\bA^4)$ as a square, with Wick contraction having the effect of gluing together edges of the squares.
The 't Hooft double line notation, which is more common in physics, represents each $\Tr(\bA^4)$ as a vertex with four incident double edges.
These representations are dual
to each other (in the sense of planar
duality); see~\cref{fig:feynman-diagrams} for comparison.

\begin{figure}[ht]
    \centering
    \begin{subfigure}[t]{0.45\linewidth}
        \centering
        \includegraphics[width=0.79\linewidth,valign=t]{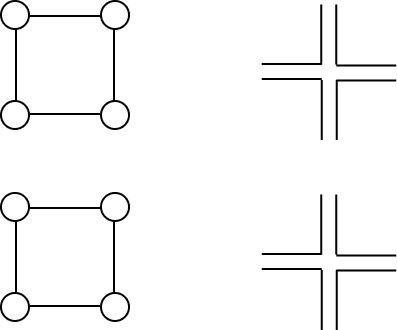}

        \vspace{5mm}
        
        \caption{$\Tr(\bA^4)^2$ represented as two squares in our notation, compared to the 't Hooft double line notation.}
        \label{fig:feynman-left}
    \end{subfigure}
    \hfill
    \begin{subfigure}[t]{0.45\linewidth}
        \centering
        \includegraphics[width=\linewidth,valign=t]{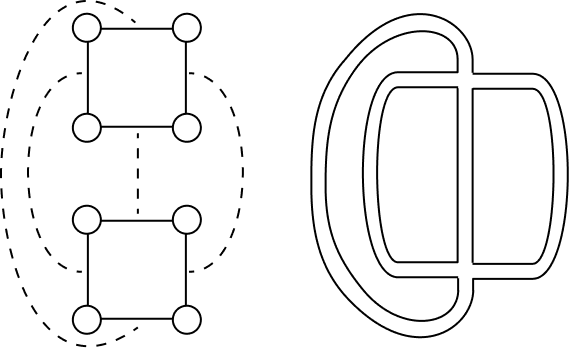}

        \vspace{3mm}
        
        \caption{One of the Wick contractions appearing in $\langle \Tr(\bA^4)^2\rangle_\goe$. The edges of the squares are glued together according to the matching to make a ``pillow''.}
        \label{fig:feynman-right}
    \end{subfigure}

    \caption{Our Feynman diagram notation vs.\ the 't Hooft double line notation.}
    \label{fig:feynman-diagrams}
\end{figure}

The delta functions in the propagator enforce that the vertices of the squares have a consistent index $i$ when the edges of the squares are glued together.
Note that the propagator in~\cref{eq:propagator} for the \goe model allows $\bA[i,j], \bA[k,\ell]$ to be glued in either orientation (in contrast to the Gaussian Unitary Ensemble which would only have one term).
Therefore, we define a Feynman diagram for the \goe to be an \emph{oriented} perfect matching between the edges of the squares.
For each Feynman diagram $\gam$, the contribution of $\gam$ to~\cref{eq:taylor-exp} is:
\begin{enumerate}[(i)]
    \item a factor $n$ per vertex of $\gam$, since each vertex holds an index from $[n]$ which is summed over in $\Tr(\bA^4)$.
    \item a factor $\frac 1n$ per paired edge of $\gam$ from the propagator,~\cref{eq:propagator}.
    \item a factor $-\frac{gn}{8}$ per square face of $\gam$ from~\cref{eq:taylor-exp}. There is also an overall factor of $\frac{1}{|F(\gam)|!}$ where $|F(\gam)|$ equals the number of square faces in $\gam$.
\end{enumerate}

For example, the $s=1$ term in~\cref{eq:taylor-exp} is
\begin{equation*}
    \left(-\frac{gn}{8}\right)\cdot \langle \Tr(\bA^4)\rangle_\goe = \left(-\frac g 8\right) \cdot  \left(2\cdot n^2 + 5 \cdot n + 5 \right)\,.
\end{equation*}
The Feynman diagrams are enumerated in~\cref{fig:feynman-square}.

\begin{figure}[ht]
    \centering
    \includegraphics[width=0.9\linewidth]{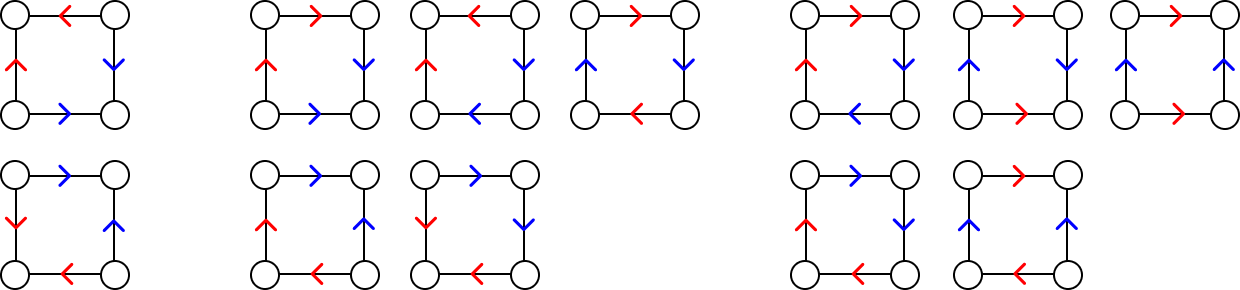}
    \caption{The 12 Feynman diagrams associated to $\langle \Tr(\bA^4)\rangle_\goe$. The two red edges are matched in the orientation specified by the arrows, and similarly for the blue edges. Gluing either of the left two diagrams results in a ``taco''. Gluing the remaining diagrams results in degenerate polyhedra. After gluing, the ``tacos'' have 3 vertices, the middle diagrams have 2 vertices, and the right diagrams have 1 vertex, respectively.}
    \label{fig:feynman-square}
\end{figure}

For a given Feynman diagram $\gam$, the total factor of $n$ is $|V(\gam)| - |E(\gam)| + |F(\gam)| =: \chi(\gam)$ which is the \emph{Euler characteristic} of the polyhedron $\gam$.
In total, we obtain a Feynman diagram expansion for the partition function,
\begin{equation}\label{eq:feynman-fml-1}
    Z = Z_{\goe} \sum_{\gam \in \Gam} \frac{1}{|F(\gam)|!}\left(-\frac{g}{8}\right)^{|F(\gam)|} n^{\chi(\gam)}\,,
\end{equation}
where $\Gam$ is the set of Feynman diagrams, the set of polyhedra built from square faces. Formally, $\Gam = \sqcup_{s \geq 0} \Gam_s$, where $\Gam_s$ is the set of oriented perfect matchings between the edges of $s$ squares.

Taking the logarithm has the effect of restricting the summation to \emph{connected} Feynman diagrams; this is the linked cluster theorem in quantum field theory~\cite[Section~3.5]{etingof2024mathematical}.
We obtain:
\begin{equation}\label{eq:feynman-free-energy}
    \ln \left(\frac Z {Z_\goe}\right) = \sum_{\gam \in \Gam_c} \frac{1}{|F(\gam)|!}\left(-\frac{g}{8}\right)^{|F(\gam)|} n^{\chi(\gam)}\,,
\end{equation}
where $\Gam_c \subseteq \Gam$ are connected Feynman diagrams.

\subsubsection{Asymptotic limit \texorpdfstring{$n \to \infty$}{n -> infty}}

As $n \to \infty$,~\cref{eq:feynman-free-energy} significantly simplifies because only the \emph{planar} diagrams survive, i.e.\ polyhedra $\gam$ with ``no holes,'' which have the maximum possible Euler characteristic among connected graphs ($\chi(\gam) = 2$).
This foundational observation goes back to 't~Hooft~\cite{tHooft1974planar}.\footnote{
    't~Hooft studies unitarily invariant matrix models instead of orthogonally
    invariant ones.
    He takes a further step by sending $g\to 0$ at the rate $\Theta(1/\sqrt{n})$, i.e., fixing $\lambda = g^2n$ to be constant. His claim is that $\lambda$ is the only parameter characterizing the physical properties of observables in the large-$n$ limit, and by taking $\lam \to \infty$ one gains some intuition on the physical phenomena of strongly interacting particles.
    The limit $g \to 0$ is less interesting for us, since the traffic distribution (hence also the spectrum) is asymptotically the same as the GUE whenever $g = o(1)$.
} We obtain, at first order,
\begin{equation}
    \frac{1}{n^2}\ln \left(\frac Z {Z_\goe}\right) = \sum_{\substack{\gam \in \Gam_c\\\text{planar}}} \frac{1}{|F(\gam)|!}\left(-\frac{g}{8}\right)^{|F(\gam)|} + O(n^{-2})\,.\label{eq:thooft}
\end{equation}

In summary,
the Feynman diagram method shows that the non-Gaussian component of the matrix model can be replaced by a
generating function for graphs/surfaces which, in the $n\to \infty$ limit,
restricts to a generating function for planar graphs/surfaces with genus 0.
This restriction leads to significant simplifications in diagrammatic calculations,
in the same way as our cactus property and treelike property in the rest of the paper.

\subsection{Calculation of general observables\texorpdfstring{: Argument for~\cref{thm:moments}}{}}

We now assume that the potential $V(\bA)$ has the general form $V(\bA)=\frac 12 \bA^2 + \sum_{j\geq 3}c_j\bA^j$ (arbitrary coefficients on $\bA$ and $\bA^2$ can be handled by centering and rescaling, respectively).
We compute the traffic distribution of $\bA$,
which consists of all $S_n$-invariant observables of $\bA$. The $z$-polynomials are a basis for these observables where, for each multigraph $\al$,
\begin{equation*}
    \frac 1n\langle z_\al(\bA) \rangle = \frac 1n \sum_{i : V(\al) \embeds [n]}\left\langle\prod_{\{u,v\} \in E(\al)} \bA[i(u), i(v)]\right\rangle\,.
\end{equation*}
Separating out the Gaussian part of the action from the higher-order interactions:
\begin{align*}
    \frac 1n\langle z_\al(\bA) \rangle = \frac 1n \cdot \frac{\left\langle z_\al(\bA) \exp(-\sum_{j \geq 3} c_j n \Tr(\bA^j)) \right\rangle_\goe}{\left\langle \exp(-\sum_{j \geq 3} c_j n \Tr(\bA^j)) \right\rangle_\goe}\,.
\end{align*}

The dual Feynman diagrams are built from polygons with $j\ge 3$ sides, each of which comes with a factor of $-c_j$, generalizing the situation from the previous section.
A small generalization of the argument
shows that the denominator is
\begin{equation}\label{eq:denom}
    \left\langle \exp(-\sum_{j \geq 3} c_j n \Tr(\bA^j)) \right\rangle_\goe = \sum_{\gam \in \Gam} \left(\prod_{j \geq 3} \frac{(-c_j)^{|F_j(\gam)|}}{|F_j(\gam)|!}\right) n^{\chi(\gam)}\,,
\end{equation}
where $|F_j(\gam)|$ denotes the number of $j$-sided faces in $\gam$.

The numerator can also be calculated diagrammatically.
The Wick contractions go between a collection of polygons as well as the additional edges $\bA[i,j]$ in $z_\al(\bA)$\,.
Let $\Gam(\al)$ be the set of Feynman diagrams, visualized as polyhedra built on a set of ``boundary'' edges $\al$.
Then

\begin{equation}\label{eq:num}
    \left\langle z_\al(\bA) \exp(-\sum_{j \geq 3} c_j n \Tr(\bA^j)) \right\rangle_\goe = \sum_{\gam \in \Gam(\al)} \left(\prod_{j \geq 3} \frac{(-c_j)^{|F_j(\gam)|}}{|F_j(\gam)|!}\right) n^{\chi(\gam)}\cdot(1-O(n^{-1}))\,.
\end{equation}

Note $\al$ is considered a boundary and does not count towards the faces $F_j(\gam)$.

To enforce that the labels of $z_\al(\bA)$ are injective, we remove from $\Gam(\al)$ any matching which causes two vertices of $\al$ to have the same label.
The factor $1-O(n^{-1})$ arises because each vertex is summed over $n - O(1)$ indices to maintain injectivity, instead of precisely $n$ which we had previously.

We obtain the final result by
dividing~\cref{eq:num} by~\cref{eq:denom}.
This has the effect of restricting to the set of connected Feynman diagrams $\Gam_c(\al) \subseteq \Gam(\al)$ by an alternate version of the linked cluster theorem.
The final Feynman diagram formula is:
\begin{align}
    \frac 1n \langle z_\al(\bA)\rangle &= \sum_{\gam \in \Gam_c(\al)} \left(\prod_{j \geq 3}\frac{\left(-c_j\right)^{|F_j(\gam)|}}{|F_j(\gam)|!} \right) \cdot n^{\chi(\gam)-1}\cdot(1-O(n^{-1}))\,. \label{eq:feynman-alpha}
\end{align}

\begin{remark}
An alternative approach to the calculation would be to first symmetrize $z_\al(\bA)$ over $O(n)$ 
which is the symmetry group of the matrix model (and is larger than $S_n$),
then to plug in the values of the $O(n)$-invariant observables (the trace polynomials).
We find it simpler to Taylor expand the action directly.
\end{remark}

\subsubsection{Asymptotic limit \texorpdfstring{$n\to\infty$}{n -> infty}}

In the asymptotic limit $n \to \infty$,
the only diagrams in~\cref{eq:feynman-alpha} with constant-order magnitude are those such that $\al$ is a cactus graph, and $\gam$ consists of polyhedra with genus 0 attached to each cycle of the cactus, which has $\chi(\gam) = 1$.
We prove this combinatorially in the forthcoming~\cref{lem:combinatorially-dominant}.

The large-$n$ combinatorial summation factors over the cycles of the cactus, since the genus-0 polyhedra on each cycle can be chosen independently.
We obtain
\begin{equation*}
    \frac 1n \langle z_\al(\bA) \rangle = \begin{cases}
        \displaystyle\prod_{\sig \in \cycles(\al)} \frac 1n \langle z_\sig(\bA)\rangle  + O(n^{-1}) &\text{ if $\al$ is a cactus}\\
        O(n^{-1}) & \text{otherwise}
    \end{cases}
\end{equation*}
The limiting value $\frac 1n \langle z_\sig(\bA)\rangle$ of the $q$-cycle diagram $\sig$ is equal to $\kappa_{q} + O(n^{-1})$ by the moment/free cumulant relation~\cref{eq:freeCumulantRelation}. Thus,~\cref{eq:feynman-alpha} recovers~\cref{thm:moments}.

\begin{lemma}\label{lem:combinatorially-dominant}
    Let $\al \in \calA$ be a connected multigraph, and let $\gam \in \Gam_c(\al)$.
    Then $\chi(\gam) = 1$ if and only if $\al$ is a cactus and $\gam$ consists of genus-0 polyhedra attached to each cycle of $\al$.
\end{lemma}
\begin{proof}
The only $\al$ for which $\Gam_c(\al)$ is nonzero are the Eulerian $\al$,
since a polyhedron $\gam \in \Gam_c(\al)$ with boundary $\al$ must have a boundary which is a union of cycles.
Therefore, it remains to argue about Eulerian graphs $\al$.

For Eulerian $\al$, the $\gam \in \Gam_c(\al)$ which maximize the quantity $\chi(\gam) = |V(\gam)| - |E(\gam)| + |F(\gam)|$ are given by decomposing $\al$ into the maximum number of simple cycles, then attaching a genus 0 polyhedron to each cycle.
This achieves $\chi(\gam) = |V(\al)| - |E(\al)| + C$ where $C$ is the number of cycles.\footnote{Note that the computational problem of, given an Eulerian graph $\al$, compute a partition of $E(\al)$ into the maximum number of cycles, is NP-hard \cite{Holyer81:EdgePartitioning}.}

We argue that:
\begin{equation}\label{eq:max-cycle-partition}
    |V(\al)| - |E(\al)| + C \leq 1
\end{equation}
for all Eulerian graphs $\al$ and this is achieved if and only if $\al$ is a cactus.
Fix a maximum cycle partition of $\al$. The $C$ cycles are edge-disjoint so we can remove one edge from each one while maintaining that the graph is connected. Let $\al'$ be the resulting graph. Then $|V(\al')| - |E(\al')| = |V(\al)| - |E(\al)| + C$. Since $\al'$ is still connected we have $|V(\al')| - |E(\al')| \leq 1$.
The final inequality is an equality if and only if $\al'$ is a tree and hence $\al$ is a cactus.
This proves~\cref{eq:max-cycle-partition} and completes the lemma.
\end{proof}

\subsection{Mathematical comments on the Feynman diagram method}
\label{app:feynman-formal}

The Feynman diagram method is not mathematically rigorous, with (in our opinion)
the main obstruction being that intermediate summations such as~\cref{eq:feynman-fml-1,eq:num,eq:feynman-free-energy} are divergent.
The Euler characteristic grows with the
number of disconnected polyhedra, but the method proceeds anyway to divide out the disconnected polyhedra,
which ultimately yields a convergent summation in~\cref{eq:thooft} (for sufficiently small values of the coupling constant $g \geq 0$).

The Feynman diagram method is a {perturbative expansion}
because it holds for sufficiently small perturbations of the \goe density,
up to the radius of convergence of the Feynman diagram summations~\cite{mcLaughlin,garouf}.
On the other hand,~\cref{thm:moments} holds beyond the radius of convergence of the Feynman diagram expansion in~\cref{eq:feynman-alpha},
so it would be impossible to prove the theorem using a perturbative expansion alone.

\section{Traffic Distributions via Weingarten Calculus}
\label{app:weingarten}

We now present different tools and calculations for the traffic distributions of orthogonally invariant matrices based on the \emph{Weingarten formula} for the moments of entries of Haar-random orthogonal matrices.
These essentially follow the ideas of similar calculations by \cite{cebron2024traffic}, but use the version of the Weingarten formula for the orthogonal group, which we review below.

\subsection{Weingarten formula for orthogonal matrices}

For $\bm i = (i_1, \dots, i_k)$ and a perfect matching $\alpha \in \PM([k])$, define
\[ \delta_{\alpha}(\bm i) = \begin{cases} 1 & \text{if } i_u = i_v \text{ for all } \{u, v\} \in \alpha, \\ 0 & \text{otherwise.} \end{cases} \]
The Weingarten calculus expresses the moments of the Haar measure on $O(n)$ in terms of a certain ``Weingarten function'' $W_n(\al, \beta)$ on pairs of matchings.

\begin{lemma}[Weingarten formula]
    Let $\bQ \sim O(n)$ be a Haar-random orthogonal matrix. There exists a function $W_n: \PM([k])^2 \to \R$ such that
    \[
        \E_{\bQ \sim O(n)} \left[\bQ[i_1,j_1]\cdots \bQ[i_k,j_k]\right] = \sum_{\alpha, \beta \in \PM([k])} W_n(\alpha, \beta) \delta_{\alpha}(\bm i) \delta_{\beta}(\bm j).
    \]
\end{lemma}
\noindent
See \cite{CS-2006-HaarMeasureMoments,Banica-2010-OrthogonalWeingartenFormula} for an explicit definition of $W_n(\al, \beta)$.
We will only be interested in asymptotics for $k$ constant and $n \to \infty$, for which the approximations below will suffice.

When $k$ is odd, $\PM([k]) = \varnothing$, so the right-hand side above is zero, and indeed the left-hand side is easily seen to be zero without invoking the Weingarten formula, because $\bQ$ has the same law as $-\bQ$.
So, the only interesting case is $k$ even.
In that case, we give $\PM([k])$ the structure of a metric space, where $\Delta(\alpha, \beta)$ is defined as the minimum number of \emph{swap operations} needed to reach $\beta$ from $\alpha$ (a swap replaces pairs $\{a, b\}$, $\{c, d\}$ with pairs $\{a, c\}$, $\{b, d\}$).
It is easy to check that $\Delta$ is a metric (indeed, it is the distance on a certain graph structure defined on $\PM([k])$).
Further, write $\cyc(\alpha, \beta)$ for the set of even cycles formed by the disjoint union of $\alpha$ and $\beta$.
Then, it is easy to show the alternative characterization
\[ \Delta(\alpha, \beta) = \frac{k}{2} - |\cyc(\alpha, \beta)|\,. \]
As a sanity check, $|\cyc(\alpha, \beta)| \leq \frac{k}{2}$ with equality achieved if and only if $\alpha = \beta$, which is precisely the case $\Delta(\alpha, \beta) = 0$.

For $\alpha, \beta \in \PM([k])$, let $\cP(\alpha, \beta)$ be the set of geodesic paths from $p$ to $q$ in $\PM([k])$, i.e., of sequences $\alpha = \gamma_0, \gamma_1, \dots, \gamma_t = \beta$ with $\gamma_i \neq \gamma_{i + 1}$ for all $i = 0, \dots, t - 1$ and with $\sum_{i = 0}^{t - 1} \Delta(\gamma_i, \gamma_{i + 1}) = \Delta(\alpha, \beta)$.
For such a path $P = (\gamma_0, \dots, \gamma_t)$, write $|P| \colonequals t$.
Then, we define
\begin{align*}
\mu(\al, \beta) 
&= \sum_{P \in \cP(\alpha, \beta)} (-1)^{|P|}
\intertext{This may be viewed as a M{\"o}bius function of the partially ordered set whose chains are geodesics from a given ``base'' matching $p$ to each other matching. An explicit formula from \cite{CS-2006-HaarMeasureMoments} is}
\mu(\al, \beta) &= \prod_{C \in \cyc(\alpha, \beta)} (-1)^{\frac{|C|}{2} - 1}\Cat\left(\frac{|C|}{2} - 1\right) \numberthis \label{eq:cycle-mobius}
\end{align*}
where $\Cat(\cdot)$ are the Catalan numbers.
The key asymptotic for the Weingarten function for our purposes is then the following:

\begin{proposition}[\cite{CS-2006-HaarMeasureMoments}]
\label{lem:weirgarten-asymptotic}
    For a fixed $k$ and $\alpha, \beta \in \PM([k])$, as $n \to \infty$ we have
    \[ W_n(\alpha, \beta) = n^{-k + \cyc(\alpha, \beta)} \left(\mu(\alpha, \beta) + O(n^{-1})\right)\,. \]
\end{proposition}
\noindent
Note that the maximum possible scaling of this quantity is $n^{-k/2}$, which corresponds to the fact that with high probability the entries of $\bQ$ are all roughly of order $n^{-1/2}$.

\subsection{M{\"o}bius inversion on non-crossing partitions}

Recall that $\NC(k)$ is the partially ordered set of non-crossing partitions, i.e., those whose parts do not cross when drawn as a partition of vertices of the $k$-cycle.
We review some standard properties of this partially ordered set; see, e.g., \cite{NS-2006-LecturesCombinatoricsFreeProbability} for a standard reference.

Each non-crossing partition $\pi \in \NC(k)$ has a natural dual partition, called the \emph{Kreweras complement} and denoted $K(\pi)$.
On the cycle graph $C_k$, this may be viewed as the maximal non-crossing partition of the midpoints of the \emph{edges} of $C_k$ that does not cross the boundaries of $\pi$.
Alternatively, one may view both partitions as placed on a single cycle graph of twice the size, $C_{2k}$, on alternating sets of vertices.
We show this viewpoint with an example in~\cref{fig:kreweras}.
The map $K: \NC(k) \to \NC(k)$ is easily checked to be an involution.

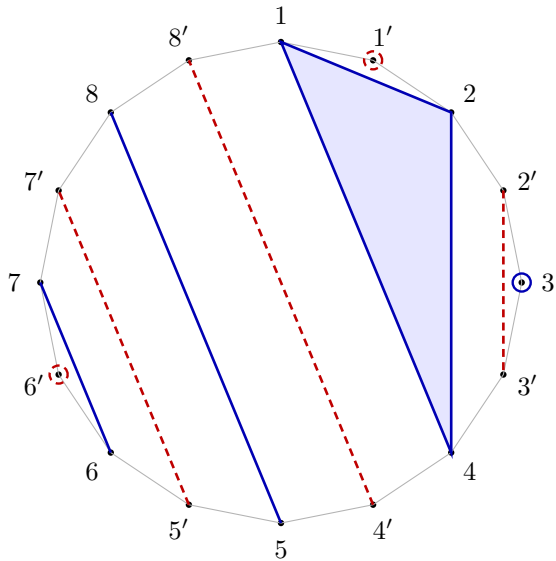
\begin{figure}[!htbp]
\begin{center}
\begin{tikzpicture}[scale=1, every node/.style={font=\small}]
  \def\R{3.2}
  \def\Rlab{3.55}

  \foreach \name/\ang/\lab in {
    u1/90/{$1$},   p1/67.5/{$1'$},
    u2/45/{$2$},   p2/22.5/{$2'$},
    u3/0/{$3$},    p3/-22.5/{$3'$},
    u4/-45/{$4$},  p4/-67.5/{$4'$},
    u5/-90/{$5$},  p5/-112.5/{$5'$},
    u6/-135/{$6$}, p6/-157.5/{$6'$},
    u7/180/{$7$},  p7/157.5/{$7'$},
    u8/135/{$8$},  p8/112.5/{$8'$}
  }{
    \coordinate (\name) at (\ang:\R);
    \fill (\name) circle (1.2pt);
    \node at (\ang:\Rlab) {\lab};
  }

  \draw[gray!60]
    (u1)--(p1)--(u2)--(p2)--(u3)--(p3)--(u4)--(p4)--
    (u5)--(p5)--(u6)--(p6)--(u7)--(p7)--(u8)--(p8)--cycle;

  \filldraw[fill=blue!10, draw=blue!70!black, line width=1pt]
    (u1)--(u2)--(u4)--cycle;
  \draw[blue!70!black, line width=1pt] (u5)--(u8);
  \draw[blue!70!black, line width=1pt] (u6)--(u7);
  \draw[blue!70!black, line width=0.9pt] (u3) circle[radius=0.12];

  \draw[red!75!black, densely dashed, line width=1pt] (p2)--(p3);
  \draw[red!75!black, densely dashed, line width=1pt] (p4)--(p8);
  \draw[red!75!black, densely dashed, line width=1pt] (p5)--(p7);
  \draw[red!75!black, densely dashed, line width=0.9pt] (p1) circle[radius=0.12];
  \draw[red!75!black, densely dashed, line width=0.9pt] (p6) circle[radius=0.12];
\end{tikzpicture}
\end{center}
\caption{An illustration of the Kreweras complement operation on non-crossing partitions. The parts of a partition $\pi \in \NC(8)$ are drawn in blue, and the parts of the Kreweras complement $K(\pi) \in \NC(8)$ in red.}
\label{fig:kreweras}
\end{figure}

We give $\NC(k)$ the usual partial ordering of refinement of partitions, written $\pi \psdleq \rho$, using that a refinement of a non-crossing partition remains non-crossing.
This partial ordering has a minimal element $\underline{0} \in \NC(k)$, the partition where every block is a singleton, and a maximal element $\underline{1} \in \NC(k)$, the partition with just one block.
The Kreweras complement is an \emph{anti-isomorphism} of this ordering: it is a bijection that reverses the ordering, i.e.\ $K(\pi) \preceq K(\rho)$ if and only if $\pi \succeq \rho$.
In particular, $K(\underline{0}) = \underline{1}$ and $K(\underline{1}) = \underline{0}$.

The \Mobius\ function for the $\NC(k)$ poset gives values $\mu(\pi, \rho)$ for each pair $\pi \preceq \rho$.
The Kreweras complement interacts with the \Mobius\ function in the following way that will be crucial for our purposes:
\begin{equation}\label{eq:kreweras-mobius}
\mu(\underline{0}, \pi) = \mu(K(\pi), \underline{1})\,. 
\end{equation}
Further, evaluations of the \Mobius\ function as on the left-hand side may be expanded as products over the blocks of $\pi$, and the factors turn out to be the same as the combinatorial quantities appearing in~\cref{eq:cycle-mobius}; there is a combinatorial explanation for this coincidence but we will just need to use that this indeed occurs:
\[ \mu(\underline{0}, \pi) = \prod_{A \in \pi} (-1)^{|A| - 1} \Cat(|A| - 1)\,. \]

Note that, applying \Mobius inversion to \cref{eq:free-cumulant}, we obtain an explicit formula for the free cumulants in terms of the moments, as mentioned earlier in the main text: if $m_k$ are the moments of a probability measure, then the free cumulants $\kappa_k$ are
\begin{equation} \label{eq:free-cumulant-explicit}
\kappa_k = \sum_{\pi \in \NC(k)} \mu(\pi, \underline{1}_k) \prod_{A \in \pi} m_{|A|}\,. 
\end{equation}

\subsection{Tracial moments concentration}

The result of~\cite{cebron2024traffic} also assumes the following formula for the joint moments of the various trace powers of a matrix, that we also use in our proof.
We show that it follows from our assumptions.

\begin{lemma}[Tracial moments concentration]\label{lem:factorization}
    Let $\bA = \bA^{(n)} \in \R^{n \times n}_\sym$ be  random  matrices that converge in tracial moments in $L^2$ to some $\mu$. Then for any cycle diagrams
    $\rho_1, \ldots, \rho_k$,
    \[
        \lim_{n\to\infty} \E\left[\prod_{j=1}^k \frac 1n w_{\rho_j}(\bmA)\right] = \prod_{j=1}^k \lim_{n\to\infty}  \E \frac 1n w_{\rho_j}(\bmA)\,.
    \]
\end{lemma}

\begin{proof}
    Let us write $T_q = T_q^{(n)} \defeq \frac 1n \Tr \bm A^q$. For
    any finite multiset of integers $\mathcal Q$, we can expand
    \[
        \E \left[\prod_{q\in \calQ} T_{q}\right] - \prod_{q\in \calQ} \E T_{q} = \sum_{\varnothing \neq \calQ'\subseteq \calQ} \E\left[\prod_{q\in \calQ'} (T_{q} - \E T_{q})\right] \prod_{q\in \calQ\setminus \calQ'} \E T_{q}\,.
    \]
    Our goal is now to show that each term in the
    sum over $\mathcal Q'$ converges to $0$ as $n\to\infty$. Fix $\mathcal Q'\subseteq \mathcal Q$ such that $\mathcal Q'\neq\varnothing$, and select an arbitrary element $q_0\in \mathcal Q'$. By Cauchy-Schwarz, we have
    \begin{equation}
        \left(\E\left[\prod_{q\in \calQ'} (T_{q} - \E T_{q})\right]\right)^2 \le \E (T_{q_0} - \E T_{q_0})^2 \cdot \E \left[\prod_{q\in \calQ'\setminus \{q_0\}} (T_q - \E T_q)^2\right]\,.\label{eq:tracialBound}
    \end{equation}
    We know that $\E (T_{q_0} - \E T_{q_0})^2$
    converges to $0$ as $n\to\infty$
    by the $L^2$ tracial moments convergence assumption.  For the remaining product of expectations from~\cref{eq:tracialBound}, we apply
    the bound $T_q^2\le T_{2p}^{q/p}$ for all $q\le p$ to get: for all $\mathcal Q''\subseteq \mathcal Q'$,
    \[
        \prod_{q\in \mathcal Q''} T_q^2\le T_{2\sum_{q\in \mathcal Q''} q}\,.
    \]
    Therefore, all terms in the expansion of~\cref{eq:tracialBound} can be bounded
    by products of terms of the form $\E T_q$ for
    $q\in\N$. These are all bounded as $n\to\infty$, since convergence in $L^2$ 
    also implies convergence in expectation.
    Together, we deduce
    \[
        \lim_{n\to\infty} \E \left[\prod_{q\in \calQ} T_{q}\right] = \prod_{q\in \calQ} \lim_{n\to\infty}\E T_{q}\,,
    \]
    which is equivalent to the desired statement.
\end{proof}

\begin{remark}
    \label{rem:cdm-factorization}
    This property is a statement about concentration of the tracial moments.
    For an example where it does not hold, one can take $\bA^{(n)} = a\bm I_n$ for $a \sim \Unif(\{\pm 1\})$, in which case $\lim_{n \to \infty}\EE \frac{1}{n}\Tr(\bA) = \lim_{n \to \infty} \EE \frac{1}{n}\Tr(\bA^3) = 0$, while $\lim_{n \to \infty} \EE [\frac{1}{n}\Tr(\bA) \cdot \frac{1}{n}\Tr(\bA^3)] = 1$.
\end{remark}
\noindent
We will further show below that analogous formulas hold for joint moments of elements of the $w$- and $z$-bases of polynomials, not just the cycle diagrams.

\subsection{Traffic distribution of orthogonally invariant matrices}

We now prove \cref{thm:moments} by computing the traffic distribution of an orthogonally invariant matrix $\bA$, which we recall consists of the limits of expressions of the form $\frac{1}{n}\EE z_{\alpha}(\bA)$ for $\alpha \in \cA_0$.

First, for a graph $\alpha = (V(\alpha), E(\alpha))$, define $\mathrm{HE} = \mathrm{HE}(\alpha)$ to be the set of \emph{half-edges} in $\alpha$, a set of size $|\mathrm{HE}| = 2|E|$ which may be identified with pairs $(v, \{v, w\})$ for each choice of $v \in V$ and $\{v, w\} \in E(\alpha)$.
Then, to $\alpha$ itself is associated a distinguished perfect matching $\widetilde{\alpha} \in \PM(\mathrm{HE})$, which matches each pair $(v, \{v, w\})$ and $(w, \{v, w\})$ of half-edges that correspond to the same edge of $\alpha$ (this is the perfect matching that would realize $\alpha$ under the configuration model).

We say that a matching $\beta \in \calM(\mathrm{HE})$ is \emph{$\alpha$-local} if all of its matches are between half-edges of the form $(v, e_1)$, $(v, e_2)$, i.e., between pairs of half-edges associated to the same \emph{vertex} (rather than the same \emph{edge}) for $\widetilde{\alpha}$.
Let $\Loc(\alpha) \subseteq \PM(\mathrm{HE}(\alpha))$ be the set of all $\alpha$-local matchings.
Note that $\Loc(\alpha) \neq \varnothing$ if and only if $\alpha$ is Eulerian, i.e., if every vertex has even degree.

At the heart of the matter is the distance between $\widetilde{\alpha}$ and the set $\Loc(\alpha)$, which is minimized precisely by the cactus graphs $\alpha\in\calC$:
\begin{proposition}
    \label{prop:cactus-loc-dist}
    For any graph $\alpha$, not necessarily connected, all of whose connected components are Eulerian, we have
    \[ \Delta(\widetilde{\alpha}, \Loc(\alpha)) = \min_{\beta \in \Loc(\alpha)} \Delta(\widetilde{\alpha}, \beta) \geq |V(\alpha)| - |\conn(\alpha)|\,, \]
    with equality if and only if every connected component of $\alpha$ is a cactus.
    Further, in that case, there is a unique $\beta \in \Loc(\alpha)$ achieving equality, which is the (unique) such $\beta$ that matches pairs of half-edges belonging to the same cycle in $\alpha$.
\end{proposition}
\begin{proof}
    It suffices to consider $\alpha$ connected; the general case follows by considering each connected component separately.

    We may rewrite
    \begin{align*}
    \Delta(\widetilde{\alpha}, \Loc(\alpha)) 
    &= \min_{\beta \in \Loc(\alpha)} \Delta(\widetilde{\alpha}, \beta) = |E| - \max_{\beta \in \Loc(\alpha)} |\cyc(\widetilde{\alpha}, \beta)|
    \end{align*}
    and therefore it suffices to show that, for all $\alpha$-local matchings of half-edges $\beta$, we have
    \[ |\cyc(\widetilde{\alpha}, \beta)| \stackrel{\text{(?)}}{\leq} |E| - |V| + 1\,. \]
    The set of cycles in the disjoint union of $\widetilde{\alpha}$ and an $\alpha$-local $\beta$ is equivalently the number of cycles in a cycle cover of $\alpha$ (i.e., a partition of its edges into cycles).

    The bound is tight for cycles.
    Suppose $C_1, \dots, C_k$ is a cycle cover of some connected multigraph $\alpha$.
    Since $\alpha$ is connected, it is possible to order the $C_i$ such that $C_{i + 1}$ has a vertex in common with the union of $C_1, \dots, C_i$ for each $i = 1, \dots, k - 1$.
    Adding each successive $C_i$ then increases $|E| - |V| + 1$ by at least 1, so the bound follows.
    If the bound is tight, then in the above ordering $C_{i + 1}$ must have exactly one vertex in common with the union of $C_1, \dots, C_i$, and thus is a cactus.
    In that case, there is only one cycle cover, and thus the minimizer $\beta$ is unique and must be as specified in the statement.
\end{proof}

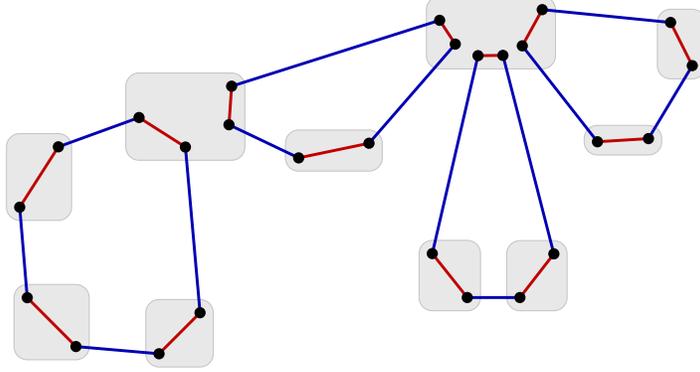
\begin{figure}
\begin{center}
\begin{tikzpicture}[
  x=1cm,y=1cm,
  halfedge/.style={circle,fill=black,inner sep=1.5pt},
  alphamatch/.style={draw=blue!70!black,line width=1.1pt},
  betamatch/.style={draw=red!75!black,line width=1.1pt},
  blob/.style={fill=gray!18,draw=gray!45,rounded corners=5pt,inner sep=5pt}
]

\coordinate (c)  at ( 0.0,  3.0);
\coordinate (a1) at (-4.1,  1.7);
\coordinate (b1) at (-2.0,  0.7);
\coordinate (a2) at (-0.9, -0.9);
\coordinate (b2) at ( 1.0, -0.9);
\coordinate (a3) at ( 3.0,  2.7);
\coordinate (b3) at ( 1.8,  0.7);
\coordinate (x)  at (-6.3,  0.9);
\coordinate (y)  at (-6.1, -1.5);
\coordinate (z)  at (-3.8, -1.7);

\coordinate (hcAone)  at ($(c)!7mm!(a1)$);
\coordinate (hcBone)  at ($(c)!7mm!(b1)$);
\coordinate (haonec)  at ($(a1)!7mm!(c)$);
\coordinate (haoneb)  at ($(a1)!7mm!(b1)$);
\coordinate (hb1a1)   at ($(b1)!6mm!(a1)$);
\coordinate (hb1c)    at ($(b1)!6mm!(c)$);

\coordinate (hcAtwo)  at ($(c)!7mm!(a2)$);
\coordinate (hcBtwo)  at ($(c)!7mm!(b2)$);
\coordinate (hatwoc)  at ($(a2)!6mm!(c)$);
\coordinate (hatwob)  at ($(a2)!6mm!(b2)$);
\coordinate (hb2a2)   at ($(b2)!6mm!(a2)$);
\coordinate (hb2c)    at ($(b2)!6mm!(c)$);

\coordinate (hcAthree) at ($(c)!7mm!(a3)$);
\coordinate (hcBthree) at ($(c)!7mm!(b3)$);
\coordinate (hathreec) at ($(a3)!6mm!(c)$);
\coordinate (hathreeb) at ($(a3)!6mm!(b3)$);
\coordinate (hb3a3)    at ($(b3)!6mm!(a3)$);
\coordinate (hb3c)     at ($(b3)!6mm!(c)$);

\coordinate (haonex)  at ($(a1)!6mm!(x)$);
\coordinate (haonez)  at ($(a1)!6mm!(z)$);
\coordinate (hxa1)    at ($(x)!6mm!(a1)$);
\coordinate (hxy)     at ($(x)!6mm!(y)$);
\coordinate (hyx)     at ($(y)!6mm!(x)$);
\coordinate (hyz)     at ($(y)!6mm!(z)$);
\coordinate (hzy)     at ($(z)!6mm!(y)$);
\coordinate (hza1)    at ($(z)!6mm!(a1)$);

\begin{scope}[on background layer]
  \node[blob,fit=(hcAone)(hcBone)(hcAtwo)(hcBtwo)(hcAthree)(hcBthree)] {};
  \node[blob,fit=(haonec)(haoneb)(haonex)(haonez)] {};
  \node[blob,fit=(hb1a1)(hb1c)] {};
  \node[blob,fit=(hatwoc)(hatwob)] {};
  \node[blob,fit=(hb2a2)(hb2c)] {};
  \node[blob,fit=(hathreec)(hathreeb)] {};
  \node[blob,fit=(hb3a3)(hb3c)] {};
  \node[blob,fit=(hxa1)(hxy)] {};
  \node[blob,fit=(hyx)(hyz)] {};
  \node[blob,fit=(hzy)(hza1)] {};
\end{scope}

\draw[alphamatch] (hcAone)   -- (haonec);
\draw[alphamatch] (haoneb)   -- (hb1a1);
\draw[alphamatch] (hb1c)     -- (hcBone);

\draw[alphamatch] (hcAtwo)   -- (hatwoc);
\draw[alphamatch] (hatwob)   -- (hb2a2);
\draw[alphamatch] (hb2c)     -- (hcBtwo);

\draw[alphamatch] (hcAthree) -- (hathreec);
\draw[alphamatch] (hathreeb) -- (hb3a3);
\draw[alphamatch] (hb3c)     -- (hcBthree);

\draw[alphamatch] (haonex)   -- (hxa1);
\draw[alphamatch] (hxy)      -- (hyx);
\draw[alphamatch] (hyz)      -- (hzy);
\draw[alphamatch] (hza1)     -- (haonez);

\draw[betamatch] (hcAone)    -- (hcBone);
\draw[betamatch] (hcAtwo)    -- (hcBtwo);
\draw[betamatch] (hcAthree)  -- (hcBthree);

\draw[betamatch] (haonec)    -- (haoneb);
\draw[betamatch] (haonex)    -- (haonez);

\draw[betamatch] (hb1a1)     -- (hb1c);
\draw[betamatch] (hatwoc)    -- (hatwob);
\draw[betamatch] (hb2a2)     -- (hb2c);
\draw[betamatch] (hathreec)  -- (hathreeb);
\draw[betamatch] (hb3a3)     -- (hb3c);
\draw[betamatch] (hxa1)      -- (hxy);
\draw[betamatch] (hyx)       -- (hyz);
\draw[betamatch] (hzy)       -- (hza1);

\foreach \p in {
  hcAone,hcBone,haonec,haoneb,hb1a1,hb1c,
  hcAtwo,hcBtwo,hatwoc,hatwob,hb2a2,hb2c,
  hcAthree,hcBthree,hathreec,hathreeb,hb3a3,hb3c,
  haonex,haonez,hxa1,hxy,hyx,hyz,hzy,hza1}
  \node[halfedge] at (\p) {};
\end{tikzpicture}
\end{center}
\caption{An illustration of the matchings involved in~\cref{prop:cactus-loc-dist} and the arguments afterwards. Gray regions represent vertices of a cactus graph $\alpha$, three triangles joined at a vertex with one 4-cycle attached to one of those triangles at a different vertex. This graph has $|V| = 10$ and $|E| = 13$. Black dots in those regions represent half-edges of $\alpha$. In blue, we draw the matching $\widetilde{\alpha}$ realizing the graph $\alpha$; if the gray regions are each contracted to a point and only blue edges are retained, then the resulting graph is the cactus $\alpha$. In red, we draw the unique $\alpha$-local matching $\beta$ that maximizes $|\cyc(\widetilde{\alpha}, \beta)| = |E| - |V| + 1 = 4$; its being $\alpha$-local corresponds to making matches only within the gray regions.}
\label{fig:matchings}
\end{figure}

We now proceed to \Cref{thm:moments} by calculating the traffic distribution of a sequence of orthogonally invariant random matrices $\bA = \bA^{(n)}$.
We could view this as $\bA = \bQ\bD\bQ^{\top}$ for a Haar-distributed orthogonal $\bQ$ and some random diagonal $\bD$, but actually this will not be necessary.
Instead, let us take a perspective similar to the calculations in, for instance, \cite{KMW-2024-TensorCumulantsInvariantInference}, which we believe is useful in general.
Our idea will be to average the $z_{\alpha}(\bA)$ over a random rotation $\bQ$ drawn independently of $\bA$.
Regardless of the structure of $\bA$, this defines another family of polynomials:
\[ \bar{z}_{\alpha}(\bA) \colonequals \E_{\bQ} z_{\alpha}(\bQ\bA\bQ^{\top})\,. \]
If $\bQ$ is drawn from Haar measure, then the $\bar{z}_{\alpha}$ will be orthogonally invariant polynomials, a greater symmetry than permutation invariance of $z_{\alpha}$.
In particular, since the invariants of matrices under the $O(n)$ action are generated by traces of matrix powers, $\bar{z}_{\alpha}$ will be a polynomial in these.

\begin{proof}[Proof of \Cref{thm:moments}]
    Let $\bQ \in O(n)$ be Haar-distributed and independent of $\bA$, and let $\alpha = (V, E)$ be a graph.
    As above, write $\mathrm{HE} = \mathrm{HE}(\alpha)$ for the set of half-edges.
    We start by directly expanding the averaged polynomial $\bar{z}$ introduced above:
\begin{align*}
    &\bar{z}_{\alpha}(\bA) \\
    &= \E_{\bQ} z_{\alpha}(\bQ\bA\bQ^{\top}) \\
    &= \E_{\bQ} \sum_{i: V \into [n]} \prod_{\{v, w\} \in E} (\bQ\bA\bQ^{\top})[i(v),i(w)] \\
    &= \E_{\bQ} \sum_{i: V \into [n]} \prod_{\{v, w\} \in E} \left(\sum_{j_1, j_2 = 1}^n \bQ[i(v),j_1] \bA[j_1,j_2] \bQ[i(w),j_2]\right) \\
    &= \E_{\bQ} \sum_{\substack{i: V \into [n] \\ j: \mathrm{HE} \to [n]}} \prod_{\{v, w\} \in E} \bQ[i(v), j(v, \{v, w\})] \bQ[i(w), j(w, \{v, w\})] \bA[j(v, \{v, w\}), j(w, \{v, w\})] \\ 
    &= \sum_{\substack{i: V \into [n] \\ j: \mathrm{HE} \to [n]}} \left( \E_{\bQ} \prod_{\{v, w\} \in E} \bQ[i(v), j(v, \{v, w\})] \bQ[i(w), j(w, \{v, w\})] \right) \\ &\hspace{2.35cm} \cdot \prod_{\{v, w\} \in E}\bA[j(v, \{v, w\}), j(w, \{v, w\})] 
    \intertext{Here, we may use the Weingarten calculus, viewing the matchings involved as matchings of half-edges, provided that we view $i: V \to [n]$ as extended to $i^{\prime}: \mathrm{HE} \to [n]$ by $i'(v, e) \colonequals i(v)$, i.e., labelling a half-edge by the vertex involved. This gives:}
    &= \sum_{\substack{i: V \into [n] \\ j: \mathrm{HE} \to [n]}} \sum_{\beta, \gamma \in \PM(\mathrm{HE})} W(\beta, \gamma) \delta_{\beta}(i^{\prime}) \delta_{\gamma}(j) \prod_{\{v, w\} \in E}\bA[j(v, \{v, w\}), j(w, \{v, w\})] \\
    &= \sum_{j: \mathrm{HE} \to [n]} \sum_{\beta, \gamma \in \PM(\mathrm{HE})} \left(\sum_{i: V \into [n]} \delta_{\beta}(i^{\prime})\right) W(\beta, \gamma) \delta_{\gamma}(j) \prod_{\{v, w\} \in E}\bA[j(v, \{v, w\}), j(w, \{v, w\})]
    \intertext{The summation over $i$ is zero unless $\beta \in \Loc(\alpha)$, and in that case each choice of $i$ contributes 1, for a total of $n^{|V|}(1 + O(n^{-1}))$. So, we have}
    &= \left(1 + O\left(\frac{1}{n}\right)\right)n^{|V|} \sum_{\gamma \in \PM(\mathrm{HE})} \left(\sum_{\beta \in \Loc(\alpha)} W(\beta, \gamma)\right) \\
    &\hspace{3.95cm}\sum_{j:\textnormal{ HE} \to [n]} \delta_{\gamma}(j) \prod_{\{v, w\} \in E}\bA[j(v, \{v, w\}), j(w, \{v, w\})]
    \intertext{The remaining summation may be grouped into summations over the cycles in the disjoint union of $\gamma$ and $\widetilde{\alpha}$, which gives}
    &= \left(1 + O\left(\frac{1}{n}\right)\right)n^{|V|} \sum_{\gamma \in \PM(\mathrm{HE})} \left(\sum_{\beta \in \Loc(\alpha)} W(\beta, \gamma)\right) \prod_{C \in \cyc(\widetilde{\alpha}, \gamma)} \Tr(\bA^{\frac{|C|}{2}})
    \intertext{Now, we may use the asymptotic formula in \cref{lem:weirgarten-asymptotic} and normalize the traces to get}
    &= \left(1 + O\left(\frac{1}{n}\right)\right)n^{|V|} \sum_{\gamma \in \PM(\mathrm{HE})} \\
    &\hspace{1cm} \left(\sum_{\beta \in \Loc(\alpha)} \left(1 + O\left(\frac{1}{n}\right)\right)n^{-|\mathrm{HE}| + \cyc(\beta, \gamma)} \mu(\beta, \gamma)\right) \prod_{C \in \cyc(\widetilde{\alpha}, \gamma)} \Tr(\bA^{\frac{|C|}{2}}) \\
    &= \left(1 + O\left(\frac{1}{n}\right)\right)n^{|V|} \sum_{\gamma \in \PM(\mathrm{HE})} \\
    &\hspace{1cm} \left(\sum_{\beta \in \Loc(\alpha)} \left(1 + O\left(\frac{1}{n}\right)\right)n^{-|\mathrm{HE}| + \cyc(\beta, \gamma) + \cyc(\widetilde{\alpha}, \gamma)} \mu(\beta, \gamma)\right) \prod_{C \in \cyc(\widetilde{\alpha}, \gamma)} \frac{1}{n}\Tr(\bA^{\frac{|C|}{2}}) \\
    &= \left(1 + O\left(\frac{1}{n}\right)\right)n^{|V|} \sum_{\gamma \in \PM(\mathrm{HE})} \\
    &\hspace{1cm}\left(\sum_{\beta \in \Loc(\alpha)} \left(1 + O\left(\frac{1}{n}\right)\right)n^{-\Delta(\beta, \gamma) - \Delta(\widetilde{\alpha}, \gamma)} \mu(\beta, \gamma)\right) \prod_{C \in \cyc(\widetilde{\alpha}, \gamma)} \frac{1}{n}\Tr(\bA^{\frac{|C|}{2}})\,.
\end{align*}
Let us pause to notice that we have achieved our initial goal, expressing the orthogonally invariant polynomial $\bar{z}_{\alpha}(\bA)$ as a polynomial in traces of powers of $\bA$.
We now use that, if $\bA$ was orthogonally invariant to begin with, then
\[ \E_{\bA} z_{\alpha}(\bA) = \E_{\bA} \bar{z}_{\alpha}(\bA) \]
and continue to determine the right-hand side as $n \to \infty$.

By the triangle inequality, we have
\[ \Delta(\beta, \gamma) + \Delta(\widetilde{\alpha}, \gamma) \geq \Delta(\beta, \widetilde{\alpha}) \geq \Delta(\widetilde{\alpha}, \Loc(\alpha)) \geq |V| - 1\,. \]
Therefore, under our assumptions, all terms are negligible as $n \to \infty$ except for those where equality is achieved throughout above.

By~\Cref{prop:cactus-loc-dist}, we then find that if $\alpha$ is not a cactus then
\[ \lim_{n \to \infty} \frac{1}{n}\E_{\bA} z_{\alpha}(\bA) = 0\,. \]
So, suppose that $\alpha$ is a cactus.
Then, using the factorization property (\cref{lem:factorization}), we have in the limit that
\begin{align*}
    \lim_{n \to \infty} \frac{1}{n}\EE_{\bA} z_{\alpha}(\bA) 
    &= \sum_{\substack{\beta \in \Loc(\alpha) \\ \Delta(\beta, \widetilde{\alpha}) = |V| - 1}} \,\,\, \sum_{\substack{\gamma \in \PM(\mathrm{HE}) \\ \Delta(\beta, \gamma) + \Delta(\gamma, \widetilde{\alpha}) = \Delta(\beta, \widetilde{\alpha})}} \mu(\beta, \gamma) \prod_{C \in \cyc(\widetilde{\alpha}, \gamma)} m_{|C| / 2}
    \intertext{where $m_k$ are the spectral moments. Letting $\eta$ be the $\alpha$-local matching of half-edges belonging to the same cycle around each vertex, by the uniqueness clause of~\Cref{prop:cactus-loc-dist} we further have that only the term $\beta = \eta$ contributes, giving}
    &= \sum_{\substack{\gamma \in \PM(\mathrm{HE}) \\ \Delta(\eta, \gamma) + \Delta(\eta, \widetilde{\alpha}) = \Delta(\eta, \widetilde{\alpha})}} \mu(\eta, \gamma) \prod_{C \in \cyc(\widetilde{\alpha}, \gamma)} m_{|C| / 2}
    \intertext{Suppose there are $k$ cycles in $\alpha$. Then, $|\cyc(\eta, \widetilde{\alpha})| = k$, and, rewriting the condition on $\gamma$ in terms of cycle counts and using the explicit formula for the \Mobius\ function from~\cref{eq:cycle-mobius}, we have}
    &= \sum_{\substack{\gamma \in \PM(\mathrm{HE}) \\ |\cyc(\eta, \gamma)| + |\cyc(\gamma, \widetilde{\alpha})| = |E| + k}} \prod_{C \in \cyc(\eta, \gamma)} (-1)^{\frac{|C|}{2} - 1} \Cat\left(\frac{|C|}{2} - 1\right) \prod_{C \in \cyc(\widetilde{\alpha}, \gamma)} m_{|C| / 2}
    \intertext{Now, we use that all $\gamma$ appearing in the sum must only match half-edges belonging to the same cycle. Since $\eta$ and $\widetilde{\alpha}$ both have this property also, the various sets of cycles above all form partitions of the cycles in $\alpha$. Thus, the entire sum factorizes over the cycles of $\alpha$. Further, those $\gamma$ that are in the sum have both the partitions of $\cyc(\eta, \gamma)$ and $\cyc(\gamma, \widetilde{\alpha})$ corresponding to non-crossing partitions of each cycle of $\alpha$, and these two non-crossing partitions are Kreweras complements of one another.
    Putting together all these combinatorial observations, we find:}
    &= \prod_{C \in \cyc(\alpha)}\left(\sum_{\pi \in \mathrm{NC}(|C|)} \prod_{A \in K(\pi)} (-1)^{|A| - 1} \Cat(|A| - 1) \cdot \prod_{B \in \pi} m_{|B|}\right)\,.
    \intertext{Now we use \cref{eq:kreweras-mobius} and \cref{eq:free-cumulant-explicit}
    to complete the proof:}
    &= \prod_{C \in \cyc(\alpha)}\left(\sum_{\pi \in \mathrm{NC}(|C|)} \mu(\underline{0}_{|C|}, K(\pi)) \cdot \prod_{B \in \pi} m_{|B|}\right) \\
    &= \prod_{C \in \cyc(\alpha)}\left(\sum_{\pi \in \mathrm{NC}(|C|)} \mu(\pi, \underline{1}_{|C|}) \cdot \prod_{B \in \pi} m_{|B|}\right) \\
    &= \prod_{C \in \cyc(\alpha)} \kappa_{|C|}\,,
\end{align*}
where we have at last identified the free cumulants, completing the calculation.
\end{proof}

We also note that, by exactly the same argument but using the disconnected case of \Cref{prop:cactus-loc-dist}, we may equally well calculate suitably normalized limits of the values of \emph{disconnected} diagrams in the $z$-basis, which factorize over their connected components:
\begin{proposition} \label{prop:invariant-disconnected}
    Let $\bA = \bA^{(n)} \in \R^{n \times n}_{\sym}$ be a sequence of orthogonally invariant random matrices that converge in tracial moments in $L^2$ to a probability measure $\mu$.
    Let $\calD$ denote their limiting traffic distribution, which exists by \Cref{thm:moments} and is given by the explicit formula stated there.
    Then, for all $k \geq 1$ and $\alpha_1, \dots, \alpha_k \in \cA$,
    \[ \lim_{n \to \infty} \frac{1}{n^k} \E z_{\alpha_1 \sqcup \cdots \sqcup \alpha_k}(\bA) = \prod_{i = 1}^k \calD(\alpha_i)\,. \]
\end{proposition}

\subsection{Concentration of traffic observables}

As a corollary, we may also conclude that the traffic distribution is concentrated in the sense of \Cref{def:traffic-concentration}.
This also extends {\cite[Theorem 4.7]{cebron2024traffic}} to orthogonally invariant distributions.
\begin{lemma}
    \label{lemma:full-factorization}
    Let $\bmA=\bmA^{(n)}$ be orthogonally invariant random matrices that converge in tracial moments in $L^2$ to a probability measure $\mu$.
    Then the traffic distribution concentrates for $\bA$ (in the sense of \cref{def:traffic-concentration}).
\end{lemma}
\begin{proof}
    Let $k\ge 2$ and $\alpha_1,\ldots,\alpha_k\in \calA$. Then, by \Cref{lem:concentration-z}, it suffices to show the concentration property in the $z$-basis, namely that:
    \begin{align*}
        \lim_{n\to\infty} \E\left[\prod_{i=1}^k \frac 1n z_{\alpha_i}(\bmA)\right] 
        \stackrel{\text{(?)}}{=} \lim_{n\to\infty} \prod_{i=1}^k \E \frac 1n z_{\alpha_i}(\bmA)\,.
    \end{align*}

    Note that, upon expanding the summations in the $z$-basis polynomials, we have
    \[ z_{\alpha_1}(\bA) \cdots z_{\alpha_k}(\bA) = z_{\alpha_1 \sqcup \cdots \sqcup \alpha_k}(\bA) + z_{\beta_1}(\bA) + \cdots + z_{\beta_M}(\bA)\,, \]
    where $\alpha_1 \sqcup \cdots \sqcup \alpha_k$ is the disjoint union, while the $\beta_i$ are various graphs formed by identifying subsets of the vertices of this disjoint union according to different non-trivial partitions of the vertices, provided that no two vertices of the same $\alpha_j$ are identified.
    In particular, all $\beta_i$ have at most $k - 1$ connected components.
    Therefore, by \Cref{prop:invariant-disconnected}, we have
    \[ \lim_{n \to \infty} \frac{1}{n^k} z_{\beta_i}(\bA) = 0 \]
    for all $i \in [M]$.
    Thus, 
    \[ \lim_{n \to \infty} \frac{1}{n^k}\E\left[\prod_{i=1}^k z_{\alpha_i}(\bmA)\right] = \lim_{n \to \infty} \frac{1}{n^k}\E[z_{\alpha_1 \sqcup \cdots \sqcup \alpha_k}(\bA)]\,, \]
    and the result then follows by \Cref{prop:invariant-disconnected}.
\end{proof}

\subsection{Traffic distribution of punctured orthogonally invariant matrices}
\label{sec:puncturedWeingarten}

Since the \rrom plays an important role in our main results, let us sketch how similar calculations can give an explicit combinatorial description of its traffic distribution, and indeed that of the puncturing of any orthogonally invariant random matrices.
Recall that in the main text we relied entirely on the implicit description of this traffic distribution via \Cref{lem:eqDiffMobius}.
The closed form we give below is completely explicit, but, being in terms of a rather complicated summation over matchings, seems less useful than the implicit one.

We follow the notation from the proof in the previous section.
Additionally, for a graph $\alpha$ and a matching $\beta$ of the half-edges of $\alpha$, we write $\loc(\beta)$ for the set of edges of $\beta$ that go between half-edges of the same vertex of $\alpha$, and $\nonloc(\beta)$ for the set of edges of $\beta$ that go between half-edges of different vertices of $\alpha$.
Recall also that $\widetilde{\alpha}$ is the matching of half-edges of $\alpha$ corresponding to the edges actually in the graph $\alpha$.
\begin{theorem}
    Let $\bA = \bA^{(n)} \in \R^{n \times n}_\sym$ be a sequence of orthogonally invariant random matrices that converges in tracial moments in $L^2$ to a probability measure $\mu$.
    Write $m_k$ for the $k$th moment of $\mu$ and $\bm\Pi = \bm\Pi^{(n)} = \bm I - \frac{1}{n}\bm 1 \bm 1^{\top}$.
    Then, for all $\al \in \calA$,
    \[ \lim_{n \to \infty} \frac{1}{n}\EE_{\bA} z_{\alpha}(\bm \Pi\bA\bm \Pi) = \sum_{\substack{\beta \in \PM(\mathrm{HE}(\alpha)) \\ \alpha \sqcup \nonloc(\beta) \text{ is a cactus}}} (-1)^{|\nonloc(\beta)|} \sum_{\substack{\gamma \in \PM(\mathrm{HE}(\alpha)) \\ \Delta(\beta, \gamma) + \Delta(\gamma, \widetilde{\alpha}) = \Delta(\beta, \widetilde{\alpha})}} \mu(\beta, \gamma) \prod_{C \in \cyc(\widetilde{\alpha}, \gamma)} m_{|C|}\,. \]
\end{theorem}
\begin{proof}
Following the same calculations as in the proof of \Cref{thm:moments} above but now applied to $\bm\Pi \bQ\bA\bQ^{\top}\bm\Pi$, we instead find:
\begin{align*}
    &\frac{1}{n}\EE_{\bQ, \bA} z_{\alpha}(\bm\Pi\bQ\bA\bQ^{\top}\bm\Pi) \\
    &= \frac{1}{n}\EE_{\bA}\sum_{\beta, \gamma \in \PM(\mathrm{HE})} W_n(\beta, \gamma) \cdot \prod_{C \in \cyc(\widetilde{\alpha}, \gamma)} \Tr(\bA^{|C|}) \cdot z_{G(\beta)}(\bm\Pi)
    \intertext{where $G(\beta)$ denotes the graph formed by ``wiring together'' the matching of half-edges $\beta$ (so that, for example, $G(\widetilde{\alpha}) = \alpha$). Note that here if we replaced $\bm\Pi$ by $\bm I$, we would get $z_{G(\beta)}(\bm I) = \mathbf 1_{\beta \in \Loc(\alpha)} n^{|V|}(1+O(n^{-1}))$, compatible with the previous calculation in the proof of \Cref{thm:moments}, and indeed the above is true for an arbitrary symmetric matrix $\bm\Pi$, not only the particular projection we are concerned with. But, in our particular case, since $\bm\Pi$ is constant on the diagonal and on the off-diagonal, we have}
    &= \frac{1}{n}\sum_{\beta, \gamma \in \PM(\mathrm{HE})} W_n(\beta, \gamma) \cdot \EE_{\bA}\prod_{C \in \cyc(\widetilde{\alpha}, \gamma)} \Tr(\bA^{|C|}) \cdot n^{\underline{|V|}}\left(-\frac{1}{n}\right)^{|\nonloc(\beta)|} \left(1 - \frac{1}{n}\right)^{|\loc(\beta)|}
    \intertext{and now by the same asymptotics as before,}
    &= \sum_{\beta, \gamma \in \PM(\mathrm{HE})} 
    \left(1 + O\left(\frac{1}{n}\right)\right) n^{-\Delta(\widetilde{\alpha}, \gamma) - \Delta(\beta, \gamma) - |\nonloc(\beta)| + |V| - 1} \mu(\beta, \gamma) (-1)^{|\nonloc(\beta)|}
     \prod_{C \in \cyc(\widetilde{\alpha}, \gamma)} m_{|C|}
\end{align*}   

We claim that, for any connected $\alpha$ realized by the matching $\widetilde{\alpha}$ of its half-edges, and any other matching $\beta$ of the half-edges of $\alpha$, we have
\[ \Delta(\widetilde{\alpha}, \beta) + |\nonloc(\beta)| \geq |V| - 1\,. \]
As before, this is equivalent to having
\[ |\cyc(\widetilde{\alpha}, \beta)| \leq |E| + |\nonloc(\beta)| - |V| + 1\,. \]
Consider an ancillary graph $\alpha^{\prime}$ constructed by adding edges to $\alpha$ for each non-local match in $\beta$.
This graph is still connected, by parity considerations it must be Eulerian, and it has a total of $|E| + |\nonloc(\beta)|$ edges.
$|\cyc(\widetilde{\alpha}, \beta)|$ is now the size of a cycle cover of $\alpha^{\prime}$, and the claim then follows by the bounds from the proof of~\Cref{prop:cactus-loc-dist} applied to $\alpha^{\prime}$.

We also again have by the triangle inequality that
\[ \Delta(\widetilde{\alpha}, \gamma) + \Delta(\beta, \gamma) \geq \Delta(\widetilde{\alpha}, \beta)\,. \]
Thus, all terms in the sum above are of at most constant order.
Further, those of constant order are those where the exponent of $n$ is zero, which are those where the above bound is tight.
By the characterization in~\Cref{prop:cactus-loc-dist}, this is precisely when $\alpha^{\prime}$ as formed above is a cactus, and the stated result follows after rearranging.
\end{proof}

\section{Convergence of Stochastic Processes}
\label{sec:convergence-process}

In~\Cref{sec:dynamics}, we deal
with convergence in distribution of
stochastic processes indexed by countably
infinite set, intended as weak convergence in the
product topology. Equivalently, this
means that every finite-dimensional
marginal converges in distribution.

\begin{definition}
    Let $\calA$ be a countable
    set. For random variables $(\bx^{(n)})_{n\ge 1}$ and
    $\bx^\infty$ taking values in $\R^\calA$, we say
    that $\bx^{(n)}$ converges in distribution to
    $\bx^\infty$ and write
    \[
        \bx^{(n)}\tod \bx^\infty
    \]
    if, for every $k\ge 1$ and $\alpha_1, \ldots,\alpha_k\in \calA$, we have
    \[
        (x^{(n)}_{\alpha_1},\ldots,x^{(n)}_{\alpha_k})\tod (X^\infty_{\alpha_1},\ldots X^{\infty}_{\alpha_k})\,.
    \]
\end{definition}

\noindent To show convergence in distribution, we will use the method of moments~\cite[Theorems 29.4, 30.1, 30.2]{billingsleyProbabilityBook}.
The following theorem follows from
Carleman's conditions on 
moment-determinacy of a distribution on
$\R$, combined with~\cite{petersenEquivalence}.

\begin{theorem}[Method of moments]\label{cor:bounded-limit-existence}
    Let $(\bx^{(n)})_{n\ge 1}$ be a sequence
    of stochastic processes indexed by a
    countable set $\calA$. Assume
    that 
    \begin{enumerate}
        \item All joint moments converge: for any
        $k\ge 1$ and $\alpha_1, \ldots, \alpha_k\in \calA$,
        the limit of the joint moments
    \begin{align}
        \lim_{n\to\infty} \E \left[\prod_{i=1}^k x^{(n)}_{\alpha_i}\right]\label{eq:joint-moments}
    \end{align} 
    exists.
        \item All marginals are subexponential: for every $\alpha\in\calA$, there exists $C_\alpha>0$ such that
        for all $p\ge 1$,
        \begin{align}
            \lim_{n\to\infty} \E \left(x^{(n)}_\alpha\right)^{2p}\le (C_\alpha p)^{2p}\,.\label{eq:subexp-growth-condition}
        \end{align}
    \end{enumerate}
    Then $\bx^{(n)}$ converges in distribution to the unique law on $\R^{\cal A}$ with
    moments given by~\cref{eq:joint-moments}.
\end{theorem}

\begin{lemma}[Truncation]\label{lem:truncation}
    Let $(x_n)_{n\ge 1}$ and $(y_n)_{n\ge 1}$
    be sequences of random variables such that
    \begin{enumerate}
        \item For any $K>0$, conditionally
        on $|x_n|\le K$, $(y_n)_{n\ge 1}$ converges
        in distribution.
        \item $(x_n)_{n\ge 1}$ is tight, i.e., $\sup_{n\ge 1} \Pr(|x_n|>K)\underset{K\to\infty} \longrightarrow 0$.
    \end{enumerate}
    Then, $(y_n)_{n\ge 1}$ converges in distribution.
\end{lemma}

\begin{proof}
    First, we prove:

    \begin{claim}\label{claim:y-tight}
        $(y_n)_{n\ge 1}$ is tight.
    \end{claim} 

    \begin{proof}
        For any $K,L>0$, we have $\Pr(|y_n|>L)\le \Pr(|y_n|>L \mid |x_n|\le K) + \Pr(|x_n|> K)$. Pick $K$ large enough so that the second term is bounded by $\eps$ uniformly in $n$. $(y_n)_{n\ge 1}$ is tight conditionally on $|x_n|\le K$, so there exists $L>0$ large enough so 
        that the first term is also bounded by $\varepsilon$ uniformly in $n$.
    \end{proof}

    \noindent By~\Cref{claim:y-tight} and Prokhorov's
    theorem, it remains to show that
    every subsequence of $(y_n)_{n\ge 1}$ that converges in distribution, converges
    to the same limit.
    Fix $f:\R\to \R$ to be a bounded continuous
    function and $\eps>0$. Then, by the law of total
    expectations, for any $n\ge 1$,
    \begin{align*}
        \left|\E f(y_n) - \E \left[f(y_n) \mid |x_n|\le K\right]\right| &=  \Pr(|x_n|> K)\left(\E \left[f(y_n) \mid |x_n|> K\right] - \E \left[f(y_n) \mid |x_n|\le K\right]\right)\\
        &\le 2 \|f\|_\infty \Pr(|x_n|> K)\\
        &\le \varepsilon
    \end{align*}
    by setting $K = K(\eps)$ to be a large enough constant (with the second assumption). By the first assumption, there exists $N\ge 1$ such that
    for any $n,m\ge N$,
    \[
        \left|\E \left[f(y_n) \mid |x_n|\le K\right] - \E \left[f(y_m) \mid |x_m|\le K\right]\right|\le \eps\,.
    \]
    In turn, this implies $\left|\E f(y_n) - \E f(y_m)\right|\le 3\eps$ by the triangle inequality, so $(\E f(y_n))_{n\ge 1}$
    is a Cauchy sequence, so it converges
    as $n\to\infty$. This implies
    that every weak subsequential limit of
    $(y_n)_{n\ge 1}$ converges to the same limit,
    which concludes the proof.
\end{proof}

\subsection{Connection with convergence of the empirical distribution}

Let us also remark on certain details concerning modes of convergence that are important to the use and interpretation of~\Cref{thm:convergence-distribution}.

Recall that we ``stack'' the $\bz_{\alpha}(\bA)$ for $\alpha\in\calA_1$ into a single vector with more complicated entries, $\bz_{\cA_1}(\bA) \in (\R^{\cA_1})^n$.
Using our notation from~\Cref{sec:intro}, we then sample a random coordinate of this vector, forming a further random countably infinite vector $\samp(\bz_{\cA_1}(\bA)) \in \R^{\cA_1}$.
This contains the $i$th entry of each $\bz_{\alpha}(\bA)$, for a single shared randomly chosen $i \sim \Unif([n])$.
Define the infinite random vector $Z_{\cA_1}^{\infty}$ similarly. \Cref{thm:convergence-distribution} states that:
\begin{equation}
\samp(\bz_{\cA_1}(\bA^{(n)})) \xrightarrow[n \to \infty]{\text{(d)}} Z_{\cA_1}^{\infty}\,,
\label{eq:samp-conv-d}
\end{equation}
By the Cram\'{e}r-Wold theorem, this is equivalent to: for any bounded continuous function $\varphi$ and any finitely supported vector of coefficients $c_{\alpha}$,
\[ \lim_{n \to \infty} \E_{\bA} \frac{1}{n}\sum_{i = 1}^n \varphi\left(\sum_{\alpha \in \cA_1} c_{\alpha} \bz_{\alpha}(\bA)[i]\right) = \E_{\bA} \varphi\left(\sum_{\alpha \in \cA_1} c_{\alpha} Z_{\alpha}^{\infty} \right). \]

Alternatively, we may also make sense of this statement in terms of empirical distributions, which are just the laws of the random variables $\samp(\bx)$ discussed above.
\begin{definition}[Empirical distribution]
    For $\bm x \in \R^n$, we write $\ed(\bm x) \colonequals \frac{1}{n} \sum_{i = 1}^n \delta_{\bmx[i]}$ for the \emph{empirical distribution} of the entries of $\bm x$.
\end{definition}
\noindent
Then, $\ed(\bz_{\cA_1}(\bA))$ is a random probability measure on the space $\R^{\cA_1}$, and the random variable $\samp(\bz_{\cA_1}(\bA^{(n)}))$ is a single draw from this random probability measure.
Its law is a \emph{deterministic} probability measure on the space $\R^{\cA_1}$, which is the expectation of the random measure $\ed(\bz_{\cA_1}(\bA))$ (if $\mu$ is a random measure, then its expectation takes values $(\EE \mu)(A) = \EE[\mu(A)]$).
Thus, the above~\cref{eq:samp-conv-d} is further equivalent to the weak convergence of probability measures
\begin{equation*}
\EE \ed(\bz_{\cA_1}(\bA^{(n)})) \xrightarrow[n \to \infty]{\text{(w)}} \Law(Z_{\cA_1}^{\infty})\,.
\end{equation*}
Again by the Cram\'{e}r-Wold theorem, this is equivalent to, for any finitely supported coefficient vector of $c_{\alpha}$, having
\[ \EE \ed\left(\sum_{\alpha \in \cA_1} c_{\alpha} \bz_{\alpha}(\bA^{(n)})\right) \xrightarrow[n \to \infty]{\text{(w)}} \Law\left(\sum_{\alpha \in \cA_1} c_{\alpha} Z_{\alpha}^{\infty}\right). \]
In particular, since the output $\bx_t$ of a GFOM can be viewed in the above way, we see that the empirical distributions of $\bx_t = \bx_t(\bA)$ are related to the asymptotic states $X_t^{\infty}$ by
\[ \EE \ed(\bx_t(\bA^{(n)})) \xrightarrow[n \to \infty]{\text{(w)}} \Law(X_t^{\infty})\,. \]

Thus our results, interpreted in terms of convergence of the random empirical distributions of GFOM iterates, give convergence of the expectations of random measures.
Often it is desirable to prove stronger modes of convergence in such situations, by proving that not only do we have
\[ \lim_{n \to \infty} \E_{\bA} \frac{1}{n} \sum_{i = 1}^n \varphi(\bx_t(\bA^{(n)})[i]) = \EE \varphi(X_t^{\infty})\,, \]
but also that the random variable inside the expectation \emph{concentrates} over the randomness in $\bA$.
We do not pursue this here, because it would require introducing additional assumptions on the matrices $\bA$ involved, which may vary from application to application.
As the example discussed in~\Cref{rem:cdm-factorization} shows, this kind of concentration does not follow automatically from the convergence in expectation that we show.
An instructive example is the argument in \cite{bayati2015universality}, which uses similar proof techniques to ours, but, to show that the above kind of convergence also happens in $L^2$ uses a trick involving the entrywise independence of the Wigner matrices they work with (see their Proposition~5).

In our much more general setting, it seems reasonable to ask instead for the convergence in the definition of the traffic distribution in~\cref{eq:traffic-conv} to happen in a stronger mode such as $L^2$.
We leave the exploration of such conditions and the determination of which random matrix distributions they hold for to future work.

\section{Omitted Proofs}
\label{sec:omitted}

\subsection{Combinatorial lemmas}
\label{sec:combinatorics}

We gather here lemmas involving only
graph combinatorics.

\begin{lemma}
    \label{lem:prod-cactus-cactus}
    For all $\sigma,\sigma'\in \calC_1$ and $\bA\in\R_\sym^{n\times n}$,
    \[
        \bz_{\sigma}(\bA) \cdot \bz_{\sigma'}(\bA) - \bz_{\sigma\oplus \sigma'}(\bA)\in \spn(\bz_{\calA_1\setminus \calC_1}(\bA))\,,
    \]
    where $\sigma\oplus \sigma'\in\calC_1$ is the grafting
    of $\sigma$ and $\sigma'$ at the root.
\end{lemma}

\begin{proof}
    In the $z$-basis expansion of
    $\bz_\sigma(\bA)\cdot \bz_{\sigma'}(\bA)$, we sum over all possible partial
    matchings of the vertices of $\sigma$
    and $\sigma'$. The empty matching 
    contributes exactly $z_{\sigma\oplus \sigma'}(\bA)$.
    Any other matching that merges some
    vertices $u\in V(\sigma)$ and $v\in V(\sigma')$ creates 4 edge-disjoint paths
    between the root and the merged vertex.
    Merging additional vertices of $\sigma$ and
    $\sigma'$ can only increase the
    number of edge-disjoint paths, so the
    resulting graphs cannot be cactuses.
\end{proof}

\begin{lemma}
    \label{lem:prod-cactus-other}
    For all $\sigma\in\calC_1$, $\alpha\in \calA_1\setminus \calC_1$ and $\bA\in\R_\sym^{n\times n}$,
    \[
        \bz_{\sigma}(\bA) \cdot \bz_{\alpha}(\bA) \in \spn(\bz_{\calA_1\setminus \calC_1}(\bA))\,.
    \]
\end{lemma}

\begin{proof}
    The proof is similar to~\Cref{lem:prod-cactus-cactus}.
    In this case, the graph corresponding
    to the empty matching is not a cactus
    because $\alpha$ is not.
    All other matchings create at least
    3 edge-disjoint paths between the root
    and the merged vertex.
\end{proof}

\begin{lemma}\label{claim:only-trees}
    For each $\al \in \calA_1 \setminus \calT_1$ and $\beta \in \calA_1$,
    \[
        \bz_{\alpha}(\bA) \cdot \bz_{\beta}(\bA) \in \spn(\bz_{\calA_1\setminus \calT_1})\,.
    \]
\end{lemma}

\begin{proof}
    The non-treelike diagrams $\calA_1 \setminus \calT_1$ can be characterized as:
    \begin{claim}\label{claim:non-treelike}
    Let $\al \in \cA_1$. Then $\al \in \cA_1 \setminus \cT_1$ if and only if one of the following holds:
    \begin{enumerate}[(i)]
    \item there exists a bridge edge which does not have a path to the root using only bridge edges,
    \item or there exist a pair of vertices with three edge-disjoint paths between them.
    \end{enumerate}
    \end{claim}
    \begin{proof}[Proof of~\Cref{claim:non-treelike}]
        It is clear that either structure forbids $\al$ from being treelike. Conversely, if there are at most two edge-disjoint paths between all pairs of vertices, then the bridge edges of $\al$ go between cactuses.
        Then condition {\em (i)} characterizes whether all bridge edges are connected to the root.
    \end{proof}

    Using the claim, if $\al$ has a structure of type {\em (ii)} then this is preserved in the product terms with any $\beta$.
    Suppose then that $\al$ has a structure of type {\em (i)} and call the bridge edge $e$.
    Note that both $\al, \beta$ are connected by definition of $\cA_1$\,.
    If no descendants of $e$ intersect with $\beta$, then the type {\em (i)} structure is preserved.
    Conversely, if any descendant of $e$ intersects with $\beta$,
    then we obtain a new path from the descendant to the root through $\beta$ which is disjoint from the other edges of $\al$.
    Edge $e$ has at least one ancestor which is not a bridge edge, hence there were already two edge-disjoint paths containing this ancestor.
    Together with the new path we obtain a structure of type {\em (ii)}.
    In all cases, the product terms remain in $\calA_1 \setminus \cT_1$.
\end{proof}

\begin{proof}[Proof of~\Cref{lem:homeomorphic}]
    First, if $P$ matches an internal vertex of a hanging cactus, then it creates three edge-disjoint paths from the root to that vertex. These paths cannot be eliminated by merging other vertices, so $\tau_P$ cannot be a cactus.
    Therefore, we may assume without loss of generality that $\tau_1$ and $\tau_2$
    contain no hanging cactuses.

    It is straightforward to check that any homeomorphic matching yields a cactus. We focus on the converse. Specifically, suppose that we are
    given a matching $P$ between the vertices
    of $\tau_1$ and $\tau_2$ such that
    $\tau_P\in \calC_1$. We prove 
    $P\in H(\tau_1, \tau_2)$ by induction
    on $|V(\tau_1)|+|V(\tau_2)|$.

    For the base case, suppose that $\tau_1$
    or $\tau_2$ has only one vertex. Then
    $\tau_P$ can be a cactus only if both
    $\tau_1$ and $\tau_2$ consist of a single
    vertex.

    For the inductive step, let $u^1_1, \ldots, u^1_k$
    be the children of the root of $\tau_1$,
    and let $u^2_1, \ldots, u^2_\ell$ be the children
    of the root of $\tau_2$. A necessary
    condition for $\tau_P$ to be a cactus is
    that $k=\ell$ (and this is also necessary for
    $P$ to be a homeomorphic matching).
    Moreover, after reordering $u^2_1, \ldots, u^2_k$ if necessary, we may assume that
    for all $i\in [k]$, $u^1_i$ and $u^2_i$ lie on the
    same cycle in $\tau_P$, and that these form
    exactly $k$ distinct cycles in $\tau_P$
    incident to the root. 
    
    For each $i\in [k]$
    and $j\in \{1,2\}$,
    let $S^j_i$ denote the non-root 
    vertices of $\tau_j$ that
    are mapped under $P$
    to the same cycle of $\tau_P$ as $u^1_i,u^2_i$.
    
    \begin{claim}\label{claim:inductive-number-mappings}
        For every $i\in [k]$, there is exactly
        one vertex $v^1_i\in S^1_i$ and exactly one vertex 
        $v^2_i\in S^2_i$ that
        are mapped to the same vertex of $\tau_P$.
    \end{claim}

    \begin{proof}
        Since $\tau_1$ and $\tau_2$ are acyclic,
        creating a cycle in $\tau_P$ requires
        identifying two other vertices than the
        root. Conversely, 
        identifying more than one pair of vertices
        would create
        three edge-disjoint paths to the root
        in $\tau_P$, contradicting the fact
        that the
        latter is a cactus.
    \end{proof}

    \begin{claim}\label{claim:allows-induction}
        For each $i\in[k]$ and $j\in\{1,2\}$, every pair in $P$ incident to a vertex
        in the subtree rooted at $v_i^j$ has its other endpoint in the subtree
        rooted at $v_i^{3-j}$.
    \end{claim}

    \begin{proof}
        Suppose for contradiction that there is a pair of $P$ between a vertex $w^1$
        in the subtree rooted at $v_i^1$ and a vertex $w^2$ in the subtree rooted at
        $v_{i'}^{2}$ for some $i'\neq i$.
        Then in $\tau_P$ there are three edge-disjoint paths from the image of
        $v_i^1$ to the root: two lie on the cycle formed by
        $S_i^1\cup S_i^2$, and the third is obtained by concatenating the path from
        $v_i^1$ to $w^1$ with the path from $w^2$
        to the root. This contradicts the fact that $\tau_P$ is a cactus.
    \end{proof}

    By~\Cref{claim:allows-induction}, we may apply the induction hypothesis for each
    $i\in[k]$ to the subtree of $\tau_1$ rooted at $v_i^1$ and the subtree of $\tau_2$
    rooted at $v_i^2$. Thus, the restriction of $P$ to these subtrees is a
    homeomorphic matching. In particular, $v_i^1$ and $v_i^2$ have the same degree.

    \begin{claim}\label{claim:both-core}
        Let $i\in[k]$. Then $v_i^1$ and $v_i^2$ are either both in the core of their
        respective trees or both outside of it. Moreover, for each $j\in\{1,2\}$, no
        vertex in $S_i^j\setminus\{v_i^j\}$ lies in the core of $\tau_j$.
    \end{claim}

    \begin{proof}
        For the first part, since $v_i^1$ and $v_i^2$ have the same degree, they are
        either both in the core or both outside the core.

        For the second part, suppose for contradiction that some
        $w\in S_i^j\setminus\{v_i^j\}$ lies in the core of $\tau_j$. Since $w$ has
        degree greater than $2$, its image in the cactus $\tau_P$ is an articulation
        vertex. Let $\rho$ be a cycle of $\tau_P$ incident to $w$ that is distinct
        from the cycle induced by $S_i^1\cup S_i^2$. Then the two neighbors of $w$ in
        $\rho$ are images of vertices of $\tau_j$. Since $\tau_j$ is acyclic, the
        cycle $\rho$ must contain a vertex $w'$ that is the image of a vertex of
        $\tau_{3-j}$. But then $\tau_P$ contains three edge-disjoint paths from $w$
        to the root: two through the cycle induced by $S_i^1\cup S_i^2$, and a third
        obtained by following $\rho$ from $w$ to $w'$ and then the path from $w'$ to
        the root. This contradicts the fact that $\tau_P$ is a cactus.      
    \end{proof}

    Let $i\in[k]$. For $j\in\{1,2\}$, let $w_i^j$ be the first descendant of
    $u_i^j$ that lies in the core of $\tau_j$. By~\Cref{claim:both-core}, there are
    only two cases:
    
    \begin{enumerate}
        \item Either $v_i^j = w_i^j$ for both $j\in\{1,2\}$. In this case, there are no
        non-core vertices to match on the path from $u_i^j$ to $v_i^j$, so the
        induced matching is empty (and hence trivially order-preserving).
        
        \item Or $v_i^j \neq w_i^j$ for both $j\in\{1,2\}$. In this case, by induction,
        the matching between $v_i^j$ and $w_i^j$ is order-preserving. Matching
        $v_i^1$ to $v_i^2$ and adding the matching from $v_i^j$ to $w_i^j$
        yields an order-preserving matching from $u_i^j$ to $w_i^j$.
    \end{enumerate}

    By induction, the restriction of $P$ induces an isomorphism between the cores of
    $\tau_1$ and $\tau_2$ within each subtree rooted at $v_i^j$. Since there is no
    core vertex on the path from $u_i^j$ to $v_i^j$ by~\Cref{claim:both-core}, these
    local isomorphisms extend to an isomorphism between the cores of $\tau_1$ and
    $\tau_2$ globally. This concludes the proof. 
\end{proof}

\begin{proof}[Proof of~\Cref{lem:product-gaussian}]
    Given $\gamma_1,\ldots,\gamma_\ell\in \calG_1$,
    we can expand
    \[
        \prod_{j=1}^\ell \bz_{\gamma_j} = \sum_P \bz_{\gamma_P}\,,
    \]
    where $P$ ranges over all partitions of
    $V(\gamma_1)\cup \ldots \cup V(\gamma_\ell)$
    such that all roots are in the same block,
    but no two vertices of the same
    $\gamma_i$ are in the same block. Suppose that ${\gamma_P}$ is treelike.

    \begin{claim}\label{claim:hanging-own}
        Every internal vertex of a hanging cactus forms a singleton block.
    \end{claim}

    \begin{proof}
        Suppose for contradiction that an internal vertex $u$ of a hanging cactus in
        $\gamma_1$ lies in the same block as some vertex $v$ of $\gamma_2$. Let $u'$
        be the attachment vertex of the cycle containing $u$. In ${\gamma_P}$,
        there are three edge-disjoint paths between the images of $u$ and~$u'$: two
        are inherited from $\gamma_1$, while the third is obtained by following the
        path in $\tau_2$ from $v$ to the root and then the path in $\tau_1$ from the
        root to~$u'$. This contradicts \Cref{claim:only-trees}, since ${\gamma_P}$
        is assumed to be treelike.
    \end{proof}

    By~\Cref{claim:hanging-own}, we may temporarily delete the hanging cactuses from
    $\gamma_1,\ldots,\gamma_\ell$ and then reattach them in ${\gamma_P}$; this
    does not affect whether ${\gamma_P}$ is treelike. Hence, we may assume
    without loss of generality that none of $\gamma_1,\ldots,\gamma_\ell$ contains a
    hanging cactus.

    \begin{claim}\label{claim:graph-is-matching}
        Let $M$ be the graph on $[\ell]$ with an edge between $i,j\in [\ell]$ if there exist
        $u\in V(\gamma_i)$ and $v\in V(\gamma_j)$ that lie in the same block of $P$.
        Then $M$ is a matching.
    \end{claim}

    \begin{proof}
        Suppose for contradiction that $M$ is not a matching. Then there exist
        non-root vertices $u\in V(\gamma_1)$, $v,v'\in V(\gamma_2)$, and
        $w\in V(\gamma_3)$ such that $u$ and $v$ (resp. $w$ and $v'$) lie in the same block of $P$. Let $v''$ be the lowest common
        ancestor of $v$ and $v'$ in $\gamma_2$. Since $\gamma_2\in\calG_1$, $v''$ is
        not the root of $\gamma_2$. In ${\gamma_P}$, there are three edge-disjoint paths from the image of
        $v''$ to the root: one is the inherited path from $v''$ to the root inside
        $\gamma_2$; the second follows the path in $\gamma_2$ from $v''$ to $v$ and
        then the path in $\gamma_1$ from $u$ to the root; and the third follows the
        path in $\gamma_2$ from $v''$ to $v'$ and then the path in $\gamma_3$ from
        $w$ to the root. This contradicts~\Cref{claim:only-trees}, since
        ${\gamma_P}$ is treelike by assumption.
    \end{proof}

    By~\Cref{claim:graph-is-matching}, it follows that
    \[
        \prod_{j=1}^\ell \bz_{\gamma_j} - \sum_{M\in\calM(\ell)} \sum_{\substack{P_{u,v}\\\forall  uv\in M}} \bz_{\bigoplus_{uv\in M}\gamma_{P_{u,v}} \,\oplus\, \bigoplus_{u\notin M} \gamma_u}\in \spn(\bz_{\calA_1\setminus \calT_1})\,,
    \]
    where for each edge $uv\in M$, the sum over $P_{u,v}$ ranges over all partial
    matchings between $V(\gamma_u)$ and $V(\gamma_v)$ that fix the roots.

    Finally, note that unless $P_{u,v}$ is empty, ${\gamma_{P_{u,v}}}$ cannot
    be a treelike diagram that is not a cactus. Indeed, no vertices in the
    hanging cactuses can be matched; otherwise it would create three edge-disjoint paths. Moreover, if we
    match two tree vertices, that would
    create two edge-disjoint paths to the root, and thus would force the diagram to
    be a cactus.
    Since the grafting of non-treelike diagrams is again non-treelike, the only treelike contributions arise when each factor
    ${\gamma_{P_{u,v}}}$ is a cactus. By~\Cref{lem:homeomorphic}, this forces
    $P_{u,v}$ to be a homeomorphic matching. Hence,
    \[
        \prod_{j=1}^\ell \bz_{\gamma_j} - \sum_{M\in\calM(\ell)} \sum_{\substack{P_{uv}\in H(\gamma_u,\gamma_v)\\\forall uv\in M}} \bz_{\bigoplus_{uv\in M}\gamma_{P_{u,v}} \,\oplus\, \bigoplus_{u\notin M} \gamma_u} \in \spn(\bz_{\calA_1\setminus \calT_1})\,,
    \]
    as desired.
\end{proof}

\subsection{Handling empirical averages}
\label{sec:empirical-averages}

To represent expressions involving empirical averages,
we allow the coefficients in a diagram representation
to be formal polynomials in the quantities 
$\{\langle \bmz_\al(\bA)\rangle : \al \in \cA_1\}$.
Another approach would be to use disconnected diagrams, as in \cite{jones2025fourier}.

\begin{lemma}\label{lem:asymptotic-state-empirical-average}
    Assume that $\bA = \bA^{(n)}$ satisfies the assumptions of \cref{thm:convergence-distribution},
    and furthermore, the traffic distribution concentrates for $\bA$ (\cref{def:traffic-concentration}).
    Let
    \begin{equation}\label{eq:asea-1}
        \bmx = \sum_{\al \in \cA_1} c_\al \bmz_\al(\bA)
    \end{equation}
    for some finitely supported coefficients $(c_\al)_{\al \in \cA_1}$ which are polynomials $c_\al \in \R[\calV]$ with $\calV := \{\langle{\bmz_\al(\bA)}\rangle : \al \in \cA_1\}$. Then,
    \begin{equation}\label{eq:asea-2}
        X := \sum_{\al \in \cA_1} c_\al(\E Z^\infty_{\calA_1}) \cdot Z_\al^\infty
    \end{equation}
    is the asymptotic state of $\bmx$.
    Moreover, if $\bmx_t$ is of the form \cref{eq:asea-1} for any $t \geq 1$ and $X_t$ is correspondingly defined as in \cref{eq:asea-2}, then $(X_t)_{t \geq 1}$ is the asymptotic state of $(\bmx_t)_{t \geq 1}$.
\end{lemma}
\begin{proof}
    For polynomial test functions, the convergence in~\cref{eq:defAsymptoticState} follows
    directly from the concentration of the traffic distribution.
    Moreover,~\Cref{lem:traffic-l2} implies that $\frac 1n z_{\alpha}(\bm A)$ converges 
    in $L^2$ to a deterministic limit for any $\alpha\in \calA_1$. So we can
    combine~\Cref{claim:cvImpliesState} with
    Slutsky's lemma to obtain that~\cref{eq:defAsymptoticState} also holds
    for bounded continuous functions.
\end{proof}

\subsection{\texorpdfstring{Proof of~\Cref{lem:switchInit}}{Proof of the puncturing lemma}}
\label{appendix:initialization}

In this section, we prove~\Cref{lem:switchInit}. We assume throughout that $\bm H$ satisfies the assumptions of~\Cref{prop:hadamard-explicit-state}.
We will prove that $\bm x_t$ and $\bm u_t$ have the same state evolution by relating them to the following intermediate iteration:
\begin{alignat}{2}
    \bm y_0 &\sim \calN(\bm 0, \bm I)\,,\quad &\bm y_t &= \bm H \bm \Pi f_{t-1}(\bm y_{t-1}) - \sum_{s=0}^{t-1} \bm b_{s,t} \cdot (\bm \Pi  f_s(\bm y_s))\quad \forall t\ge 1\,,\label{eq:ampUnpunctured}
\end{alignat}
where $\bm b_{s,t}$ is defined in~\cref{eq:backtrackPunctured}.
Unless specified otherwise, all expectations in this section are taken with respect to both $\bm H$ and $\bm y_0$.

\Cref{thm:full-onsager} does not apply 
to $\bm y_t$ because of the Gaussian initialization $\by_0 \sim \calN(\bzero, \bI)$ (instead of $\by_0 = \bm{1}$).
To analyze this initialization, we extend the class of diagrams to {\em generalized diagrams}, that is, graphs $\alpha = (V(\alpha),E(\alpha))$ together with an additional label $p(v) \in \N$ assigned to each vertex.
The $z$-polynomial associated with a graph $\alpha$ is
\[
    z_\al(\bA, \by_0) \defeq \sum_{i : V(\al) \hookrightarrow [n]} \prod_{\{u,v\} \in E(\al)} \bA[i(u),i(v)] \prod_{v \in V(\al)} \by_0[i(v)]^{p(v)}\,.
\]
The collection of generalized vector diagrams $\calA_1(\bm y_0)$ is defined analogously.
Definitions such as $\mathcal T_1$, $\mathcal G_1$, and $\eqinf$ extend to generalized diagrams by simply ignoring the labels $p(v)$.

As in the proof of~\Cref{thm:orthogonal-amp}, one caveat is that $\bm y_t$ cannot be directly expanded as a linear combination of {\em connected} generalized vector diagrams, because the iteration involves the scalar quantity $\langle f_t(\bm y_t)\rangle$. 
We therefore proceed as in~\Cref{lem:asymptotic-state-empirical-average}, viewing the coefficients in the diagram expansion as formal polynomials in these variables whenever necessary.

Our first observation is that taking
expectation over $\bm y_0$ in the $z$-basis
turns (up to a
scaling factor) a
generalized diagram $\alpha$
into the same diagram $\alpha$ where
the labels are ignored.

\begin{lemma}\label{lem:initialization-expectation}
    For any generalized scalar diagram $\al$ (not necessarily connected) and any $\bmH\in \R^{n\times n}_\sym$,
    \begin{align*}
        \E_{\by_0} z_{\al}(\bmH, \bmy_0) &= \begin{cases}
            \left(\prod_{v \in V(\al)} (p(v) - 1)!!\right) z_{\al}(\bH) & \text{if $p(v)$ is even for every $v\in V(\alpha)$}\\
            0 & \text{otherwise}
        \end{cases}
    \end{align*}
\end{lemma}
    
\begin{proof}        
    In the $z$-basis, all vertices are assigned distinct labels.
    Therefore, we may take the expectation over $\by_0$ separately at each vertex, since the coordinates of $\by_0$ are independent. For each vertex $v$, we have $\E_{Z \sim \calN(0,1)} Z^{p(v)} = (p(v)-1)!!$ if $p(v)$ is even or $0$ if $p(v)$ is odd.
\end{proof}

Next, we describe structural properties of the labels $p(v)$ appearing in the diagram expansion of the iterates of the AMP iteration~\cref{eq:ampUnpunctured}.

\begin{lemma}\label{claim:initialization-structure}
    We have
    $\bmy_t \eqinf \sum_{\tau} c_\tau \bmz_\tau(\bH, \bm y_0)$ and $f_t(\bm y_t) \eqinf \sum_{\tau} c'_\tau \bmz_\tau(\bH, \bm y_0)$,
    where $c_\tau$ and $c'_\tau$ are supported on (generalized) treelike diagrams $\tau$ such that,
    for all $v \in V(\tau)$: 
    \[
        p(v) = \begin{cases}
            1 & \text{if $v$ is a leaf vertex of $\tau$}\\
            0\text{ or } 2 & \text{if $v$ is in a hanging cactus}\\
            0 & \text{otherwise}\\
        \end{cases}\,.
    \] 
    Leaves of treelike diagrams are defined after
    removing hanging cactuses.
\end{lemma}

\begin{proof}
    First, the proof of~\Cref{lem:memory-formula} still goes through with the nonlinearities $g_t(y) = f_t(y) - \langle f_t(\bm y_t)\rangle$, after extending the coefficient field from $\R$ to the ring of formal polynomials in $\{\langle \bm z_{\alpha}(\bm A)\rangle : \alpha\in \calA_1\}$. Therefore, we obtain
    \begin{align}
         \bm y_t \eqinf \sum_{s=0}^{t-1} \bm B_{s,t} (\bm \Pi  f_s(\bm y_s))^{\neq 1}\,.\label{eq:extensionRandomInit}
    \end{align}
    We now argue by induction on $t$. The base case is $f_0(\bm y_0)=\bm y_0$ which is the singleton with $p(v) = 1$.
    
    Now, suppose that the claim holds for $\bm y_t$.
    The treelike diagrams appearing in $f_t(\bm y_t)$ are obtained by considering all possible products of treelike diagrams $\gam_1, \dots, \gam_\ell \in \calG_1 \cup \cC_1$ appearing in $\bm y_t$. 
    By~\Cref{lem:product-gaussian}, each such product can be written as a sum over matchings among the $\gam_i$, where each $\gam_i$ is either paired into a cactus or does not intersect any other $\gam_j$. 
    In the second case, the values $p(v)$ within $\gam_i$ are unchanged.
    In the first case, the values $p(v)$ at the leaves are updated from 1 to 2, while all other values $p(v)$ within $\gam_i$ remain unchanged.

    Moreover, no non-trivial intersection between $\bm B_{s,t}$ and $(\bm \Pi f_s(\bm y_s))^{\neq 1}$ can produce a treelike diagram. Hence, the decomposition of $\bm y_{t+1}$ given by~\cref{eq:extensionRandomInit}, together with the induction hypothesis, shows that in every treelike diagram appearing in $\bm y_{t+1}$, the condition on $p(v)$ is inherited directly from the corresponding property of $f_s(\bm y_s)$.
    This completes the induction.
\end{proof}

\begin{lemma}\label{lem:puncToUnpunc}
    For any $t\ge 1$ and any polynomial $\varphi:\R^t \to \R$,
    \begin{align*}
        \lim_{n\to\infty} \E_{\bm H,\bm x_0} \langle \varphi(\bm x_1, \ldots, \bm x_t)\rangle &= \lim_{n\to\infty} \E_{\bm H,\bm y_0} \langle \varphi(\bm y_1, \ldots, \bm y_t)\rangle\,.
    \end{align*}
\end{lemma}

\begin{proof}[Proof of~\Cref{lem:puncToUnpunc}]
    An iteration involving $\bm A$ can be reduced to one involving $\bm H$ by expanding
    \begin{align}
        \bm Af_{t}(\bm y_{t}) = \bm H f_{t}(\bm y_{t}) - \langle f_{t}(\bm y_{t})\rangle \bm H \bm 1 - \langle \bm H \bm \Pi f_{t}(\bm y_{t})\rangle \bm 1\,.\label{eq:decompositionPunctured}
    \end{align}
    Set $\delta_t\defeq \langle \bm H \bm \Pi f_{t}({\bm y}_{t})\rangle$ and $m_t \defeq \langle f_t(\bm y_t)\rangle$. 
    We first compare $\bm y_t$ with the following modified iteration, which differs from $\bm x_t$ only in the formula for the Onsager correction term:
    \[
        \tilde {\bm y}_0 = \bm y_0\,,\quad \tilde {\bm y}_t = \bm A f_{t-1}(\tilde {\bm y}_{t-1}) - \sum_{s=0}^{t-1} {\bm b}_{s,t} \cdot (\bm \Pi f_s(\tilde {\bm y}_s))\,,
    \]
    where $\bm b_{s,t}$ is defined in~\cref{eq:backtrackPunctured}.

    \begin{claim}\label{claim:divisibility}
        For any $t\in \N$, we have
        \begin{align}
            \tilde {\bm y}_t-\bm y_t = \sum_{\alpha} c_{t,\alpha}(\delta_0, \ldots, \delta_{t-1}, m_0, \ldots, m_{t-1}) \bm z_{\alpha}(\bm H, \bm y_0)\,,\label{eq:idealDifference}
        \end{align}
        where the sum runs over finitely
        many generalized vector diagrams, and each $c_{t,\alpha}$ is a polynomial in $\delta_0, \ldots, \delta_{t-1}, m_0, \ldots, m_{t-1}$ that is divisible by $\delta_s$ for some $s \in \{0, \ldots, t-1\}$.
    \end{claim}

    \begin{proof}[Proof of~\Cref{claim:divisibility}]
        We argue by induction on $t$.
        For $t=0$, $\bm y_0-\tilde{\bm y}_0=0$, establishing the base case. Let $t\ge 1$ and suppose that~\cref{eq:idealDifference}
        holds for all $s< t$. 
        First, one easily verifies from the induction hypothesis that the same property~\cref{eq:idealDifference} holds for $\bm \Delta_s \defeq f_s(\tilde{\bm y}_s) - f_s(\bm y_s)$ for every $s < t$.
        By~\cref{eq:decompositionPunctured}, we can then write
        \[
            \bm A f_{t-1}(\tilde {\bm y}_{t-1}) - \bm H \bm \Pi f_{t-1}(\bm y_{t-1}) = \bm H \bm \Pi \bm \Delta_{t-1} - \delta_{t-1} \bm 1 - \langle \bm H \bm \Pi \bm \Delta_{t-1}\rangle \bm 1\,,
        \]
        and each of the three terms on the right-hand side satisfies a decomposition of the form~\cref{eq:idealDifference} by the induction hypothesis.
        Finally,  the correction terms differ by
        \[
            \sum_{s=0}^{t-1} \bm b_{s,t} \cdot (\bm \Pi f_s(\tilde {\bm y}_s)) - \sum_{s=0}^{t-1} \bm b_{s,t} \cdot (\bm \Pi f_s({\bm y}_s)) = \sum_{s=0}^{t-1} \bm b_{s,t} \cdot \bm \Pi \bm \Delta_s\,,
        \]
        which again satisfies the property~\cref{eq:idealDifference} by the induction hypothesis. Combining these observations, we conclude that $\tilde{\bm y}_t-\bm y_t$ satisfies~\cref{eq:idealDifference}, completing the induction.
    \end{proof}

    Next, fix any polynomial $\varphi:\R^t\to \R$. By~\Cref{claim:divisibility}, we have
    \begin{equation}
        \langle \varphi (\bm y_1, \ldots, \bm y_t)\rangle - \langle \varphi(\tilde {\bm y}_1, \ldots, \tilde{\bm y}_t)\rangle = \sum_{\alpha} c_{\alpha}(\delta_0, \ldots, \delta_{t-1},m_0,\ldots,m_{t-1}) \langle \bm z_{\alpha}(\bm H, \bm y_0)\rangle\,,\label{eq:divisibilityFinal}
    \end{equation}
    where the sum runs over finitely many generalized scalar diagrams, and each $c_\alpha$ is a polynomial in $\delta_0,\ldots,\delta_{t-1},m_0,\ldots,m_{t-1}$ that is divisible by some $\delta_s$. 
    In the remainder of the proof, we show that each term on the right-hand side of~\cref{eq:divisibilityFinal} converges to $0$ in expectation.
    The reason is that each coefficient $c_\alpha$ contains a factor $\delta_s$, and these quantities converge to $0$ in $L^2$:
    \begin{claim}\label{lem:correctionNegligible}
        $\langle \bm H \bm \Pi f_t(\bm y_{t})\rangle \overset{L^2}\longrightarrow 0$ for any $t\ge 1$.
    \end{claim}
    
    \begin{proof}[Proof of~\Cref{lem:correctionNegligible}]
        The claim is equivalent to the statement that $\frac{1}{n^2}\E \langle \bm H \bm 1, \bm \Pi f_t(\bm y_t)\rangle^2$ converges to $0$.
        This quantity can be expanded as a linear combination of terms of the form
        \[\frac{1}{n^2} \E \left[z_\alpha (\bm H, \bm y_0) z_\beta (\bm H, \bm y_0)\right]\,
        \]
        where $\alpha,\beta\in\calA$ both belong to the support of the expansion of
        $\langle \bm H \bm 1, \bm \Pi f_t(\bm y_t)\rangle$.
        As in the proof of~\Cref{lemma:full-factorization},
        \[
            \frac 1 {n^2} \E \left[z_\alpha (\bm H, \bm y_0) z_\beta (\bm H, \bm y_0)\right] = \frac 1{n^2} \E \left[z_{\alpha\sqcup \beta} (\bm H, \bm y_0)\right]+o(1)\,.
        \]
        Indeed, each identification of vertices across the two copies yields a connected diagram whose expectation, after normalization by $1/n^2$, converges to $0$ by the existence of the traffic distribution. This holds for every realization of $\bm y_0$, and therefore also after taking expectation over $\bm y_0$.
        
        Taking expectation over $\bm y_0$ and using~\Cref{lem:initialization-expectation}, each term either vanishes or becomes a constant multiple of $\E_{\bm H} \left[z_{\alpha\sqcup \beta} (\bm H)\right]$, where $\alpha\sqcup\beta$ is viewed as an ordinary scalar diagram obtained by ignoring the labels $p(v)$. By~\Cref{lemma:full-factorization} and the strong cactus property, the only terms that contribute to the limit are those for which both $\alpha$ and $\beta$ are cactuses.
        Viewing $\bm H\bm 1$ as a rooted tree with one edge, the cactuses in the $z$-basis expansion of $\langle \bm H \bm 1, \bm \Pi f_t(\bm y_t)\rangle$ arise when the child of $\bm H\bm 1$ is merged with a leaf of a diagram from $\bm \Pi f_t(\bm y_t)$. By~\Cref{claim:initialization-structure}, such leaves satisfy $p(v)=1$. Applying~\Cref{lem:initialization-expectation} once again, we find that each of these cactus terms has expectation $0$ over $\bm y_0$, which concludes the proof.
    \end{proof}

    After taking expectation over $\bm H$ and $\bm y_0$, any monomial appearing on the right-hand side of~\cref{eq:divisibilityFinal} has the following form for some $p_i,q_i\in \N$:
    \begin{equation}
        \E \left[\delta_s \prod_{i=0}^{t-1} \delta_i^{p_i} m_i^{q_i} \langle z_{\alpha}(\bm H, \bm z_0)\rangle\right] \le \left(\E \delta_s^2\right)^{\frac 12} \cdot \left(\E \left[\prod_{i=0}^{t-1} \delta_i^{2p_i} m_i^{2q_i} \langle z_{\alpha}(\bm H, \bm z_0)\rangle^2\right]\right)^{\frac 12}\,,\label{eq:anyTerm}
    \end{equation}
    where the inequality follows from Cauchy-Schwarz.
    
    The first factor on the right-hand side
    of~\cref{eq:anyTerm} converges to $0$ as $n\to\infty$
    by~\Cref{lem:correctionNegligible}.
    The second factor can be expanded in the $z$-basis as a finite linear combination of products of generalized $z$-diagrams. 
    Taking expectation over $\bm y_0$ and using~\Cref{lem:initialization-expectation}, each such term either vanishes or becomes a constant multiple of a product of ordinary scalar $z$-diagrams.
    By~\Cref{lemma:full-factorization}, the normalized expectation of each of these terms has a finite limit as $n\to\infty$ and in particular, is uniformly bounded in $n$. 
    Therefore, the second factor on the right-hand side of~\cref{eq:anyTerm} is bounded, and hence the right-hand side of~\cref{eq:divisibilityFinal} converges to $0$ in expectation.
    In summary, we have shown:
    \begin{equation}
        \lim_{n\to\infty} \E \langle \varphi(\bm y_1, \ldots, \bm y_t)\rangle - \E \langle \varphi(\tilde{\bm y}_1, \ldots, \tilde{\bm y}_t)\rangle = 0\,.\label{eq:semiConclusion}
    \end{equation}
    Finally, as in the proof of~\Cref{thm:orthogonal-amp}, we may use the traffic concentration property (\Cref{lemma:full-factorization}) to replace ${\bm b}_{s,t}$ by $\kappa_{t-s} \prod_{s<r<t} \langle f_r'(\tilde{\bm y}_r)\rangle$ in the iteration for $\tilde{\bm y}_t$ without affecting the asymptotic state. This yields
    \[
        \lim_{n\to\infty} \E \langle \varphi (\bm x_1, \ldots, \bm x_t)\rangle - \E \langle \varphi(\tilde {\bm y}_1, \ldots, \tilde{\bm y}_t)\rangle = 0\,.
    \]
    Combining this with~\cref{eq:semiConclusion} completes the proof.
\end{proof}

\begin{proof}[Proof of~\Cref{lem:switchInit}]
    First, we can replace every
    occurrence of $\bm \Pi f_0(\bm y_0)$
    in $\bm y_t$ by $f_0(\bm y_0)$ using
    the traffic concentration property,
    since $\langle f_0(\bm y_0)\rangle = \langle \bm y_0\rangle$, which converges to $0$ as $n\to\infty$.
    After this update, the iterates $\bm y_t$ and $\bm u_t$ have the same generalized diagram expansion as functions of their initializations $\bm y_0$ and $\bm u_0$. 
    Note that this expansion is formal in the variables $\langle \bm z_{\alpha}(\bm A)\rangle$ for $\alpha\in \calA_1(\bm y_0)$, because the puncturing operation introduces terms of the form $\langle f_t(\bm y_t)\rangle$.
    
    First, by the strong cactus property and~\Cref{lem:initialization-expectation}, all non-cactus terms in the generalized diagram expansions of $\langle\varphi(\bm y_1,\ldots,\bm y_t)\rangle$ and $\langle\varphi(\bm u_1,\ldots,\bm u_t)\rangle$ converge to $0$ in expectation.
    Second, using~\Cref{claim:initialization-structure} and extending the same argument one further step to $\varphi$,
    all cactus diagrams in the generalized diagram expansions of $\bm y_t$ and $\bm u_t$ satisfy $p(v) \in \{0,2\}$, since they have no non-root leaves (the iterates for $t\ge 1$ have no singleton component).
    Therefore, by~\Cref{lem:initialization-expectation}, the expectations of the cactus terms remain unchanged as $n\to\infty$ if we replace $\bm y_0$ by $\bm u_0=\bm 1$.

    Combining these facts with the traffic concentration property for $\bm H$ (\Cref{lemma:full-factorization}) shows that $\langle \varphi(\bm y_1, \ldots, \bm y_t)\rangle - \langle \varphi(\bm u_1, \ldots, \bm u_t)\rangle$ converges to $0$ in expectation, as desired.
\end{proof}

\subsection{\texorpdfstring{Proof of~\Cref{lem:block-matrix-cactus-limits}}{Proof of block matrix limits}}
\label{app:block-matrix}

In this section, we prove the auxiliary lemmas for block matrices.

\begin{definition}
    Let $\al \in \cC_1$ be a cactus diagram.
    For a {coloring} $\chi:V(\al)\to [q]$
    of the vertices of $\al$ with $q$ colors, we say that $\chi$ is {valid} if for every cycle $\rho =(u_1, \ldots, u_k, u_1)\in \cyc(\al)$, there exist $r,c\in [q]$ such that $\chi(u_i) = r$ when $i$ is even and $\chi(u_i) = c$ when $i$ is odd, with $r=c$ if $k$ is odd. 
    We write
    $\chi(\rho)=\{r,c\}$ in this case.
\end{definition}

Our main diagrammatic calculation for block models is the following, which gives the traffic distribution on each block:

\begin{lemma}\label{claim:insideBlock}
    Let $\bA$ be as in the setting of \cref{lem:block-matrix-cactus-limits}.
    Then for all $\al \in \cA_1$ and $r \in [q]$:
    \[
        \frac qn \sum_{\substack{i\in [n]\\\block(i)=r}} \E \bm z_\sigma(\bm A)[i] \underset{n\to\infty}{\longrightarrow} \begin{cases}\displaystyle
        \sum_{\substack{\chi:V(\al)\to [q]\\\chi \textnormal{ valid}\\\chi(\textnormal{root})=r}} \prod_{\rho \in \cyc(\al)} \kappa_{|\rho|}^{\chi(\rho)} & \text{ if }\al \in \cC_1\\
        0 &\text{ if } \al \in \cA_1 \setminus \cC_1
        \end{cases}
    \]
\end{lemma}
\begin{proof}
    We partition the sum defining $\bmz_\al(\bA)$ according to the block of each vertex, as in the proof of~\Cref{prop:block-matrix-strong-cactus}:
    \begin{align*}
        \bmz_\al(\bA) &= \sum_{\chi : V(\al) \setminus \{\text{root}\} \to [q]} \bmz_{\al_\chi}((\bA_{r,c})_{r,c \in [q]})\\
        \intertext{where $\al_\chi$ is a diagram whose edges are colored by the matrices $\bA_{r,c}$. For a fixed $r' \in [q]$, we get}
        \frac qn \sum_{\substack{i \in [n]\\ \block(i) = r'}} \E\bmz_\al(\bA)[i] &= \sum_{\chi : V(\al) \setminus \{\text{root}\} \to [q]} \frac qn \sum_{\substack{i \in [n]\\ \block(i) = r'}}\E\bmz_{\al_\chi}((\bA_{r,c})_{r,c \in [q]})[i]\,.
    \end{align*}
    By the definition of traffic independence (\Cref{def:traffic independence}), the limit as $n\to\infty$ exists for each term indexed by $\chi$ on the right-hand side.
    Hence, the limit of the left-hand side also exists.
    Arguing as in the proof of~\Cref{prop:block-matrix-strong-cactus}, we find that the limit is zero for all $\al \in \cA_1 \setminus \cC_1$.
    
    For cactus diagrams $\al \in \cC_1$, asymptotic traffic independence and the strong factorizing cactus property of the individual blocks imply that the only nonzero contributions arise when every cycle of $\al_\chi$ is monochromatic, in the sense that it involves only a single matrix $\bm A_{r,c}$. 
    This
    happens if and only if $\chi$ is a valid coloring, in which case the corresponding term contributes asymptotically
    \[\prod_{\rho\in \cyc(\sigma)} \kappa_{|\rho|}^{\chi(\rho)}\,,\] as desired.
\end{proof}

\begin{proof}[Proof of \cref{lem:block-matrix-cactus-limits}.]    
    Let $\calL_{\cC_1}(r)$ denote the values from \cref{claim:insideBlock}:
    \[
        \mathcal L_{\sigma}(r) := \sum_{\substack{\chi:V(\sig)\to [q]\\\chi \textnormal{ valid}\\\chi(\textnormal{root})=r}} \prod_{\rho \in \cyc(\sig)} \kappa_{|\rho|}^{\chi(\rho)}\,.
    \]
    We first prove that all joint moments
    of
    $\bm z_{\calC_1}(\bm A)[i]$ conditioned on $\block(i)=r$ converge to the moments
    of the deterministic sequence $Z^\infty_{\calC_1}(r)$.
    For any $\sig_1, \ldots, \sig_k\in \cC_1$, we have
    \begin{align*}
        &\frac qn \sum_{\substack{i\in[n]\\\block(i) = r}} \E \left[\bm z_{\sigma_1}(\bA)[i]\cdots \bm z_{\sigma_k}(\bA)[i]\right]\\
        =\;&\frac qn \sum_{\substack{i\in [n]\\\block(i) = r}} \E \bm z_{\sigma_1\oplus \ldots \oplus\sigma_k}(\bm A)[i] + o(1)\tag*{\textnormal{(by~\Cref{lem:prod-cactus-cactus,claim:insideBlock})}}\\
        =\;& \mathcal L_{\sig_1 \oplus \ldots \oplus \sig_k}(r) + o(1) = \prod_{j = 1}^k \mathcal L_{\sigma_j}(r) + o(1)\,. \tag*{\textnormal{(by \cref{claim:insideBlock})}}
    \end{align*}

    So it remains to prove that $\mathcal L_{\mathcal C_1}(r)$ satisfies the same recursion as $Z^\infty_{\calC_1}(r)$.
    First, one readily checks that $\mathcal L_{\textnormal{singleton}}(r) = 1$, as in {\em (i)}. Next, suppose that $\sigma$ is rooted at a vertex of degree 2, and let $\ell$ and $\sigma_1, \ldots, \sigma_{\ell-1}$ be as in {\em (ii)}. Then, by decomposing according to the value of cycle containing the root, we have
    \[
        \mathcal L_\sigma(r) =
        \begin{cases}
            \displaystyle
            \sum_{c\in [q]} \left[\kappa_\ell^{\{r,c\}}
            \prod_{\substack{k=2\\k\textnormal{ odd}}}^{\ell} \mathcal L_{\sigma_k}(r) \prod_{\substack{k=2\\k\textnormal{ even}}}^{\ell} \mathcal L_{\sigma_k}(c)\right]
            & \text{if $\ell$ is even} \\
            \displaystyle
            \kappa_\ell^{\{r,r\}}
            \prod_{k=2}^{\ell} \mathcal L_{\sigma_k}(r)
            & \text{if $\ell$ is odd}
        \end{cases}
    \]
    just like the recursion in {\em (ii)}. Similarly, {\em (iii)} follows from the
    fact that the definition of $\mathcal L_\sigma(r)$
    factorizes over graftings at the root.
    Together, this shows that $\mathcal L_{\calC_1}(r) = Z^\infty_{\calC_1}(r)$.

    Since the limit is deterministic, 
    we have shown that
    conditionally on $\block(i)=r$,
    $\bm z_{\calC_1}(\bm A)[i]$ converges
    to $Z^\infty_{\calC_1}(r)$ in $L^2$.
    Since $\block(i)\sim \Unif([q])$, it follows that $(\block(i),\bm z_{\calC_1}(\bm A)[i])$ converges in distribution
    to $(R, Z^\infty_{\calC_1}(R))$, where $R\sim \Unif([q])$.
\end{proof}

\end{document}